\documentclass[a4paper,11pt,reqno]{amsart}
\usepackage[top=2.5cm,bottom=2.5cm,left=2.5cm,right=2.5cm]{geometry}
\usepackage[numbers,sort&compress]{natbib}
\usepackage[colorlinks]{hyperref}
\usepackage[utf8]{inputenc}
\usepackage{amsthm,amsmath,amssymb,dsfont}
\usepackage{mathrsfs,amsfonts,functan,extarrows,mathtools,stmaryrd}
\usepackage{xcolor}
\usepackage{bm,scalerel}
\usepackage{enumitem}
\usepackage{calligra}
\newtheorem{theorem}{Theorem}[section]

\numberwithin{equation}{section}
\allowdisplaybreaks
\usepackage[capitalize,nameinlink]{cleveref}
\usepackage{aliascnt}
  \newaliascnt{definition}{theorem}
  \aliascntresetthe{definition}
  \Crefname{definition}{Definition}{Definitions}
  \newaliascnt{lemma}{theorem}
  \newtheorem{lemma}[lemma]{Lemma}
  \aliascntresetthe{lemma}
  \Crefname{lemma}{Lemma}{Lemmas}
  \newaliascnt{proposition}{theorem}
  \newtheorem{proposition}[proposition]{Proposition} %
  \aliascntresetthe{proposition}
  \Crefname{proposition}{Proposition}{Propositions}
  \newaliascnt{remark}{theorem}
  \newtheorem{remark}[remark]{Remark} %
  \aliascntresetthe{remark}
  \Crefname{remark}{Remark}{Remarks} %
  \crefname{equation}{eq.}{eqs.}
\def\d{\mathrm{d}}
\makeatletter
\def\subsection{\@startsection{subsection}{2}%
     \z@{.5\linespacing\@plus.7\linespacing}{.5\linespacing}%
     {\normalfont\bfseries}}
\makeatother

\begin{document}

\title[Non-isothermal Maxwell-Stefan asymptotics]{The non-isothermal Maxwell-Stefan asymptotics of the multi-species  Boltzmann equations}

\author[Xinqiu Chen]{Xinqiu Chen}
\address[Xinqiu Chen]{\newline School of Mathematics and Statistics, Wuhan University, Wuhan, 430072, P. R. China}
\email{xinqiu\_chen@whu.edu.cn}

\author[Ning Jiang]{Ning Jiang}
\address[Ning Jiang]{\newline School of Mathematics and Statistics, Wuhan University, Wuhan, 430072, P. R. China}
\email{njiang@whu.edu.cn}

\author[Yi-Long Luo]{Yi-Long Luo}
\address[Yi-Long Luo]
{\newline School of Mathematics, Hunan University, Changsha, 410082, P. R. China}
\email{luoylmath@hnu.edu.cn}

\thanks{\today}

\maketitle

\begin{abstract}
We study the convergence from the multi-species Boltzmann equations to the non-isothermal Maxwell-Stefan system. The global-in-time well-posedness of the Maxwell-Stefan system is first established. The solution is utilized as the fluid quantities to construct a local Maxwellian vector. The Maxwell-Stefan system can be derived from the multi-species Boltzmann equations under diffusive scaling by adding a relation on the total concentration. Different with the classical hydrodynamic limits of the Boltzmann equations, the Maxwellian based on the Maxwell-Stefan system is not a local equilibrium for the mixtures due to cross-interactions. A local coercivity property for the operator linearized around the local Maxwellian is established, based on the explicit spectral gap of the operator linearized around 
the global equilibrium. The global-in-time solution to the multi-species 
Boltzmann equations uniform in Knudsen number $\varepsilon$ is established in this scaling, thus the first non-isothermal Maxwell-Stefan asymptotics is rigorously justified. This generalizes
Bondesan and Briant’s work \cite{briant2021stability} from isothermal to non-isothermal case. \\

\noindent\textsc{Keywords.} Non-isothermal Maxwell-Stefan equations; Multi-species Boltzmann equations; Knudsen number; Asymptotical behaviors \\

\noindent\textsc{AMS subject classifications.} 35Q61; 82C40

\end{abstract}

%\markboth{From Multi-species Boltzmann to non-isothermal Maxwell-Stefan}{X. Chen, N. Jiang and Y.-L. Luo}

\tableofcontents
\section{Introduction}
\subsection{The Multi-species Boltzmann Equations}
The multi-species Boltzmann equations \cite{Sirovich1962kinetic,Morse1964kinetic}
represent an extension of the standard Boltzmann equation for 
mono-species \cite{Cercignani1988applications,Cercignani1994dilute,Villani2002handbook}. They describe the evolution of 
a dilute gas composed of different species. More precisely, we consider a monatomic inert gaseous mixture composed of $N \geq 2$ different species, having atomic masses $(m_i)_{1 \leq i \leq N}$, evolving on $\mathbb{T}^3$.  
We use a vector-valued distribution function $\mathbf{F} = (F_1 , \cdots , F_N)$ to model this evolution, where $F_i=F_i(t,x,v)$ describes  the evolution of the $i$-th specie of mixtures satisfying the following equation: for any $1\leq i\leq N$, 
\begin{equation}
\partial_t F_i + v \cdot \nabla_x F_i = \mathit{Q}_i (\mathbf{F} , \mathbf{F}) ,  \label{def of Bz-i equation}
\end{equation}
with initial data 
\[ F_i(0,x,v) = F_i^{in} (x,v), \quad (x,v) \in \mathbb{T}^3 \times \mathbb{R}^3.
\]
The Boltzmann operator $\mathbf{Q}=(\mathit{Q}_1,\dots,\mathit{Q}_N)$ is given for any $1\leq i\leq N$,
\begin{equation}
    \begin{aligned}
    \mathit{Q}_i (\mathbf{F} , \mathbf{F}) (v) & = \sum_{j=1}^N \mathit{Q}_{ij} (F_i , F_j) (v)\\
    &=\sum_{j=1}^N \int_{\mathbb{R}^3 \times \mathbb{S}^2 } \mathit{B}_{ij} (|v-v_\ast| , \cos{\theta}) [F_i^\prime F_j^{\prime\ast} 
    -F_i F_j^\ast]  \,\d \sigma \d v_\ast ,\quad \cos{\theta} = \frac{(v - v_\ast) \cdot \sigma}{|v - v_\ast|},
    \end{aligned}
\end{equation}
where $\mathit{Q}_{ij}$ models interactions between particles of species $i$ and species $j$. The notations are defined as follows: 
$F_i^\prime = F_i(v^\prime)$, $F_i = F_i(v)$, $F_j^{\prime\ast}=F_j(v_\ast^\prime)$, and $F_j^{\ast} = F_j(v_\ast)$, where 
$v , v_\ast$ represent precollision velocities, while $v^\prime , v_\ast^\prime$ indicate post-collision velocities. 
We focus on binary and elastic collisions, ensuring that the collision velocities $v , v_\ast , v^\prime$ and $v_\ast^\prime$ 
satisfy the conservation of momentum and kinetic energy
\begin{equation}
 m_i v + m_j v_\ast = m_i v^\prime + m_j v_\ast^\prime ,\quad
        \frac{1}{2} m_i |v|^2 + \frac{1}{2} m_j |v_\ast|^2 = \frac{1}{2} m_i |v^\prime|^2 + \frac{1}{2} m_j |v_\ast^\prime|^2 .
\end{equation}
The expression of post-collision velocities are
 \begin{equation}
    \left\{
    \begin{aligned}
        v^\prime & = \frac{m_i v + m_j v_\ast}{m_i + m_j} + \frac{m_j}{m_i + m_j} |v - v_\ast| \sigma ,\\
        v_\ast^\prime & = \frac{m_i v + m_j v_\ast}{m_i + m_j} - \frac{m_i}{m_i + m_j} |v - v_\ast| \sigma.
    \end{aligned}
    \label{post-collision vel}
    \right. 
    \end{equation}

Moreover, the collision kernels $( \mathit{B}_{ij} )_{1 \leq i,j \leq N}$ are nonnegative and model the physics of the binary collisions between particles. We make the following assumptions on each $\mathit{B}_{ij}$ (as in \cite{briant2021stability}).
\begin{enumerate}[label=\textbf{(H\arabic*)}]
    \item It satisfies a symmetry property with respect to the interchange of both species indices $i,j$
\[ \mathit{B}_{ij} (|v - v_\ast| , \cos{\theta} ) = \mathit{B}_{ji} (|v - v_\ast| , \cos{\theta} ),\quad \forall v , v_\ast \in 
\mathbb{R}^3, \quad \forall \theta \in [0,\pi].
\]
    \item It decomposes into the product of a kinetic part $\Phi_{ij}\geq0$ and an angular part $\mathit{b}_{ij}>0$, namely
\[ \mathit{B}_{ij} (|v - v_\ast| , \cos{\theta} )=\Phi_{ij} (|v - v_\ast| ) \mathit{b}_{ij} (\cos{\theta}), \quad \forall v , 
v_\ast \in \mathbb{R}^3, \quad \forall \theta \in [0,\pi].  \]
    \item The kinetic part has the form
\[ \Phi_{ij} (|v - v_\ast| )=C_{ij}^{\Phi} |v - v_\ast|^\gamma, \quad C_{ij}^\Phi > 0, \quad \gamma \in (-3,1], \quad \forall v , 
v_\ast \in \mathbb{R}^3,  \]
and is usually classified into three classes: hard potential, when $\gamma>0$, soft potential, when $\gamma<0$, and Maxwellian potential, 
when $\gamma=0$.
    \item For the angular part, we consider a strong form of Grad's angular cutoff \cite{Grad1958bookprinciples}. 
We suppose that $\mathit{b}_{ij}$ is $C^1$, and there exists a constant $C>0$ such that 
\[ \mathit{b}_{ij} (\cos{ \theta}), \quad \mathit{b}^\prime_{ij} (\cos{ \theta}) \leq C,\quad \forall \theta \in [0,\pi].
\]
Furthermore, we assume that 
\[ \min_{1\leq i \leq N} \inf_{ \sigma_1, \sigma_2 \in \mathbb{S}^2} \int_{\mathbb{S}^2} \min\{\mathit{b}_{ii} (\sigma_1 \cdot \sigma_3), 
\mathit{b}_{ii}( \sigma_2 \cdot \sigma_3)\} \,\d \sigma_3 > 0.
\]
\end{enumerate}
The above hypotheses on the collision kernels are standard in both mono-species cases \cite{mouhot2006explicit,Baranger2005explicit} 
and multi-species cases \cite{mouhot2016SIAMJMAhypercoercivity,briant2016ARMAglobal, briant2021stability}. In these assumptions, 
(H1) indicates that the collisions are micro-reversible, (H2) is commonly used in many physical models, and (H3) applies to 
collision kernels derived from interaction potentials that behave like power laws. The positivity assumption on the infimum of 
the integral appearing in (H4) is satisfied by most physical models and is necessary to obtain an explicit spectral gap property for 
mono-species case \cite{mouhot2006explicit,Baranger2005explicit} and multi-species case 
\cite{briant2016ARMAglobal,mouhot2016SIAMJMAhypercoercivity}.

By using changes of variables $(v , v_\ast ) \rightarrow (v^\prime , v_\ast^\prime )$ and $(v , v_\ast) \rightarrow (v_\ast , v)$, together with the symmetries of the collision kernels 
(see \cite{mouhot2016SIAMJMAhypercoercivity,Boudin2015AAMdiffusion,Desvillettes2005polytropic} for more details), we recover the weak 
forms
\begin{equation}
  \int_{\mathbb{R}^3} \mathit{Q}_{ij} (f , g)(v) \psi(v) \,\d v = \int_{\mathbb{R}^6 \times \mathbb{S}^2} \mathit{B}_{ij}(v , v_\ast , 
  \sigma) f(v) g(v_\ast) \big[ \psi (v^\prime) - \psi(v) \big ] \, \d \sigma \d v_\ast \d v,
  \label{weak form-1}
\end{equation}
and 
\begin{equation}
  \begin{aligned}
  & \int_{\mathbb{R}^3} \mathit{Q}_{ij} (f , g) (v) \psi (v) \,\d v + \int_{\mathbb{R}^3} \mathit{Q}_{ji} (g , f) (v) \phi (v) 
  \,\d v \\
  & = - \frac{1}{2} \int_{\mathbb{R}^6 \times \mathbb{S}^2} \mathit{B}_{ij} (v , v_\ast,\sigma) \big[ f (v^\prime) g (v_\ast^\prime)
  - f (v) g (v_\ast )  \big]  \big[ \psi (v^\prime ) - \psi (v) + \phi (v_\ast^\prime) - \phi (v_\ast ) \big] \, \d \sigma \d v_\ast \d v.
\label{weak form-2}  
\end{aligned}
\end{equation}
Thus, the operator $\mathbf{Q} = (\mathit{Q}_1 , \dots , \mathit{Q}_N)$ has collision invariants for multi-species version 
if and only if 
$\bm { \psi } =(\psi_1 , \dots , \psi_N) \in \mathrm{Span} \left\{ \mathbf{e}^{(1)} , \dots, \mathbf{e}^{(N)} , v_1 \mathbf{m} , v_2 \mathbf{m} , 
v_3\mathbf{m}, |v|^2 \mathbf{m} \right\}$ with notation $\mathbf{m}=(m_1, \dots, m_N)$. 
Moreover, the operator $\mathbf{Q}$ satisfies a multi-species version of the classical $\mathit{H}\mbox{-}$Theorem \cite{Desvillettes2005polytropic}, 
which implies that any local equilibrium takes the form of a local Maxwellian. Specifically, the local Maxwellian vector 
$\mathbf{M} = (M_1 , \dots , M_N)$ have the following form, for any $1 \leq i \leq N$,
\[  M_i (t , x , v)=c_{loc,i} (t,x) \left (\frac{m_i}{2 \pi T_{loc} (t,x)} \right )^{3/2} \exp\left\{ -\frac{m_i |v - u_{loc} 
(t,x)|^2}{2 T_{loc} (t,x)} \right\},
\]
where we assume that the Boltzmann constant equals to 1. The macroscopic quantities $c_{loc,i} : \mathbb{R}_+ \times \mathbb{T}^3 \rightarrow \mathbb{R}_+$ represent local molar concentration of the $i$-th species, 
$u_{loc} : \mathbb{R}_+ \times \mathbb{T}^3 \rightarrow \mathbb{R}^3$ denotes the bulk velocity, and $T_{loc} : \mathbb{R}_+ \times \mathbb{T}^3 \rightarrow \mathbb{R}_+$ expresses the temperature of the mixture. These quantities can be computed from the local Maxwellians
\[  c_{loc,i} = \int_{\mathbb{R}^3} M_i \,\d v, \quad \forall 1 \leq i \leq N, \quad u_{loc} = \frac{1}{\rho_{loc}} 
\sum_{i=1}^N \int_{\mathbb{R}^3} m_i v M_i \,\d v,
\]
\[  T_{loc} = \frac{1}{3\rho_{loc}} \sum_{i=1}^N \int_{\mathbb{R}^3} m_i |v - u_{loc} |^2 M_i \,\d v,
\]
where $\rho_{loc}=\sum_{i=1}^N m_i c_{loc,i}$ is the total mass density.

On torus, $\mathit{H}\mbox{-}$Theorem for multi-species also implies that the global equilibrium is a stationary solution to the equations 
\eqref{def of Bz-i equation} associated with the initial data $\mathbf{F}^{in}=(F_1^{in},\dots,F_N^{in})$. 
The macroscopic quantities of the global equilibrium can be calculated as follows
\[ \bar{c}_i = \int_{\mathbb{T}^3 \times \mathbb{R}^3} F_i^{in}(x,v) \,\d x\d v, \quad \forall 1 \leq i \leq N,\quad 
\bar{u} = \frac{1}{\bar{\rho}} \sum_{i=1}^N \int_{\mathbb{T}^3 \times \mathbb{R}^3} m_i v F_i^{in} (x , v) \,\d x\d v,
\]
\[ \bar{ T} = \frac{1}{3\bar{\rho}} \sum_{i=1}^N \int_{\mathbb{T}^3 \times \mathbb{R}^3} m_i |v - \bar{ u} |^2 
F_i^{in} \,\d x \d v,
\]
where $\bar{ \rho} = \sum_{i=1}^N m_i \bar{c}_i$ denotes the global density of the mixtures. The global equilibrium is given by 
$\mathbf{M} = (M_1 , \dots , M_N)$ with 
\[  M_i (v) = \bar{c}_i (\frac{m_i}{2 \pi \bar{T} })^{3/2} \exp\{- \frac{m_i |v - \bar{u}|^2}
{2 \bar{T}}\},
\] 
for any $1 \leq i \leq N$. By translating and rescaling the coordinate system, we can always assume $\bar{u} = 0$ and $\bar{T} = 1$. 
The only global equilibrium is a normalized Maxwellian, denoted by $\bm{ \mu}=( \mu_1, \dots, \mu_N)$ with
\begin{equation}\label{def of mu-i in introduction}
\mu_i (v) = \bar{c}_i \left ( \frac{m_i}{2 \pi} \right )^{3/2} e^{ - \frac{m_i |v|^2}{2}} , \quad \forall 1 \leq i \leq N.
\end{equation}

The Cauchy theory and the trend to the global equilibrium for the solution of equations \eqref{def of Bz-i equation} have been investigated in \cite{Briant2016DCDSglobal} and \cite{briant2016ARMAglobal}. There have been some works investigating the classical fluid limits (Euler and Navier-Stokes) from the multi-species Boltzmann equations (or multi-species Vlasov-Maxwell-Boltzmann systems) in different scaling, such as \cite{ABT-JSP2003, Wu-Yang-JDE2023, Wu-Yang-AA2024,JL-APDE-2022,JLZ-ARMA-2023}. The current paper aims to study a non-fluid macroscopic asymptotics from the multi-species Boltzmann equations in the diffusive scaling, namely, the Maxwell-Stefan equations, which are fundamental for mixture gases in chemical engineering.  Specifically, by introducing a dimensionless parameter $\varepsilon>0$ (Knudsen number, which is the ratio of the mean free path to the macroscopic length scale, such as the length unit of the torus), the rescaled Boltzmann equations are given by:
\begin{equation}
\partial_t F_i^\varepsilon + \frac{1}{ \varepsilon} v \cdot \nabla_x F_i^\varepsilon = \frac{1}{\varepsilon^2} \mathit{Q}_i 
(\mathbf{F}^\varepsilon , \mathbf{F}^\varepsilon ) , \quad \forall 1 \leq i \leq N.  \label{diffusive scaled Bz-i}
\end{equation}
Starting from these multi-species Boltzmann equations in the diffusive scaling, the Maxwell-Stefan system can be derived. 

\subsection{The Maxwell-Stefan System}
Maxwell-Stefan system was derived in the 19-th century by Maxwell \cite{maxwell1867iv} for dilute gases 
and Stefan \cite{stefan1871gleichgewicht} for fluids, providing a generalization of Fick's law of mono-species diffusion 
\cite{fick1855ueber}. Fick's law was the most classical law to describe diffusive transport. It postulated the flux goes 
from higher concentration regions to lower concentration regions, and the magnitude is proportional to the concentration 
gradient. In many situations, it provides accurate description of diffusive transport. However, it fails to describe 
diffusion in the case of multi-component gaseous mixtures, as three distinct diffusion phenomena: 
osmotic diffusion (diffusion without a gradient), reverse diffusion (up-hill diffusion in direction of the
gradient) and diffusion barrier (diffusion with a zero flux), were observed in experiments, see \cite{Duncan-Toor-1962} for example. 
Maxwell-Stefan system was applied to overcome the shortcomings of the Fick's law, relying on the fact that the driving force 
of the species in a mixture is in local equilibrium with the total inter-species drag/friction force. 

Here we first present a simplest case, i.e. isothermal, ideal, and non-reactive gaseous mixtures, which takes the following form:
\begin{equation} \left\{
\begin{aligned}
&\partial_t c_i + \nabla_x \cdot \mathcal{F}_i = 0,\quad \forall  1 \leq i \leq N,\\
    & - c_{tot} \nabla_x n_i = \sum_{j \neq i} \frac{n_j \mathcal{F}_i - n_i \mathcal{F}_j}{D_{ij}} , \quad \forall 1 \leq i \leq N.
\end{aligned}\right.\label{classical M-S}
\end{equation}
In this context, the notation $c_i=c_i(t,x)$ represents the molar concentration of $i$-th species, while 
$\mathcal{F}_i=\mathcal{F}_i(t,x)$ indicates its molar flux. The total molar concentration of the mixture is given by 
$c_{tot}=\sum_{i=1}^N c_i$, and $n_i=\frac{c_i}{c_{tot}}$ represents the mole fraction of species $i$. 
The binary diffusion coefficients $D_{ij} = D_{ji}$ quantify the interactions between species $i$ and $j$ for any fixed $i,j\in\{1, \dots ,N\}$. 

The first  equation of \eqref{classical M-S} is the conservation law of the molar concentration. The second part corresponds to the so-called Maxwell-Stefan equations/law, describing the relations between the molar fluxes and the mole fractions in a nonlinear way. Additionally, due to the symmetry of the binary diffusion coefficients, another relation is necessary to close the system. The usual assumption made is the equimolar diffusion condition \cite{krishna1997maxwell}
\[ \sum_{i=1}^N \mathcal{F}_i=0.
\]
It ensures that the total concentration is fixed. If the binary diffusion coefficients 
$\{D_{ij}\}_{1\leq i,j \leq N}$ are all equal, the Maxwell-Stefan law becomes Fick’s law, and the equations in this system reduce into heat equations.

The Maxwell-Stefan system belongs to the so-called cross-diffusion models \cite{MBG-2017,DLM-SIAMJMA2014,Jüngel-Nonlinearity2015,Luo-Ni-JDE1996}, which describe the process in which the gradient in the concentration or density of one chemical or biological species induces a flux of another species. It is widely used in the context of chemical engineering \cite{taylor1993multicomponent} 
and plays an important role in physics and medicine \cite{Boudin2010diffusionlung,chang1980multicomponent}. Numerical 
studies \cite{ern1994multicomponent,giovangigli2012multicomponent} have been applied to the Maxwell-Stefan system in 20-th century, and mathematical community arose interest on it starting from \cite{boudin2012mathematical,bothe2011maxwell}. 

For the well-posedness of this system, the works \cite{bothe2011maxwell,HMP-NA2017,Jungel-Stelzer-SIAM2013} proved the existence of local-in-time smooth solutions and global-in-time weak solutions. For the nonisothermal gas mixtures, \cite{Hutridurga-Salvarani-AML2018} focused on the local-in-time solution for a special Maxwell-Stefan system. 
\cite{Helmer-Jungel-NARWA2021,Georgiadis-Jungel-Nonlinearity2024} investigated the global existence of weak solution, and the former one established a conditional weak–strong uniqueness property.

The first derivation \cite{Boudin2015AAMdiffusion} established a link between the Maxwell-Stefan system and the 
multi-species Boltzmann equations in diffusive scaling, using a moment method. Subsequently, more formal studies emerged 
that employed the same approach \cite{hutridurga2017maxwell,Anwasia2021hardsphere,Salvarani2021survery}, starting from the 
Boltzmann equations with various types of collision kernels. Rigorous analysis of this 
diffusion asymptotics was presented in \cite{briant2021stability} under isothermal condition. 
The study on the connection between the multi-species Boltzmann equations and the Fick cross-diffusion system followed in \cite{Briant2023fickcross}. In chemical engineering, considering temperature variations is more realistic. However, the studies on the non-isothermal cases are much rarer. Hutridurga and Salvarani formally derived the non-isothermal Maxwell-Stefan system using the moment method in \cite{Hutridurga2017noniso}, and then proved the existence and uniqueness of the local-in-time solution for this system in \cite{Hutridurga-Salvarani-AML2018}. The reactive mixtures case was investigated in \cite{Anwasia2020KRMpolyatomic}.

We present a brief proof for the non-isothermal Maxwell-Stefan system. It generalizes the result 
in \cite{Hutridurga2017noniso} to the case where the multi-species Boltzmann equations have general cross sections, following the computations in \cite{Boudin2017general}. We assume that the 
distributions for the rescaled multi-species Boltzmann equations \eqref{diffusive scaled Bz-i} under assumptions (H1)-(H2)-(H3)-(H4) keep in a local Maxwellian state, i.e. for any $1 \leq i \leq N$,
\[  M_i^\varepsilon (t,x,v)= c_i^\varepsilon (t,x) \left ( \frac{m_i}{2\pi T^\varepsilon (t,x)} \right )^{\frac{3}{2}} 
e^{-\frac{m_i|v-\varepsilon u_i^\varepsilon (t,x)|^2}{2 T^\varepsilon (t,x)}}.
\]
The choices of $\mathcal{O}(\varepsilon)$ in the macroscopic velocities are caused by our focus on the pure diffusion phenomenon. Multiplying the rescaled equation \eqref{diffusive scaled Bz-i} by $1, m_i v$, and integrating over $\mathbb{R}^3$ with respect to $v$, 
we obtain that for any $1 \leq i \leq N$,
\begin{equation}
   \partial_t c_i^\varepsilon + \nabla \cdot (c_i^\varepsilon u_i^\varepsilon ) =0 ,  \label{formal deri of equation of i}
\end{equation}
\begin{equation}  
 \varepsilon \partial_t (m_i c_i^\varepsilon u_i^\varepsilon ) + \frac{1}{\varepsilon} \nabla_x (c_i^\varepsilon T^\varepsilon) 
+ \varepsilon \nabla_x \cdot ( m_i c_i^\varepsilon u_i^\varepsilon \otimes u_i^\varepsilon) = \frac{1}{\varepsilon ^2} \sum_{j \neq i} \int_{\mathbb{R}^3} 
\mathit{Q}_{ij} ( F_i^\varepsilon, F_j^\varepsilon ) (v) m_i v \,\d v .  \label{formal deri of equation of m-i v}
\end{equation}
Moreover, by multiplying \eqref{diffusive scaled Bz-i} by $m_i |v|^2$, summing over the indices $i$, 
and integrating over $\mathbb{R}^3$, we get the equation
\begin{equation}
    3 \partial_t (\sum_{i=1}^N c_i^\varepsilon T^\varepsilon ) +5 \nabla_x \cdot (\sum_{i=1}^N  c_i^\varepsilon u_i^\varepsilon T^\varepsilon ) 
    + \varepsilon^2 \left[ \partial_t (\sum_{i=1}^N m_i c_i^\varepsilon |u_i^\varepsilon|^2 ) + \nabla_x \cdot (\sum_{i=1}^N m_i c_i^\varepsilon 
    |u_i^\varepsilon| u_i^\varepsilon )     \right] =0.  \label{formal deri of equation of T}
\end{equation}
Assuming these quantities are pertubated as
\[  c_i^\varepsilon = c_i + \mathcal{O} (\varepsilon), \quad u_i^\varepsilon = u_i +  \mathcal{O} (\varepsilon), \quad 
T^\varepsilon = T + \mathcal{O} (\varepsilon)\,.
\]
The equations for $c_i$ and $T$ can be derived  from equations \eqref{formal deri of equation of i}  and \eqref{formal deri of equation of T} by taking the limit as $\varepsilon \rightarrow 0$: 
\begin{equation}
    \partial_t c_i + \nabla \cdot (c_i u_i) =0, \quad \forall 1\leq i\leq N, \label{formal limit equ for c-i}
\end{equation}
\begin{equation}
    3 \partial_t (\sum_{i=1}^N c_i T) +5 \nabla_x \cdot (\sum_{i=1}^N  c_i u_i T ) = 0 . \label{formal limit equ for T}
\end{equation}
Moreover, we notice that the leading order of \eqref{formal deri of equation of m-i v} is 
$\mathcal{O}(\frac{1}{\varepsilon})$. By multiplying it by $\varepsilon$ and taking the limit as $\varepsilon \rightarrow 0$, we get 
\[ \nabla_x ( c_i T ) =  \lim_{\varepsilon \rightarrow 0} \left (\frac{1}{\varepsilon} \sum_{j \neq i} \int_{\mathbb{R}^3} 
\mathit{Q}_{ij} ( F_i^\varepsilon, F_j^\varepsilon ) (v) m_i v \,\d v \right ).
\]
Repeating the computations in \cite{Boudin2017general} with some minor modifications, as they considered a constant temperature there, we obtain the limit equation:
\begin{equation}
    \nabla_x ( c_i T ) = - \sum_{j \neq i} k_{ij} c_i c_j (u_i -u_j ), \quad \forall 1\leq i\leq N. \label{formal limit equ for flux-i}
\end{equation}
The expression of $k_{ij}$ is, for any $1 \leq i,j \leq N$, 
\begin{equation}
\begin{aligned}
    k_{ij} = &\frac{C_{ij}^\Phi \mu_{ij}^2}{6 T} \left ( \frac{m_i}{2 \pi T} \right )^{\frac{3}{2}} \left ( \frac{m_j}{2 \pi T} \right )^{\frac{3}{2}} 
    \int_{\mathbb{R}^6 \times \mathbb{S}^2 } |v-v_\ast|^\gamma \mathit{b}_{ij} (\cos{\theta})\\
    & \times \exp{\left[ -\frac{m_i |v|^2 +m_j |v_\ast|^2}{2 T} \right]} [ (v_\ast -v) + |v-v_\ast| \sigma  ]^2 \,\d \sigma \d v_\ast \d v\\
       = & \frac{8 \pi^{1/2} \mu_{ij} \| \mathit{b}_{ij}\|_{L^1(-1,1)} } {3} \Gamma(\frac{\gamma+5}{2}) 
       \left ( \frac{2 T}{\mu_{ij}}   \right )^{\gamma/2} :=  \frac{T^{\gamma/2}}{\Delta_{ij}} ,  \label{explicit expression of the diffusion coff}
\end{aligned}
\end{equation}
where $\mu_{ij} : =\frac{m_i m_j}{m_i + m_j}$ is the reduced mass. The diffusion coefficients $\frac{\Delta_{ij}}{T^{\gamma/2}}$ are 
symmetric with respect to the indices $i,j$, and they only depend on the atomic masses $(m_i)_{1\leq i\leq N}$ and the collision kernels $(\mathit{B}_{ij})_{1 \leq i,j \leq N}$. This implies that some additional relations must be satisfied by taking the summation with respect to $i=1, \cdots, N$ in the equation \eqref{formal limit equ for flux-i}:  $\nabla_x ( \sum_{i=1}^N c_i T )=0$. 
These relations infer that \eqref{formal limit equ for c-i}-\eqref{formal limit equ for T}-\eqref{formal limit equ for flux-i} are $4N-2$ independent equations, while the number of the limit macroscopic quantities $(c_i,u_i,T)_{1\leq i\leq N}$ is $4N+1$. Therefore, to close the limiting equations,
it is necessary to introduce three additional equations as supplements. We introduce the closure relations 
proposed in \cite{Hutridurga2017noniso}
\begin{equation}
    \sum_{i=1}^N c_i u_i = - \alpha \nabla_x c_{tot}, \quad \alpha >0,  \label{closure relation in intro}
\end{equation}
i.e., the total concentration $c_{tot} = \sum_{i=1}^N c_i$ follows a Fickian behaviour. We denote the flux 
    $J_i=c_i u_i$ for the species $i$. When the temperature is a constant, the relation $\nabla_x (c_{tot} T) =0$ implies $\sum_{i=1}^N c_i u_i =0 $ from the closure relations \eqref{closure relation in intro}, which coincides with the equimolar diffusion condition. 
The equations \eqref{formal limit equ for c-i} and \eqref{formal limit equ for T} deduce that 
$\partial_t c_{tot}=0$ and $\nabla_x \cdot (\sum_{i=1}^N c_i u_i)=0$, leading to the total concentration $c_{tot}$ to be a constant. Furthermore, if we assume the temperature is not same for different species and denote $T_i^\varepsilon$ for species $i$, then we can formally deduce that 
$T_i^\varepsilon - T_j^\varepsilon = \mathcal{O} (\varepsilon^2)$. This result is consistent with that in \cite{Hutridurga2017noniso}, obtained by multiplying the rescaled Boltzmann equations by $m_i |v|^2$, integrating it with respect to $v$ and repeating the above process. Now, the equations \eqref{formal limit equ for c-i}-\eqref{formal limit equ for T}-\eqref{formal limit equ for flux-i}-\eqref{closure relation in intro} together are called the {\bf non-isothermal Maxwell-Stefan system}, which is the system  we will prove the global well-posedness and rigorously justify the derivation from the scaled multi-species Boltzmann equations \eqref{diffusive scaled Bz-i}. This generalizes Bondesan and Briant's work \cite{briant2021stability} from isothermal to non-isothermal case. 

Our aim in this paper is to investigate the convergence from the multi-species Boltzmann equations to the non-isothermal Maxwell-Stefan system stated above, i.e. the equations \eqref{formal limit equ for c-i}, \eqref{formal limit equ for T}, \eqref{formal limit equ for flux-i}, and \eqref{closure relation in intro}. We first analyze this Maxwell-Stefan 
system around a constant state, and then use the perturbation solution to construct a local Maxwellian. 
Furthermore, we prove the stability of the diffusion asymptotics 
from this Maxwellian to the multi-species Boltzmann equations by introducing a suitable modified Sobolev norm, 
as used in \cite{mouhot2006Nonlinearity,briant2015JDEnavierstokes} for applications to the mono-species Boltzmann equation. 
In this part, we focus on the Boltzmann equations with hard potential and Maxwellian potential collision kernels ($\gamma\in[0,1]$), as we require the spectral gap property for the operator linearized around the global equilibrium \cite{briant2016ARMAglobal,mouhot2016SIAMJMAhypercoercivity}. We present notations and our main results in Section \ref{Sec:2}, and the sketch of the proofs of our work is at the end of this section. The perturbative Cauchy theory for the non-isothermal Maxwell-Stefan system is presented in Section \ref{Sec:3}. Finally, our analysis of diffusion asymptotics is in Section \ref{Sec:4}, following methods in \cite{briant2021stability}.

\section{Main Results}\label{Sec:2}
\subsection{Notations and Conventions}
We first introduce main notations used in this paper. For any vector-valued operator $\mathbf{w} \in (\mathbb{R} )^N, \mathbf{W} \in (\mathbb{R})^N$, $\mathbf{v} \in (\mathbb{R_+})^N$, $u \in \mathbb{R}$ and $q \in \mathbb{Q}$, we denote by
\[\mathbf{w} \mathbf{W} = (w_i W_i)_{1 \leq i \leq N}, \quad u \mathbf{w} = (u w_i)_{1 \leq i \leq N}, \quad \mathbf{v}^q 
= (v_i^q )_{1 \leq i \leq N}.
\]
The Euclidean scalar product in $\mathbb{R}^N$ weighted by a vector $\mathbf{w} \in (\mathbb{R}_+)^N$ is defined as 
\[  \langle \mathbf{c} , \mathbf{d} \rangle_{ \mathbf{w}} = \sum_{i=1}^N c_i d_i w_i, \quad \text{for any }\mathbf{c} , \mathbf{d} 
\in \mathbb{R}^N,
\]
and the corresponding norm is $\|\mathbf{c} \|_{ \mathbf{w}}^2 = \langle \mathbf{c} , \mathbf{c} \rangle_{ \mathbf{w}}$. When 
$\mathbf{w} = \mathbf{1} := (1 , \dots , 1)$, the index $ \mathbf{1}$ will be dropped in both notations for 
the scalar product and the norm. When $w\in \mathbb{R}_+$, we define
\[\langle \mathbf{c} , \mathbf{d} \rangle_w := \langle \mathbf{c} , \mathbf{d} \rangle_{w \mathbf{1}}=
\sum_{i=1}^N c_i d_i w,
\]
and the corresponding norm $\|\mathbf{c}\|_w = \langle \mathbf{c} , \mathbf{c} \rangle_w = \sum_{i=1}^N c_i^2 w$.

Then we introduce the functional spaces, for $p \in [ 1 , + \infty]$
\[L_t^p = L^p (0 , +\infty) , \quad L_x^p = L^p (\mathbb{T}^3 ), \quad L_{t,x}^p = L^p( \mathbb{R}_+ \times \mathbb{T}^3).
\]
For any positive measurable functions $\mathbf{w}: \mathbb{T}^3 \rightarrow (\mathbb{R}_+)^N$ and $w: \mathbb{T}^3 \rightarrow 
\mathbb{R}_+$ in the variable $x$, we define the weighted Hilbert space $L^2 (\mathbb{T}^3 , \mathbf{w})$ and its norm as
\[   \langle \mathbf{c} , \mathbf{d} \rangle_{L_x^2(\mathbf{w})} = \sum_{i=1}^N \langle c_i , d_i \rangle_{L_x^2(w_i)}=
\sum_{i=1}^N \int_{\mathbb{T}^3} c_i d_i w_i \,\d x = \int_{\mathbb{T}^3} \langle \mathbf{c} , \mathbf{d}\rangle_{ \mathbf{w}} \,\d x , 
\]
\[   \| \mathbf{c} \|_{L_x^2 ( \mathbf{w})}^2 = \sum_{i=1}^N \|c_i \|_{L_x^2(w_i)}^2 = \sum_{i=1}^N 
\int_{\mathbb{T}^3} c_i^2 w_i \,\d x = \int_{\mathbb{T}^3} \|\mathbf{c} \|^2_{\mathbf{w}} \,\d x .
\]
The weighted Hilbert space $L^2(\mathbb{T}^3,w)$ and its norm are defined as 
\[  \langle \mathbf{c}, \mathbf{d} \rangle_{L_x^2(w)} := \langle \mathbf{c} , \mathbf{d}\rangle_{L_x^2(w \mathbf{1})}
= \sum_{i=1}^N \langle c_i , d_i \rangle_{L_x^2(w)} = \sum_{i=1}^N \int_{\mathbb{T}^3} c_i d_i w \,\d x = 
\int_{\mathbb{T}^3} \langle \mathbf{c},\mathbf{d} \rangle_w \,\d x, 
\]
\[  \|\mathbf{c} \|_{L_x^2(w)} := \|\mathbf{c} \|_{L_x^2(w \mathbf{1})} = \sum_{i=1}^N \|c_i \|_{L_x^2(w)} \,\d x 
=\sum_{i=1}^N \int_{\mathbb{T}^3} c_i^2 w \,\d x = \int_{ \mathbb{T}^3} \|\mathbf{c}\|_{w} \,\d x.
\]

Furthermore, we introduce the corresponding weighted Sobolev spaces, for any $s \in \mathbb{N}$ and the vector-valued function 
$\mathbf{c} \in H^s (\mathbb{T}^3,\mathbf{w})$
\[\|\mathbf{c} \|_{H_x^s (\mathbf{w})} := \left ( \sum_{|\beta| \leq s} \|\partial_x^\beta \mathbf{c}\|^2_{L_x^2( \mathbf{w})} 
\right )^{1/2} = \left ( \sum_{i=1}^N \sum_{|\beta| \leq s} \|\partial_x^\beta c_i \|_{L_x^2 ( w_i)}^2 \right 
)^{1/2} ,
\]
with the multi-index $\beta \in \mathbb{N}^3$ and the length $|\beta| = \sum_{k=1}^3 \beta_k$. Similarly, the weighted Sobolev spaces $\mathbf{\mathbf{c}} \in H^s( \mathbb{T}^3, w)$ is defined as
\[\|\mathbf{c}\|_{H_x^s( w)}:=\left ( \sum_{|\beta| \leq s} \|\partial_x^\beta \mathbf{c}\|_{L_x^2(w)}
\right )^{1/2} = \left ( \sum_{i=1}^N \sum_{|\beta|\leq s} \|\partial_x^\beta c_i\|_{L_x^2(w)}^2 \right )^{1/2}.
\]
\subsection{Main Theorems}
Using the above notations, we can rewrite the non-isothermal Maxwell-Stefan system \eqref{formal limit equ for c-i}-\eqref{formal limit equ for T}-\eqref{formal limit equ for flux-i}-\eqref{closure relation in intro} in vectorial form
\begin{gather}
    \partial_t \mathbf{c} + \nabla_x \cdot (\mathbf{c} \mathbf{u}) = 0, \label{M-S 1}\\
    \nabla_x (\mathbf{c} T) = T^{\gamma/2} A(\mathbf{c}) \mathbf{u} , \quad \gamma \in (-3,1] \label{M-S 2}\\
    \partial_t (c_{tot} T) + \frac{5}{3} \nabla_x \cdot (\sum_{i=1}^N c_i u_i T)=0, \label{M-S 3}\\
    \langle \mathbf{c},\mathbf{u} \rangle  = -\alpha \nabla_x c_{tot} \quad \alpha >0. \label{M-S 4}
\end{gather}
The Maxwell-Stefan matrix $A$ is defined as 
\begin{equation}\label{def of matrix A in introduction}
    [A(\mathbf{c})]_{ij}= \frac{c_i c_j}{\Delta_{ij}} - \left ( \sum_{r=1}^N \frac{c_i c_r}{ \Delta_{ir}} \right ) \delta_{ij},
\end{equation}
which depends on the concentrations $\mathbf{c}=(c_1 , \dots , c_N)$ in a nonlinear way. The Maxwell-Stefan (in short, MS) matrix $A$ is non-positive defined for any $\mathbf{c} \in (\mathbb{R}_+ )^N$ with the kernel space on $\mathbb{R}^N$ given by $\mathrm{Ker} A= \mathrm{Span} ( \mathbf{1})$, see \cite{bothe2011maxwell,jungel2013existence,briant2022perturbativ}.

The first step in this paper is to establish the global-in-time Cauchy theory for the non-isothermal Maxwell-Stefan system \eqref{M-S 1}-\eqref{M-S 2}-\eqref{M-S 3}-\eqref{M-S 4} around any constant state of $(\bar{\mathbf{c}},\mathbf{0},1)$, where $\bar{\mathbf{c}}\in(\mathbb{R}_+^\ast)^N$ is a positive constant vector and $\mathbf{0} = (0, \dots, 0)$ represents the zero vector function from $\mathbb{R}_+ \times \mathbb{T}^3$ to $(\mathbb{R}^3)^N$. Before we present the first main theorem, we introduce the following norms for the perturbation variables $(\tilde{\mathbf{c}}, \tilde{\mathbf{U}}, \tilde{T})$, which are defined in the following Theorem \ref{theorem for M-S in main result}: 
\begin{equation} 
 \begin{aligned}
     \mathscr{E}_s (t) &:= \|\tilde{\mathbf{c}}\|_{H_x^s (\bar{\mathbf{c}}^{-1})}^2 (t) + \|\tilde{T}\|_{H_x^s}^2 (t) 
        + \chi \| \sum_{i=1}^N \tilde{c}_i \|_{H_x^s}^2 (t),\\
            \mathscr{D}_s (t) &:= d_1 \|\tilde{\mathbf{U}}\|_{H_x^s (\omega)}^2 (t) + d_2 
        \|\nabla_x  \sum_{i=1}^N  \tilde{c}_i\|_{H_x^s}^2 (t),
        \label{def of MS functionals in intro}
 \end{aligned}
\end{equation}
 where $\omega := (1 + \lambda \tilde{T})^{\gamma/2 - 1}$ and $\tilde{\mathbf{U}}$ is the projection of the velocity perturbation $\tilde{\mathbf{u}}$ onto space $\mathrm{Span} (\mathbf{1})^\perp$.

 \begin{theorem}\label{theorem for M-S in main result}
        Let $s > 3$ be an integer, $\lambda \in (0,1]$ and $\bar{\mathbf{c}}$ be a positive constant $N$-vector. There exist explicit constants 
        $\delta_{MS}, d_1,d_2,\chi > 0$ such that if the initial data $(\tilde{ \mathbf{c}}^{in}, \tilde{\mathbf{u}}^{in}, \tilde{T}^{in} ) \in H_x^s( \mathbb{T}^3) \times H_x^{s-1}
        (\mathbb{T}^3) \times H_x^s (\mathbb{T}^3)$ satisfy the following conditions (A1)-(A2)-(A3): for almost any $x \in \mathbb{T}^3$, and for any $1 \leq i \leq N$
    \begin{enumerate}[label=\textbf{(A\arabic*)}]
        \item  \textbf{Mass and temperature positivities:} \quad $\bar{c}_i+\lambda\tilde{c}_i^{in}(x)>0,$ 
        $1 + \lambda \tilde{T}^{in} (x)>0$,
        \item \textbf{Mass and temperature compatibility:}
        \[(\bar{c}_{tot} + \lambda \tilde{c}_{tot}^{in} (x))
        \nabla_x \tilde{T}^{in} (x) + (1 + \lambda \tilde{T}^{in} (x)) \nabla_x \tilde{c}_{tot}^{in} (x)=0,
        \quad \tilde{c}_{tot}^{in}=\sum_{i=1}^N \tilde{c}_i^{in},
        \]
        \item  \textbf{Moment compatibility:}
         \[   (\bar{c}_i + \lambda \tilde{c}_i^{in}) \nabla_x \tilde{T}^{in} + (1 + \lambda \tilde{T}^{in})
         \nabla_x \tilde{c}_i^{in} = (1 + \lambda \tilde{T}^{in})^{\gamma/2} \sum_{j \neq i} \frac{( \bar{c}_i +
         \lambda \tilde{c}_i^{in}) (\bar{c}_j + \lambda \tilde{c}_j^{in})}{\Delta_{ij}} \left (\tilde{u}_j^{in}
         - \tilde{u}_i^{in} \right ),
        \]
    \end{enumerate}
      and obey the smallness assumption
      \begin{equation}
      \|\tilde{ \mathbf{c}}^{in}\|^2_{H_x^s( \bar{\mathbf{c}}^{-1})}  + \chi 
        \|\sum_{i=1}^N \tilde{c}_i^{in}\|^2_{H_x^s} + \|\tilde{T}^{in}\|^2_{H_x^s} \leq \delta_{MS}^2,
      \end{equation}
      then there exists a unique global-in-time classical solution 
        \[(\mathbf{c} , \mathbf{u}, T) = (\bar{ \mathbf{c}} + \lambda \tilde{ \mathbf{c}}, \lambda \tilde{ \mathbf{u}}, 1+ \lambda 
        \tilde{T})  \]
        in $L^\infty (\mathbb{R}_+ ; H_x^s({\mathbb{T}^3})) \times L^\infty (\mathbb{R}_+ ; H_x^{s-1} ({\mathbb{T}^3})) \times L^\infty
        (\mathbb{R}_+ ; H_x^s({\mathbb{T}^3}) )$ to the non-isothermal Maxwell-Stefan system 
        \eqref{M-S 1}-\eqref{M-S 2}-\eqref{M-S 3}-\eqref{M-S 4} with the initial data
        $(\tilde{\mathbf{c}}, \tilde{\mathbf{u}}, \tilde{T})|_{t=0} = (\tilde{\mathbf{c}}^{in}, \tilde{\mathbf{u}}^{in}, 
        \tilde{T}^{in}) $, a.e. on $\mathbb{T}^3$. In particular, the positivities hold
        \[  \min_{1\leq i \leq N} \bar{c}_i + \lambda \tilde{c}_i (t,x)>0, \quad 1 + \lambda \tilde{T} (t,x)>0,\quad \text{a.e. } (t,x) \in 
        \mathbb{R}_+ \times \mathbb{T}^3.        
        \]
        Moreover, the folowing inequality holds for any $t\in\mathbb{R}_+$,
        \begin{equation}
         \mathscr{E}_s (t) + \int_0^t \mathscr{D}_s (\tau) \,d\tau \leq \mathscr{E}_s (0) \leq \delta_{MS}^2.
         \label{energy estimate for M-S in main thm}
        \end{equation}
         Explicit expressions of $d_1,d_2,\chi$ can be found in the following context, which are independent of parameters $\lambda$ and $\delta_{MS}$. 
    \end{theorem}
\begin{remark}
In fact, our calculations can be performed for the case when the spatial domain is the whole space $\mathbb{R}^3$ with trivial 
adjustments. Our results are applied to the space $\mathbb{T}^3$ due to the diffusion asymptotics, 
as we introduce a functional from \cite{mouhot2006Nonlinearity}, where the Poincaré inequality plays a crucial role in the analysis.
\end{remark}

The vectorial form for the multi-species Boltzmann equations in diffusive scaling is
\begin{equation}
\partial_t \mathbf{F}^\varepsilon + \frac{1}{\varepsilon} v \cdot \nabla_x \mathbf{F}^\varepsilon = \frac{1}{\varepsilon^2} 
\mathbf{Q}(\mathbf{F}^\varepsilon, \mathbf{F}^\varepsilon)  \label{Bz in diffusive scaling}
\end{equation}
with the initial data given
\begin{equation}
\mathbf{F}^\varepsilon (0,x,v) = \mathbf{F}^{\varepsilon, in}(x,v), \quad (x,v) \in \mathbb{T}^3 \times \mathbb{R}^3.
\end{equation}
The local Maxwellian vector $\mathbf{M}^\varepsilon = \mathbf{M}(\mathbf{c}, \varepsilon \mathbf{u}, T \mathbf{1})
= (M_1^\varepsilon, \dots, M_N^\varepsilon)$ is defined as the following, for
$1 \leq i \leq N$,
\begin{equation}\label{def of M-i}
    M_i^\varepsilon (t,x,v) = c_i(t,x) \left (\frac{m_i}{2 \pi T(t,x)} \right )^{3/2} \exp\left\{-\frac{m_i |v - \varepsilon u_i(t,x)|^2 }
    {2T(t,x)}\right\} ,
\end{equation}
where $(\mathbf{c}, \mathbf{u}, T)$ is the unique classical solution of the Maxwell-Stefan system 
\eqref{M-S 1}-\eqref{M-S 2}-\eqref{M-S 3}-\eqref{M-S 4} stated in Theorem \ref{theorem for M-S in main result}. 
We set the parameter $\lambda$ to be the 
Knudsen number $\varepsilon$, which means that $(\mathbf{c},\mathbf{u},T)$ is in 
the perturbative form that for any $1 \leq i \leq N$,
\begin{equation}
\begin{cases}
c_i = \bar{c}_i + \varepsilon \tilde{c}_i,\\
u_i = \varepsilon \tilde{u}_i,\\
T = 1 + \varepsilon \tilde{T}. \label{Perturbative form of the perturbations for M-S}
\end{cases}
\end{equation}

Starting from this choice of local Maxwellian, we plug the perturbation $\mathbf{F}^\varepsilon = \mathbf{M}^\varepsilon + \varepsilon 
\bm{\mu}^{1/2} \mathbf{f}^\varepsilon$ into the rescaled equation \eqref{Bz in diffusive scaling}, where $\bm{\mu}$ is the unique global 
equilibrium defined in \eqref{def of mu-i in introduction}. The perturbed equation is
\begin{equation}
\partial_t \mathbf{f}^\varepsilon + \frac{1}{\varepsilon} v \cdot \nabla_x \mathbf{f}^\varepsilon = \frac{1}{\varepsilon^2} 
\mathbf{L}^\varepsilon (\mathbf{f}^\varepsilon) + \frac{1}{\varepsilon} \mathbf{\Gamma} (\mathbf{f}^\varepsilon , 
\mathbf{f}^\varepsilon) + \mathbf{S}^\varepsilon , \label{pert Bz}
\end{equation}
where the linear operator $\mathbf{L}^\varepsilon = (\mathit{L}_1^\varepsilon, \dots, \mathit{L}_N^\varepsilon)$ is given by 
\begin{equation}
\mathit{L}_i^\varepsilon (\mathbf{g}) = \mu_i^{-1/2} \sum_{j=1}^N \left (\mathit{Q}_{ij} (M_i^\varepsilon, \mu_j^{1/2} g_j) + 
\mathit{Q}_{ij} (\mu_i^{1/2} g_i, M_j^\varepsilon) \right ). \label{def of L-eps}
\end{equation}
The nonlinear operator $\mathbf{\Gamma} = (\Gamma_1, \dots, \Gamma_N)$ is defined as 
\begin{equation}
\Gamma_i (\mathbf{f}, \mathbf{g}) = \frac{1}{2} \mu_i^{-1/2} \sum_{j=1}^N \left (\mathit{Q}_{ij} (\mu_i^{1/2} f_i, \mu_j^{1/2} g_j) + \mathit{Q}_{ij} 
(\mu_i^{1/2} g_i, \mu_j^{1/2} f_j) \right ),
\end{equation}
and the source term is given by 
\begin{equation}\label{equ for source term}
\mathbf{S}^\varepsilon = \frac{1}{\varepsilon^3} \bm{\mu}^{-1/2} \mathbf{Q}(\mathbf{M}^\varepsilon, \mathbf{M}^\varepsilon) - 
\frac{1}{\varepsilon} \bm{\mu}^{-1/2} \partial_t \mathbf{M}^\varepsilon -\frac{1}{\varepsilon^2} \bm{\mu}^{-1/2} v \cdot \nabla_x 
\mathbf{M}^\varepsilon.
\end{equation}

Moreover, we introduce the dissipative operator $\mathbf{L}=(L_1,\dots,L_N)$ linearized around the global equilibrium $\bm{\mu}$, defined as, for 
any $1 \leq i \leq N$,
\begin{equation}
\mathit{L}_i (\mathbf{g}) = \mu_i^{-1/2} \sum_{j=1}^N \left ( \mathit{Q}_{ij} (\mu_i, \mu_j^{1/2} g_j) + \mathit{Q}_{ij} (\mu_i^{1/2} 
g_i, \mu_j) \right ),
\end{equation}
which is a closed self-adjoint operator in the space $L^2(\mathbb{R}^3)$. The weak forms \eqref{weak form-1} and \eqref{weak form-2} imply that the kernel space for operator $\mathbf{L}$ is spanned by the orthonormal basis $(\bm{\phi}^{(1)}, \dots, \bm{\phi}^{N+4})$ with the expressions
\begin{equation*}
\begin{cases}
    \bm{\phi}^{(i)} = \frac{1}{ \sqrt{\bar{c}_i}} \mu_i^{1/2} \mathbf{e}^{(i)},  1 \leq i \leq N,\\
    \bm{\phi}^{(N+l)} = \frac{v_l}{\left ( \sum_{j=1}^N m_j \bar{c}_j \right )^{\frac{1}{2}} } (m_i \mu_i^{1/2})_{1 \leq i \leq N}, 
    \quad  1 \leq l \leq 3,\\
    \bm{\phi}^{(N+4)} = \frac{1}{\left (\sum_{j=1}^N \bar{c}_j \right )^{\frac{1}{2}}} \left ( \frac{m_i|v|^2-3}{\sqrt{6}} 
    \mu_i^{1/2}  \right )_{1 \leq i \leq N}.
\end{cases}
\end{equation*}
The orthogonal projection onto $\mathrm{Ker}{ \mathbf{L}}$ in $L^2(\mathbb{R}^3 )$ is denoted by
\begin{equation}
\bm{\pi}_{ \mathbf{L}}  (\mathbf{g}) (v) = \sum_{k=1}^{N+4} \langle \mathbf{g}, \bm{\phi}^{(k)} \rangle_{L_v^2}  \bm{\phi}^{(k)} (v).
\end{equation}
The explicit expression is 
\begin{equation}
  \begin{aligned}
  \bm{\pi}_\mathbf{L} ( \mathbf{g}) = & \sum_{i=1}^N \left ( \frac{1}{\bar{c}_i} \int_{\mathbb{R}^3} g_i \mu_i^{1/2} \,\d v \right ) 
  \mu_i^{1/2} \mathbf{e}^{(i)}\\
  & + \sum_{k=1}^3 \frac{v_k}{ \sum_{i=1}^N m_i \bar{c}_i} \left ( \sum_{i=1}^N \int_{\mathbb{R}^3} m_i v_k g_i \mu_i^{1/2} \,\d v 
  \right ) ( m_i \mu_i^{1/2} )_{1\leq i\leq N}\\
  & + \frac{1}{\sum_{i=1}^N \bar{c}_i} \left (\sum_{i=1}^N \int_{\mathbb{R}^3} \frac{m_i |v|^2 - 3}{\sqrt{6}} g_i \mu_i^{1/2} 
  \,\d v \right )  \left ( \frac{m_i |v|^2 - 3}{\sqrt{6}} \mu_i^{1/2} \right )_{1 \leq i \leq N}. \label{proje pi-L express}
  \end{aligned}
\end{equation}
We denote by $\mathbf{f}^\perp = \mathbf{f} - \bm{\pi}_\mathbf{L} (\mathbf{f})$. Since the Maxwellian $\mathbf{M}^\varepsilon$ 
is not an equilibrium for the collision operator $\mathbf{Q}$, the linear operator $\mathbf{L}^\varepsilon$ has no clear self-adjointness 
in the usual Sobolev space. The works \cite{Bondesan2020CPAAnonequilibrium,Briant2023hypocoercivity} established 
a local coercivity property by applying the decomposition 
$\mathbf{L}^\varepsilon = \mathbf{L} + ( \mathbf{L}^\varepsilon - \mathbf{L} )$ and utilizing the explicit spectral gap form \cite{briant2016ARMAglobal,mouhot2016SIAMJMAhypercoercivity} for $\mathbf{L}$ under the condition that the collision kernels of the Boltzmann equations satisfy assumptions (H1)-(H2)-(H3)-(H4) with $\gamma\in[0,1]$.

We choose the functional in \cite{mouhot2006Nonlinearity,briant2015JDEnavierstokes} to derive a priori estimate on 
$\mathbf{f}^\varepsilon$ in equation \eqref{pert Bz}, which is defined on the space $H^s ( \mathbb{T}^3 \times \mathbb{R}^3 )$ with 
$s \in \mathbb{N}^\ast$ and $\varepsilon \in (0,1]$,
\begin{equation}
\begin{aligned}
\|\cdot \|_{\mathcal{H}_\varepsilon^s} = \left\{ \sum_{ |\alpha| \leq s} a_\alpha^{(s)} \|\partial_x^\alpha \cdot\|_{L_{x,v}^2}^2
 + \varepsilon \sum_{\substack{ |\alpha| \leq s \\ k, \alpha_k > 0}} b_{\alpha,k}^{(s)} \langle \partial_x^\alpha \cdot, 
 \partial_{v}^{e_k} \partial_x^{\alpha-e_k} \cdot \rangle_{L_{x,v}^2} \right.\\
 + \left. \varepsilon^2 \sum_{\substack{ |\alpha| + |\beta| \leq s \\ |\beta| \geq 1}} d_{\alpha, \beta}^{(s)} \|\partial_v^\beta 
\partial_x^\alpha \cdot \|_{L_{x,v}^2}^2 \right\},\label{express of functional H-eps-s}
\end{aligned}
\end{equation}
for some positive constants $ \{a_\alpha^{(s)}\},  \{b_{\alpha, k}^{(s)} \}$ and $\{ d_{\alpha, \beta}^{(s)}\}$ to be chosen later. We remark that $\alpha_k > 0$ for all $k \in \{1,2,3\}$ means that $\alpha-e_k \geq 0$ for all $k \in \{1,2,3\}$. Here $e_k = (e_k^1, e_k^2, e_k^3)$ with $e_k^i = 1$ for $i = k$ and $e_k^i = 0$ for $i \neq k$

Based on these choices, we can establish the following result.
\begin{theorem}\label{theorem for perturbed f}
Let the collision kernels $(\mathit{B}_{ij})_{1 \leq i,j \leq N}$ satisfy assumptions (H1)-(H2)-(H3)-(H4) with the parameter in (H3) choosing $\gamma \in [0,1]$, and consider the local Maxwellian $\mathbf{M}^\varepsilon$ defined by \eqref{def of M-i}. There exist constants $s_0 \in \mathbb{N}^\ast$, $\delta_B>0$, $\bar{\delta}_{MS}>0$ and $\varepsilon_0 \in (0,1]$ independent of $\varepsilon$ such that for any $s \geq s_0$, $\varepsilon \in (0, \varepsilon_0]$ and $\delta_{MS} \in [0, \bar{\delta}_{MS}]$, if the initial data $\mathbf{f}^{in} \in H^s( \mathbb{T}^3 \times \mathbb{R}^3 )$ satisfies 
\[ \|\mathbf{f}^{in} \|_{ \mathcal{H}_\varepsilon^s } \leq \delta_B/2, \quad \|\bm{\pi}_\mathbf{T}^\varepsilon (\mathbf{f}^{in}) 
\|_{L_{x,v}^2} \leq C \delta_{MS}
\]
for some positive constant $C$ independent of $\varepsilon$ and $\delta_{MS}$, then there exists a 
unique $\mathbf{f} (t, x, v) \in C^0(\mathbb{R}_+ ; H_{x,v}^s)$ such that $\mathbf{F}^\varepsilon = 
\mathbf{M}^\varepsilon + \varepsilon \bm{\mu}^{1/2} \mathbf{f}$ is the classical solution to the multi-species Boltzmann equations \eqref{Bz in diffusive scaling}. Here the projection operator $\bm{\pi}_\mathbf{T}^\varepsilon$ is defined in \eqref{def of pi T}. 

In particular, if $\mathbf{F}^{\varepsilon, in} = \mathbf{M}^{ \varepsilon, in} + \varepsilon \bm{\mu}^{1/2} \mathbf{f}^{in} 
\geq0$, then $\mathbf{F}^\varepsilon (t,x,v) \geq 0$ holds for almost all $(t,x,v) \in \mathbb{R}_+ \times \mathbb{T}^3 \times \mathbb{R}^3$. 
Moreover, for any time $t \geq 0$, $\mathbf{F}^\varepsilon$ satisfies the stability property 
\[ \| \frac{\mathbf{F}^\varepsilon - \mathbf{M}^\varepsilon }{ \bm{\mu}^{1/2} } \|_{\mathcal{H}_\varepsilon^s} \leq \varepsilon \delta_B, 
\quad \forall \varepsilon \in (0, \varepsilon_0].
\]
Here the constant $\delta_B > 0$ explicitly only depends on the number of species $N$, the atomic masses $(m_i)_{1 \leq i \leq N}$ and on the collision kernels $(\mathit{B}_{ij})_{1 \leq i \leq N}$, and is independent of parameters $\varepsilon$ and $\delta_{MS}$. The upper bound 
$\bar{\delta}_{MS}$ depends on the choice of $\delta_B$.
\end{theorem}

    Here we present the sketch of our proofs and the novelties. 
    \begin{enumerate}
        \item In the non-isothermal Maxwell-Stefan system \eqref{M-S 1}-\eqref{M-S 2}-\eqref{M-S 3}-\eqref{M-S 4}, the unknowns $(\mathbf{c}, \mathbf{u}, $ $ c_{tot}, T )$ are visually coupled with each other, which presents the highly nonlinear behavior. However, together with the flux-force relations \eqref{M-S 2} and the Fick's laws \eqref{M-S 4}, the unknowns $c_{tot}$ and $T$ can be completely decoupled from the functions $(\mathbf{c}, \mathbf{u})$. More precisely, from substituting the Fick's laws \eqref{M-S 4} into the mass conservations law \eqref{M-S 1}, one follows that the total concentration $c_{tot}$ subjects to the heat equation \eqref{orig equ c-tot}, and the temperature $T$ obeys the equation \eqref{orig equ T} with coefficients depending only on $c_{tot}$. By using the property of the Maxwell-Stefan (in short, MS) matrix $A$, the flux-force relation \eqref{M-S 2} provides an important relation $\nabla_x (c_{tot} T)=0$ with the total concentration $c_{tot}=\sum_{i=1}^N c_i$. This relation is derived from the kernel space of the MS matrix $A$, i.e. the space $\mathrm{Span} (\mathbf{1})$. The relation $\nabla_x (c_{tot} T)=0$ will enforce the $T$-equation \eqref{orig equ T} into the form \eqref{another eq for T}, which is a nonlinear backward parabolic equation with constant coefficients. Unfortunately, the backward form \eqref{another eq for T} cannot directly be used to solve the temperature globally in time. This is completely different with the isothermal case. 

        For this non-isothermal MS system, the work \cite{Hutridurga-Salvarani-AML2018} used a matrix transformation method. They used fluxes $\mathbf{J}=\mathbf{c} \mathbf{u}$ representation for the equation \eqref{M-S 2}, instead of the velocities $\mathbf{u}$, that was, 
         \begin{equation*}
        \nabla_x \mathbf{c} = F \mathbf{J}.
        \end{equation*}
        The matrix $F$ corresponds to the MS matrix, defined as 
        \begin{equation*}
        F_{ij} :=  \begin{cases}
        \frac{ c_i}{ D_{ij}} & \text{for } j \neq i, \; i, j = 1, \ldots, N, \\
        -\sum_{r \neq i} \frac{c_r}{ D_{ir}} & \text{for } j = i = 1, \ldots, N,
        \end{cases}
        \end{equation*}
        where $\{ D_{ij} \}_{1\leq i,j \leq N}$ are the diffusion coefficients. The known total concentration $c_{tot}$ implied that one only need to solve $N-1$ concentrations, i.e. $\mathbf{c}^\prime = \mathbf{c}^\prime =(c_1, \dots c_{N-1})$. The fluxes were same as the Fick's laws 
        $\sum_{i=1}^N J_i = -\alpha \nabla_x c_{tot}$ indicated the sum of fluxes, and they denoted $\mathbf{J}^\prime = (J_1, \dots, J_{N-1})$. After matrix transformation, they finally obtained equations for concentrations $\mathbf{c}^\prime$, 
        \begin{equation}
            \partial_t \mathbf{c}^\prime -\nabla_x (T F_0^{-1} \nabla_x \mathbf{c}^\prime ) = \mathbf{r} (\mathbf{c}^\prime).
            \label{quasi-parab eq by matrix trans}
        \end{equation}
        The $(N-1) \times (N-1 ) $ matrix $F_0$ was derived from matrix $F$, which is invertible with spectra $\sigma(F_0) \subset [ \delta, \eta)$ for some positive constants $\delta$ and $\eta$. The right hand side was the lower order term, 
        \begin{equation*}
            \mathbf{r} (\mathbf{c}^\prime) =\nabla_x \cdot (F_0^{-1} (\mathbf{c}^\prime \otimes \nabla_x T)) +\alpha \nabla_x 
            \cdot ( \hat{\mathbf{c}}^\prime \otimes \nabla_x c_{tot}),
        \end{equation*}
        with $\hat{\mathbf{c}}^\prime =(\frac{c_1}{D_{1N}}, \dots, 
        \frac{c_{N-1}}{D_{N-1, N}})$. Since they pre-solved the advection equation for $T$, the unique local-in-time in $L^p$ solution for this quasi-linear parabolic system \eqref{quasi-parab eq by matrix trans} was derived in \cite{Hutridurga-Salvarani-AML2018}. 
        However, this method seems fail when we looking for a global-in-time solution.

        Recently, the work \cite{briant2022perturbativ} proposed a different method in isothermal case, in which the total concentration $c_{tot}$ is a constant. 
        In energy estimate, they computed the inner product between the mass conservation law and $\mathbf{c}$. Using integration by parts, 
        the term $\langle \mathbf{c}\mathbf{u}, \nabla_x \mathbf{c}\rangle$ arose. 
        They replaced the term $\nabla_x \mathbf{c}$ with the flux force relations $\nabla_x \mathbf{c}= A(\mathbf{c}) \mathbf{u}$. 
        The matrix $A$ shares the same structure with the MS matrix here. 
        The non-positivity of $A$, hence, for any $\mathbf{c} \geq 0$ and 
        $\mathbf{U} \in \mathrm{Span} (\mathbf{1}) ^\perp$, 
                \begin{equation*}
                    \langle A(\mathbf{c}) \mathbf{U}, \mathbf{U} \rangle \leq  -\lambda_A 
                    (\min_{1\leq i \leq N} c_i)^2 
                \| \mathbf{\mathbf{U}} \|,
                \end{equation*}
        provided coercivity, which supplied the dissipation effect on velocities. However, for the non-isothermal case here, the replacement is much complex, i.e. $\nabla_x \mathbf{c} = T^{\gamma/2 -1} A (\mathbf{c}) \mathbf{U} -
        \frac{\nabla_x T}{T}\mathbf{c}$. We denote $\mathbf{U}$ is the projection of $\mathbf{u}$ onto $\mathrm{Span} (\mathbf{1})^\perp$, 
        and apply the velocities decomposition 
        $\mathbf{u} = \left ( \mathbf{U} - \frac{ \langle \mathbf{c}, \mathbf{U} 
        \rangle }{c_{tot}} \mathbf{1} \right ) - \frac{ \alpha \nabla_x c_{tot}}{c_{tot}} 
        \mathbf{1} $ by using the Fick's laws \eqref{M-S 4}. Then the first part of the replacement is similar to the isothermal case, which generates the dissipation on velocities with a temperature weight due to the non-isothermal effect, i.e. 
        $\| \mathbf{U}\|_{H_x^s (\omega)}$ with $\omega := T^{\gamma/2-1}$. Note that the second part meets the same problem with the $T$-equation \eqref{orig equ T}, i.e. the lacks of the dissipation term about $\nabla_x T$. Fortunately, the relation $\nabla_x (c_{tot} T) = 0$ provides a link between $\nabla_x c_{tot}$ and $\nabla_x T $, in a nonlinear way. By transforming the term $\nabla_x T$ into $\nabla_x c_{tot}$, we can use the dissipation contributed by the heat equation for $c_{tot}$ as a supplement.

        Since we consider the solutions around a constant state, the equations satisfied by the perturbations are split into linear and nonlinear parts. In energy estimate, the nonlinear parts in equations can always be controlled by terms of the form that the dissipation rate functional $\mathscr{D}_s$ multiplied by the energy functional $\mathscr{E}_s$. They could be absorbed by the dissipation contributed by the linear parts when $\mathscr{E}_s$ is small enough. 

        While designing the approximated scheme of the non-isothermal MS system \eqref{M-S 1}-\eqref{M-S 2}-\eqref{M-S 3}-\eqref{M-S 4}, we construct a so-called {\em pseudo-nonlinear iterative approximated scheme} \eqref{solvability of n+1}-\eqref{iter mass conser}-\eqref{iter T equa}-\eqref{equ for total con of n+1} later. Usually, the various linear iterative approximated schemes are employed to study the existence theory of the nonlinear PDEs, such as Navier-Stokes equations, and so on. To our best acknowledgment, the only nonlinear iterative approximated approach was developed by the last two authors of current paper in the work \cite{JL-2019-SIMA} (see also \cite{JLT-M3AS-2019}), in which the nonlinear iterative approximated scheme was designed to guarantee the highly nonlinear constraint $|d| = 1$ of the orientation for liquid crystal molecules. In current work, the approximated equations \eqref{iter mass conser} and \eqref{equ for total con of n+1} are linear forms. However, the temperature approximated equation \eqref{iter T equa} are explicitly nonlinear (involving the nonlinear forms $|\nabla_x \tilde{c}_{tot}^{(n+1)}|^2$ and $\nabla_x \tilde{c}_{tot}^{(n+1)} \cdot \nabla_x \tilde{T}^{(n+1)}$), which is such that all gradient forms can be controlled by the dissipation rate $\mathscr{D}_s^{(n+1)}$ defined in \eqref{dissiaption functional n+1}. Fortunately, the equation \eqref{equ for total con of n+1} for $\tilde{c}_{tot}^{(n+1)}$ is completely decoupled with the other iterative unknowns, which such that the  temperature approximated equation \eqref{iter T equa} can be solved by the linear theory.

        \item In the diffusion asymptotics process, the main difference to the diffusion limit 
        from the Boltzmann equation for mono-species to incompressible Navier-Stokes 
        equation is that the Maxwellian \eqref{def of M-i} is not an equilibrium to the mixtures. 
        This leads to two difficulties. One is that the linear operator $\mathbf{L}^\varepsilon$ is no longer self-adjoint in $L^2(\mathbb{R}^3)$, and it is hard to explicitly describe its kernel space. Other is that the source term $\mathbf{S}^\varepsilon$ arises in the perturbed equations due to the non-constants macroscopic quantities. The control for the source term is of order $\mathcal{O} (\delta_{MS})$. Thus the a prior estimate for $\mathbf{f}$ remains a term of $\mathcal{O} (\delta_{MS})$, which can not obtain an exponential decay in time. 

        Luckily, the work \cite{Bondesan2020CPAAnonequilibrium} noticed the constructed Maxwellian was close to the global equilibrium $\bm{\mu}$, which derived a local coercivity for the corresponding linear operator. We therefore use the decomposition $\mathbf{M}^\varepsilon = \bar{\mathbf{c}} 
        \bm{\mathcal{M}} + \bar{\mathbf{c}} (\bm{\mathcal{M}}^\varepsilon - 
        \bm{\mathcal{M}}) + \varepsilon \tilde{\mathbf{c}} \bm{\mathcal{M}}^\varepsilon$ in our analysis, where $\bm{\mathcal{M}}=\mathbf{M}(\mathbf{1}, 0, \mathbf{1})$ and 
        $\bm{\mathcal{M}}^\varepsilon = \mathbf{M}(\mathbf{1}, 
        \varepsilon \mathbf{u}, T\mathbf{1})$. The operator $\mathbf{L}$ related to the first item has an explicit spectral gap property, when the cross-sections $\{ \mathit{B}_{ij}\}_{1\leq i,j \leq N}$ are hard potential or Maxwellian potential \cite{briant2016ARMAglobal,mouhot2016SIAMJMAhypercoercivity}. The second term $\bm{\mathcal{M}}^\varepsilon - \bm{\mathcal{M}}$ is of order $\mathcal{O}(\varepsilon)$, and the third term is clearly $\mathcal{O}(\varepsilon)$. Similar method as that in \cite{Bondesan2020CPAAnonequilibrium} allows us to recover a coercivity property for the linear operator $\mathbf{L}^\varepsilon$. Moreover, the macroscopic quantities of $\mathbf{f}$ have the same order as that of the source term $\mathbf{S}^\varepsilon$. As a result, they have the same order as that of the local Maxwellian $\mathbf{M}^\varepsilon$.

        Our analysis follows the structure of the isothermal case \cite{briant2021stability}. 
        We now display the microscopic dissipative effect in the perturbed equations \eqref{pert Bz}. First the decomposition $\mathbf{L}^\varepsilon = \mathbf{K}^\varepsilon- \bm{\nu}^\varepsilon$ is applied, where $\bm{\nu}^\varepsilon$ is the multiplicative part. The Carleman representation allows us to rewrite the operator $\mathbf{K}^\varepsilon$ as a kernel operator. Due to the aimed functional \eqref{express of functional H-eps-s}, the lower bounds for the terms of $v$ derivatives corresponding to $\bm{\nu}^\varepsilon$ provide negative terms about $\|\partial_v^\beta \partial_x^\alpha \mathbf{f}\|_{H_{x,v}^s (\langle v\rangle^\gamma)}$ for $|\beta|\neq 0$. The commutators in this functional involve the negative terms associated with the $x$ derivatives, i.e. $\|\partial_x^\alpha \mathbf{f} \|_{H_{x,v}^s}$ for $|\alpha| \neq 0$. They can be bounded from below by $\|\partial_x^\alpha \mathbf{f} \|_{H_{x,v}^s(\langle v \rangle^\gamma)}$ for $\alpha \neq 0$, applying the coercivity property that contributes terms about $\|\partial_x^\alpha \mathbf{f}^\perp \|_{H_{x,v}^s(\langle v \rangle^\gamma)}$, together with the equivalence between two norms $\|\cdot \|_{L_{x,v}}^2$ and $\|\cdot \|_{L_{x,v}(\langle v\rangle^\gamma)}^2$ on the space $\mathrm{Ker} \mathbf{L}$. Moreover, the Poincar\'e inequality \eqref{pi-L estimate for pert f} is crucial, as there lacks a negative term about $\|\bm{\pi}_{\mathbf{L}} ( \mathbf{f}) \|_{L_{x,v}^2}$. The properties for the nonlinear operator $\mathbf{\Gamma}$ are extensions of the property for the Boltzmann operator in mono-species case \cite{briant2015JDEnavierstokes}, leading us to derive a prepared estimate \eqref{Prepared priori estimate}.

        Estimates for source term $\mathbf{S}^\varepsilon$ 
        make full use of its expression \eqref{equ for source term}, 
        and the Poincaré inequality mentioned above does the same. 
        We impose a stronger control on the projection 
        $\bm{\pi}_{\mathbf{L}}(\mathbf{S}^\varepsilon)$ for terms involving 
        pure $x$ derivatives, as they did in \cite{briant2021stability}. 
        This works because $x$ derivatives do not change the property that 
        the operator $\mathbf{Q}$ is orthogonal to $\mathrm{Ker} \mathbf{L}$. 
        Moreover, all of above estimates depend on the upper and lower bounds for 
        $\mathbf{M}^\varepsilon$, which is controlled by the terms $\bm{\mathcal{M}}^\delta$ 
        and $\bm{\mathcal{M}}^{1/\delta}$ for some positive constant $\delta$ uniformly with respect to $v$. 
        The varying temperature leads to various choices for $\delta$ in these upper and lower bounds, 
        which are no longer simply confined to 
        the interval $(0,1)$ as in the isothermal case. They require more careful 
        discussions, especially in the discussions for the operator 
        $\mathbf{L}^\varepsilon$ and the source term $\mathbf{S}^\varepsilon$.

        At last, the a prior estimate for $\mathbf{f}$ derived from inequality \eqref{Prepared priori estimate} relies on the equivalence 
        between the modified norm $\|\cdot \|_{\mathcal{H}^s_\varepsilon}$ in 
        \eqref{express of functional H-eps-s} and the norm 
            \[\left ( \sum_{|\alpha| \leq s}  \|\partial_x^\alpha 
            \cdot\|_{L_{x,v}^2}^2 
         + \varepsilon^2 \sum_{ \substack{|\alpha| + |\beta| \leq s \\ 
         |\beta| \geq 1} }  
         \|\partial_v^\beta \partial_x^\alpha \cdot\|_{L_{x,v}^2}^2 \right )^{1/2},
         \]
        as stated in \cite{briant2015JDEnavierstokes}, establishing the link between $\|\cdot \|_{\mathcal{H}^s_\varepsilon}$ 
        and the standard Sobolev norm. 

        \end{enumerate}

    % We introduce a varying temperature in the convergence from the multi-species Boltzmann equations to the Maxwell-Stefan system. This varying temperature is accompanied by a non-constant total concentration, 
    % leading to a more complex Maxwell-Stefan system compared to the isothermal case. The key to this highly coupled system lies in the flux-force relations \eqref{M-S 2}. The Maxwell-Stefan matrix in these relations provides parabolicity on the space $\mathrm{Span} (\mathbf{1})^\perp$. They also provide the link between the gradients $\nabla_x c_{tot}$ and $\nabla_x T$, which is significant in energy estimate. For the decoupled equations for $c_{tot}$ and $T$, derived from the closure equations, the latter one does not have a dissipative term about $\nabla_x T$. This is where the link works, making it possible to compensate for this system by the dissipation on $\mathrm{Span} (\mathbf{1})$, provided by the heat equation for $c_{tot}$. This is achieved by transforming all terms involving $\nabla_x T$ into terms $\nabla_x c_{tot}$. In the diffusion asymptotics process, 
    % the varying temperature leads to various choices for the constant $\delta$, the exponent of the Maxwellian $\bm{\mathcal{M}}$. It is no longer simply confined to the interval $(0,1)$ as in \cite{briant2021stability}, requiring more careful discussions, especially in the discussions for the operator $\mathbf{L}^\varepsilon$ and the source term $\mathbf{S}^\varepsilon$. Moreover, the consideration of temperature results in complicated calculations in both processes.

%%%%%%%%%%%%%%%%%%%%%%%%%%%%%%%%%%%%%%%%%
\section{Global well-posedness for the Non-Isothermal Maxwell-Stefan System}\label{Sec:3} 
We first observe that the closure relations \eqref{M-S 4} (i.e., Fick's laws) can decouple the equations for the total concentration and the temperature respectively from \eqref{M-S 1} and \eqref{M-S 3}, leading to 
\begin{gather}
\partial_t c_{tot} = \alpha \Delta_x c_{tot}, \label{orig equ c-tot}\\
\partial_t T - \frac{2\alpha}{3} \frac{ \Delta_x c_{tot} }{ c_{tot} } T - \frac{5\alpha}{3} \frac{\nabla_x c_{tot}}
{c_{tot}} \cdot \nabla_x T = 0. \label{orig equ T}
\end{gather}
The derivation of \eqref{orig equ T} utilizes the equation \eqref{orig equ c-tot}, and needs the fact that $c_{tot} (t,x) \neq 0$ for any 
$(t,x) \in \mathbb{R}_+ \times \mathbb{T}^3$. 
This holds under the basic assumption on initial data that $c_{tot}^{in} (x) > 0$ for any $x \in \mathbb{T}^3$, 
due to the property of the heat equation that $c_{tot}$ satisfied. Moreover, property of the MS matrix $A$ implies a relation $\nabla_x (c_{tot} T) =0$ 
from \eqref{M-S 2}, which provides a link between $\nabla_x c_{tot}$ and $\nabla_x T$, leading to another expression of equation for $T$, i.e. 
\begin{equation}
\partial_t T + \frac{\alpha }{3} \frac{|\nabla_x T|^2 }{T}+ \frac{2 \alpha}{3} \Delta_x T=0. \label{another eq for T}
\end{equation}
This nonlinear backward parabolic equation clearly indicates that $T$-equation dose not have a dissipative term, thus requiring the dissipation contributed by \eqref{orig equ c-tot}, achieved by the link $\frac{\nabla_x T}{T} = -\frac{\nabla_x c_{tot}}{c_{tot}}$.

Since the MS matrix $A = A( \mathbf{c} )$ defined in \eqref{def of matrix A in introduction} has the kernel space $\mathrm{Ker} A = \mathrm{Span} (\mathbf{1})$, the velocity vector has a unique decomposition 
$\mathbf{u} = \bm{\pi}_A (\mathbf{u}) + \mathbf{U}$, where $\bm{\pi}_A$ is the projection operator onto 
$\mathrm{Ker} A = \mathrm{Span} (\mathbf{1})$ and $\mathbf{U}$ is the projection onto $\mathrm{Span} (\mathbf{1})^\perp$. 
The projection $\bm{\pi}_A (\mathbf{u})$ is in the form $\bm{\pi}_A (\mathbf{u}) = \pi_A (\mathbf{u}) \mathbf{1}$, thus the closure relations \eqref{M-S 4} can be expressed as  
\[-\alpha \nabla_x c_{tot} = \langle \mathbf{c}, \mathbf{U} + \pi_A (\mathbf{u}) \mathbf{1} \rangle = \langle\mathbf{c}, \mathbf{U}\rangle 
+ c_{tot} \pi_A (\mathbf{u}),
\]
deducing that 
\[\pi_A (\mathbf{u}) = - \frac{ \langle \mathbf{c}, \mathbf{U} \rangle }{ c_{tot} } - \frac{\alpha \nabla_x c_{tot}}{c_{tot}}.
\]
Then we can write the velocity vector as 
\begin{equation}
    \mathbf{u} = \left ( \mathbf{U} - \frac{ \langle \mathbf{c}, \mathbf{U} \rangle }{c_{tot}} \mathbf{1} \right ) - 
\frac{ \alpha \nabla_x c_{tot}}{c_{tot}} \mathbf{1} := \mathbf{U}^\ast - \frac{\alpha \nabla_x c_{tot}}{c_{tot}} \mathbf{1}. \label{velocity decomposition}
\end{equation}

Under the basic assumption that $\inf_{x \in \mathbb{T}^3} c^{in}_{tot}> 0$, 
by taking the above velocity decomposition into the Maxwell-Stefan system \eqref{M-S 1}-\eqref{M-S 2}-\eqref{M-S 3}-\eqref{M-S 4}, 
we can get an equivalent system: 
\begin{gather}
\partial_t \mathbf{c} + \nabla_x \cdot (\mathbf{c} \mathbf{U}^\ast) - \nabla_x \cdot (\mathbf{c} \bar{u}) = 0,  \label{vect mass cons}\\
\nabla_x (\mathbf{c} T) = T^{\gamma/2} A (\mathbf{c}) \mathbf{U} , \quad \gamma \in (-3,1],  \label{vect M-S rela}\\
\partial_t T -\frac{2\alpha}{3}  \frac{ \Delta_x c_{tot} }{c_{tot}} - \frac{5\alpha}{3} \frac{\nabla_x c_{tot}}{c_{tot}} \cdot 
\nabla_x T =0, \label{equation for temperature T}
\end{gather}
with notations 
\begin{equation}
    c_{tot} := \langle \mathbf{c}, \mathbf{1} \rangle, \quad \mathbf{U}^\ast := \mathbf{U} - \frac{ \langle \mathbf{c}, \mathbf{U} \rangle }
    {c_{tot}}\mathbf{1}, \quad \bar{u} := \frac{\alpha \nabla_x c_{tot}}{c_{tot}}. \label{vect mean velo} 
\end{equation}
We point out that the evolution equation for the total concentration \eqref{orig equ c-tot} is a natural derivation, by taking the scalar product of the mass conservation equations \eqref{vect mass cons} and $\mathbf{1}$ in $\mathbb{R}^N$. 

Next we provide a simple analysis on the stationary solution for system 
\eqref{vect mass cons}-\eqref{vect M-S rela}-\eqref{equation for temperature T}. Let $\bar{\mathbf{c}}$ to be a 
positive constant N-vector, then the total concentration $\bar{c}_{tot} = \sum_{i=1}^N \bar{c}_i$ remains a positive constant. This leads to a constant temperature, derived from the equation \eqref{orig equ T} and the relation $\nabla_x (c_{tot}T)=0$, which also is the solvability condition for $\mathbf{U}$ obtained from \eqref{vect M-S rela}. Without loss of generality, we assume that the temperature is $1$. 
The property of matrix $A$ implies that $\mathbf{U} = 0$, meaning that $(\bar{ \mathbf{c} }, \mathbf{0}, 1)$ 
is a stationary solution for the above equivalent system. In other words, $(\bar{\mathbf{c}},\mathbf{0},1)$ is a stationary solution for the original Maxwell-Stefan system \eqref{M-S 1}-\eqref{M-S 2}-\eqref{M-S 3}-\eqref{M-S 4}.

We conduct our analysis on the equivalent system, and consider the solution perturbed around the stationary solution $(\bar{ \mathbf{c}}, \mathbf{0}, 1)$ in the form 
\begin{equation}
    \begin{cases}
    \mathbf{c} (t,x) = \bar{\mathbf{c}} + \lambda \tilde{\mathbf{c}} (t,x),\\
    \mathbf{U} (t,x) = \lambda \tilde{\mathbf{U}} (t,x),\\
    T (t,x) = 1 + \lambda \tilde{T} (t,x), \label{pert form for M-S in ortho pro}
    \end{cases}
\end{equation}
where $\lambda\in(0,1]$ is a parameter. Actually, perturbation $\tilde{\mathbf{U}}$ is the projection of the velocity perturbation $\tilde{\mathbf{u}}$ in Theorem \ref{theorem for M-S in main result} onto $\mathrm{Span} (\mathbf{1})^\perp$. Assumption (A3) in this Theorem transforms to 
        \begin{itemize}
    \item[(A3$^\prime$)] \textbf{Moment compatibility:}
    \begin{equation*}
    (\bar{c}_i + \lambda \tilde{c}_i^{in}) \nabla_x \tilde{T}^{in} + (1 + \lambda \tilde{T}^{in})
         \nabla_x \tilde{c}_i^{in} = ( 1 + \lambda \tilde{T}^{in} )^{\gamma/2} \sum_{j \neq i} \frac{ (\bar{c}_i + \lambda \tilde{c}_i^{in}) 
        (\bar{c}_j + \lambda \tilde{c}_j^{in}) }{ \Delta_{ij} } \left ( \tilde{U}_j^{in} - \tilde{U}_i^{in} \right ).
    \end{equation*}
\end{itemize}

By substituting the unknowns $(\mathbf{c},\mathbf{U},T )$ with $\mathbf{U} \in \mathrm{Span} (\mathbf{1})^\perp$ into the equations 
\eqref{vect mass cons}-\eqref{vect M-S rela}-\eqref{equation for temperature T} in the form \eqref{pert form for M-S in ortho pro}, 
we can derive the equations for the perturbations $(\tilde{\mathbf{c}},\tilde{\mathbf{U}},\tilde{T})$. 
We first present the notations in \eqref{vect mean velo} in a perturbative form, dividing them into two parts that 
depend linearly and nonlinearly on the perturbations,
\begin{equation}
    c_{tot} = \bar{c}_{tot} + \lambda \tilde{c}_{tot}, \quad \text{with } \bar{c}_{tot}=\sum_{i=1}^N \bar{c}_i, 
    \text{ and } \tilde{c}_{tot}=\sum_{i=1}^N \tilde{c}_i,    \label{express of c-tot}
\end{equation}
\begin{equation}
\begin{aligned}
    \bar{u} & = \alpha \frac{\lambda \nabla_x \tilde{c}_{tot} }{ \bar{c}_{tot} + \lambda \tilde{c}_{tot}}
    = \lambda \alpha \frac{\nabla_x \tilde{c}_{tot}}{ \bar{c}_{tot} } - \lambda^2 \frac{\alpha \tilde{c}_{tot} \nabla_x \tilde{c}_{tot}}
    {\bar{c}_{tot} (\bar{c}_{tot} + \lambda \tilde{c}_{tot})}\\
    & := \lambda \tilde{\bar{u}} - \lambda^2 N_1 ( \tilde{\bar{u}} ),  \label{express of u}
\end{aligned}
\end{equation}
\begin{equation}
    \begin{aligned}
     \mathbf{U}^\ast & =\lambda \tilde{\mathbf{U}} - \lambda \frac{\langle \bar{\mathbf{c}} + \lambda \tilde{\mathbf{c}}, 
     \tilde{\mathbf{U}}  \rangle}{ \bar{c}_{tot} + \lambda \tilde{c}_{tot} } \mathbf{1}\\ 
     & = \lambda \left (\tilde{\mathbf{U}} - \frac{ \langle \bar{\mathbf{c}}, \tilde{\mathbf{U}} \rangle}{ \bar{c}_{tot} } \mathbf{1} 
     \right ) - \lambda^2 \left ( \frac{\langle \tilde{\mathbf{c}}, \tilde{\mathbf{U}} \rangle}{ \bar{c}_{tot} } 
     - \frac{ \tilde{c}_{tot} \langle \bar{\mathbf{c}} + \lambda \tilde{\mathbf{c}}, \tilde{\mathbf{U}} \rangle }{ \bar{c}_{tot} 
     (\bar{c}_{tot} + \lambda \tilde{c}_{tot}) } \right ) \mathbf{1}\\
        & := \lambda \tilde{\mathbf{U}}^\ast - \lambda^2 N_2 (\tilde{\mathbf{U}}) \mathbf{1}.  \label{express of U}
\end{aligned}
\end{equation}
The operator $N_1$ depends on $\tilde{c}_{tot}$ in a nonlinear form, while the operator $N_2$ nonlinearly depends on $\tilde{\mathbf{c}}$ and $\tilde{\mathbf{U}}$. 
Under the basic assumption that $\inf_{x \in \mathbb{T}^3} \left (\bar{c}_{tot} + \lambda \tilde{c}_{tot}^{in} \right ) (x) > 0$ 
for all $\lambda \in (0,1]$, i.e. a direct derivation of assumption (A1) in Theorem \ref{theorem for M-S in main result}, 
the equations for the perturbations can be written as follows
\begin{equation}
    \begin{aligned}
    \partial_t \tilde{\mathbf{c}} + & \bar{\mathbf{c}} (\nabla_x \cdot \tilde{\mathbf{U}}^*)-\bar{\mathbf{c}} (\nabla_x\cdot\tilde{\bar{u}})\\
    &+\lambda \nabla_x\cdot(\tilde{\mathbf{c}} \tilde{\mathbf{U}}^\ast) - \lambda \nabla_x \cdot[ (\bar{\mathbf{c}} + \lambda 
    \tilde{\mathbf{c}}) N_2 (\tilde{\mathbf{U}}) ] - \lambda \nabla_x \cdot ( \tilde{\mathbf{c}} \tilde{\bar{u}} ) + \lambda \nabla_x \cdot 
    [ (\bar{\mathbf{c}} + \lambda \tilde{\mathbf{c}}) N_1 (\bar{u}) ] = 0,  \label{per mass-linear}
\end{aligned}
\end{equation}
the flux-force relations  
\begin{equation}
    (\bar{\mathbf{c}} + \lambda \tilde{\mathbf{c}}) \nabla_x \tilde{T} + (1 + \lambda \tilde{T}) \nabla_x \tilde{\mathbf{c}}
    = (1 + \lambda \tilde{T})^{\gamma/2} A (\bar{\mathbf{c}} + \lambda \tilde{\mathbf{c}} ) \tilde{\mathbf{U}}, \quad \gamma\in(-3,1], 
    \label{equa for M-S rela}
\end{equation}
and the evolution equation for temperature
\begin{equation}
    \partial_t \tilde{T} - \frac{2\alpha}{3} \nabla_x \cdot \left[ \frac{ \nabla_x \tilde{c}_{tot} }{ \bar{c}_{tot} + \lambda 
    \tilde{c}_{tot} } (1 + \lambda \tilde{T})  \right] - \frac{2\alpha \lambda}{3}
    \frac{ |\nabla_x \tilde{c}_{tot}|^2 }{ (\bar{c}_{tot} + \lambda \tilde{c}_{tot})^2 } (1 + \lambda \tilde{T})
    - \alpha \lambda \frac{\nabla_x \tilde{c}_{tot}}{ \bar{c}_{tot} + \lambda \tilde{c}_{tot} } \cdot \nabla_x \tilde{T} = 0.
    \label{equa for per T}
\end{equation}
Here uses the equality 
\begin{equation}
\nabla_x \cdot \left[ \frac{ \nabla_x \tilde{c}_{tot} }{\bar{c}_{tot} + \lambda \tilde{c}_{tot}} (1 + \lambda \tilde{T})  \right] 
= \frac{ \Delta_x \tilde{c}_{tot} }{\bar{c}_{tot} + \lambda \tilde{c}_{tot}} (1 + \lambda \tilde{T}) 
- \frac{ \lambda |\nabla_x \tilde{c}_{tot}|^2 }{ (\bar{c}_{tot} + \lambda \tilde{c}_{tot})^2 }( 1 + \lambda \tilde{T})
+ \frac{\lambda \nabla_x \tilde{c}_{tot} \cdot \nabla_x \tilde{T}}{ \bar{c}_{tot} + \lambda \tilde{c}_{tot} }.
\label{transformation used in T}
\end{equation}
The mass conservation equations are split into linear and nonlinear parts, with respect to the perturbations, facilitating subsequent calculations. We also give its direct expression 
\[ \partial_t \tilde{\mathbf{c}} +\nabla_x\cdot \left[(\bar{\mathbf{c}}+ \lambda \tilde{\mathbf{c}})(\tilde{\mathbf{U}}-\frac{\langle 
\bar{\mathbf{c}}+\lambda \tilde{\mathbf{c}}, \tilde{\mathbf{U}}\rangle  }{\bar{c}_{tot} +\lambda \tilde{c}_{tot}} \mathbf{1})\right] 
-\nabla_x \cdot \left[ (\bar{\mathbf{c}} + \lambda \tilde{\mathbf{c}} ) \frac{\alpha \nabla_x \tilde{c}_{tot}}{\bar{c}_{tot} 
+\lambda \tilde{c}_{tot}} \mathbf{1} \right] =0.
\]
It also has a natural deduction for the total concentration perturbation
\begin{equation}
    \partial_t \tilde{c}_{tot} = \alpha \Delta_x \tilde{c}_{tot}. \label{equation for per total concentration}
\end{equation}
Furthermore, we give another expression of the flux-force relations \eqref{equa for M-S rela}
\begin{equation}
    \nabla_x \tilde{\mathbf{c}} = ( 1 + \lambda \tilde{T} )^{\gamma/2 -1} A (\bar{\mathbf{c}} + \lambda \tilde{\mathbf{c}}) 
    \tilde{\mathbf{U}} + \frac{\nabla_x \tilde{c}_{tot}}{\bar{c}_{tot} + \lambda \tilde{c}_{tot}} (\bar{\mathbf{c}} + \lambda \tilde{\mathbf{c}}),     \label{per M-S rela}
\end{equation}
which applies an important relation, specifically the solvability condition for $\tilde{\mathbf{U}}$ 
\begin{equation}
    \frac{ \nabla_x \tilde{c}_{tot} }{ \bar{c}_{tot} + \lambda \tilde{c}_{tot} } = - \frac{\nabla_x \tilde{T}}{ 1 + \lambda \tilde{T} }. 
    \label{implicit rela}
\end{equation} 

Now we focus on the Cauchy problem of the system of equations 
\eqref{per mass-linear}-\eqref{equa for M-S rela}-\eqref{equa for per T} for the 
unknowns $(\tilde{\mathbf{c}}, \tilde{\mathbf{U}}, \tilde{T})$ with $\tilde{\mathbf{U}} \in \mathrm{Span} (\mathbf{1})^\perp$, 
given the initial data
\begin{equation}
    \begin{aligned}
      \tilde{\mathbf{c}} (0,x) = \tilde{\mathbf{c}}^{in} (x), \quad \tilde{\mathbf{U}} (0,x) & = \tilde{\mathbf{U}}^{in} (x), 
      \quad  \tilde{T} (0,x) = \tilde{T}^{in} (x) \quad \text{a.e. on } \mathbb{T}^3,\\
    \text{and } &\langle \tilde{\mathbf{U}}^{in} , \mathbf{1} \rangle = 0, \quad \forall x \in \mathbb{T}^3,\label{initial data set}
    \end{aligned}
    \end{equation}
    satisfying assumptions (A1), (A2) and (A3$^\prime$). 
The orthogonal reformulation of Theorem \ref{theorem for M-S in main result} then takes the following form.
\begin{theorem}\label{theorem for M-S in orthogonal version}
        Let $s > 3$ be an integer, $\lambda \in (0,1]$ and $\bar{\mathbf{c}}>0$. There exist explicit constants $\delta_{MS},d_1,d_2,\chi>0$ 
        such that if the initial data 
        $(\tilde{\mathbf{c}}^{in}, \tilde{\mathbf{U}}^{in}, \tilde{T}^{in}) \in H_x^s (\mathbb{T}^3)
        \times H_x^{s-1} (\mathbb{T}^3) \times H_x^s (\mathbb{T}^3)$ satisfies (A1)-(A2)-(A3$^\prime$) and the smallness assumption
        \begin{equation}
        \| \tilde{\mathbf{c}}^{in} \|^2_{H_x^s (\bar{\mathbf{c}}^{-1}) } + \|\tilde{T}^{in}\|^2_{H_x^s} + \chi \|\sum_{i=1}^N 
        \tilde{c}_i^{in} \|^2_{H_x^s} \leq \delta_{MS}^2, \label{smallness for E-0 in ortho thm}
        \end{equation}
        then there exists a unique classical solution $(\tilde{\mathbf{c}}, \tilde{\mathbf{U}}, \tilde{T} ) \in L^\infty ( \mathbb{R}_+ ; H_x^s({\mathbb{T}^3}) ) \times L^\infty ( \mathbb{R}_+; H_x^{s-1} ({\mathbb{T}^3}) ) \times L^\infty 
        ( \mathbb{R}_+; H_x^s ({\mathbb{T}^3}) )$ to the system of equations \eqref{per mass-linear}-\eqref{equa for M-S rela}-\eqref{equa for per T} with the initial data satisfying \eqref{initial data set}. In particular, for almost everywhere $(t,x) \in \mathbb{R}_+ \times \mathbb{T}^3$, we have 
        \[ \min_{1 \leq i \leq N}  \bar{c}_i + \lambda \tilde{c}_i (t,x) > 0, \quad 1 + \lambda \tilde{T} (t,x) > 0, \quad \tilde{\mathbf{U}} (t,x) \in \mathrm{Span} (\mathbf{1}) ^\perp.
        \]
        Moreover, the following estimate holds for any $t\geq0$,
       \[ \mathscr{E}_s (t) + \int_0^t \mathscr{D}_s (\tau) \,d\tau \leq \mathscr{E}_s (0) \leq \delta_{MS}^2,
       \]
       where the functionals $\mathscr{E}_s$ and $\mathscr{D}_s$ are defined in \eqref{def of MS functionals in intro}.
    \end{theorem}

\subsection{Energy Estimate for the Maxwell-Stefan System}
Before starting to calculate the energy estimate, we first refer Proposition 3.1, Proposition 3.2 and Proposition 3.3 in \cite{briant2022perturbativ}:
\begin{proposition}\label{prop A}
    For any $\mathbf{c} \geq 0$ the matrix $A (\mathbf{c})$ is nonpositive, in the sense that there exist two positive constants $\lambda_A$ 
    and $\mu_A$ such that, for any 
    $\mathbf{X} \in \mathbb{R}^N$,
    \[  \| A (\mathbf{c}) \mathbf{X} \| \leq \mu_A \langle \mathbf{c}, \mathbf{1} \rangle^2 \|\mathbf{X}\|,
    \]
    \[  \langle \mathbf{X}, A(\mathbf{c}) \mathbf{X} \rangle \leq - \lambda_A ( \min_{1\leq i\leq N} c_i )^2 
    [ \| \mathbf{X} \|^2 - \langle\mathbf{X}, \mathbf{1} \rangle^2 ] \leq 0.
    \]
  \end{proposition}
  \begin{proposition}\label{prpo scalar estimate for A}
  Consider a multi-index $\beta \in \mathbb{N}^3$ and let $\mathbf{c}, \mathbf{U} \in H^{ |\beta| } (\mathbb{T}^3)$, with $\mathbf{c} \geq 0$. 
  Then, for any $\mathbf{X} \in \mathbb{R}^N$,
  \begin{align*}
    \langle \partial_x^\beta [ A (\mathbf{c}) \mathbf{U} ], \mathbf{X} \rangle \leq & \langle A (\mathbf{c}) \partial_x^\beta \mathbf{U}, 
    \mathbf{X} \rangle + 2N^2 \mu_A \langle \mathbf{c}, \mathbf{1}\rangle \|\mathbf{X}\|3^{|\beta|} 
    \sum_{ \substack{\beta_1 + \beta_3 = \beta \\ |\beta_1| \geq 1} } \|\partial_x^{\beta_1} \mathbf{c}\| \|\partial_x^{\beta_3} \mathbf{U}\|\\
      & + N^2 \mu_A \|\mathbf{X}\| 3^{|\beta|} \sum{ \substack{\beta_1 + \beta_2 + \beta_3 = \beta \\ |\beta_1|, |\beta_2| \geq 1} }
      \|\partial_x^{\beta_1} \mathbf{c}\| \|\partial_x^{\beta_2}\mathbf{c}\|  \|\partial_x^{\beta_3} \mathbf{U}\|.
  \end{align*}
  \end{proposition}
  \begin{proposition}\label{prop A-1}
  For any $\mathbf{c} \in (\mathbb{R}_+^\ast)^N$ and any $\mathbf{U} \in \mathrm{Span} (\mathbf{1})^\perp$, the following estimates hold:
  \[  \|A(\mathbf{c})^{-1} \mathbf{U}\| \leq \frac{1}{ \lambda_A (\min_{1 \leq i \leq N} c_i)^2 } \|\mathbf{U}\|,
  \]
  \[\langle A (\mathbf{c})^{-1} \mathbf{U}, \mathbf{U} \rangle \leq -\frac{\lambda_A (\min_{1 \leq i \leq N} c_i )^2}
  {\mu_A^2 \langle \mathbf{c}, \mathbf{1}\rangle^4}  \|\mathbf{U}\|^2. 
  \]
  \end{proposition}
%%%%%%%%%%%%%%%%%%%%%%%%%%%%%%%%%%%%%%%%%%%%%%%%%%%%%%%%%%%%%%%%%%%%%%%%%%%%%%%%%%%%%
Furthermore, we present some inequalities frequently used in the following computations. Lemma 3.1 in \cite{jiang2021CMS} 
implies that if $f \in \mathbb{R}^N \rightarrow \mathbb{R}$ is a smooth function and 
$\mathbf{n} = (n_1, \dots, n_N) : \mathbb{T}^3 \rightarrow \mathbb{R}^N$ is a vector-valued function belonging to $H^{|\beta|}$ 
with $\beta \neq 0$, then for any $|\beta| \leq s$ $(s \geq 2)$ we have the following inequality 
\begin{equation*}
\|\partial_x^\beta f (\mathbf{n}) \|_{L_x^2} \lesssim \sum_{i=1}^N \sum_{j=1}^s \|\frac{\partial^j f}{\partial n_i^j} (\mathbf{n}) 
\|_{L_x^\infty} \|\nabla_x n_i\|_{H^{s-1}} (1 + \|n_i\|_{ H^{s-1} }^{s-1}).
\end{equation*}
Inspired from this inequality, we define $H_1(x):=(1+x)^{\gamma/2-1}$, which enables us to deduce that for any integer $s>3$,
\begin{equation} 
    \|(1 + \lambda \tilde{T})^{\gamma/2 - 1}\|_{ \dot{H}_x^s } \leq \sum_{ \substack{|\beta| \leq s \\ |\beta| \neq 0} }
    \|\partial^\beta H_1\|_{L_x^\infty} \|\nabla \tilde{T}\|_{H_x^{s-1}} (1 + \|\nabla \tilde{T}\|_{H_x^{s-1}}^{s-1} ), \quad 
    \forall \lambda \in (0,1]. \label{per T homo Hs esti}
\end{equation}
And its $H_x^s$ norm is bounded by
\[ \| (1 + \lambda \tilde{T})^{ \gamma/2 - 1}\|_{H_x^s}^2 \leq \| (1 + \lambda \tilde{T})^{\gamma/2-1} \|_{\dot{H}_x^s}^2 
+ \|H_1\|_{L_x^\infty}^2 |\mathbb{T}^3|^2.
\]
We introduce a functional $G_1 (I_1) = \sum_{ |\beta| \leq s } \|\partial^\beta H_1 \|_{L^\infty (I_1)}$. Since the composite function $H_1(\lambda \tilde{T})$ is considered here, we choose the internal to be 
$I_1 = \left[ - \|\tilde{T}\|_{L_x^\infty}, + \|\tilde{T}\|_{L_x^\infty} \right]$. Then we have 
\begin{equation}
    \|(1 + \lambda \tilde{T})^{\gamma/2 - 1}\|_{H_x^s} \leq C_{|\mathbb{T}^3|} G_1 (I_1) (1 + \|\tilde{T}\|_{H_x^s} + \|\tilde{T}\|_{H_x^s}^s),
    \quad \forall \lambda \in (0,1],  \label{per T Hs esti}
\end{equation}
where $C_{|\mathbb{T}^3|} > 0$ is a constant that depends on $s, N$ and $|\mathbb{T}^3|$. Similarly, we define $H_2(x) := \frac{1}{\bar{c}_{tot} + x}$, and the following inequality holds
\begin{equation}
\| \frac{1}{\bar{c}_{tot} + \lambda \tilde{c}_{tot}} \|_{H_x^s} \leq C_{|\mathbb{T}^3|} G_2 (I_2) (1 + \|\tilde{c}_{tot}\|_{H_x^s} 
+ \|\tilde{c}_{tot}\|_{H_x^s}^s), \quad \forall \lambda \in (0,1].  \label{estimate frac c-tot Hs}
\end{equation}
The functional is defined as $G_2 (I_2) = \sum_{ |\beta| \leq s}  \|\partial^\beta H_2\|_{L^\infty (I_2)}$, and 
the interval is chosen to be $I_2 = \left[ -\|\tilde{c}_{tot}\|_{L_x^\infty}, + \|\tilde{c}_{tot}\|_{L_x^\infty} \right]$. 
We will omit the intervals $I_1$ and $I_2$ in subsequent computations, discussing them only when necessary.

%%%%%%%%%%%%%%%%%%%%%%%%%%%%%%%%%%%%%%%%%%%%%%%%%%%%%%%%%%%%%%%%%%%%%%%%%%%%%%%%%%%%%%%%%%%%%%%%
Here we recall functionals defined in \eqref{def of MS functionals in intro}, in which the energy functional is
\begin{equation}
    \mathscr{E}_s (t) := \|\tilde{\mathbf{c}}\|_{H_x^s (\bar{\mathbf{c}}^{-1}) }^2 (t) + \|\tilde{T}\|_{H_x^s}^2 (t) 
    +\chi \| \tilde{c}_{tot} \|_{H_x^s}^2 (t) ,
\end{equation}
where $\chi$ is a positive constant defined as
\begin{equation}
\chi = \frac{4}{3\bar{c}_{tot}^2} + \frac{3 (\frac{3}{2})^{1-\gamma/2} }{ \alpha \lambda_A (\min_{1 \leq i\leq N}\bar{c}_i)^2 }.
\label{defi of chi}
\end{equation}
The energy dissipation rate functional is
\begin{equation}
\mathscr{D}_s (t) := d_1 \|\tilde{\mathbf{U}}\|^2_{H_x^s(\omega)} (t) + d_2 \|\nabla_x \tilde{c}_{tot}\|_{H_x^s}^2 (t),
\end{equation}
with the notation 
\begin{equation}
\omega := (1 + \lambda \tilde{T})^{\gamma/2 - 1}. \label{def of omega}
\end{equation}
The constants $d_1$ and $d_2$ are defined as 
\begin{align}
d_1 := & \frac{\lambda_A (\min_{1\leq i\leq N}\bar{c}_i)^2 }{3},\\
d_2 := & \frac{\alpha}{\bar{c}_{tot}} + \frac{2 \alpha}{3\bar{c}_{tot}^2} + \frac{3 (\frac{3}{2})^{1 - \gamma/2} }{2 \lambda_A 
(\min_{1 \leq i \leq N} \bar{c}_i)^2 }.
\end{align}

Here we present the energy estimate for the system of equations that the perturbations $(\tilde{\mathbf{c}}, \tilde{\mathbf{U}}, \tilde{T})$ satisfied.
\begin{proposition}\label{proposition for energy esti for M-S}
    Let $s>3$ be an integer and $\bar{\mathbf{c}} >0$, assume $(\tilde{\mathbf{c}}, \tilde{\mathbf{U}}, \tilde{T})$ with 
    $ \tilde{\mathbf{U}} \in \mathrm{Span} (\mathbf{1})^\perp$  is a sufficiently smooth solution for the perturbed system \eqref{per mass-linear}-\eqref{equa for M-S rela}-\eqref{equa for per T} with the initial data $(\tilde{\mathbf{c}}^{in}, \tilde{\mathbf{U}}^{in}, \tilde{T}^{in})$ satisfying (A1)-(A2)-(A3$^\prime$). 
    There exists a positive constant $C_s$, such that 
    \begin{equation}
    \begin{aligned}
    \frac{1}{2} \frac{\d}{\d t} \mathscr{E}_s & (t) + d_1 \|\tilde{\mathbf{U}}\|_{H_x^s(\omega)}^2 (t) + \left ( 
      \frac{\alpha}{\bar{c}_{tot}} + \frac{2\alpha}{ 3 \bar{c}_{tot}^2 } + \frac{ 6 (\frac{3}{2})^{1 - \frac{\gamma}{2}} 
      - 3 (1 + C_{Sob} \mathscr{E}_s^{1/2} (t) )^{1 - \frac{\gamma}{2}} }{2 \lambda_A ( \min_{1 \leq i \leq N} \bar{c}_i )^2}  \right ) 
      \|\nabla_x \tilde{c}_{tot}\|_{H_x^s}^2 (t)\\
&\leq C_s (1 + G_1 + G_2 + G_1 G_2 + G_2^2) (1 + \mathscr{E}_s^{1/2} (t))^{1 - \gamma/2}(1 + \mathscr{E}_s^{1/2} (t) 
     + \mathscr{E}_s^{s + 2} (t)) \mathscr{E}_s^{1/2} (t) \mathscr{D}_s (t).  \label{original energy esti}
    \end{aligned}
\end{equation}
The constant $C_s$ depends on $s, N ,\mu_A ,\lambda_A, \alpha, \bar{\mathbf{c}}, C_{Sob}, |\mathbb{T}^3|$, and 
is independent of the parameter $\lambda$. The functionals 
$G_1,G_2$ are defined as above, which depend on the ranges of 
$\tilde{T}$ and $\tilde{c}_{tot}$. In particular, if the functional 
    $\mathscr{E}_s (0) := \| \tilde{\mathbf{c}}^{in} \|_{H^s (\bar{\mathbf{c}}^{-1}) }^2 + \| \tilde{T}^{in} \|_{H^s}^2 
    + \chi \| \sum_{i=1}^N \tilde{c}_i^{in}\|_{H^s}^2 $ is bounded by $\delta_{MS}^2$, and the bound $\delta_{MS}$ satisfies
 \begin{equation}
 \begin{aligned}
   \max\{ 1, \frac{1}{\sqrt{\chi}}\}  C_{sob} \delta_{MS} &\leq \frac{ \min_{1 \leq i \leq N} \bar{c}_i }{4 \max_{1 \leq i \leq N} \bar{c}_i} 
   \leq \frac{1}{4}, \\
    \text{ and} \quad C_s (1 + \gamma_1) (1 + \delta_{MS})^{1 - \gamma/2} &(1 + \delta_{MS} + \delta_{MS}^{2s + 4}) 
    \delta_{MS} \leq 1/4, \label{assum for delta-MS}
    \end{aligned}
 \end{equation}
 where $\gamma_1$ is a positive constant defined in the following context, then the following inequality 
 holds for any $t \in \mathbb{R}_+$,
    \begin{equation}
        \frac{\d}{\d t} \mathscr{E}_s (t) + \mathscr{D}_s (t) \leq 0.  \label{energy estimate under assu}
        \end{equation}
        Furthermore, the positivities hold for almost any $(t,x)\in \mathbb{R}_+ \times \mathbb{T}^3$, and for any $\lambda \in(0,1]$
         \[  \bar{c}_i + \lambda \tilde{c}_i (t,x) > 0, \quad \forall 1 \leq i \leq N, \quad  
         1 + \lambda \tilde{T} (t,x) > 0.
        \]
  \end{proposition}
  \begin{proof}\textit{Proof of the Proposition \ref{proposition for energy esti for M-S} }\\
\textbf{The $H^s$ estimate for $\tilde{\mathbf{c}}$}

Now we apply the derivative $\partial_x^\beta$ to the equation \eqref{per mass-linear} with $|\beta| \leq s$, where $s > 3$ is an integer and the multi-index 
$\beta \in \mathbb{N}^3$ has length $|\beta| = \sum_{k=1}^3 \beta_k$. We take the scalar product with the vector 
$( \frac{1}{\bar{c}_i} \partial_x^\beta \tilde{c}_i )_{1 \leq i \leq N}$ and integrate it over $\mathbb{T}^3$, that is
\begin{equation}
    \begin{aligned}
    \frac{1}{2} \frac{\d}{\d t} \|\partial_x^\beta \tilde{\mathbf{c}}\|^2_{L_x^2 (\bar{\mathbf{c}}^{-1}) }  & =  
    \underbrace{ \int_{\mathbb{T}^3} \langle \partial_x^\beta \tilde{\mathbf{U}}^\ast, \nabla_x \partial_x^\beta \tilde{\mathbf{c}} \rangle 
    \,\d x }_{I} 
    \underbrace{ - \int_{\mathbb{T}^3} \langle \partial_x^\beta \tilde{\bar{u}} , \nabla_x \partial_x^\beta \tilde{c}_{tot} \rangle 
    \,\d x }_{II} \\
    &+\underbrace{ \lambda \int_{\mathbb{T}^3} \langle \partial_x^\beta ( \tilde{\mathbf{c}} \tilde{\mathbf{U}}^\ast ), \nabla_x 
    \partial_x^\beta \tilde{\mathbf{c}} \rangle_{ \bar{\mathbf{c}}^{-1} } \,\d x }_{III}
         \underbrace{ - \lambda \int_{\mathbb{T}^3} \langle \partial_x^\beta [ ( \bar{\mathbf{c}} + \lambda \tilde{\mathbf{c}} ) 
        N_2 (\tilde{\mathbf{U}}) ], \nabla_x \partial_x^\beta \tilde{\mathbf{c}} \rangle_{ \bar{\mathbf{c}}^{-1} } \,\d x}_{IV}\\
        & \underbrace{ - \lambda \int_{\mathbb{T}^3} \langle \partial_x^\beta (\tilde{\mathbf{c}} \tilde{\bar{u}}), \nabla_x \partial_x^\beta 
        \tilde{\mathbf{c}} \rangle_{\bar{\mathbf{c}}^{-1}} \,\d x }_{V}
         + \underbrace{ \lambda \int_{\mathbb{T}^3} \langle \partial_x^\beta [ ( \bar{\mathbf{c}} + \lambda \tilde{\mathbf{c}} )
        N_1 (\bar{u}) ], \nabla_x \partial_x^\beta \tilde{\mathbf{c}} \rangle_{ \bar{\mathbf{c}}^{-1} } \,\d x }_{VI}.
        \label{equation of beta-x deri for pert c}
\end{aligned}
\end{equation}
\textit{Control of integral I}

The following inequality in \cite{majda2002vorticity} will be frequently used  
\begin{equation}
    \| f g \|_{H_x^s} \lesssim \|f\|_{H_x^s} \|g\|_{H_x^s}, \quad \text{for } s > \frac{3}{2}.  \label{G-N Hs}
    \end{equation}
The integral $I$ can be written in the following form due to the expression 
$\tilde{\mathbf{U}}^\ast = \tilde{\mathbf{U}} - \frac{\langle \bar{\mathbf{c}}, \tilde{\mathbf{U}} \rangle}{\bar{c}_{tot}} \mathbf{1}$, 
\[  I = \underbrace{ \int_{\mathbb{T}^3} \langle \partial_x^\beta \tilde{\mathbf{U}}, \nabla_x \partial_x^\beta \tilde{\mathbf{c}} \rangle 
\,\d x }_{I_1} \underbrace{ - \frac{1}{\bar{c}_{tot}} \int_{\mathbb{T}^3} \partial_x^\beta 
(\langle \bar{\mathbf{c}}, \tilde{\mathbf{U}} \rangle) \nabla_x \partial_x^\beta \tilde{c}_{tot} \,\d x }_{I_2}.
\]
Using Hölder's inequality, the integral $I_2$ can be bounded as
\begin{equation}
    \begin{aligned}
    |I_2| &\leq \frac{1}{\bar{c}_{tot}} (\int_{\mathbb{T}^3} |\partial_x^\beta(\langle \bar{\mathbf{c}},\tilde{\mathbf{U}}\rangle)|^2 
    \,\d x)^{1/2}  \|\nabla_x \partial_x^\beta \tilde{c}_{tot}\|_{L_x^2}\\
    &\leq \frac{\max_{1\leq i\leq N}\bar{c}_i}{\bar{c}_{tot}}
  \|\omega^{-1}\|^{1/2}_{L_x^\infty} \|\partial_x^\beta\tilde{\mathbf{U}}\|_{L_x^2(\omega)} \|\partial_x^\beta\nabla_x\tilde{c}_{tot}\|_{L_x^2}\\
  & \leq \eta_1 \|\partial_x^\beta \tilde{\mathbf{U}}\|_{L_x^2(\omega)}^2 +\frac{  \|\omega^{-1}\|_{L_x^\infty}}{4\eta_1}
  \|\partial_x^\beta\nabla_x\tilde{c}_{tot}\|_{L_x^2}^2,
\end{aligned}
  \end{equation}
  for any constant $\eta_1 > 0$. Young's inequality is applied in the third step. Taking the expression of $\nabla_x \tilde{\mathbf{c}}$ \eqref{per M-S rela} 
  into $I_1$, then we obtain
  \[  I_1 = \underbrace{ \int_{\mathbb{T}^3} \langle \partial_x^\beta \tilde{\mathbf{U}}, \partial_x^\beta \left[ \omega A(\mathbf{c}) 
  \tilde{\mathbf{U}}  \right] \rangle \,\d x }_{I_{1,1}} + 
  \underbrace{ \int_{\mathbb{T}^3} \langle \partial_x^\beta \tilde{\mathbf{U}}, \partial_x^\beta \left[ 
    \frac{ ( \bar{\mathbf{c}} + \lambda \tilde{\mathbf{c}} ) \nabla_x \tilde{c}_{tot} }{\bar{c}_{tot} + \lambda \tilde{c}_{tot}} 
    \right] \rangle \,\d x}_{I_{1,2}},
\]
where $\omega$ is defined in \eqref{def of omega}. 
First, the integral $I_{1,2}$ can be split into linear and nonlinear parts
\begin{align*}
I_{1,2} = & \underbrace{ \int_{\mathbb{T}^3} \langle \partial_x^\beta \tilde{\mathbf{U}}, \partial_x^\beta 
\left ( \frac{ \nabla_x \tilde{c}_{tot} }{\bar{c}_{tot}} \right ) \bar{\mathbf{c}} \rangle \,\d x }_{I_{1,2}^1}  
\underbrace{-\lambda \int_{\mathbb{T}^3} \langle \partial_x^\beta \tilde{\mathbf{U}}, \partial_x^\beta 
\left ( \frac{ \tilde{c}_{tot} \nabla_x \tilde{c}_{tot} }{ \bar{c}_{tot} (\bar{c}_{tot} + \lambda \tilde{c}_{tot}) } \right ) 
\bar{\mathbf{c}} \rangle \,\d x }_{I_{1,2}^2}\\
    & + \underbrace{ \lambda \int_{\mathbb{T}^3} \langle \partial_x^\beta \tilde{\mathbf{U}}, \partial_x^\beta \left[ 
      \frac{\nabla_x \tilde{c}_{tot}}{ \bar{c}_{tot} + \lambda \tilde{c}_{tot} } \tilde{\mathbf{c}} \right] \rangle \,\d x }_{I_{1,2}^3}.
\end{align*}
The bound for $I_{1,2}^1$ uses Young's inequality, deducing that for any $\eta_2 > 0$
\begin{equation}
    \begin{aligned}
I_{1,2}^1 & \leq \frac{ \max_{1\leq i\leq N} \bar{c}_i }{\bar{c}_{tot}} \|\omega^{-1}\|_{L_x^\infty}^{1/2} 
\|\partial_x^\beta \tilde{\mathbf{U}}\|_{L_x^2(\omega)} \|\partial_x^\beta \nabla_x \tilde{c}_{tot}\|_{L_x^2} \\
&\leq \eta_2 \|\partial_x^\beta \tilde{\mathbf{U}}\|_{L_x^2(\omega)}^2 
+ \frac{ \|\omega^{-1}\|_{L_x^\infty} }{4 \eta_2}  \| \partial_x^\beta \nabla_x \tilde{c}_{tot}\|_{L_x^2}^2.
    \end{aligned}
\end{equation}
Since $\bar{\mathbf{c}}$ is a constant vector, we can easily verify that the norms $\|\cdot\|_{H_x^s}$ and 
$\|\cdot \|_{H_x^s(\bar{\mathbf{c}}^{-1})}$ are equivalent
\begin{equation}
\|\tilde{\mathbf{c}}\|_{H_x^s} \leq \max_{1 \leq i \leq N} \bar{c}_i \|\tilde{\mathbf{c}}\|_{H_x^s (\bar{\mathbf{c}}^{-1}) }
\leq \frac{ \max_{1 \leq i \leq N} }{ \min_{1 \leq i \leq N} } \|\tilde{\mathbf{c}}\|_{H_x^s}.
\end{equation}
Moreover, the upper bound for $\omega^{-1}$ is derived by using the Sobolev embedding $H^2 \hookrightarrow L^\infty$, leading to 
\begin{equation}
\|\omega^{-1}\|_{L_x^\infty} = \|(1 + \lambda \tilde{T})^{1 - \gamma/2}\|_{L_x^\infty} \lesssim 
(1 + \|\tilde{T}\|_{H_x^s})^{1 - \gamma/2} \lesssim \left (1 + \mathscr{E}_s^{1/2} (t)  \right )^{1-\gamma/2},  \label{L infty w-1}
\end{equation}
for any $\lambda \in (0,1]$. Combining these results inequality \eqref{estimate frac c-tot Hs}, the remaining parts of $I_{1,2}$ are bounded by
\begin{equation}
    \begin{aligned}
    I_{1,2}^2 & \lesssim \frac{ \lambda \max_{1 \leq i \leq N} \bar{c}_i}{\bar{c}_{tot}} \|\omega^{-1}\|_{L_x^\infty}^{1/2}
    \|\tilde{c}_{tot}\|_{H_x^s} \|\frac{1}{\bar{c}_{tot} + \lambda \tilde{c}_{tot} }\|_{H_x^s} \|\nabla_x \tilde{c}_{tot}\|_{H_x^s} 
    \|\tilde{\mathbf{U}}\|_{H_x^s (\omega) }\\
& \lesssim G_2 (1 + \mathscr{E}_s^{1/2} (t))^{1/2 - \gamma/4} (1 + \mathscr{E}_s^{1/2} (t) + \mathscr{E}_s^{s/2} (t)) 
\mathscr{E}_s^{1/2} (t) \mathscr{D}_s(t),
    \end{aligned}
\end{equation}
and 
\begin{equation}
    \begin{aligned}
   I_{1,2}^3 & \lesssim \|\omega^{-1}\|_{ L_x^\infty }^{1/2} \|\tilde{\mathbf{c}}\|_{H^s}
   \| \frac{1}{\bar{c}_{tot} + \lambda \tilde{c}_{tot} }\|_{H_x^s} \|\nabla_x \tilde{c}_{tot}\|_{H_x^s} 
   \|\tilde{\mathbf{U}}\|_{H^s (\omega) } \\
 & \lesssim G_2 (1 + \mathscr{E}_s^{1/2} (t))^{1/2 - \gamma/4}(1 + \mathscr{E}_s^{1/2} (t) + \mathscr{E}_s^{s/2} (t)) 
 \mathscr{E}_s^{1/2} (t) \mathscr{D}_s (t) .
    \end{aligned}
\end{equation}
Next, we note that there is a natural expansion for $I_{1,1}$
\[I_{1,1} = \int_{\mathbb{T}^3} \langle \partial_x^\beta \tilde{\mathbf{U}}, \partial_x^\beta \left[ \omega A(\mathbf{c}) \tilde{\mathbf{U}}
  \right] \rangle \,\d x = \int_{\mathbb{T}^3} \sum_{\tilde{\beta} \leq \beta} \langle \partial_x^\beta \tilde{\mathbf{U}} 
  \partial_x^{\beta - \tilde{\beta}} \omega, \partial_x^{\tilde{\beta}} [ A (\mathbf{c}) \tilde{\mathbf{U}} ] \rangle\,\d x.
  \]
Applying Proposition \ref{prpo scalar estimate for A} with $\mathbf{X} = \partial_x^\beta \tilde{\mathbf{U}} 
  \partial_x^{\beta - \tilde{\beta}}w$, it can be bounded by
  \begin{align*}
  I_{1,1} \leq & \underbrace{ \int_{\mathbb{T}^3} \sum_{ \tilde{\beta} \leq \beta } \langle A (\mathbf{c}) \partial_x^{\tilde{\beta}}
  \tilde{\mathbf{U}}, \partial_x^\beta \tilde{\mathbf{U}} \partial_x^{\beta - \tilde{\beta}} \omega \rangle \,\d x }_{I_{1,1}^1}\\
  & + \underbrace{ 2 \lambda N^2 \mu_A \int_{\mathbb{T}^3} \sum_{\tilde{\beta} \leq \beta} 3^{|\tilde{\beta}|} 
  (\bar{c}_{tot}  + \lambda \tilde{c}_{tot}) |\partial_x^{\beta - \tilde{\beta}} \omega| \|\partial_x^\beta \tilde{\mathbf{U}}\|
  \sum{ \substack{\beta_1 + \beta_3 = \tilde{\beta} \\ |\beta_1| \geq 1} } \|\partial_x^{\beta_1} \tilde{\mathbf{c}}\| 
  \|\partial_x^{\beta_3} \tilde{\mathbf{U}} \| \,\d x  }_{I_{1,1}^2}\\
  & + \underbrace{ \lambda^2 N^2 \mu_A \int_{\mathbb{T}^3} \sum_{\tilde{\beta} \leq \beta} 3^{|\tilde{\beta}|}
  |\partial_x^{\beta - \tilde{\beta}} \omega| \|\partial_x^\beta \tilde{\mathbf{U}}\| \sum_{ \substack{\beta_1 + \beta_2 + \beta_3 =
  \tilde{\beta} \\ |\beta_1|, |\beta_2| \geq 1}} \|\partial_x^{\beta_1} \tilde{\mathbf{c}}\|  \|\partial_x^{\beta_2} \tilde{\mathbf{c}}\| 
  \|\partial_x^{\beta_3} \tilde{\mathbf{U}}\| \,\d x }_{I_{1,1}^3}.
  \end{align*}
For the case $\tilde{\beta} = \beta$ in $I_{1,1}^1$, Proposition \ref{prop A} infers that 
\begin{equation}
    \begin{aligned}
\int_{\mathbb{T}^3} \omega \langle A (\mathbf{c}) \partial_x^{\beta} \tilde{\mathbf{U}}, \partial_x^\beta \tilde{\mathbf{U}} \rangle \,\d x
& \leq - \lambda_A \int_{\mathbb{T}^3} \left ( \min_{1 \leq i \leq N} \bar{c}_i + \lambda \tilde{c}_i \right )^2
\omega \|\partial_x^\beta \tilde{\mathbf{U}}\|^2 \,\d x\\
& \leq - \lambda_A \left ( \min_{1 \leq i \leq N} \bar{c}_i - \|\tilde{\mathbf{c}}\|_{L_x^\infty} \right )^2
\|\partial_x^\beta \tilde{\mathbf{U}}\|_{L_x^2 (\omega)}^2\\
& \leq - \lambda_A (\min_{1 \leq i \leq N} \bar{c}_i)^2\|\partial_x^\beta \tilde{\mathbf{U}}\|_{L_x^2(\omega)}^2
 + 2 \lambda_A ( \min_{1 \leq i \leq N} \bar{c}_i ) \|\tilde{\mathbf{c}}\|_{L_x^\infty} \|\partial_x^\beta \tilde{\mathbf{U}}\|_{L_x^2 (\omega)}^2.
    \end{aligned}
\end{equation}
The second part of the right hand side can be bounded as follows, by using the Sobolev embedding $ H_x^2 \hookrightarrow L_x^\infty$ 
and the assumption that $s>3$,
\[2 \lambda_A ( \min_{1 \leq i \leq N} \bar{c}_i) \|\tilde{\mathbf{c}}\|_{L_x^\infty} \|\partial_x^\beta \tilde{\mathbf{U}}\|_{L_x^2(\omega)}^2
\lesssim \|\tilde{\mathbf{c}}\|_{H_x^s} \|\tilde{\mathbf{U}}\|_{H_x^s (\omega)}^2 \lesssim \mathscr{E}_s^{1/2} (t) \mathscr{D}_s (t).
\]
Again, using Proposition \ref{prop A}, we have
\[\langle A(\mathbf{c}) \partial_x^{\tilde{\beta}} \tilde{\mathbf{U}}, \partial_x^\beta \tilde{\mathbf{U}}\rangle \leq
\|A (\mathbf{c}) \partial_x^{\tilde{\beta}} \tilde{\mathbf{U}}\|  \|\partial_x^\beta \tilde{\mathbf{U}}\| \leq \mu_A 
(\bar{c}_{tot} + \lambda \tilde{c}_{tot})^2 \|\partial_x^{\tilde{\beta}} \tilde{\mathbf{U}}\| \|\partial_x^\beta \tilde{\mathbf{U}}\|,
\]
which leads to the remaining parts of $I_{1,1}^1$ being bounded by
    \begin{align*}
   I_{1,1}^1(\tilde{\beta} \neq \beta) & \leq \mu_A \int_{\mathbb{T}^3} \sum_{ 
    \substack{ \tilde{\beta} \leq \beta\\ \tilde{\beta} \neq \beta} } (\bar{c}_{tot} + \lambda \tilde{c}_{tot})^2 
    \partial_x^{\beta-\tilde{\beta}} \omega \|\partial_x^{\tilde{\beta}} \tilde{\mathbf{U}}\| \|\partial_x^\beta \tilde{\mathbf{U}}\| \,\d x \\
& \leq \mu_A (\bar{c}_{tot} + \lambda \|\tilde{c}_{tot}\|_{L_x^\infty})^2 \|\partial_x^\beta \tilde{\mathbf{U}}\|_{L_x^2}
\left ( \int_{\mathbb{T}^3} \sum_{\substack{\tilde{\beta}\leq \beta\\ \tilde{\beta}\neq \beta} }
(\partial_x^{\beta - \tilde{\beta}} \omega )^2 \|\partial_x^{\tilde{\beta}} \tilde{\mathbf{U}}\|^2 \,\d x  \right )^{1/2}.
    \end{align*}
The calculus inequality in \cite[Lemma 3.4]{majda2002vorticity} implies
\begin{align*}
\left ( \int_{\mathbb{T}^3} \sum_{\substack{\tilde{\beta}\leq \beta\\ \tilde{\beta}\neq \beta} }
(\partial_x^{\beta-\tilde{\beta}} \omega )^2\|\partial_x^{\tilde{\beta}} \tilde{\mathbf{U}}\|^2 \,\d x \right )^{1/2}
&\lesssim \|\partial_x \omega \|_{L_x^\infty} \|\tilde{\mathbf{U}}\|_{H_x^{s-1}} +\|\tilde{\mathbf{U}}\|_{L_x^\infty}
\|\partial_x \omega \|_{H_x^{s-1}} \\
& \lesssim \|\omega\|_{\dot{H}_x^s}\|\tilde{\mathbf{U}}\|_{H_x^s} \lesssim G_1(\|\tilde{T}\|_{H_x^s}+\|\tilde{T}\|_{H_x^s}^s)
\|\tilde{\mathbf{U}}\|_{H_x^s}.
\end{align*}
The second step applies the Sobolev embedding $H^2 \hookrightarrow L^\infty$, assumption $s>3$ and the estimate for $\|\omega\|_{\dot{H}_x^s}$ in \eqref{per T homo Hs esti}. Therefore, the following inequality holds 
\begin{equation}
    \begin{aligned}
     I_{1,1}^1 (\tilde{\beta} \neq \beta) & \lesssim ( \bar{c}_{tot} + \lambda \|\tilde{c}_{tot}\|_{L_x^\infty})^2
    \|\omega^{-1}\|_{L_x^\infty} \|\omega\|_{\dot{H}_x^s} \|\tilde{\mathbf{U}}\|_{H_x^s (\omega) }^2\\
    & \lesssim G_1 (1 + \mathscr{E}_s^{1/2} (t))^{1 - \gamma/2}(1 + \mathscr{E}_s (t))(1 + \mathscr{E}_s^{(s-1)/2} (t))
    \mathscr{E}_s^{1/2} (t) \mathscr{D}_s (t).
    \end{aligned}
\end{equation} 
Moreover, the integrals $I_{1,1}^2$ and $I_{1,1}^3$ can be estimated as
\begin{align*}
    I_{1,1}^2 & = 2 \lambda N^2 \mu_A \int_{\mathbb{T}^3} \sum_{\tilde{\beta} \leq \beta} 3^{|\tilde{\beta}|}(\bar{c}_{tot}  
    + \lambda \tilde{c}_{tot}) |\partial_x^{\beta - \tilde{\beta}} \omega| \|\partial_x^\beta \tilde{\mathbf{U}}\|
    \sum{\substack{\beta_1 + \beta_3 = \tilde{\beta} \\ |\beta_1| \geq 1}} \|\partial_x^{\beta_1} \tilde{\mathbf{c}}\| 
    \|\partial_x^{\beta_3} \tilde{\mathbf{U}}\| \,\d x\\
    & \leq 2 \lambda 3^s N^2 \mu_A (\bar{c}_{tot} + \|\tilde{c}_{tot}\|_{L_x^\infty}) \|\partial_x^\beta 
    \tilde{\mathbf{U}}\|_{L_x^2} \left ( \int_{\mathbb{T}^3}  \sum{ \substack{\beta^\prime + \beta_1 + \beta_3 = \beta \\ |\beta_1| \geq 1} }
    |\partial_x^{\beta^\prime} \omega|^2 \|\partial_x^{\beta_1} \tilde{\mathbf{c}}\|^2 \|\partial_x^{\beta_3} \tilde{\mathbf{U}}\|^2
    \,\d x \right )^{1/2},
\end{align*}
\begin{align*}
    I_{1,1}^3 = & \lambda^2 N^2 \mu_A \int_{\mathbb{T}^3}  \sum_{\tilde{\beta} \leq \beta} 3^{ |\tilde{\beta}| }
    |\partial_x^{\beta - \tilde{\beta}} \omega| \|\partial_x^\beta \tilde{\mathbf{U}}\| \sum_{ 
      \substack{ \beta_1 + \beta_2 + \beta_3 = \tilde{\beta} \\ |\beta_1|, |\beta_2| \geq 1} }\|\partial_x^{\beta_1} \tilde{\mathbf{c}}\|  
      \|\partial_x^{\beta_2} \tilde{\mathbf{c}}\|  \|\partial_x^{\beta_3} \tilde{\mathbf{U}}\| \,\d x ,\\
    & \leq \lambda^2 N^2 3^s \mu_A \|\partial_x^\beta \tilde{\mathbf{U}}\|_{L_x^2} \left ( \int_{\mathbb{T}^3} \sum_{ 
      \substack{ \beta^\prime + \beta_1 + \beta_2 + \beta_3 = \beta \\ |\beta_1|, |\beta_2| \geq 1}} |\partial_x^{\beta^\prime} \omega|^2 
        \|\partial_x^{\beta_1} \tilde{\mathbf{c}}\|^2  \|\partial_x^{\beta_2} \tilde{\mathbf{c}}\|^2  
        \|\partial_x^{\beta_3} \tilde{\mathbf{U}}\|^2 \,\d x    \right )^{1/2},
\end{align*}
with the notation $\beta^\prime = \beta - \tilde{\beta}$. Since $s > 3$, we can use the Sobolev embedding 
$ H_x^{s/2} \hookrightarrow L_x^\infty$, which holds as $s/2 > 3/2$. In the bound for $I_{1,1}^2$, at most one of 
$|\beta^\prime|, |\beta_1|, |\beta_3|$ is strictly larger than $|\beta|/2$, leading to the split as follows
\[\sum{ \substack{ \beta^\prime + \beta_1 + \beta_3 = \beta \\ |\beta_1| \geq 1} }
 = \sum{ \substack{ \beta^\prime + \beta_1 + \beta_3 = \beta \\ |\beta_1| \geq 1\\ |\beta^\prime|, |\beta_1| \leq \frac{s}{2}}}
 + \sum{ \substack{ \beta^\prime + \beta_1 + \beta_3 = \beta \\ |\beta_1| \geq 1\\ |\beta^\prime|, |\beta_3| \leq \frac{s}{2}}}
 + \sum{ \substack{ \beta^\prime + \beta_1 + \beta_3 = \beta \\ |\beta_1| \geq 1\\ |\beta_1|, |\beta_3| \leq \frac{s}{2}}}.
\]
Each one of them can be bounded by
\begin{align*} 
  \sum{ \substack{ \beta^\prime + \beta_1 + \beta_3 = \beta \\ |\beta_1| \geq 1\\ |\beta^\prime|, |\beta_1| \leq \frac{s}{2} } } 
  \int_{\mathbb{T}^3} & |\partial_x^{\beta^\prime} \omega|^2 \|\partial_x^{\beta_1} \tilde{\mathbf{c}}\|^2 
  \|\partial_x^{\beta_3} \tilde{\mathbf{U}}\|^2 \,\d x\\
    & \leq \sum{ \substack{ \beta^\prime + \beta_1 + \beta_3 = \beta \\ |\beta_1|\geq 1\\ |\beta^\prime|, |\beta_1| \leq \frac{s}{2} } } 
    \|\partial_x^{\beta^\prime} \omega\|_{L_x^\infty}^2  \|\partial_x^{\beta_1} \tilde{\mathbf{c}}\|_{L_x^\infty}^2 
    \|\partial_x^{\beta_3} \tilde{\mathbf{U}}\|_{L_x^2}^2\\
    & \lesssim \sum{ \substack{ \beta^\prime + \beta_1 + \beta_3 = \beta \\ |\beta_1|\geq 1\\ |\beta^\prime|, |\beta_1| \leq \frac{s}{2} } } 
    \|\partial_x^{\beta^\prime} \omega\|_{H_x^{s/2}}^2  \|\partial_x^{\beta_1} \tilde{\mathbf{c}}\|_{H_x^{s/2}}^2 
    \|\partial_x^{\beta_3} \tilde{\mathbf{U}}\|_{L_x^2}^2\\
    & \lesssim \|\omega\|_{H_x^s}^2 \|\tilde{\mathbf{c}} \|_{H_x^s}^2  \|\tilde{\mathbf{U}}\|_{H_x^s}^2.
\end{align*}
Therefore, we can deduce the bound for $I_{1,1}^2$, using the inequalities \eqref{per T Hs esti} and \eqref{L infty w-1},
\begin{equation}
    \begin{aligned}
        I_{1,1}^2 & \lesssim \|\omega^{-1}\|_{L_x^\infty} (\bar{c}_{tot} + \|\tilde{c}_{tot}\|_{L_x^\infty}) \|\omega\|_{H_x^s}
    \|\tilde{\mathbf{c}}\|_{H_x^s}  \|\tilde{\mathbf{U}}\|^2_{H_x^s (\omega)}\\
    & \lesssim G_1 (1 + \mathscr{E}_s^{1/2} (t))^{1 - \gamma/2} (1 + \mathscr{E}_s^{1/2} (t)) 
    (1 + \mathscr{E}_s^{1/2} (t) + \mathscr{E}_s^{s/2} (t)) \mathscr{E}_s^{1/2} (t) \mathscr{D}_s (t).
    \end{aligned}
\end{equation}
The estimate for $I_{1,1}^3$ is similar, leading to
\begin{equation}
    \begin{aligned}
      I_{1,1}^3 & \lesssim \|\omega^{-1} \|_{L_x^\infty} \|\omega\|_{H_x^s} \|\tilde{\mathbf{c}}\|_{H_x^s}^2 
      \|\tilde{\mathbf{U}}\|_{H_x^s (\omega)}^2\\
    & \lesssim G_1 (1+\mathscr{E}_s^{1/2} (t))^{1 - \gamma/2}(1 + \mathscr{E}_s^{1/2} (t) + \mathscr{E}_s^{s/2} (t))
    \mathscr{E}_s (t) \mathscr{D}_s (t).
    \end{aligned}
\end{equation}
\textit{Controls of integrals II and III}

We apply the expressions $\tilde{\bar{u}} = \frac{\alpha \nabla_x \tilde{c}_{tot}}{\bar{c}_{tot}}$ and 
$\tilde{\mathbf{U}}^\ast = \tilde{\mathbf{U}} - \frac{ \langle \bar{\mathbf{c}}, \tilde{\mathbf{U}}\rangle }{\bar{c}_{tot}}\mathbf{1}$ 
to the integrals $II$ and $III$, deducing that
\begin{equation}
II = - \int_{\mathbb{T}^3} \langle \partial_x^\beta \tilde{\bar{u}} , \nabla_x \partial_x^\beta \tilde{\mathbf{c}} \rangle \,\d x
 = - \frac{\alpha}{\bar{c}_{tot}} \|\partial_x^\beta \nabla_x \tilde{c}_{tot}\|_{L_x^2}^2,
\end{equation}
\begin{align*}
III & = \lambda \int_{\mathbb{T}^3}  \langle \partial_x^\beta (\tilde{\mathbf{c}} \tilde{\mathbf{U}}^\ast ), \nabla_x \partial_x^\beta 
\tilde{\mathbf{c}}\rangle_{ \bar{\mathbf{c}}^{-1} } \,\d x\\
& = \underbrace{ \lambda \int_{\mathbb{T}^3}  \langle \partial_x^\beta (\tilde{\mathbf{c}} \tilde{\mathbf{U}} ), \nabla_x \partial_x^\beta 
\tilde{\mathbf{c}}\rangle_{ \bar{\mathbf{c}}^{-1} } \,\d x }_{III_1} 
\underbrace{ - \frac{\lambda}{\bar{c}_{tot}} \int_{\mathbb{T}^3} \langle \partial_x^\beta 
[ \tilde{\mathbf{c}} \langle \bar{\mathbf{c}}, \tilde{\mathbf{U}}\rangle ], \nabla_x \partial_x^\beta \tilde{\mathbf{c}}
\rangle_{ \bar{\mathbf{c}}^{-1} } \,\d x }_{III_2}.
\end{align*}
To begin the estimate for the integral $III$, we present the bound for $\|\nabla_x \partial_x^\beta \tilde{\mathbf{c}}\|_{L_x^2}$, 
using the expression \eqref{per M-S rela} and the estimate on $\|A(\mathbf{c})\tilde{\mathbf{U}}\|_{H_x^s}$ in \cite{briant2022perturbativ},
\begin{equation}
    \begin{aligned}
     \|\nabla_x \partial_x^\beta \tilde{\mathbf{c}}\|_{L_x^2} & \leq \|\omega\|_{H_x^s}\| A (\mathbf{c}) \tilde{\mathbf{U}}\|_{H_x^s} 
     + \|\frac{ \bar{\mathbf{c}} + \lambda \tilde{\mathbf{c}} }{ \bar{c}_{tot} + \lambda \tilde{c}_{tot} } \nabla_x 
     \tilde{c}_{tot} \|_{H_x^s}\\
    & \lesssim \| \omega^{-1} \|_{L_x^\infty}^{1/2} \| \omega \|_{H_x^s} (1 + \|\tilde{\mathbf{c}}\|_{H_x^s}^2) 
    \|\tilde{\mathbf{U}}\|_{H_x^s (\omega)} + (1 + \|\tilde{\mathbf{c}}\|_{H_x^s}) \|\frac{1}
    {\bar{c}_{tot} + \lambda \tilde{c}_{tot}}\|_{H_x^s} \|\nabla_x \tilde{c}_{tot}\|_{H_x^s}\\
    & \lesssim G_1 (1 + \mathscr{E}_s^{1/2} (t))^{1/2 - \gamma/4} (1 + \mathscr{E}_s^{1/2} (t) + \mathscr{E}_s^{s/2} (t))
(1 + \mathscr{E}_s (t))  \mathscr{D}^{1/2}_s (t)\\
&+ G_2  (1 + \mathscr{E}_s^{1/2} (t))(1 + \mathscr{E}_s^{1/2} (t) + \mathscr{E}_s^{s/2} (t)) \mathscr{D}^{1/2}_s (t).
    \label{estim for grad c}
    \end{aligned}
\end{equation}
The deduction uses the inequalities \eqref{per T Hs esti}, \eqref{estimate frac c-tot Hs} and \eqref{L infty w-1}. Based on this inequality, we can dominate $III_1$ and $III_2$ as
\begin{equation}
    \begin{aligned}
    III_1 \leq & \frac{\lambda}{\min_{1 \leq i \leq N}\bar{c}_i} \|\omega^{-1} \|_{L_x^\infty}^{1/2} \|\tilde{\mathbf{c}}\|_{H_x^s}
\| \tilde{\mathbf{U}}\|_{H_x^s(\omega)} \|\nabla_x \partial_x^\beta \tilde{\mathbf{c}}\|_{L_x^2}\\
\lesssim & G_1 (1 + \mathscr{E}_s^{1/2} (t))^{1 - \gamma/2} (1 + \mathscr{E}_s^{1/2} (t) + \mathscr{E}_s^{s/2} (t))
(1 + \mathscr{E}_s (t)) \mathscr{E}_s^{1/2} (t) \mathscr{D}_s (t)\\
&+ G_2 (1 + \mathscr{E}_s^{1/2} (t))^{1/2 - \gamma/4} (1 + \mathscr{E}_s^{1/2} (t))(1 + \mathscr{E}_s^{1/2} (t) 
+ \mathscr{E}_s^{s/2} (t)) \mathscr{E}_s^{1/2} (t) \mathscr{D}_s (t),
    \end{aligned}
\end{equation}
\begin{equation}
\begin{aligned}
III_2 \leq & \frac{ \lambda \max_{1 \leq i \leq N} \bar{c}_i }{ \bar{c}_{tot} \min_{1 \leq i \leq N}\bar{c}_i }
\|\omega^{-1} \|_{L_x^\infty}^{1/2} \|\tilde{\mathbf{c}}\|_{H_x^s} \| \tilde{\mathbf{U}}\|_{H_x^s(\omega)}  
\|\nabla_x \partial_x^\beta \tilde{\mathbf{c}}\|_{L_x^2}\\
\lesssim & G_1 (1 + \mathscr{E}_s^{1/2} (t))^{1 - \gamma/2} (1 + \mathscr{E}_s^{1/2} (t) + \mathscr{E}_s^{s/2} (t)) 
(1 + \mathscr{E}_s (t)) \mathscr{E}_s^{1/2} (t) \mathscr{D}_s (t)\\
& + G_2 (1 + \mathscr{E}_s^{1/2} (t))^{1/2 - \gamma/4} (1 + \mathscr{E}_s^{1/2} (t)) 
(1 + \mathscr{E}_s^{1/2} (t) + \mathscr{E}_s^{s/2} (t)) \mathscr{E}_s^{1/2} (t) \mathscr{D}_s (t).
\end{aligned}
\end{equation}
\textit{Control of integral IV}

Since the expression $N_2 ( \tilde{\mathbf{U}} ) = \frac{\langle \tilde{\mathbf{c}}, \tilde{\mathbf{U}} \rangle}{\bar{c}_{tot}} 
- \frac{\tilde{c}_{tot} \langle \bar{\mathbf{c}}, \tilde{\mathbf{U}} \rangle}{\bar{c}_{tot} (\bar{c}_{tot} + \lambda \tilde{c}_{tot}) } 
- \frac{\lambda \tilde{c}_{tot} \langle \tilde{\mathbf{c}}, \tilde{\mathbf{U}} \rangle }
{\bar{c}_{tot} (\bar{c}_{tot} + \lambda \tilde{c}_{tot})}$ holds, the integral $IV$ are divided into the following parts,
\begin{align*}
 IV = & - \lambda \int_{\mathbb{T}^3} \langle \partial_x^\beta [ (\bar{\mathbf{c}} + \lambda \tilde{\mathbf{c}} ) N_2 (\tilde{\mathbf{U}}) ], 
 \nabla_x \partial_x^\beta \tilde{\mathbf{c}} \rangle_{ \bar{\mathbf{c}}^{-1} } \,\d x\\
 = & \underbrace{ - \lambda \int_{\mathbb{T}^3} \partial_x^\beta \left ( \frac{\langle \tilde{\mathbf{c}}, \tilde{\mathbf{U}} \rangle}
 {\bar{c}_{tot}} \right ) \nabla_x \partial_x^\beta \tilde{c}_{tot} \,\d x }_{IV_1} 
 + \underbrace{ \lambda \int_{\mathbb{T}^3} \partial_x^\beta \left[ 
  \frac{\tilde{c}_{tot} \langle \bar{\mathbf{c}}, \tilde{\mathbf{U}} \rangle}{\bar{c}_{tot} (\bar{c}_{tot} + \lambda \tilde{c}_{tot}) }
\right] \nabla_x \partial_x^\beta \tilde{c}_{tot} \,\d x }_{IV_2}\\
    & + \underbrace{ \lambda^2 \int_{\mathbb{T}^3} \partial_x^\beta \left[ 
      \frac{ \tilde{c}_{tot} \langle \tilde{\mathbf{c}}, \tilde{\mathbf{U}} \rangle }{\bar{c}_{tot} (\bar{c}_{tot} 
      + \lambda \tilde{c}_{tot}) }\right] \nabla_x \partial_x^\beta \tilde{c}_{tot} \,\d x  }_{IV_3} 
      \underbrace{ - \frac{\lambda^2}{\bar{c}_{tot}} \int_{\mathbb{T}^3} \langle \partial_x^\beta \left[ 
        \langle \tilde{\mathbf{c}}, \tilde{\mathbf{U}} \rangle \tilde{\mathbf{c}} \right], \nabla_x \partial_x^\beta \tilde{\mathbf{c}} 
        \rangle_{ \bar{\mathbf{c}}^{-1} } \,\d x}_{IV_4} \\
        & + \underbrace{ \frac{\lambda^2}{\bar{c}_{tot}} \int_{\mathbb{T}^3} \langle \partial_x^\beta \left[ 
          \frac{\tilde{c}_{tot} \langle \bar{\mathbf{c}}, \tilde{\mathbf{U}} \rangle}{ \bar{c}_{tot} + \lambda \tilde{c}_{tot} } 
          \tilde{\mathbf{c}} \right], \nabla_x \partial_x^\beta \tilde{\mathbf{c}} \rangle_{ \bar{\mathbf{c}}^{-1} } \,\d x}_{IV_5} 
          + \underbrace{\frac{\lambda^3}{\bar{c}_{tot}} \int_{\mathbb{T}^3} \langle \partial_x^\beta \left[ 
            \frac{ \tilde{c}_{tot} \langle \tilde{\mathbf{c}}, \tilde{\mathbf{U}} \rangle }{ \bar{c}_{tot} + \lambda \tilde{c}_{tot} } 
            \tilde{\mathbf{c}} \right], \nabla_x \partial_x^\beta \tilde{\mathbf{c}} \rangle_{ \bar{\mathbf{c}}^{-1} } \,\d x}_{IV_6}.
\end{align*}
We first control $IV_1-IV_3$ by applying inequalities \eqref{per T Hs esti}, \eqref{estimate frac c-tot Hs} and \eqref{L infty w-1}, 
that are
\begin{equation}
    \begin{aligned}
    IV_1\leq& \frac{\lambda}{\bar{c}_{tot}} \|\omega^{-1}\|_{L_x^\infty}^{1/2}\|\tilde{\mathbf{c}}\|_{H_x^s} 
    \|\tilde{\mathbf{U}}\|_{H_x^s(\omega)}\|\nabla_x \partial_x^\beta\tilde{c}_{tot}\|_{L_x^2} 
    \lesssim (1+\mathscr{E}_s^{1/2}(t))^{1/2-\gamma/4} \mathscr{E}_s^{1/2}(t) \mathscr{D}_s(t),
    \end{aligned}
\end{equation}
\begin{equation}
    \begin{aligned}
    IV_2 \leq & \frac{ \lambda \max_{1 \leq i \leq N} \bar{c}_i}{\bar{c}_{tot}} \|\omega^{-1}\|_{L_x^\infty}^{1/2}
    \|\tilde{c}_{tot}\|_{H_x^s} \|\frac{1}{\bar{c}_{tot} + \lambda \tilde{c}_{tot}}\|_{H_x^s} \|\tilde{\mathbf{U}}\|_{H_x^s (\omega)}
    \|\nabla_x \partial_x^\beta \tilde{c}_{tot}\|_{L_x^\beta}\\
    \lesssim & G_2 (1 + \mathscr{E}_s^{1/2} (t))^{1/2 - \gamma/4} (1 + \mathscr{E}_s^{1/2} (t) + \mathscr{E}_s^{s/2} (t))
    \mathscr{E}_s^{1/2} (t) \mathscr{D}_s (t),
    \end{aligned}
\end{equation}
and 
\begin{equation}
\begin{aligned}
IV_3 \leq & \frac{\lambda^2}{\bar{c}_{tot}} \|\omega^{-1}\|_{L_x^\infty}^{1/2}  
\|\frac{1}{\bar{c}_{tot} + \lambda \tilde{c}_{tot}} \|_{H_x^s}  \|\tilde{c}_{tot}\|_{H_x^s}   \|\tilde{\mathbf{c}}\|_{H_x^s}  
\|\tilde{\mathbf{U}}\|_{H_x^s (\omega)}  \|\nabla_x \partial_x^\beta \tilde{c}_{tot}\|_{L_x^\beta}\\
\lesssim & G_2 (1 + \mathscr{E}_s^{1/2} (t))^{1/2 - \gamma/4} (1 + \mathscr{E}_s^{1/2} (t) + \mathscr{E}_s^{s/2} (t)) 
\mathscr{E}_s (t) \mathscr{D}_s (t).
\end{aligned}
\end{equation}
The estimates for $IV_4-IV_6$ rely on the control for $\|\nabla_x \partial_x^\beta \tilde{\mathbf{c}}\|_{L_x^2}$ in \eqref{estim for grad c}, writing as follows,
\begin{equation}
\begin{aligned}
IV_4 \leq & \frac{\lambda^2}{\bar{c}_{tot} \min_{1 \leq i \leq N} \bar{c}_i}  \|\omega^{-1}\|_{L_x^\infty}^{1/2}
\|\tilde{\mathbf{c}}\|_{H_x^s}^2  \|\tilde{\mathbf{U}}\|_{H_x^s (\omega)}  \|\nabla_x \partial_x^\beta \tilde{\mathbf{c}}\|_{L_x^2}\\
 \lesssim & G_1 (1 + \mathscr{E}_s^{1/2} (t))^{1 - \gamma/2} (1 + \mathscr{E}_s^{1/2} (t) + \mathscr{E}_s^{s/2} (t)) 
 (1 + \mathscr{E}_s(t)) \mathscr{E}_s (t) \mathscr{D}_s (t)\\
 & + G_2 (1 + \mathscr{E}_s^{1/2} (t))^{1/2 - \gamma/4} (1 + \mathscr{E}_s^{1/2} (t)) (1 + \mathscr{E}_s^{1/2} (t) + \mathscr{E}_s^{s/2} (t)) 
 \mathscr{E}_s (t) \mathscr{D}_s (t),
\end{aligned}
\end{equation}
\begin{equation}
\begin{aligned}
IV_5 \leq & \frac{ \lambda^2 \max_{1 \leq i \leq N} \bar{c}_i}{\bar{c}_{tot} \min_{1 \leq i \leq N} \bar{c}_i}  
\|\omega^{-1}\|_{L_x^\infty}^{1/2}  \|\frac{1}{\bar{c}_{tot} + \lambda \tilde{c}_{tot}}\|_{H_x^s} \|\tilde{c}_{tot}\|_{H_x^s} 
\|\tilde{\mathbf{c}}\|_{H_x^s}  \|\tilde{\mathbf{U}}\|_{H_x^s (\omega)}  \|\nabla_x \partial_x^\beta \tilde{\mathbf{c}}\|_{L_x^2}\\
\lesssim & G_1 G_2 (1 + \mathscr{E}_s^{1/2} (t))^{1 - \gamma/2} (1+\mathscr{E}_s^{1/2} (t) + \mathscr{E}_s^s (t)) 
 (1 + \mathscr{E}_s (t)) \mathscr{E}_s (t) \mathscr{D}_s (t)\\
 &+ G_2^2 (1 + \mathscr{E}_s^{1/2} (t))^{1/2 - \gamma/4} (1 + \mathscr{E}_s^{1/2} (t))( 1+ \mathscr{E}_s^{1/2} (t) + \mathscr{E}_s^s (t)) 
 \mathscr{E}_s (t) \mathscr{D}_s (t),
\end{aligned}
\end{equation}
and 
\begin{equation}
\begin{aligned}
IV_6 \leq & \frac{\lambda^3}{\bar{c}_{tot} \min_{1 \leq i \leq N} \bar{c}_i}  \|\omega^{-1}\|_{L_x^\infty}^{1/2}
\|\frac{1}{\bar{c}_{tot} + \lambda \tilde{c}_{tot}}\|_{H_x^s} \|\tilde{c}_{tot}\|_{H_x^s} \|\tilde{\mathbf{c}}\|_{H_x^s}^2
\|\tilde{\mathbf{U}}\|_{H_x^s(\omega)}\|\nabla_x \partial_x^\beta\tilde{\mathbf{c}}\|_{L_x^2}\\
\lesssim & G_1 G_2 (1 + \mathscr{E}_s^{1/2} (t))^{1 - \gamma/2} (1 + \mathscr{E}_s^{1/2} (t) + \mathscr{E}_s^s (t)) 
 (1 + \mathscr{E}_s (t)) \mathscr{E}^{3/2}_s (t) \mathscr{D}_s (t)\\
 & + G_2^2 (1 + \mathscr{E}_s^{1/2} (t))^{1/2 - \gamma/4} (1 + \mathscr{E}_s^{1/2} (t)) (1 + \mathscr{E}_s^{1/2} (t) + \mathscr{E}_s^s (t)) 
 \mathscr{E}^{3/2}_s (t) \mathscr{D}_s (t).
\end{aligned}
\end{equation}
\textit{Controls of integrals V and VI}

Taking the expression $\tilde{\bar{u}} = \frac{\alpha \nabla_x \tilde{c}_{tot}}{\bar{c}_{tot}}$ into the integral $V$ and 
using the bound for $\|\nabla_x \partial_x^\beta \tilde{\mathbf{c}}\|_{L_x^2}$, we obtain
\begin{equation}
\begin{aligned}
V =  - & \lambda \int_{\mathbb{T}^3} \langle \partial_x^\beta (\tilde{\mathbf{c}} \tilde{\bar{u}}), \nabla_x \partial_x^\beta 
\tilde{\mathbf{c}} \rangle_{\bar{\mathbf{c}}^{-1}} \,\d x = - \frac{\alpha \lambda}{\bar{c}_{tot}} \int_{\mathbb{T}^3} 
\langle \partial_x^\beta  (\tilde{\mathbf{c}} \nabla_x\tilde{c}_{tot}), \nabla_x \partial_x^\beta \tilde{\mathbf{c}} 
\rangle_{ \bar{\mathbf{c}}^{-1} } \,\d x\\
\leq & \frac{\alpha \lambda}{ \min_{1 \leq i \leq N} \bar{c}_i \bar{c}_{tot}}  \|\tilde{\mathbf{c}}\|_{H_x^s}
\|\nabla_x \tilde{c}_{tot}\|_{H_x^s} \|\nabla_x \partial_x^\beta \tilde{\mathbf{c}}\|_{L_x^2}\\
\lesssim & G_1  (1 + \mathscr{E}_s^{1/2} (t))^{1/2 - \gamma/4} (1 + \mathscr{E}_s^{1/2} (t) + \mathscr{E}_s^{s/2} (t)) 
 (1 + \mathscr{E}_s (t)) \mathscr{E}^{1/2}_s (t) \mathscr{D}_s (t)\\
 & + G_2 (1 + \mathscr{E}_s^{1/2} (t) + \mathscr{E}_s^{s/2} (t)) (1 + \mathscr{E}_s^{1/2} (t)) \mathscr{E}^{1/2}_s (t) \mathscr{D}_s (t).
\end{aligned}
\end{equation}
The expression $N_1 ( \tilde{\bar{u}} ) = \frac{\alpha \tilde{c}_{tot} \nabla_x \tilde{c}_{tot}}{ \bar{c}_{tot} 
(\bar{c}_{tot} + \lambda \tilde{c}_{tot}) }$ implies that the integral $VI$ can be expanded as
\begin{align*}
     VI = & \lambda \int_{\mathbb{T}^3} \langle \partial_x^\beta [ (\bar{\mathbf{c}}+\lambda\tilde{\mathbf{c}}) N_1(\bar{u}) ], 
     \nabla_x \partial_x^\beta \tilde{\mathbf{c}} \rangle_{ \bar{\mathbf{c}}^{-1} } \,\d x\\
 = & \underbrace{ \frac{\alpha\lambda}{\bar{c}_{tot}} \int_{\mathbb{T}^3} \partial_x^\beta \left[ 
  \frac{ \tilde{c}_{tot} \nabla_x \tilde{c}_{tot} }{ \bar{c}_{tot} + \lambda \tilde{c}_{tot} } \right] 
  \nabla_x \partial_x^\beta \tilde{c}_{tot} \,\d x}_{VI_1} 
  + \underbrace{ \frac{\alpha \lambda^2}{\bar{c}_{tot}} \int_{\mathbb{T}^3}\langle 
    \partial_x^\beta \left[ \frac{ \tilde{c}_{tot} \nabla_x \tilde{c}_{tot}}{\bar{c}_{tot} + \lambda \tilde{c}_{tot}}
    \tilde{\mathbf{c}} \right], \nabla_x \partial_x^\beta \tilde{\mathbf{c}} \rangle_{ \bar{\mathbf{c}}^{-1} } \,\d x}_{VI_2}.
\end{align*}
Estimate for $VI_1$ is derived by using the inequality \eqref{estimate frac c-tot Hs},
\begin{equation}
\begin{aligned}
VI_1 \leq & \frac{\alpha \lambda}{\bar{c}_{tot}}  \|\tilde{c}_{tot}\|_{H_x^s}  \|\frac{1}{\bar{c}_{tot} + \lambda  \tilde{c}_{tot}}\|_{H_x^s}  
\|\nabla_x\tilde{c}_{tot}\|_{H_x^s}^2 \lesssim G_2 (1 + \mathscr{E}_s^{1/2} (t) + \mathscr{E}_s^{s/2} (t)) \mathscr{E}^{1/2}_s (t) 
\mathscr{D}_s (t).
\end{aligned}
\end{equation}
The method to control $VI_2$ is similar to that for $IV$, employing the inequality \eqref{estimate frac c-tot Hs} and the estimate for 
$\|\nabla_x \partial_x^\beta \tilde{\mathbf{c}}\|_{L_x^2}$ in \eqref{estim for grad c},
\begin{equation}
\begin{aligned}
VI_2 \leq & \frac{\alpha \lambda^2}{\bar{c}_{tot} \min_{1 \leq i \leq N} \bar{c}_i} 
\|\frac{1}{\bar{c}_{tot} + \lambda \tilde{c}_{tot}}\|_{H_x^s} \|\tilde{c}_{tot}\|_{H_x^s} \|\nabla_x \tilde{c}_{tot}\|_{H_x^s}
\|\tilde{\mathbf{c}}\|_{H_x^s}  \|\nabla_x \partial_x^\beta \tilde{\mathbf{c}}\|_{L_x^2}\\
\lesssim & G_1 G_2 (1 + \mathscr{E}_s^{1/2} (t))^{1/2 - \gamma/4} (1 + \mathscr{E}_s^{1/2} (t) + \mathscr{E}_s^s (t))
 (1 + \mathscr{E}_s (t)) \mathscr{E}_s (t) \mathscr{D}_s (t)\\
 & +  G_2^2 (1 + \mathscr{E}_s^{1/2} (t)) (1 + \mathscr{E}_s^{1/2} (t) + \mathscr{E}_s^s (t)) \mathscr{E}_s (t) \mathscr{D}_s (t).
\end{aligned}
\end{equation}
Finally, by summing all inequalities above for $|\beta| \leq s$, 
we complete the estimate for $\| \tilde{\mathbf{c}}\|_{H_x^s(\bar{\mathbf{c}}^{-1})}$, 
that is
\begin{equation}
    \begin{aligned}
    \frac{1}{2} \frac{\d}{\d t}&  \|\tilde{\mathbf{c}}\|_{H_x^s}^2 +  \left ( \lambda ( \min_{1 \leq i \leq N} \bar{c}_i)^2 
    - \eta_1 - \eta_2 \right ) \|\tilde{\mathbf{U}}\|_{H_x^s (\omega)}^2 + \left ( \frac{\alpha}{\bar{c}_{tot}} - 
    \frac{\|\omega^{-1}\|_{L_x^\infty}}{4 \eta_1} - \frac{\|\omega^{-1}\|_{L_x^\infty}}{4 \eta_2} \right ) 
     \|\nabla_x\tilde{c}_{tot}\|_{H_x^s}\\
& \lesssim (1 + G_1 + G_2 + G_1 G_2 + G_2^2) (1 + \mathscr{E}_s^{1/2} (t))^{1 - \gamma/2} 
(1 + \mathscr{E}_s^{1/2} (t) + \mathscr{E}_s^{s+2} (t)) \mathscr{E}_s^{1/2} (t) \mathscr{D}_s (t).  \label{prior Hs for per c}
    \end{aligned}
\end{equation}
\textbf{The $H^s$ estimates for $\tilde{T}$ and $\tilde{c}_{tot}$}

The evolution equation for $\tilde{T}$ is
\begin{gather*}
    \partial_t \tilde{T} - \frac{2\alpha}{3} \nabla_x \cdot \left[ \frac{ \nabla_x \tilde{c}_{tot}}{\bar{c}_{tot} + \lambda \tilde{c}_{tot}}  
    (1+\lambda \tilde{T})  \right] - \frac{2 \alpha \lambda}{3} \frac{ |\nabla_x \tilde{c}_{tot}|^2 }
    {(\bar{c}_{tot} + \lambda \tilde{c}_{tot})^2} (1 + \lambda \tilde{T}) - \alpha \lambda 
    \frac{ \nabla_x \tilde{c}_{tot}}{\bar{c}_{tot} + \lambda \tilde{c}_{tot}} \cdot \nabla_x \tilde{T}=0.
\end{gather*}
By taking derivative $\partial_x^\beta$ of this equation with $|\beta| \leq s$, multiplying by $\partial_x^\beta \tilde{T}$ and 
integrating with respect to $x$ over $\mathbb{T}^3$, we deduce that
\begin{equation}
    \begin{aligned}
\frac{1}{2} \frac{\d}{\d t} \|\partial_x^\beta \tilde{T} \|_{L_x^2}^2 = & \underbrace{ - \frac{2\alpha}{3} 
\langle \partial_x^\beta \left (\frac{ \nabla_x \tilde{c}_{tot}}{\bar{c}_{tot} + \lambda \tilde{c}_{tot}} \right ), 
\nabla_x \partial_x^\beta \tilde{T} \rangle_{L_x^2} }_{VII_1} 
\underbrace{-\frac{2 \alpha \lambda}{3} \langle \partial_x^\beta \left ( 
  \frac{\nabla_x \tilde{c}_{tot}}{\bar{c}_{tot} + \lambda \tilde{c}_{tot}} \tilde{T} \right ), \nabla_x \partial_x^\beta 
    \tilde{T} \rangle_{L_x^2} }_{VII_2}\\
    & + \underbrace{\frac{2 \alpha \lambda}{3} \langle \partial_x^\beta \left[ \frac{|\nabla_x \tilde{c}_{tot}|^2}
    { (\bar{c}_{tot} + \lambda \tilde{c}_{tot})^2 } \left (1 + \lambda \tilde{T} \right ) \right], \partial_x^\beta \tilde{T}
    \rangle_{L_x^2} }_{VII_3}\\
    & + \underbrace{\alpha \lambda \langle \partial_x^\beta \left ( \frac{\nabla_x\tilde{c}_{tot}}
    {\bar{c}_{tot} + \lambda \tilde{c}_{tot}} \cdot \nabla_x \tilde{T} \right ), \partial_x^\beta \tilde{T}\rangle_{L_x^2} }_{VII_4}.
    \end{aligned}
\end{equation}
Before dominating these integrals, we rewrite the relation \eqref{implicit rela}
\begin{equation*}
    \frac{\nabla_x\tilde{c}_{tot}}{\bar{c}_{tot}+\lambda\tilde{c}_{tot}}=-\frac{\nabla_x\tilde{T}}
    {1+\lambda \tilde{T}}.
\end{equation*}
Using it to replace $\nabla_x\tilde{T}$, the integral $VII_1$ can be expressed as
\begin{align*}
VII_1 = & \underbrace{ \frac{2\alpha}{3} \langle \partial_x^\beta \left (\frac{\nabla_x\tilde{c}_{tot}}{\bar{c}_{tot} 
+ \lambda \tilde{c}_{tot}} \right ), \partial_x^\beta \left ( \frac{\nabla_x \tilde{c}_{tot}}{\bar{c}_{tot} + \lambda \tilde{c}_{tot}} 
\right ) \rangle_{L_x^2} }_{VII_{1,1}}\\
& + \underbrace{ \frac{2 \alpha \lambda}{3} \langle \partial_x^\beta \left ( \frac{\nabla_x \tilde{c}_{tot}}{\bar{c}_{tot} 
+ \lambda \tilde{c}_{tot}} \right ), \partial_x^\beta \left ( \frac{\nabla_x \tilde{c}_{tot}}{\bar{c}_{tot} + \lambda \tilde{c}_{tot}} 
\tilde{T} \right ) \rangle_{L_x^2} }_{VII_{1,2}}.
\end{align*}
Applying the inequalities \eqref{estimate frac c-tot Hs} and \eqref{G-N Hs}, we control the second term as
\begin{equation}
    VII_{1,2} \lesssim \|\frac{1}{\bar{c}_{tot} + \lambda \tilde{c}_{tot}}\|_{H_x^s}^2  \|\tilde{T}\|_{H_x^s}  
    \|\nabla_x\tilde{c}_{tot}\|_{H_x^s}^2 \lesssim G_2^2 (1 + \mathscr{E}_s^{1/2} (t) + \mathscr{E}_s^s (t)) 
    \mathscr{E}_s^{1/2} (t) \mathscr{D}_s (t).
\end{equation}
The integral $VII_{1,1}$ needs a more refined estimate. Due to the decomposition 
\[  \frac{1}{\bar{c}_{tot} + \lambda \tilde{c}_{tot}} = \frac{1}{\bar{c}_{tot}} - \frac{\lambda \tilde{c}_{tot}}{\bar{c}_{tot} 
(\bar{c}_{tot} + \lambda \tilde{c}_{tot}) },
\]
this integral can be divided into the following parts 
\begin{align*}
    VII_{1,1} = & \underbrace{ \frac{2 \alpha}{3 \bar{c}_{tot}^2} \langle \partial_x^\beta \nabla_x \tilde{c}_{tot}, 
    \partial_x^\beta \nabla_x \tilde{c}_{tot} \rangle_{L_x^2} }_{VII_{1,1}^1}
    \underbrace{ -\frac{4 \alpha \lambda}{3 \bar{c}_{tot}^2} \langle \partial_x^\beta \nabla_x \tilde{c}_{tot}, 
    \partial_x^\beta \left ( \frac{\tilde{c}_{tot} \nabla_x \tilde{c}_{tot}}{\bar{c}_{tot} + \lambda \tilde{c}_{tot}} \right ) 
    \rangle_{L_x^2} }_{VII_{1,1}^2}\\
    & + \underbrace{ \frac{2 \alpha \lambda^2}{3 \bar{c}_{tot}^2} \langle \partial_x^\beta \left ( 
      \frac{\tilde{c}_{tot} \nabla_x \tilde{c}_{tot}}{\bar{c}_{tot} + \lambda \tilde{c}_{tot}} \right ), \partial_x^\beta 
      \left ( \frac{\tilde{c}_{tot} \nabla_x\tilde{c}_{tot}}{\bar{c}_{tot} + \lambda \tilde{c}_{tot}} \right )\rangle_{L_x^2} }_{VII_{1,1}^3}.
\end{align*}
We control them separately, writing as follows
\begin{equation}
    VII_{1, 1}^1 = \frac{2 \alpha }{3 \bar{c}_{tot}^2}  \|\partial_x^\beta \nabla_x \tilde{c}_{tot}\|_{L_x^2}^2,
\end{equation}
\begin{equation}
\begin{aligned}
    VII_{1, 1}^2 \lesssim \|\frac{1}{\bar{c}_{tot} + \lambda \tilde{c}_{tot}}\|_{H_x^s}  \|\tilde{c}_{tot}\|_{H_x^s}
    \|\nabla_x \tilde{c}_{tot}\|^2_{H_x^s} \lesssim G_2 (1 + \mathscr{E}_s^{1/2} (t) + \mathscr{E}_s^{s/2} (t)) \mathscr{E}_s^{1/2} (t) 
    \mathscr{D}_s (t),
\end{aligned}
\end{equation}
and
\begin{equation}
\begin{aligned}
    VII_{1, 1}^3 \lesssim \|\frac{1}{\bar{c}_{tot} + \lambda \tilde{c}_{tot}}\|^2_{H_x^s}  \|\tilde{c}_{tot}\|^2_{H_x^s}
    \|\nabla_x\tilde{c}_{tot}\|^2_{H_x^s} \lesssim G_2^2 (1 + \mathscr{E}_s^{1/2} (t) + \mathscr{E}_s^s (t))
    \mathscr{E}_s (t) \mathscr{D}_s (t).
\end{aligned}
\end{equation}
For the integrals $VII_2$ and $VII_4$, we again replace $\nabla_x\tilde{T}$ with the relation \eqref{implicit rela} and apply the inequality \eqref{G-N Hs}, deducing that
\begin{equation}
    \begin{aligned}
        VII_2 = \frac{2 \alpha \lambda}{3} \langle \partial_x^\beta \left ( 
          \frac{\nabla_x \tilde{c}_{tot}}{\bar{c}_{tot} + \lambda \tilde{c}_{tot}} \tilde{T} \right ), 
          \partial_x^\beta \left[ \frac{\nabla_x \tilde{c}_{tot}}{\bar{c}_{tot} + \lambda \tilde{c}_{tot}} (1 + \lambda \tilde{T}) \right] 
          \rangle_{L_x^2} \\
        \lesssim \|\frac{1}{\bar{c}_{tot} + \lambda \tilde{c}_{tot}}\|_{H_x^s}^2 (1 + \|\tilde{T}\|_{H_x^s}) \|\tilde{T}\|_{H_x^s}  
        \|\nabla_x \tilde{c}_{tot}\|_{H_x^s}^2\\
        \lesssim G_2^2 (1 + \mathscr{E}_s^{1/2} (t) + \mathscr{E}_s^s (t)) (1 + \mathscr{E}_s^{1/2} (t)) \mathscr{E}_s^{1/2} (t) 
        \mathscr{D}_s (t),
    \end{aligned}
    \end{equation}
\begin{equation}
\begin{aligned}
VII_4 = & - \alpha \lambda \langle \partial_x^\beta \left[ \frac{|\nabla_x\tilde{c}_{tot}|^2}{(\bar{c}_{tot} + \lambda \tilde{c}_{tot})^2} 
(1 + \lambda \tilde{T}) \right], \partial_x^\beta \tilde{T} \rangle_{L_x^2}\\
    \lesssim & (1 + \|\tilde{T}\|_{H_x^s})  \|\tilde{T}\|_{H_x^s}  \|\frac{1}{\bar{c}_{tot} + \lambda \tilde{c}_{tot}}\|_{H_x^s}^2 
    \|\nabla_x \tilde{c}_{tot}\|^2_{H_x^s}\\
    \lesssim & G_2^2 (1 + \mathscr{E}_s^{1/2} (t))(1 + \mathscr{E}_s^{1/2} (t) + \mathscr{E}_s^s (t)) \mathscr{E}_s^{1/2} (t) 
    \mathscr{D}_s (t).
\end{aligned}
\end{equation}
The calculations for $VII_3$ is same to that for $VII_4$,
\begin{equation}
\begin{aligned}
    VII_3 = & \frac{2 \alpha \lambda}{3} \langle \partial_x^\beta \left[ \frac{ |\nabla_x\tilde{c}_{tot}|^2 }
    { (\bar{c}_{tot} + \lambda \tilde{c}_{tot})^2 } \left ( 1 + \lambda \tilde{T} \right ) \right], 
    \partial_x^\beta \tilde{T} \rangle_{L_x^2}\\
    \lesssim & G_2^2 (1 + \mathscr{E}_s^{1/2} (t)) (1 + \mathscr{E}_s^{1/2} (t) + \mathscr{E}_s^{s} (t)) \mathscr{E}_s^{1/2} (t)  
    \mathscr{D}_s (t).
\end{aligned}
\end{equation}
Summing all above inequalities for $|\beta| \leq s$, now we establish the $H_x^s$ norm estimate for $\tilde{T}$, i.e. 
\begin{equation}
\frac{1}{2} \frac{\d}{\d t}  \|\tilde{T}\|_{H_x^s}^2 - \frac{2 \alpha}{3 \bar{c}_{tot}^2} \|\nabla_x \tilde{c}_{tot}\|_{H_x^s}^2
\lesssim (G_2 + G_2^2) (1 + \mathscr{E}_s^{1/2} (t) + \mathscr{E}_s^{s+1/2} (t)) \mathscr{E}_s^{1/2} (t) \mathscr{D}_s (t).  \label{prior Hs for per T}
\end{equation}
At last, the mass conservation equations \eqref{per mass-linear} have a derivation for $\tilde{c}_{tot}$, i.e.
\[ \partial_t \tilde{c}_{tot} = \alpha \Delta_x \tilde{c}_{tot}.
\]
Applying derivative $\partial_x^\beta$ to this equation ($|\beta| \leq s$), multiplying by $\chi \partial_x^\beta \tilde{c}_{tot}$ 
and integrating with respect to $x$ over $\mathbb{T}^3$, we obtain the equality 
\begin{equation*}
\frac{\chi}{2} \frac{\d}{\d t} \| \partial_x^\beta \tilde{c}_{tot}\|_{L_x^2}^2 + \alpha \chi 
\|\nabla_x \partial_x^\beta  \tilde{c}_{tot}\|_{L_x^2}^2 = 0,
\end{equation*}
where $\chi$ is defined in \eqref{defi of chi}. It is clear that the estimate for $\|\tilde{c}_{tot}\|_{H_x^s}$ 
can be expressed as 
\begin{equation}
\frac{\chi}{2} \frac{\d}{\d t}  \|\tilde{c}_{tot}\|_{H_x^s}^2 + \alpha \chi \|\nabla_x\tilde{c}_{tot}\|_{H_x^s}^2 = 0
\label{prior Hs for per c-tot}.
\end{equation}
\textbf{The energy estimate for the perturbations}

Now we sum the inequalities \eqref{prior Hs for per c}, \eqref{prior Hs for per T} and \eqref{prior Hs for per c-tot}, obtaining 
\begin{equation}
    \begin{aligned}
    \frac{1}{2} \frac{\d}{ \d t}\mathscr{E}_s & (t) + \left ( \lambda_A ( \min_{1 \leq i \leq N} \bar{c}_i)^2 - \eta_1 - \eta_2 \right )  
    \|\tilde{\mathbf{U}}\|_{H_x^s (\omega)}^2 (t)\\
& + \left ( \alpha \chi + \frac{\alpha}{\bar{c}_{tot}} - \frac{2\alpha}{3\bar{c}_{tot}^2} - \frac{\|\omega^{-1}\|_{L_x^\infty}(t)}{4\eta_1} 
- \frac{ \|\omega^{-1}\|_{L_x^\infty} (t) }{4 \eta_2} \right )  \|\nabla_x\tilde{c}_{tot}\|_{H_x^s}^2 (t)\\
&\leq C_s (1 + G_1 + G_2 + G_1 G_2 + G_2^2) (1 + \mathscr{E}_s^{1/2} (t))^{1 - \gamma/2}(1 + \mathscr{E}_s^{1/2} (t) 
+ \mathscr{E}_s^{s+2} (t)) \mathscr{E}_s^{1/2} (t) \mathscr{D}_s (t).
    \end{aligned}
\end{equation}
The constant $C_s > 0$ depends on $s, N, \mu_A, \lambda_A, \alpha, \bar{\mathbf{c}}, C_{Sob}, |\mathbb{T}^3|$.
The bound for $\omega^{-1}$ uses the Sobolev embedding, that is
\[  \|\omega^{-1}\|_{L_x^\infty} = \|(1 + \lambda \tilde{T})^{1 - \gamma/2}\|_{L_x^\infty} 
\leq (1 + \|\tilde{T}\|_{L_x^\infty})^{1 - \gamma/2}  \leq (1 + C_{Sob} \|\tilde{T}\|_{H_x^s})^{1 - \gamma/2} .
\]
Let the constant $\eta_1 = \eta_2 = \frac{ \lambda_A ( \min_{1 \leq i \leq N} \bar{c}_i)^2 }{3}$, the definitions of $d_1$ and $\chi$ deduce that we recover inequality \eqref{original energy esti}.

Since the functionals $G_1 = G_1 (I_1)$ and $G_2 = G_2 (I_2)$ depend on the composite functions $H_1(\lambda \tilde{T})$ and 
$H_2(\lambda \tilde{c}_{tot})$, there exists a positive constant $\gamma_1$ such that
\begin{equation}
G_1 + G_2 + G_1 G_2 + G_2^2 |_{I_1=I_2=[-\frac{1}{2},+\frac{1}{2}]} \leq \gamma_1.
\end{equation}
This is caused by the smoothness and continuity of the functions $H_1,H_2$ and $\tilde{T},\tilde{c}_{tot}$. Moreover, we can find a positive 
constant $\gamma_0<\gamma_1$ such that 
\begin{equation}
    G_1 + G_2 + G_1 G_2 + G_2^2 |_{I_1=I_2=[-\frac{1}{4},+\frac{1}{4}]} \leq \gamma_0,
\end{equation}
and there exists a constant $0<\iota_0 < 1$ such that $\gamma_0 \leq \iota_0 \gamma_1< \gamma_1$. Recall that we choose 
$I_1 = [ - \|\tilde{T}\|_{L_x^\infty}, + \|\tilde{T}\|_{L_x^\infty} ]$ and 
$I_2 = [ - \|\tilde{c}_{tot}, +\|\tilde{c}_{tot}\|_{L_x^\infty} ]$ to ensure the inequalities 
\eqref{per T Hs esti} and \eqref{estimate frac c-tot Hs} hold in the previous discussion. Using the Sobolev embedding, we can deduce the following bounds from the assumptions on the initial data in Proposition \ref{proposition for energy esti for M-S},
\[  \|\tilde{c}_{tot}^{in} \|_{L_x^\infty}, \|\tilde{T}^{in}\|_{L_x^\infty} \leq  \max\{1, \frac{1}{\sqrt{\chi}}\} 
C_{Sob} \delta_{MS} \leq 1/2,
\]
\[\text{ and} \quad C_s (1 + \gamma_1) (1 + \delta_{MS})^{1 - \gamma/2}(1 + \delta_{MS} + \delta_{MS}^{2 s + 4}) \delta_{MS} 
\leq 1/2.
\]
Obviously, the initial data satisfy $I_1,I_2 \subseteq [ - \frac{1}{2}, + \frac{1}{2}]$. Thus, for the initial time, the functionals can be bounded by 
\[ G_1 + G_2 + G_1 G_2 + G_2^2 \leq \gamma_1. 
\]
Due to the definitions of the constants $d_1$ and $d_2$, the inequality \eqref{original energy esti} becomes
\[\frac{\d}{\d t} \mathscr{E}_s (t) + \mathscr{D}_s (t) \leq 0, \]
for time $t=0$. We introduce
\begin{align*}
T^\ast := \sup & \left\{ \tau>0; \sup_{t \in [0,\tau]} \max\{ 1,\frac{1}{ \sqrt{\chi} } \} C_{Sob} \mathscr{E}_s^{1/2} (t)\leq  
\frac{ \min_{1 \leq i \leq N} \bar{c}_i }{2 \max_{1 \leq i \leq N} \bar{c}_i} \right.\\
 & \left.\text{ and } C_s (1 + \gamma_1) (1 + \mathscr{E}_s^{1/2} (t))^{1 - \gamma/2} (1 + \mathscr{E}_s^{1/2} (t) + \mathscr{E}_s^{s+2} (t)) 
\mathscr{E}_s^{1/2} (t) \leq \frac{1}{2} \right\},
\end{align*}
constant $C_s$ coincides with that in \eqref{original energy esti}. Then the above inequality holds on $t\in[0,T^\ast]$. We then prove $T^\ast = + \infty$. Otherwise, if $T^\ast < + \infty$, it is clear that for all $t\in[0,T^\ast]$, we have
\[  \mathscr{E}_s (t) + \int_0^t \mathscr{D}_s (\tau) \,d\tau \leq \mathscr{E}_s (0) \leq \delta_{MS}^2.
\]
The assumptions on $\delta_{MS}$ \eqref{assum for delta-MS} imply that
\[  \max\{ 1,\frac{1}{ \sqrt{\chi} } \}  C_{Sob} \mathscr{E}_s^{1/2} (t) \leq \frac{ \min_{1 \leq i \leq N} \bar{c}_i}{4 \max_{1 \leq i \leq N} \bar{c}_i} \leq \frac{1}{4} . 
 \]
Thus, the above discussions on functionals $G_1$ and $G_2$ deduce that for any $t \in [0, T^\ast]$, 
\[ 1 + G_1 + G_2 + G_1 G_2 + G_2^2 \leq \gamma_0 \leq \iota_0 \gamma_1 <\gamma_1.
\]
It guarantees that the following inequality holds for any $t \in [0, T^\ast]$
\begin{align*}
C_s (1 + G_1 + G_2 + G_1 G_2 + G_2^2) & (1 + \mathscr{E}_s^{1/2} (t))^{1 - \gamma/2}(1 + \mathscr{E}_s^{1/2} (t) + \mathscr{E}_s^{s+2} (t)) 
\mathscr{E}_s^{1/2} (t)\\
&\leq C_s(1+\iota\gamma_1)(1+\mathscr{E}_s^{1/2}(t))^{1-\gamma/2}(1+\mathscr{E}_s^{1/2}(t)+\mathscr{E}_s^{s+2}(t))
\mathscr{E}_s^{1/2}(t)\\
& < C_s(1 + \gamma_1)(1 + \delta_{MS})^{1 - \gamma/2}(1 + \delta_{MS} + \delta_{MS}^{2s + 4}) \delta_{MS} 
\leq \frac{1}{4} < \frac{1}{2}.
\end{align*}
Due to the continuity of the energy functional $\mathscr{E}_s(t)$, there exists a small constant $\tau_0>0$ such that 
for all $t \in [0, T^\ast + \tau_0]$,
\begin{gather*}
    \max\{ 1,\frac{1}{ \sqrt{\chi} } \} C_{Sob} \mathscr{E}_s^{1/2} (t)\leq \frac{ \min_{1 \leq i \leq N} \bar{c}_i}{2 \max_{1 \leq i \leq N} \bar{c}_i} \\
   C_s (1 + \gamma_1) (1 + \mathscr{E}_s^{1/2} (t))^{1 - \gamma/2} (1 + \mathscr{E}_s^{1/2} (t) + \mathscr{E}_s^{s+2} (t)) 
\mathscr{E}_s^{1/2} (t) \leq \frac{1}{2} ,
\end{gather*}
contradicting the definition of $T^\ast$. Therefore, $T^\ast = + \infty$. We conclude that the inequality \eqref{energy estimate under assu} holds for all $t \in \mathbb{R}_+$.

Moreover, we prove the positivities of the concentrations and the temperature. Applying the Sobolev embedding and the assumption \eqref{assum for delta-MS} for the bound $\delta_{MS}$, we have, for all $t\in\mathbb{R}_+$
\[  \|\tilde{\mathbf{c}}\|_{L_x^\infty} (t) \leq C_{Sob} \|\tilde{\mathbf{c}}\|_{H_x^s} (t) \leq C_{Sob} \max_{1 \leq i \leq N} 
\bar{c}_i  \|\tilde{\mathbf{c}}\|_{H_x^s (\bar{\mathbf{c}}^{-1}) }(t) \leq C_{Sob} (\max_{1 \leq i \leq N} \bar{c}_i) \delta_{MS}
\leq \frac{ \min_{1 \leq i \leq N}\bar{c}_i}{2},
\]
\[  \|\tilde{T}\|_{L_x^\infty} (t) \leq C_{Sob} \|\tilde{T}\|_{H_x^s} (t) \leq C_{Sob} \delta_{MS} \leq \frac{1}{2}.
\]
Thus the positivities hold as
\[  \bar{c}_{i} + \lambda \tilde{c}_i (t,x) \geq \bar{c}_{i} - \|\tilde{\mathbf{c}}\|_{L_x^\infty} (t) \geq 
\frac{ \min_{1 \leq i \leq N} \bar{c}_i}{2} > 0, \quad \forall 1 \leq i \leq N, (t,x) \in \mathbb{R}_+ \times \mathbb{T}^3, \lambda \in (0,1]
\]
and 
\[1 + \lambda \tilde{T} (t,x) \geq 1 - \|\tilde{T}\|_{L_x^\infty} (t) \geq \frac{1}{2} > 0 \quad \forall (t,x) \in \mathbb{R}_+ \times 
\mathbb{T}^3, \lambda \in (0,1].\]

In addition, the flux-force relations \eqref{equa for M-S rela} and Proposition \ref{prop A-1} imply that there exists a positive constant $C_{MS}$, such that the following inequality holds for any $t \geq 0$,
\begin{equation}
    \begin{aligned}
  \|\tilde{\mathbf{U}}\|_{H_x^{s-1}} (t) = & \|(1 + \lambda \tilde{T})^{1 - \gamma/2} A (\mathbf{c})^{-1} \nabla_x \tilde{\mathbf{c}} 
  + (1 + \lambda \tilde{T})^{- \gamma/2} \nabla_x \tilde{T} A (\mathbf{c})^{-1} (\bar{\mathbf{c}} + \lambda \tilde{\mathbf{c}})\|_{H_x^s}  \\
& \leq C (\bar{\mathbf{c}}, s, N, \delta_{MS}, \lambda_A, \mu_A,C_{Sob},|\mathbb{T}^3|) ( \|\nabla_x \tilde{\mathbf{c}}\|_{H_x^{s-1}} 
+ \|\nabla_x \tilde{T}\|_{H_x^{s-1}} )\\
& \leq C_{MS} \delta_{MS}.\label{pert U bound in L-t-infty}
    \end{aligned}
\end{equation}
The constant $C_{MS}$ is independent of the parameters $\lambda$ and $\delta_{MS}$, since we impose an upper bound for $\delta_{MS}$ in the assumption.
\end{proof}

\subsection{Existence and Uniqueness of the Solution to the Maxwell-Stefan System}
\textbf{Existence of the perturbation solution}

We first emphasize that the results in Proposition \ref{proposition for energy esti for M-S} can be directly applied, since the assumptions there 
are the same as those in Theorem \ref{theorem for M-S in orthogonal version}. We will develop a so-called {\em pseudo-nonlinear iterative approximated method} on the time interval $[0,T_0]$ 
with any positive constant $T_0$, and initially set 
\[  \tilde{\mathbf{c}}^{(0)} := \tilde{\mathbf{c}}^{in}, \quad \tilde{T}^{(0)} := \tilde{T}^{in}.
\]
The assumptions on the initial data ensure that 
\[  \mathscr{E}_s^{(0)} := \|\tilde{\mathbf{c}}^{(0)}\|_{H_x^s (\bar{\mathbf{c}}^{-1}) }^2 + \|\tilde{T}^{(0)}\|_{H_x^s}^2 + \chi \|\sum_{i=1}^N \tilde{c}_i^{(0)}\|_{H_x^s} \leq \delta_{MS}^2.
\]
The assumptions (A2) and (A3$^\prime$) infer that we can define $\tilde{\mathbf{U}}^{(0)} := \tilde{\mathbf{U}}^{in}$, which satisfies 
not only $\tilde{\mathbf{U}}^{(0)} \in \mathrm{Span} (\mathbf{1})^\perp$ but also the flux-force relations. We assume that 
there are $\tilde{\mathbf{c}}^{(n)} \in L^\infty (0, T_0; H_x^s)$ and $\tilde{T}^{(n)} \in L^\infty (0, T_0; H_x^s)$ with $n\in \mathbb{N}^\ast$, satisfying that 
for any $(t,x) \in [0,T_0] \times \mathbb{T}^3$
\begin{equation}
     ( \bar{c}_{tot} + \lambda \tilde{c}_{tot}^{(n-1)} ) \nabla_x \tilde{T}^{(n)} + (1 + \lambda \tilde{T}^{(n-1)} ) 
     \nabla_x \tilde{c}_{tot}^{(n)} = 0 , \quad \text{with } \tilde{c}_{tot}^{(n)}=\sum_{i=1}^N \tilde{c}_i^{(n)},
\end{equation}
and for any $t\in[0,T_0]$
\[  \mathscr{E}_s^{(n)} (t) := \|\tilde{\mathbf{c}}^{(n)}\|_{H_x^s (\bar{\mathbf{c}}^{-1}) }^2 (t) + \|\tilde{T}^{(n)}\|_{H_x^s}^2 (t) 
+ \chi \|\tilde{c}_{tot}^{(n)}\|_{H_x^s} (t) \leq \delta_{MS}^2.
\]
Define $\tilde{\mathbf{U}}^{(n)} \in \mathrm{Span} (\mathbf{1})^\perp$ as 
\begin{align*}
 \tilde{\mathbf{U}}^{(n)} : =& (1 + \lambda \tilde{T}^{(n-1)})^{1-\gamma/2} A(\bar{\mathbf{c}} + \lambda \tilde{\mathbf{c}}^{(n-1)} )^{-1} 
\nabla_ x \tilde{\mathbf{c}}^{(n)}\\
& + (1 + \lambda \tilde{T}^{(n-1)})^{-\gamma/2} A(\bar{\mathbf{c}} + \lambda \tilde{\mathbf{c}}^{(n-1)} )^{-1} 
(\bar{\mathbf{c}} + \lambda \tilde{\mathbf{c}}^{(n-1))}) \nabla_x \tilde{T}^{(n)},
\end{align*}
then $\tilde{\mathbf{U}}^{(n)}$ belongs to $L^\infty (0, T_0; H_x^{s-1})$. 
By induction, we define $(\tilde{\mathbf{c}}^{(n+1)}, \tilde{T}^{(n+1)})$ as 
 \begin{equation}
   \nabla_x \tilde{c}_{tot}^{(n+1)} (1 + \lambda \tilde{T}^{(n)}) + \nabla_x \tilde{T}^{(n+1)} (\bar{c}_{tot} + \lambda 
   \tilde{c}_{tot}^{(n)} ) = 0, \quad \text{with }\tilde{c}_{tot}^{(n+1)} = \sum_{i=1}^N \tilde{c}_i^{(n+1)},
   \label{solvability of n+1}
    \end{equation}
    which enables us to define $\tilde{\mathbf{U}}^{(n+1)} \in \mathrm{Span} (\mathbf{1})^\perp$ as 
\begin{equation}
    \begin{aligned}
         \tilde{\mathbf{U}}^{(n+1)} : =& (1 + \lambda \tilde{T}^{(n)})^{1-\gamma/2} A(\bar{\mathbf{c}} + \lambda \tilde{\mathbf{c}}^{(n)} )^{-1} 
\nabla_ x \tilde{\mathbf{c}}^{(n+1)}\\
& + (1 + \lambda \tilde{T}^{(n)})^{-\gamma/2} A(\bar{\mathbf{c}} + \lambda \tilde{\mathbf{c}}^{(n)} )^{-1} 
(\bar{\mathbf{c}} + \lambda \tilde{\mathbf{c}}^{(n))}) \nabla_x \tilde{T}^{(n+1)}.
    \end{aligned}
\end{equation}
The quantities $\tilde{\mathbf{c}}^{(n+1)}$ and $\tilde{T}^{(n+1)}$ satisfy the following equations 
\begin{equation}
\begin{aligned}
    \partial_t \tilde{\mathbf{c}}^{(n+1)} + \nabla_x \cdot \left[ ( \bar{\mathbf{c}} + \lambda \tilde{\mathbf{c}}^{(n)} ) \left ( 
    \tilde{\mathbf{U}}^{(n+1)} -\frac{\langle \bar{\mathbf{c}} +\lambda \tilde{\mathbf{c}}^{(n)} , \tilde{\mathbf{U}}^{(n+1)} \rangle }
    {\bar{c}_{tot} + \lambda \tilde{c}_{tot}^{(n)} }   \mathbf{1} \right )    \right] \\
    - \nabla_x \cdot \left[ ( \bar{\mathbf{c}} + \lambda \tilde{\mathbf{c}}^{(n)} )  \frac{\alpha \nabla_x \tilde{c}_{tot}^{(n+1)}}
    {\bar{c}_{tot} + \lambda \tilde{c}_{tot}^{(n)} } \mathbf{1}   \right] =0, \label{iter mass conser}
\end{aligned}
\end{equation}
\begin{equation}
    \begin{aligned}
\partial_t \tilde{T}^{(n+1)} - \frac{2\alpha}{3} \nabla_x \cdot \left[ \frac{1 + \lambda \tilde{T}^{(n)}}
{\bar{c}_{tot} + \lambda \tilde{c}_{tot}^{(n)}} \nabla_x \tilde{c}_{tot}^{(n+1)} \right]
& - \frac{2 \alpha \lambda}{3} \frac{1 + \lambda \tilde{T}^{(n)}}{ (\bar{c}_{tot} + \lambda \tilde{c}_{tot}^{(n)})^2 } 
| \nabla_x \tilde{c}_{tot}^{(n+1)} |^2\\
& - \alpha \lambda \frac{ \nabla_x \tilde{c}_{tot}^{(n+1)} \cdot \nabla_x \tilde{T}^{(n+1)}}{ \bar{c}_{tot} +
\lambda \tilde{c}_{tot}^{(n)}} = 0,  \label{iter T equa}
    \end{aligned}
\end{equation}
supplemented with the initial data that for almost any $x\in\mathbb{T}^3$ 
\begin{equation}
\tilde{\mathbf{c}}^{(n+1)} (0,x) = \tilde{\mathbf{c}}^{in} (x), 
\quad  \tilde{T}^{(n+1)} (0,x) = \tilde{T}^{in} (x). \label{initial data set for per M-S n+1}
\end{equation}
The equation \eqref{iter mass conser} for $\tilde{\mathbf{c}}^{(n+1)}$ can also be written in a form similar to the equation 
\eqref{per mass-linear}. It has a natural derivation by taking the scalar product with $\mathbf{1}$ in $\mathbb{R}^N$,
\begin{equation}
\partial_t \tilde{c}_{tot}^{(n+1)} = \alpha \Delta_x \tilde{c}_{tot}^{(n+1)}. \label{equ for total con of n+1}
\end{equation}
The main features of the pseudo-nonlinear iterative approximated method are as follows: Although the equation \eqref{iter T equa} appears nonlinear due to the terms $| \nabla_x \tilde{c}_{tot}^{(n+1)} |^2$ and $\nabla_x \tilde{c}_{tot}^{(n+1)} \cdot \nabla_x \tilde{T}^{(n+1)}$, 
we clarify that it is actually linear with respect to $\tilde{T}^{(n+1)}$. 
This is because $\tilde{c}_{tot}^{(n+1)}$ is solvable due to the heat equation \eqref{equ for total con of n+1}.

Now we introduce the functionals 
\begin{gather}
\mathscr{E}_s^{(n+1)} (t) := \|\tilde{\mathbf{c}}^{(n+1)}\|_{H_x^s( \bar{\mathbf{c}}^{-1} )}^2 (t) + \|\tilde{T}^{(n+1)}\|_{H_x^s}^2 (t) 
+ \chi \|\tilde{c}_{tot}^{(n+1)}\|_{H_x^s}, \label{energy functional n+1}
\\
\mathscr{D}_s^{(n+1)} (t) := d_1 \|\tilde{\mathbf{U}}^{(n+1)} \|_{H_x^s(\omega^{(n)})}^2 (t) +d_2 \|\nabla_x\tilde{c}_{tot}^{(n+1)}\|_{H_x^s}^2 (t), \label{dissiaption functional n+1}
\end{gather}
with the notation $\omega^{(n)} := (1 + \lambda \tilde{T}^{(n)})^{\gamma/2 - 1}$. Following the above calculations in the energy estimate, we can obtain 
    \begin{align*}
        \frac{1}{2} \frac{\d}{\d t} & \mathscr{E}^{(n+1)}_s (t) + d_1\|\tilde{\mathbf{U}}^{(n+1)}\|_{H_x^s(\omega^{(n)})}^2 (t)\\
         + & \left ( \frac{\alpha}{\bar{c}_{tot}} + \frac{2\alpha}{3\bar{c}_{tot}^2} + \frac{6 (\frac{3}{2})^{1 - \frac{\gamma}{2}} 
        - 3 [1 + C_{Sob} (\mathscr{E}^{(n)}_s)^{1/2} (t) ]^{1 - \frac{\gamma}{2}} }{2 \lambda_A ( \min_{1 \leq i \leq N} \bar{c}_i)^2}  
        \right )  \|\nabla_x \tilde{c}_{tot}^{(n+1)}\|_{H_x^s}^2 (t)\\
    & \leq C_s (1 + G_1 + G_2 + G_1 G_2 + G_2^2)[ 1 + (\mathscr{E}^{(n)}_s)^{1/2} (t) ]^{1 - \gamma/2}\\
    &\times [1 + (\mathscr{E}^{(n)}_s)^{1/2} (t) + (\mathscr{E}^{(n)}_s)^{s+2} (t) ] 
    (\mathscr{E}^{(n)}_s)^{1/2} (t) \mathscr{D}^{(n+1)}_s (t).  \label{energy esti for n+1}
    \end{align*}
    The constant $C_s$ coincides with the one in inequality \eqref{original energy esti}. Here the functionals $G_1$ and $G_2$ depend on the intervals $I_1=[- \|\tilde{T}^{(n)}\|_{L_x^\infty} , +\|\tilde{T}^{(n)}\|_{L_x^\infty} ]$ and 
$I_2= [- \|\tilde{c}_{tot}^{(n)}\|_{L_x^\infty}, + \|\tilde{c}_{tot}^{(n)}\|_{L_x^\infty}]$. 
Since $\mathscr{E}_s^{(n)} (t) \leq \delta_{MS}^2$ for any $t\in[0,T_0]$, by repeating the previous discussions in the energy estimate, we finally can deduce that for any $t\in[0,T_0]$
\begin{equation}
 \mathscr{E}^{(n+1)}_s (t) + \int_0^t \mathscr{D}^{(n+1)}_s (\tau) \,d\tau \leq \mathscr{E}^{(n+1)}_s (0) \leq \delta_{MS}^2.
 \label{energy esti for n+1}
\end{equation}

Based on this inequality, there exist $\tilde{c}_{tot}^{(n+1)}$ to the equation \eqref{equ for total con of n+1} that 
belongs to $L^\infty (0, T_0; H_x^s)$ and 
$\nabla_x \tilde{c}_{tot}^{(n+1)}$ that belongs to $L^2(0,T_0;H_x^s)$. When $\tilde{c}_{tot}^{(n+1)}$ is known and 
the relation \eqref{solvability of n+1} holds, the equation \eqref{iter T equa} is linear to $\tilde{T}^{(n+1)}$. 
It is not hard to deduce that there exists $\tilde{T}^{(n+1)}$ belonging to $L^\infty(0,T_0;H_x^s)$ for this equation. 
Then \eqref{iter mass conser} is a linear hyperbolic-parabolic equation for $\tilde{\mathbf{c}}^{(n+1)}$, 
standard method raises the existence of $\tilde{\mathbf{c}}^{(n+1)}$ belonging to $L^\infty (0, T_0; H_x^s)$. 
Moreover, the relation $\tilde{c}_{tot}^{(n+1)} =\sum_{i=1}^N \tilde{c}_i^{(n+1)}$ holds 
since the solution to \eqref{equ for total con of n+1} is unique. 
At last, we derive $\tilde{\mathbf{U}}^{(n+1)} \in L^\infty (0,T_0; H_x^{s-1})$ from its definition, 
and $(\tilde{\mathbf{c}}^{(n+1)}, \tilde{\mathbf{U}}^{(n+1)},\tilde{T}^{(n+1)} )$ satisfies the inequality \eqref{energy esti for n+1}.

By induction, there exist sequences $(\tilde{\mathbf{c}}^{(n)})_{n \in \mathbb{N}}$ and $(\tilde{T}^{(n)})_{n \in \mathbb{N}}$, having a uniform bound. We rewrite the equation for $\tilde{T}^{(n+1)}$ as 
\begin{equation}
    \partial_t \tilde{T}^{(n+1)} +\frac{2 \alpha}{3} \Delta_x \tilde{T}^{(n+1)} +\frac{\alpha \lambda}{3} \frac{|\nabla_x \tilde{T}^{(n+1)}|^2} 
    {1 + \lambda \tilde{T}^{(n)}} =0, \label{transformation of the iter T equ}
\end{equation}
by applying the relation \eqref{solvability of n+1} in the form 
$\frac{\nabla_x \tilde{c}_{tot}^{(n+1)}}{\bar{c}_{tot} +\lambda \tilde{c}_{tot}^{(n)}} = - \frac{\nabla_x \tilde{T}^{(n+1)}}{1 +
\lambda \tilde{T}^{(n)}}$. Due to the Sobolev embedding $H_x^2 \hookrightarrow L_x^\infty$ and the assumption $s>3$, 
it is not hard to find a polynomial $\mathcal{P}_1$, with coefficients depending on $ \alpha, s$, 
such that for any $t \in [0,T_0]$, 
\begin{align*} 
\| \partial_t \tilde{T}^{(n+1)}\|_{L_x^2} & \lesssim \mathcal{P}_1 \left ( \|\tilde{T}^{(n)}\|_{H_x^s}, 
\|\tilde{T}^{(n+1)}\|_{H_x^s} \right ) \lesssim \mathcal{P}_1 (\delta_{MS}).
\end{align*}
Using the Aubin-Lions Theorem, it deduces that there exists $\tilde{T}^\infty \in L^\infty (0, T_0; H_x^s)$ 
such that for any $s^\prime <s$, up to a subsequence, 
\begin{itemize}
    \item  $(\tilde{T}^{(n)})_{n \in \mathbb{N}}$ converges to $\tilde{T}^\infty$ weakly in 
    $L^\infty (0, T_0 ; H_x^s)$ (specifically, weakly-$\ast$ in $L^\infty (0, T_0)$), 
    \item $(\tilde{T}^{(n)})_{n \in \mathbb{N}}$ converges to $\tilde{T}^\infty$ 
    strongly in $C^0(0, T_0; H_x^{s^\prime})$.
\end{itemize}
Integrating \eqref{transformation of the iter T equ} against a test function, the nonlinear terms in it can be bounded 
and dealt with by applying the strong convergence in $H_x^{s^\prime}$. Therefore, it can pass to the limit as $n \rightarrow +\infty $, writing as 
\begin{equation}
    \partial_t \tilde{T}^\infty + \frac{2 \alpha}{3} \Delta_x \tilde{T}^\infty +\frac{\alpha \lambda}{3} \frac{|\nabla_x \tilde{T}^\infty|^2} 
    {1 + \lambda \tilde{T}^\infty} =0.
\end{equation}

Similar discussions can be applied on the sequence $(\tilde{\mathbf{c}}^{(n)})_{n \in \mathbb{N}}$, deducing that there exists $\tilde{\mathbf{c}}^\infty \in L^\infty (0,T_0; H_x^s)$, such that for any $s^\prime <s$, up to a subsequence,
\begin{itemize}
    \item $(\tilde{\mathbf{c}}^{(n)})_{n \in \mathbb{N}}$ converges to $\tilde{\mathbf{c}}^\infty$ 
    weakly in $L^\infty (0,T_0 ; H_x^s)$ (specifically, weakly-$\ast$ in $L^\infty (0, T_0)$),
    \item $(\tilde{\mathbf{c}}^{(n)})_{n \in \mathbb{N}}$ converges to $\tilde{\mathbf{c}}^\infty$ 
    strongly in $C^0(0,T_0;H_x^{s^\prime})$.
\end{itemize}
Define $\tilde{c}_{tot}^\infty := \sum_{i=1}^N \tilde{c}_i^\infty$, it is not hard to notice that the equation \eqref{solvability of n+1} can also pass to 
the limit as $n \rightarrow +\infty$, leading to
\[  \nabla_x \tilde{c}_{tot}^\infty (1 +\lambda \tilde{T}^\infty ) + \nabla_x \tilde{T}^\infty ( \bar{c}_{tot} +\lambda \tilde{c}_{tot}^\infty )=0.
\]
Then we can define $\tilde{\mathbf{U}}^\infty \in \mathrm{Span} (\mathbf{1})^\perp$ as 
\[  \tilde{\mathbf{U}}^\infty = (1 + \lambda \tilde{T}^\infty)^{1-\gamma/2} [A (\bar{\mathbf{c}} + \lambda \tilde{\mathbf{c}}^\infty)]^{-1} 
\nabla_x \tilde{\mathbf{c}} + (1 + \lambda \tilde{T}^\infty)^{-\gamma/2} [A (\bar{\mathbf{c}} + \lambda \tilde{\mathbf{c}}^\infty)]^{-1} (\bar{\mathbf{c}} + \lambda  \tilde{\mathbf{c}}^\infty) \nabla_x \tilde{T},
\]
which belongs to $L^\infty (0,T_0; H_x^{s-1})$. Since we can deal with the nonlinear terms by applying the strong convergences, the evolution equations for 
$(\tilde{\mathbf{c}}^{(n)})_{n \in \mathbb{N}}$ can also pass to the limit 
\begin{equation*}
\begin{aligned}
    \partial_t \tilde{\mathbf{c}}^\infty + \nabla_x \cdot \left[ ( \bar{\mathbf{c}} + \lambda \tilde{\mathbf{c}}^\infty ) \left ( 
    \tilde{\mathbf{U}}^\infty -\frac{\langle \bar{\mathbf{c}} +\lambda \tilde{\mathbf{c}}^\infty , \tilde{\mathbf{U}}^\infty \rangle }
    {\bar{c}_{tot} + \lambda \tilde{c}_{tot}^\infty }   \mathbf{1} \right )    \right] 
    - \nabla_x \cdot \left[ ( \bar{\mathbf{c}} + \lambda \tilde{\mathbf{c}}^\infty )  \frac{\alpha \nabla_x \tilde{c}_{tot}^\infty }
    {\bar{c}_{tot} + \lambda \tilde{c}_{tot}^\infty } \mathbf{1}   \right] =0. 
\end{aligned}
\end{equation*}
Here we use the fact that $(1 + \lambda \tilde{T}^{(n)})^{\gamma/2}$ converges strongly to $(1 + \lambda \tilde{T}^\infty)^{\gamma/2}$ 
in $H_x^{s^\prime}$, obtained from the Dominated Convergence Theorem. We conclude that 
$(\tilde{\mathbf{c}}^\infty, \tilde{\mathbf{U}}^\infty, \tilde{T}^\infty)$ is a classical solution to the system of equations \eqref{per mass-linear}-\eqref{equa for M-S rela}-\eqref{equa for per T}.  Moreover, it is not hard to check 
$\partial_t \tilde{\mathbf{U}}^\infty \in L^\infty (0,T_0; L_x^2)$, resulting in 
$\tilde{\mathbf{U}}^\infty \in C^0 (0,T_0; H_x^{s^\prime-1})$. 
Finally, by using the continuity of $(\tilde{\mathbf{c}}^\infty, \tilde{\mathbf{U}}^\infty, \tilde{T}^\infty)$, we conclude that the following 
inequality holds for $t\in [0,T_0]$
\[  \mathscr{E}_s^\infty(t) := \|\tilde{\mathbf{c}}^\infty\|_{H_x^s (\bar{\mathbf{c}}^{-1}) }^2 (t) + 
\|\tilde{T}^\infty\|_{H_x^s}^2 (t) + \chi \|\tilde{c}_{tot}^\infty\|_{H_x^s}^2 \leq \delta_{MS}^2,
\]
owing to the energy estimates established in the previous. 

Reset $t=T_0$ as an initial moment and repeat the above processes on the time interval 
$[T_0, 2T_0]$, the corresponding energy functional can be bounded as $\mathscr{E}_s^\infty (2T_0) \leq \delta_{MS}^2$. 
By induction, we conclude that there is a classical solution for the perturbative system 
\eqref{per mass-linear}-\eqref{equa for M-S rela}-\eqref{equa for per T} on $t \in \mathbb{R}_+$.

\textbf{Uniqueness of the perturbation solution}

Consider $(\tilde{\mathbf{c}}, \tilde{\mathbf{U}}, \tilde{T})$ and $(\tilde{\mathbf{d}}, \tilde{\mathbf{V}}, \tilde{\theta})$ are two solutions to the system of equations \eqref{per mass-linear}-\eqref{equa for M-S rela}-\eqref{equa for per T} and satisfy the assumptions in Theorem \ref{theorem for M-S in orthogonal version}. We introduce the notations frequently used below
\[  \mathbf{d} = \bar{\mathbf{c}} + \lambda \tilde{\mathbf{d}}, \quad \tilde{d}_{tot} = \sum_{i=1}^N \tilde{d}_i, 
\quad \tilde{\bar{v}} = \frac{\alpha \nabla_x \tilde{d}_{tot}}{\bar{c}_{tot}}, \quad 
M_1 (\tilde{\bar{v}})= \frac{\tilde{d}_{tot}}{\bar{c}_{tot} + \lambda \tilde{d}_{tot}} \tilde{\bar{v}},
\]
\[ \tilde{\mathbf{V}}^\ast = \tilde{\mathbf{V}} - \frac{\langle \bar{\mathbf{c}}, \tilde{\mathbf{V}}\rangle}{\bar{c}_{tot}} \mathbf{1}, 
\quad M_2 (\tilde{\mathbf{V}}) = \frac{\langle \tilde{\mathbf{d}}, \tilde{\mathbf{V}} \rangle}{\bar{c}_{tot}} 
- \frac{\tilde{d}_{tot} \langle \bar{\mathbf{c}} + \lambda \tilde{\mathbf{d}}, \tilde{\mathbf{V}} \rangle}{\bar{c}_{tot} 
(\bar{c}_{tot} + \lambda \tilde{d}_{tot})},
\]
and denote $\tilde{\mathbf{h}} = \tilde{\mathbf{c}} - \tilde{\mathbf{d}}$, $\tilde{\mathbf{W}} = \tilde{\mathbf{U}} - \tilde{\mathbf{V}}$ with $\tilde{\mathbf{W}}\in \mathrm{Span} (\mathbf{1})^\perp$ and 
$\tilde{I} = \tilde{T} - \tilde{\theta}$. The corresponding notations are $\tilde{h}_{tot} = \langle \tilde{\mathbf{h}}, \mathbf{1} \rangle$ and 
$\tilde{\mathbf{W}}^\ast = \tilde{\mathbf{W}} - \frac{\langle \bar{\mathbf{c}}, \tilde{\mathbf{W}} \rangle}{\bar{c}_{tot}} \mathbf{1}$, 
$\tilde{\bar{w}} = \frac{\alpha \nabla_x \tilde{h}_{tot}}{\bar{c}_{tot}}$. 
Recalling the equation \eqref{equation for per total concentration} that $\tilde{c}_{tot}$ and $\tilde{d}_{tot}$ satisfy, the equation for $\tilde{h}_{tot}$ is
\[ \partial_t \tilde{h}_{tot} = \alpha \Delta_x \tilde{h}_{tot}, \quad \alpha >0,
\]
with the initial data $\tilde{h}_{tot} (0,x) = 0$ a.e. $x \in \mathbb{T}^3$. It is clear that $\tilde{h}_{tot} = 0$ for a.e. $(t,x) \in \mathbb{R}_+ \times \mathbb{T}^3$. 
% Next we focus on the subtraction of the evolution equations for temperatures, recall 
Next, the subtraction of the evolution equations for temperatures is 
\begin{equation}
\begin{aligned}
\partial_t \tilde{I} - \frac{2 \alpha \lambda }{3} \nabla_x \cdot \left[ \frac{\nabla_x \tilde{c}_{tot}}{\bar{c}_{tot} + \lambda 
\tilde{d}_{tot}} \tilde{I} \right] - \frac{2 \alpha \lambda^2}{3} \frac{ |\nabla_x \tilde{c}_{tot}|^2}
{ (\bar{c}_{tot} + \lambda \tilde{d}_{tot})^2} \tilde{I} - \alpha \lambda \frac{\nabla_x \tilde{d}_{tot}}{\bar{c}_{tot} 
+ \lambda \tilde{d}_{tot}} \cdot \nabla_x \tilde{I} = 0.
\end{aligned}
\end{equation}
It uses the following calculations 
\[  \frac{1}{\bar{c}_{tot} + \lambda \tilde{c}_{tot}} - \frac{1}{\bar{c}_{tot} + \lambda \tilde{d}_{tot}} = 
\frac{-\lambda \tilde{h}_{tot}}{ (\bar{c}_{tot} + \lambda \tilde{c}_{tot}) (\bar{c}_{tot} + \lambda \tilde{d}_{tot}) } = 0,
 \] 
\[  \frac{1}{ (\bar{c}_{tot} + \lambda \tilde{c}_{tot})^2 } - \frac{1}{ (\bar{c}_{tot} + \lambda \tilde{d}_{tot})^2 } 
= \frac{ - 2 \lambda \bar{c}_{tot} \tilde{h}_{tot} - \lambda^2 (\tilde{c}_{tot} + \tilde{d}_{tot}) \tilde{h}_{tot}}
{ (\bar{c}_{tot} + \lambda \tilde{c}_{tot})^2 (\bar{c}_{tot} + \lambda \tilde{d}_{tot})^2 } = 0,
\]
and
\[  |\nabla_x \tilde{c}_{tot}|^2 - |\nabla_x\tilde{d}_{tot}|^2 = \nabla_x \tilde{c}_{tot} \cdot \nabla_x \tilde{h}_{tot} 
+ \nabla_x \tilde{d}_{tot} \cdot \tilde{h}_{tot} = 0,
\]
which hold since $\tilde{h}_{tot}=0$. Moreover, the solvability conditions for $\tilde{\mathbf{U}}$ and $ \tilde{\mathbf{V}}$ provide relations
\[\nabla_x \tilde{I} (\bar{c}_{tot} + \lambda \tilde{d}_{tot}) + \lambda \nabla_x \tilde{d}_{tot} \tilde{I} = 0,
\]
and
\[  \nabla_x \tilde{I} (\bar{c}_{tot} + \lambda \tilde{d}_{tot}) + \lambda \nabla_x \tilde{c}_{tot} \tilde{I} = 0.
\]
The difference is caused by the equality
\begin{align*}
\nabla_x \tilde{c}_{tot} \tilde{T} - \nabla_x \tilde{d}_{tot} \tilde{\theta} = \nabla_x \tilde{c}_{tot} \tilde{T} 
- \nabla_x \tilde{d}_{tot} \tilde{T} + \nabla_x \tilde{d}_{tot} \tilde{T} - \nabla_x \tilde{d}_{tot} \tilde{\theta} 
= \nabla_x \tilde{h}_{tot} \tilde{T} + \nabla_x \tilde{d}_{tot} \tilde{I}\\
= \nabla_x \tilde{c}_{tot} \tilde{T} - \nabla_x\tilde{c}_{tot} \tilde{\theta} + \nabla_x \tilde{c}_{tot} \tilde{\theta} - \nabla_x 
\tilde{d}_{tot} \tilde{\theta} = \nabla_x \tilde{c}_{tot} \tilde{I} + \nabla_x \tilde{h}_{tot} \tilde{\theta}.
\end{align*}
We further deduce the expression 
\[\nabla_x \tilde{I} = - \frac{\lambda \tilde{d}_{tot}}{\bar{c}_{tot} + \lambda \tilde{d}_{tot}} \tilde{I} 
= - \frac{\lambda \tilde{c}_{tot}}{\bar{c}_{tot} + \lambda \tilde{d}_{tot}} \tilde{I}.
\]
Now combining the relation $|\tilde{c}_{tot}|^2 = |\tilde{d}_{tot}|^2$, we can obtain the $L_x^2$ norm of $\tilde{I}$:
\begin{equation}
\frac{1}{2} \frac{\d}{\d t}\|\tilde{I} \|_{L_x^2}^2 = \frac{\alpha \lambda^2}{3} \int_{\mathbb{T}^3} \frac{|\nabla_x \tilde{d}_{tot}|^2}
{ (\bar{c}_{tot} + \lambda \tilde{d}_{tot})^2 } \tilde{I}^2 \,\d x \lesssim \delta_{MS}^2 \|\tilde{I}\|_{L_x^2}^2,
\end{equation}
where we use the assumptions that $\|\tilde{d}_{tot}\|_{H_x^s}^2 \leq \delta_{MS}^2$ with $s>3$ and the Sobolev embedding. We conclude that $\tilde{I} = 0$ 
almost everywhere in $(t,x) \in \mathbb{R}_+ \times \mathbb{T}^3$, deduced from Grönwall's inequality given the initial data $\tilde{I}^{in} = 0$.

Finally, we focus on the subtraction of the mass conservation equations and the flux-force relations, writing as
\begin{equation}
    \begin{aligned}
    \partial_t \tilde{\mathbf{h}} + \bar{\mathbf{c}} (\nabla_x \cdot \tilde{\mathbf{W}}^\ast) & + \lambda \nabla_x \cdot 
    (\tilde{\mathbf{h}} \tilde{\mathbf{U}}^\ast) + \lambda \nabla_x \cdot (\tilde{\mathbf{d}} \tilde{\mathbf{W}}^\ast) - \lambda \nabla_x \cdot 
    (\tilde{\mathbf{h}} \tilde{\bar{u}}) - \lambda \nabla_x \cdot [ \bar{\mathbf{c}} (N_2 (\tilde{\mathbf{U}}) - M_2 (\tilde{\mathbf{V}}) ) ]\\
    & - \lambda^2 \nabla_x \cdot (\tilde{\mathbf{h}} N_2 (\tilde{\mathbf{U}})) - \lambda^2 \nabla_x \cdot [ \tilde{\mathbf{d}} 
    (N_2 (\tilde{\mathbf{U}}) - M_2 (\tilde{\mathbf{V}}) ) ] + \lambda^2 \nabla_x \cdot  ( \tilde{\mathbf{h}} N_1 (\tilde{\bar{u}}) ) = 0,
    \label{subtra of mass cons}
\end{aligned}
\end{equation}
\begin{equation}
    \begin{aligned}
    \nabla_x \tilde{\mathbf{h}} (1 + \lambda \tilde{\theta}) + \lambda \tilde{\mathbf{h}} \nabla_x \tilde{\theta} = (1 + \lambda 
    \tilde{\theta})^{\gamma/2} [A (\mathbf{c}) - A(\mathbf{d})] \tilde{\mathbf{U}} + (1 + \lambda \tilde{\theta})^{\gamma/2} 
    A (\mathbf{d}) \tilde{\mathbf{W}} . \label{subtra of M-S rela}
      \end{aligned}
\end{equation} 
The deduction uses the fact that $\tilde{h}_{tot} = 0$, $\tilde{I} = 0$ $\tilde{T} = \tilde{\theta}$, 
$\tilde{\bar{w}} = \frac{\alpha \nabla_x \tilde{h}_{tot}}{\bar{c}_{tot}} = 0$, and 
\[ N_1 (\tilde{\bar{u}}) - M_1(\tilde{\bar{v}}) 
=\frac{\tilde{c}_{tot} \tilde{\bar{w}}}{\bar{c}_{tot}+\lambda \tilde{c}_{tot}}  - \frac{\lambda \tilde{c}_{tot} \tilde{h}_{tot} \tilde{\bar{v}}}{(\bar{c}_{tot}+\lambda \tilde{c}_{tot})(\bar{c}_{tot}+\lambda \tilde{d}_{tot})} +\frac{\tilde{h}_{tot} \tilde{\bar{v}}}{\bar{c}_{tot}+\lambda \tilde{c}_{tot}}=0.
\]
 By applying the expression
\begin{equation}
N_2 (\tilde{\mathbf{U}}) - M_2 (\tilde{\mathbf{V}}) = \frac{\langle \mathbf{h}, \tilde{\mathbf{U}} \rangle 
+ \langle \mathbf{d}, \tilde{\mathbf{W}}\rangle}{\bar{c}_{tot}} - \frac{\tilde{d}_{tot} (\lambda 
\langle \tilde{\mathbf{h}}, \tilde{\mathbf{U}} \rangle + \langle \bar{\mathbf{c}} + \lambda \tilde{\mathbf{d}}, \tilde{\mathbf{W}}\rangle)} 
 {\bar{c}_{tot} (\bar{c}_{tot} + \lambda \tilde{d}_{tot})},
\end{equation}
we obtain the $L_x^2 (\bar{\mathbf{c}}^{-1})$ norm of $\tilde{\mathbf{h}}$ as 
\begin{align*}
\frac{1}{2} \frac{\d}{\d t} \|\tilde{\mathbf{h}}\|_{L_x^2 (\bar{\mathbf{c}}^{-1}) }^2 & = \int_{\mathbb{T}^3} 
\langle \tilde{\mathbf{W}}, \nabla_x\tilde{\mathbf{h}} \rangle \,\d x + \lambda \int_{\mathbb{T}^3} 
\langle \tilde{\mathbf{h}} \tilde{\mathbf{U}}^\ast, \nabla_x \tilde{\mathbf{h}} \rangle_{ \bar{\mathbf{c}}^{-1} } \,\d x 
+ \lambda \int_{\mathbb{T}^3} \langle \tilde{\mathbf{d}} \tilde{\mathbf{W}}^\ast, \nabla_x \tilde{\mathbf{h}} 
\rangle_{ \bar{\mathbf{c}}^{-1} } \,\d x\\
& - \lambda \int_{\mathbb{T}^3} \langle \tilde{\mathbf{h}} \tilde{\bar{u}}, \nabla_x \tilde{\mathbf{h}} \rangle_{ \bar{\mathbf{c}}^{-1} } 
\,\d x - \lambda^2 \int_{\mathbb{T}^3} \langle \tilde{\mathbf{h}} N_2 (\tilde{\mathbf{U}}), \nabla_x \tilde{\mathbf{h}} 
\rangle_{ \bar{\mathbf{c}}^{-1} } \,\d x\\
& - \lambda^2 \int_{\mathbb{T}^3}\langle \tilde{\mathbf{d}} \left ( 
  \frac{\langle \tilde{\mathbf{h}}, \tilde{\mathbf{U}} \rangle }{\bar{c}_{tot}} - \frac{\lambda \tilde{d}_{tot} 
  \langle \tilde{\mathbf{h}}, \tilde{\mathbf{U}}\rangle }{\bar{c}_{tot} (\bar{c}_{tot} + \lambda \tilde{d}_{tot})}    \right ), 
  \nabla_x \tilde{\mathbf{h}} \rangle_{ \bar{\mathbf{c}}^{-1} } \,\d x\\
& - \lambda^2 \int_{\mathbb{T}^3} \langle \tilde{\mathbf{d}} \left ( 
  \frac{\langle \tilde{\mathbf{d}}, \tilde{\mathbf{W}} \rangle }{\bar{c}_{tot}} - \frac{\tilde{d}_{tot} 
  \langle \bar{\mathbf{c}} + \lambda \tilde{\mathbf{d}}, \tilde{\mathbf{W}} \rangle }{\bar{c}_{tot} (\bar{c}_{tot} 
  + \lambda \tilde{d}_{tot})}    \right ), \nabla_x\tilde{\mathbf{h}} \rangle_{ \bar{\mathbf{c}}^{-1} } \,\d x\\
& + \lambda^2 \int_{\mathbb{T}^3} \langle \tilde{\mathbf{h}} N_1 (\tilde{\bar{u}}), \nabla_x \tilde{\mathbf{h}} 
\rangle_{ \bar{\mathbf{c}}^{-1} } \,\d x .
\end{align*}
Replacing $\nabla_x \tilde{\mathbf{h}}$ as we have done in the energy estimate, the first integral on the right hand side can be bounded as 
\begin{align*}
    \int_{\mathbb{T}^3} \langle \tilde{\mathbf{W}}, \nabla_x\mathbf{h} \rangle \,\d x = & \int_{\mathbb{T}^3} \langle 
      \tilde{\mathbf{W}}, (1 + \lambda \tilde{\theta})^{\gamma/2-1}  A (\mathbf{d}) \tilde{\mathbf{W}}\rangle \,\d x \\
      &  + \int_{\mathbb{T}^3} \langle \tilde{\mathbf{W}}, (1 + \lambda \tilde{\theta})^{\gamma/2 - 1} [ A (\mathbf{c}) - A (\mathbf{d}) ] 
    \tilde{\mathbf{U}} \rangle \,\d x - \lambda \int_{\mathbb{T}^3} \langle \tilde{\mathbf{W}}, \tilde{\mathbf{h}} \nabla_x \tilde{\theta}\rangle \,\d x\\
    & \leq -\frac{\lambda_A}{2} (\min_{1 \leq i \leq N} \bar{c}_i)^2  \|\tilde{\mathbf{W}}\|_{L_x^2}^2 + \delta_{MS} 
    \mathcal{P} (\delta_{MS})  \|\tilde{\mathbf{h}}\|_{L_x^2}  \|\tilde{\mathbf{W}}\|_{L_x^2},
\end{align*}
where $\mathcal{P} (\delta_{MS}) > 0$ is a polynomial in $\delta_{MS}$ whose coefficients only depend on 
$\bar{\mathbf{c}}$ and $N$. This inequality also uses the expression of the MS matrix $A$ and the estimate
\[  |c_i c_j - d_i d_j| = |c_i c_j - c_i d_j + c_i d_j - d_i d_j| = \lambda |c_i \tilde{h}_j + \tilde{h}_i d_j| \leq \lambda 
C (|\tilde{h}_i| + |\tilde{h}_j|).
\]
After calculations, we finally deduce that 
\begin{align*}
  \frac{1}{2} \frac{\d}{\d t} \|\tilde{\mathbf{h}}\|_{L_x^2 (\bar{\mathbf{c}}^{-1}) }^2 \leq & - \frac{\lambda_A}{2} (\min_{1 \leq i \leq N} \bar{c}_i)^2  
\|\tilde{\mathbf{W}}\|_{L_x^2}^2 + \delta_{MS} \mathcal{P}^\prime (\delta_{MS}) (\|\tilde{\mathbf{h}}\|_{L_x^2} + 
\|\tilde{\mathbf{W}}\|_{L_x^2})^2\\
  \leq & \left (- \frac{\lambda_A}{2} (\min_{1 \leq i \leq N} \bar{c}_i)^2 +2 \delta_{MS} \mathcal{P}^\prime (\delta_{MS}) \right )
\|\tilde{\mathbf{W}}\|_{L_x^2}^2 \\
&+ 2  \delta_{MS} \mathcal{P}^\prime (\delta_{MS}) \|\tilde{\mathbf{h}}\|_{L_x^2(\bar{\mathbf{c}}^{-1})}^2.   
\end{align*}
Here we also replace $\nabla_x \tilde{\mathbf{h}}$ using the equation \eqref{subtra of M-S rela}. $\mathcal{P}^\prime (\delta_{MS}) > 0$ denotes another polynomial in $\delta_{MS}$. Choosing $\delta_{MS}$ small enough so that 
$- \frac{\lambda_A}{2} (\min_{1 \leq i \leq N} \bar{c}_i)^2 +2 \delta_{MS} \mathcal{P}^\prime (\delta_{MS}) \leq 0$, we deduce that 
\[  \frac{1}{2} \frac{\d}{\d t} \|\tilde{\mathbf{h}}\|_{L_x^2 (\bar{\mathbf{c}}^{-1}) }^2  \leq 2 (\max_{1\leq i\leq N}\bar{c}_i) \delta_{MS} \mathcal{P}^\prime (\delta_{MS}) \|\tilde{\mathbf{h}}\|_{L_x^2(\bar{\mathbf{c}}^{-1})}.
\]
Grönwall's inequality implies that $\tilde{\mathbf{h}}(t,x)=0$ for almost any $(t,x)\in \mathbb{R}_+ \times \mathbb{T}^3$, due to 
the initial data $\tilde{\mathbf{h}}(0,x)=0$ for any $x\in \mathbb{T}^3$. It means $\tilde{\mathbf{c}}=\tilde{\mathbf{d}}$, then the 
equation \eqref{subtra of M-S rela} becomes
\[ (1 + \lambda \tilde{\theta})^{\gamma/2} A (\mathbf{d}) \tilde{\mathbf{W}} = 0 ,
\]
inferring that $\tilde{\mathbf{W}} = 0$, as $1+\lambda \tilde{\theta}(t,x)>0$ for almost any $(t,x)\in \mathbb{R}_+ \times \mathbb{T}^3$ and $\tilde{\mathbf{W}} \in \mathrm{Span}  (\mathbf{1})^\perp$. We conclude that the solution $(\tilde{\mathbf{c}}, \tilde{\mathbf{U}}, \tilde{T})$ is equal to 
$(\tilde{\mathbf{d}}, \tilde{\mathbf{V}}, \tilde{\theta})$ and the uniqueness is established.

%%%%%%%%%%%%%%%%%%%%%%%%%%%%%%%%
\section{The Diffusion Asymptotics Process}\label{Sec:4}
We again emphasize that we only consider the Boltzmann equations with collision kernels satisfying assumptions (H1)-(H2)-(H3)-(H4) 
with $\gamma \in [0,1]$ (i.e. the hard potential and the Maxwellian potential cases) in this Section, 
since we need the explicit spectral gap for the operator 
$\mathbf{L}$ \cite{briant2016ARMAglobal,mouhot2016SIAMJMAhypercoercivity}. 

In the first part, we will establish the hypocoercive properties satisfied by the Boltzmann operators $\mathbf{L}^\varepsilon$ and 
$\mathbf{\Gamma}$, some estimates for the source term $\mathbf{S}^\varepsilon$ following methods in \cite{briant2021stability}, and a crucial Poincaré inequality for $\mathbf{f}$, as preparations. The proof for the perturbative Cauchy theory of the Boltzmann multi-species equations states in the second part.
\subsection{Preparations for A Priori Estimate of the Perturbation $\mathbf{f}$}
We consider the local Maxwellian $\mathbf{M}^\varepsilon$ defined in \eqref{def of M-i}, whose macroscopic quantities are the unique solution in Theorem \ref{theorem for M-S in main result} in the form \eqref{Perturbative form of the perturbations for M-S}. These macroscopic quantities have the following properties and estimates 
\begin{gather*}
\min_{1 \leq i \leq N} \bar{c}_i + \varepsilon \tilde{c}_i (t,x) >0, \quad 1 + \varepsilon \tilde{T} (t,x) >0,\quad 
\text{a.e. } (t,x) \in \mathbb{R}_+ \times \mathbb{T}^3,\\
\|\tilde{\mathbf{c}}\|_{L_t^\infty H_x^s( \bar{\mathbf{c}}^{-1} )} \leq \delta_{MS},\\
\|\tilde{T}\|_{L_t^\infty H_x^s} \leq \delta_{MS}.
\end{gather*}
Moreover, we can easily find a positive constant $C_{MS}$ from the decomposition \eqref{velocity decomposition}, such that 
\[ \|\tilde{\mathbf{u}}\|_{L_t^\infty H_x^{s-1}} \leq C_{MS} \delta_{MS}.
\]
The constant $C_{MS}$ is independent of $\varepsilon$ and $\delta_{MS}$. Computation details for its $\mathrm{Span} (\mathbf{1})^\perp$ part lie in \eqref{pert U bound in L-t-infty} 

We first present some estimates for the local Maxwellian $\mathbf{M}^\varepsilon$. 
Following the notations in \cite{briant2021stability}, we define the local Maxwellian 
$\bm{\mathcal{M}}^\varepsilon = \mathbf{M}(\mathbf{1}, \varepsilon \mathbf{u}, T \mathbf{1})$ and the global Maxwellian 
$\bm{\mathcal{M}} = \mathbf{M}(\mathbf{1}, 0, \mathbf{1})$. Then we have the expressions 
$\mathbf{M}^\varepsilon = (\bar{\mathbf{c}} + \varepsilon \tilde{\mathbf{c}}) \bm{\mathcal{M}}$ and 
$\bm{\mu} = \bar{\mathbf{c}} \bm{\mathcal{M}}$.
\begin{lemma}\label{lemma for bound of M-i}
 Consider the local Maxwellian $\mathbf{M}^\varepsilon$ defined in \eqref{def of M-i}, where 
$(\mathbf{c}, \mathbf{u},T )$ is the unique perturbation solution 
 stated in Theorem \ref{theorem for M-S in main result}. Then for any 
 $\varepsilon \in (0,1]$ and $1 \leq i \leq N$, when parameter 
 \begin{equation}\label{delta comm-low}
 \delta \in (0, 1 - \delta_{MS}),
 \end{equation}
 there exist explicit constants $C_\delta^{low}$ and $\mathcal{R}_i^{low}(\delta)$, independent of $\varepsilon$, such that for any 
  $(t,x,v) \in \mathbb{R}_+ \times \mathbb{T}^3 \times \mathbb{R}^3$,
 \begin{equation}\label{lower bound for M-i}
  M_i^\varepsilon (t,x,v) \geq C_\delta^{low} \mathcal{R}_{i}^{low}(\delta) c_i \mathcal{M}_i^{1/\delta}.
 \end{equation}
 
For any $\varepsilon \in (0,1]$ and $1 \leq i \leq N$, when parameter 
\begin{equation}
        \delta \in (0, \frac{1}{1 + \delta_{MS}} ),  \label{delta comm-up}
        \end{equation}
there exist explicit constants $C_\delta^{up}$ and $\mathcal{R}_i^{up}(\delta)$, such that for any 
  $(t,x,v) \in \mathbb{R}_+ \times \mathbb{T}^3 \times \mathbb{R}^3$,
 \begin{equation}\label{upper bound for M-i}
        M_i^\varepsilon (t,x,v)\leq C_\delta^{up} \mathcal{R}_i^{up} (\delta) c_i \mathcal{M}_i^\delta .
        \end{equation}
        
  Moreover, for any $\varepsilon \in (0,1]$ and $1 \leq i \leq N$, when $\delta$ ranges in \eqref{delta comm-up}, there 
  exists a constant $\mathcal{R}_i^{sub} (\delta)$, such that we can bound the following subtraction, for any 
  $(t,x,v) \in \mathbb{R}_+ \times \mathbb{T}^3 \times \mathbb{R}^3$,
   \begin{equation}\label{upper bound for sub of M-i-eps and M-i}
       |\mathcal{M}_i^\varepsilon - \mathcal{M}_i| (t,x,v)\leq C_\delta^{up} (\varepsilon |u_i| + \varepsilon |\tilde{T}|) \mathcal{R}_i^{sub} 
       (\delta) \mathcal{M}_i^\delta.
       \end{equation} 
\end{lemma}
\begin{proof}
        For any $\delta > 0$ and for any $1\leq i \leq N$, we first establish the lower bound for $M_i^\varepsilon = c_i \mathcal{M}_i^\varepsilon$, 
        writing $\mathcal{M}_i^\varepsilon$ in the form
        \[\mathcal{M}_i^\varepsilon (t,x,v) = \mathcal{M}_i^{1/\delta} (v) \mathcal{M}_i^\varepsilon (t,x,v) \mathcal{M}_i^{-1/\delta} (v), 
        \quad t \geq 0, x \in \mathbb{T}^3, v \in \mathbb{R}^3  .
          \]
        The product is expressed as
        \[  \mathcal{M}_i^\varepsilon \mathcal{M}_i^{-1/\delta} =\left ( \frac{m_i}{2\pi} \right )^{\frac{3 (\delta -1)}{2 \delta}}
        \left ( \frac{1}{T} \right )^{\frac{3}{2}} \exp\left\{ \frac{m_i |v|^2}{2} ( \frac{1}{\delta} - \frac{1}{T}) 
        + \frac{\varepsilon m_i v \cdot u_i}{T} - \frac{\varepsilon^2 m_i |u_i|^2}{2T}   \right\},
        \]
        where $T = 1 + \varepsilon \tilde{T} > 0$, since $\tilde{T}$ is small enough. To find a uniform lower bound for this product 
        with respect to $v \in \mathbb{R}^3$, we choose 
        \begin{equation}
        \delta\in (0,1-\delta_{MS}),
        \end{equation} 
        which ensures that $\frac{1}{\delta} - \frac{1}{1 + \varepsilon \tilde{T} }>0$ for any $\varepsilon \in (0,1]$. 
        Then we can obtain
        \[  \mathcal{M}_i^\varepsilon \mathcal{M}_i^{-\frac{1}{\delta}} \geq \left ( \frac{m_i}{2\pi} \right )^{\frac{3 (\delta -1)}{2\delta}} 
        \left (\frac{1}{1 + \varepsilon |\tilde{T}|} \right )^{\frac{3}{2}} \exp\left\{ \frac{m_i |v|^2}{2} 
        (\frac{1}{\delta} - \frac{1}{1 - \varepsilon |\tilde{T}|}) - \frac{\varepsilon m_i |u_i| |v|}{1 - \varepsilon |\tilde{T}|} 
        - \frac{\varepsilon^2 m_i |u_i|^2}{2 (1 - \varepsilon |\tilde{T}|)} \right\},
        \]
        where the index of the exponent on the right hand side is a quadratic function with respect to $|v|$. Its minimum value occurs at 
        $|v| = \frac{\delta \varepsilon |u_i|}{(1 - \delta)- \varepsilon |\tilde{T}|}$, that is, for any $v\in \mathbb{R}^3$
        \[  \mathcal{M}_i^\varepsilon \mathcal{M}_i^{-1/\delta} (v) \geq \left ( \frac{m_i}{2 \pi} 
        \right )^{\frac{3 (\delta -1)}{2\delta}} \left (\frac{1}{1 + \varepsilon |\tilde{T}|} \right )^{\frac{3}{2}} 
        \exp\left\{ - \frac{\varepsilon^2 m_i |u_i|^2}{2 (1 - \delta)- 2 \varepsilon |\tilde{T}|} \right\}.
        \]
        The lower bound for $M_i^\varepsilon$ is, for any $1 \leq i \leq N$
        \begin{equation*}
        M_i^\varepsilon (v) \geq C_\delta^{low} \mathcal{R}_{i}^{low} (\delta) c_i \mathcal{M}_i^{1/\delta}, 
        \quad \forall v \in \mathbb{R}^3,
        \end{equation*}
        with notations
        \[ C_\delta^{low} = \min_{1 \leq i \leq N} \left ( \frac{m_i}{2 \pi} \right )^{\frac{3 (\delta -1  )}{2\delta}}, 
        \quad \mathcal{R}_i^{low} (\delta) = \left ( \frac{1}{1 + |\tilde{T}|} \right )^{\frac{3}{2}} \exp\left\{  
          - \frac{ m_i |u_i|^2}{2 (1 - |\tilde{T}| -\delta) } \right\}.
        \]
        
        Next, the upper bound for $M_i^\varepsilon$ ($1 \leq i \leq N$) uses similar calculations. For any $\delta>0$, the local Maxwellian $\mathcal{M}_i^\varepsilon$ can be expressed as
        \[  \mathcal{M}_i^\varepsilon (t,x,v) = \mathcal{M}_i^\delta (v) \mathcal{M}_i^\varepsilon (t,x,v) \mathcal{M}_i^{-\delta} (v),
        \]
        where the product has the bound 
        \begin{align*}
          \mathcal{M}_i^\varepsilon \mathcal{M}_i^{-\delta} & =\left ( \frac{m_i}{2\pi} \right )^{\frac{3(1-\delta)}{2}}
          (\frac{1}{1 + \varepsilon \tilde{T}})^{\frac{3}{2}} \exp\left\{ - \frac{m_i}{2}\left ( \frac{1}{1 + \varepsilon \tilde{T}} 
          - \delta \right ) |v|^2 + \frac{\varepsilon m_i u_i \cdot v}{1 + \varepsilon \tilde{T}} - \frac{\varepsilon^2 m_i |u_i|^2}
          {2 (1 + \varepsilon \tilde{T})} \right\}\\
        & \leq \left ( \frac{m_i}{2 \pi} \right )^{\frac{3(1-\delta)}{2}} (\frac{1}{1 + \varepsilon \tilde{T}})^{\frac{3}{2}} 
        \exp\left\{ - \frac{m_i}{2} \left ( \frac{1}{1 + \varepsilon\tilde{T}} - \delta \right ) |v|^2
         + \frac{\varepsilon m_i |u_i| |v|}{1 + \varepsilon \tilde{T}} - \frac{\varepsilon^2 m_i |u_i|^2}{2 (1 + \varepsilon \tilde{T})} 
         \right\}.
        \end{align*}
        If the parameter is in
        \begin{equation}
        \delta \in \left (0, \frac{1}{1 + \delta_{MS}} \right ),
        \end{equation}
        then the inequality $ \frac{1}{1 + \varepsilon \tilde{T}} - \delta >0$ holds for any $\varepsilon \in (0,1]$. 
        Since $1+ \varepsilon \tilde{T}>0$, the exponent part arrives its maximum value at $|v| = \frac{\varepsilon |u_i|}{1 - (1 + \varepsilon \tilde{T}) \delta}$, leading to
        \begin{align*}
        \mathcal{M}_i^\varepsilon \mathcal{M}_i^{ - \delta} & \leq \left ( \frac{m_i}{2 \pi} \right )^{\frac{3 (1 - \delta)}{2}} 
        (\frac{1}{1 + \varepsilon \tilde{T}})^{\frac{3}{2}} \exp\left\{ \frac{\delta \varepsilon^2 m_i |u_i|^2}
        {2 [1 - (1 + \varepsilon \tilde{T}) \delta] }  \right\}\\
        & \leq \left ( \frac{m_i}{2 \pi} \right )^{\frac{3 (1 - \delta)}{2}} (\frac{1}{1 - \varepsilon |\tilde{T}|})^{\frac{3}{2}} 
        \exp \left\{ \frac{\delta \varepsilon^2  m_i |u_i|^2}{2 [1 - (1 + \varepsilon |\tilde{T}|) \delta] } \right\}.\\
        \end{align*} 
        The upper bound for $M_i^\varepsilon$ is recovered as, for any $1 \leq i \leq N$,
        \begin{equation}
        M_i^\varepsilon (v) \leq C_\delta^{up} \mathcal{R}_i^{up} (\delta) c_i \mathcal{M}_i^\delta \quad \forall v \in \mathbb{R}^3,
        \end{equation}
        with notations
        \[ C_\delta^{up} = \max_{1 \leq i \leq N} \left ( \frac{m_i}{2\pi} \right )^{\frac{3 (1 - \delta)}{2}}, \quad 
        \mathcal{R}_i^{up} (\delta) = (\frac{1}{1 - |\tilde{T}| })^{\frac{3}{2}}\exp \left\{ 
          \frac{\delta m_i |u_i|^2}{2 (1- (1 + |\tilde{T}|) \delta)} \right\} .
          \]
          
        Finally, we establish the upper bound for $|\mathcal{M}_i^\varepsilon - \mathcal{M}_i|$. The subtraction can express as
        \begin{align*}
        \mathcal{M}_i^\varepsilon - \mathcal{M}_i & = \int_0^1 \frac{d}{ds} \mathcal{A}_i^{s, \varepsilon} \,ds = \int_0^1 
        \mathcal{A}_i^{s, \varepsilon} \mathcal{B}_i^{s, \varepsilon} \,ds = \int_0^1 \mathcal{A}_i^{s, \varepsilon} 
        \mathcal{B}_i^{s,\varepsilon} \mathcal{M}_i^{-\delta}\,ds \mathcal{M}_i^\delta,
        \end{align*}
        with the notations
        \[  \mathcal{A}_i^{s, \varepsilon} := \left ( \frac{m_i}{2 \pi} \right )^{\frac{3}{2}} \left ( 
          \frac{1}{1 + s \varepsilon \tilde{T}} \right )^{\frac{3}{2}} \exp \left\{ - \frac{m_i |v - s \varepsilon u_i|^2}
          {2 (1 + s \varepsilon \tilde{T})} \right\},
          \]
        \[  \mathcal{B}_i^{s, \varepsilon} := \frac{m_i (v - s \varepsilon u_i) \cdot \varepsilon u_i}{1 + s \varepsilon \tilde{T}} 
        + \frac{m_i |v - s \varepsilon u_i|^2 \varepsilon \tilde{T}}{2 (1 + s\varepsilon \tilde{T})^2} 
        - \frac{3 \varepsilon \tilde{T}}{2 (1 + s \varepsilon \tilde{T})}.
        \]
        Applying the Cauchy-Schwarz inequality, we deduce the following bounds
        \begin{align*}
        \mathcal{A}_i^{s, \varepsilon} \mathcal{M}_i^{ - \delta}  = &\left ( \frac{m_i}{2 \pi} \right )^{ \frac{3 (1 - \delta)}{2}} 
        \left ( \frac{1}{1 + s \varepsilon \tilde{T}} \right )^{\frac{3}{2}}\\
        &\times \exp \left\{ - \frac{m_i |v|^2}{2} \left ( 
          \frac{1}{1 + s \varepsilon \tilde{T}} - \delta \right ) + \frac{ s \varepsilon m_i  u_i \cdot v}{1 + s \varepsilon \tilde{T}} 
          - \frac{s^2 \varepsilon^2 m_i |u_i|^2}{2 (1 + s \varepsilon \tilde{T}) } \right\} \\
        \leq & \left ( \frac{m_i}{2 \pi} \right )^{\frac{3 (1 - \delta)}{2}} \left ( \frac{1}{1 + s \varepsilon \tilde{T}} 
        \right )^{\frac{3}{2}} \exp \left\{ - \frac{m_i |v|^2}{2} \left ( \frac{1}{1 + s \varepsilon \tilde{T}} - \delta \right ) 
        + \frac{\varepsilon m_i |u_i| |v|}{1 + s \varepsilon \tilde{T}}    \right\},
        \end{align*}
       and 
       \begin{align*}
       |\mathcal{B}_i^{s, \varepsilon}| \leq \varepsilon |u_i| \frac{m_i (|v| + \varepsilon |u_i|)}{1 + s \varepsilon \tilde{T}} 
       + \varepsilon |\tilde{T}| \frac{m_i (|v| + \varepsilon |u_i|)^2}{2 (1 + s \varepsilon \tilde{T})^2} + \varepsilon |\tilde{T}| 
       \frac{3}{2 (1 + s \varepsilon \tilde{T})}.
       \end{align*}
       Here the term $1+ s \varepsilon \tilde{T}$ is positive for any $s\in[0,1]$ and for any $\varepsilon \in (0,1]$, as we can assume $\tilde{T}$ is small enough. 
       When the parameter $\delta$ satisfies \eqref{delta comm-up}, which guarantees
       $\delta (1 +s \varepsilon \tilde{T}) <1$, 
       we can distinguish the velocity space into two parts. For $|v|\geq 
       \frac{4 \varepsilon |u_i|}{1 - \delta (1 + s \varepsilon \tilde{T})}$, we obtain the upper bounds 
       \begin{align*}
       \mathcal{A}_i^{s, \varepsilon} \mathcal{M}_i^{-\delta} & \leq \left ( \frac{m_i}{2 \pi} \right )^{ \frac{3 (1 - \delta)}{2}} 
       (\frac{1}{1 + s \varepsilon \tilde{T}})^{\frac{3}{2}} \exp \left\{ - \frac{m_i |v|^2}{4} \left ( 
        \frac{1}{1 + s \varepsilon \tilde{T}} - \delta \right ) \right\}\\
        & \leq \left ( \frac{m_i}{2 \pi} \right )^{ \frac{3 (1 - \delta)}{2}}  (\frac{1}{1 - |\tilde{T}|})^{\frac{3}{2}} \exp 
        \left\{ - \frac{m_i |v|^2}{4} \left ( \frac{1}{1 + |\tilde{T}| } - \delta \right )\right\},
       \end{align*}
       and 
       \begin{align*}
       |\mathcal{B}_i^{s, \varepsilon}| & \leq \varepsilon |u_i| \left ( \frac{5}{1 + s \varepsilon \tilde{T}} - \delta \right ) 
       \frac{m_i |v|}{4} + \varepsilon |\tilde{T}|  \left (  \left ( \frac{5}{1 + s \varepsilon \tilde{T}} - \delta \right )^2 
       \frac{m_i |v|^2}{32} + \frac{3}{2 (1 + s \varepsilon \tilde{T})}  \right )\\
       & \leq \varepsilon |u_i| \left ( \frac{5}{1 - |\tilde{T}|} - \delta \right ) \frac{m_i |v|}{4} + \varepsilon |\tilde{T}|
       \left ( \left ( \frac{5}{1 - |\tilde{T}|} - \delta \right )^2 \frac{m_i |v|^2}{32} + \frac{3}{2 (1 - |\tilde{T}|)} \right ).
       \end{align*}
       For $|v| \leq \frac{4 \varepsilon |u_i|}{1 - \delta (1 + s \varepsilon \tilde{T})}$, the exponent part of 
       $\mathcal{A}_i^{s, \varepsilon} \mathcal{M}_i^{ - \delta}$ arrives its maximum value at 
       $|v| = \frac{ \varepsilon |u_i|}{1 - \delta (1 + s \varepsilon \tilde{T})}$, and we get
       \begin{align*}
        \mathcal{A}_i^{s, \varepsilon} \mathcal{M}_i^{ - \delta} & \leq \left ( \frac{m_i}{2 \pi} \right )^{ \frac{3 (1 - \delta)}{2}} 
        (\frac{1}{1 - |\tilde{T}|})^{\frac{3}{2}} \exp \left\{ \frac{m_i |u_i|^2}{2 (1 - |\tilde{T}|) 
        (1 - \delta (1 + |\tilde{T}|) ) }\right\},
       \end{align*}
       \begin{align*}
       |\mathcal{B}_i^{s, \varepsilon}| \leq &\varepsilon |u_i| \frac{m_i |u_i|}{1 - |\tilde{T}|} \left ( 1 + \frac{4}{1 - \delta 
       (1 + |\tilde{T}|) }   \right ) \\
       &+ \varepsilon |\tilde{T}| \left ( \frac{m_i |u_i|^2}{2 (1 - |\tilde{T}|^2)} 
       \left ( 1 + \frac{4}{1 - \delta (1 + |\tilde{T}|)}    \right )^2 + \frac{3}{2 (1 - |\tilde{T}|)} \right ).
       \end{align*}
       These estimates hold for any $s\in[0,1]$, thus the upper bound reads
       \begin{equation}
       |\mathcal{M}_i^\varepsilon - \mathcal{M}_i| (v) \leq  C_\delta^{up} (\varepsilon |u_i| + \varepsilon |\tilde{T}|) 
       \mathcal{R}_i^{sub} (\delta) \mathcal{M}_i^\delta, \quad \forall v \in\mathbb{R}^3.
       \end{equation} 
       The notation $\mathcal{R}_i^{sub} (\delta) = \max\left\{ \mathcal{R}_{i,1}^{sub} (\delta), \mathcal{R}_{i,2}^{sub} (\delta) 
       \right\}$, with 
       \begin{align*}
       \mathcal{R}_{i,1}^{sub} (\delta) = & (\frac{1}{1 - |\tilde{T}|})^{\frac{3}{2}} \sup_{v \in \mathbb{R}^3} 
       \left ( \frac{5}{1 - |\tilde{T}|} - \delta \right ) \frac{m_i |v|}{4} \exp \left\{ - \frac{m_i |v|^2}{4} 
       \left ( \frac{1}{1 + |\tilde{T}|} - \delta \right )    \right\}\\
       & + (\frac{1}{1 - |\tilde{T}|})^{\frac{3}{2}} \frac{m_i |u_i|}{1 - |\tilde{T}|} \left (1 + \frac{4}{1 - \delta (1 + |\tilde{T}|)} \right )
        \exp \left\{ \frac{m_i |u_i|^2}{2 (1 - |\tilde{T}|)(1 - \delta (1 + |\tilde{T}|))} \right\},
       \end{align*}
       \begin{align*}
        \mathcal{R}_{i,2}^{sub} (\delta) & = (\frac{1}{1 - |\tilde{T}|})^{\frac{3}{2}} \sup_{v \in \mathbb{R}^3} 
        \left ( \left ( \frac{5}{1 - |\tilde{T}|} - \delta \right )^2 \frac{m_i |v|^2}{32} + \frac{3}{2 (1 - |\tilde{T}|)} \right )\\
        & \times \exp \left\{ - \frac{m_i |v|^2}{4} \left ( \frac{1}{1 + |\tilde{T}|} - \delta \right ) \right\} 
        + (\frac{1}{1 - |\tilde{T}|})^{\frac{3}{2}} \left[ \frac{m_i |u_i|^2}{2 (1 - |\tilde{T}|^2)} 
        \left (1 + \frac{4}{1 - \delta (1 + |\tilde{T}|)} \right )^2 \right.\\
        & \left. + \frac{3}{2 (1 - |\tilde{T}|)} \right]   \times \exp \left\{ \frac{m_i |u_i|^2}{2 (1 - |\tilde{T}|) 
        (1- \delta (1 + |\tilde{T}|)) }\right\}.
       \end{align*}
       
       We end this proof by explaining that $\mathcal{R}_{i}^{low} (\delta), \mathcal{R}_i^{up} (\delta)$ and
       $\mathcal{R}_i^{sub} (\delta)$ can be viewed as positive constants that only depend on $\delta$ 
       (note there are different commands for $\delta$). This is because the macroscopic quantities have upper bounds of 
       $\mathcal{O} (\delta_{MS})$, and the parameter $\delta_{MS}$ has an upper bound in the smallness assumption \eqref{assum for delta-MS}.
\end{proof}
%%%%%%%%%%%%%%%%%%%%%%%%%%%%%%%%%%%%%%%%%%%%%%%%%%
Based on these estimates for the local Maxwellians $(M_i^\varepsilon)_{1 \leq i \leq N}$, we can establish some controls for the operators and the source term arose in the perturbed equation \eqref{pert Bz}, following the analysis in \cite{briant2021stability}. We first focus on the 
linear operator $\mathbf{L}^\varepsilon$ defined in \eqref{def of L-eps}, and split it into two parts, as one usually does for the linearized Boltzmann operator in mono-species case,
\[ \mathbf{L}^\varepsilon = \mathbf{K}^\varepsilon - \bm{\nu}^\varepsilon.
 \]
The operator $\mathbf{K}^\varepsilon = (\mathit{K}_1^\varepsilon, \dots, \mathit{K}_N^\varepsilon)$ is defined by, for any $1 \leq i \leq N$
\begin{equation}
    \begin{aligned}
   \mathit{K}_i^\varepsilon (\mathbf{g}) = \mu_i^{-1/2} \sum_{j=1}^N \int_{\mathbb{R}^3 \times \mathbb{S}^2} \mathit{B}_{ij} 
   \left ( M_i^{\varepsilon, \prime} (\mu_j^{\prime\ast})^{1/2} g_j^{\prime\ast} + M_j^{\varepsilon, \prime\ast} (\mu_i^\prime)^{1/2} 
   g_i^\prime - M_i^\varepsilon (\mu_j^\ast)^{1/2} g_j^\ast \right ) \,\d \sigma \d v_\ast.  \label{def of operator K} 
    \end{aligned}
\end{equation}
And $\bm{\nu}^\varepsilon = (\nu_1^\varepsilon, \dots, \nu_N^\varepsilon)$ acts like a multiplicative operator, i.e. for 
any $1 \leq i \leq N$,
\begin{equation}
\nu_i^\varepsilon (\mathbf{g}) (v) = \sum_{j=1}^N \nu_{ij}^\varepsilon (v) g_i (v),   \label{express of nu}
\end{equation}  
where 
\begin{equation}
\nu_{ij}^\varepsilon (v) = \int_{\mathbb{R}^3 \times \mathbb{S}^2} \mathit{B}_{ij} (|v - v_\ast|, \cos{\theta}) M_j^\varepsilon(v_\ast) 
\,\d \sigma \d v_\ast   \label{explicit express of nu-ij}
\end{equation}
Using estimates on $M_i^\varepsilon$ in the above Lemma, we can establish some controls for $\bm{\nu}^\varepsilon$.
\begin{lemma} \label{lemma for bound of nu and its v-derivative}
For any $\varepsilon \in (0, 1]$ and any $\mathbf{f}, \mathbf{g} \in L^2(\mathbb{R}^3)$, there exist some explicit 
constants $C^{\bm{\nu}, \delta}_{low}$ and $C^{\bm{\nu}, \delta}_{up}$, such that, 
\begin{equation}
 C^{\bm{\nu}, \delta}_{low} \|\mathbf{f}\|_{L_v^2 (\langle v \rangle^\gamma) }  \|\mathbf{g}\|_{L_v^2(\langle v \rangle^\gamma) } 
 \leq \langle \bm{\nu}^\varepsilon (\mathbf{f}), \mathbf{g} \rangle_{L_v^2} \leq C^{\bm{\nu}, \delta}_{up}  
 \|\mathbf{f}\|_{L_v^2 (\langle v \rangle^\gamma) }  \|\mathbf{g}\|_{L_v^2 (\langle v \rangle^\gamma) },  \label{up and lower bound for nu-eps}
\end{equation}
 for any $\delta$ in $C^{\bm{\nu}, \delta}_{low}$ satisfying \eqref{delta comm-low}, and for any $\delta$ in $C^{\bm{\nu}, \delta}_{up}$ 
 satisfying \eqref{delta comm-up}.
 
Moreover, for any $s \in \mathbb{N}^\ast$, with the multi-indices $\alpha, \beta$ satisfying $|\alpha| + |\beta| = s$ and $|\beta| \geq 1$, 
there exist explicit constants $(C_k^{\bm{\nu}})_{1 \leq k \leq 5}$, such that for any $\mathbf{f} \in H^s(\mathbb{T}^3 \times \mathbb{R}^3)$,
\begin{equation}
\begin{aligned}
\langle \partial_v^\beta \partial_x^\alpha \bm{\nu}^\varepsilon (\mathbf{f}), \partial_v^\beta \partial_x^\alpha \mathbf{f}
\rangle_{L_{x,v}^2} \geq (C_1^{\bm{\nu}} - \varepsilon \textbf{1}_{|\alpha| \geq 1} C_2^{\bm{\nu}})  
\|\partial_v^\beta \partial_x^\alpha \mathbf{f}\|_{L_{x,v}^2 (\langle v \rangle^\gamma) }^2\\
 - (C_3^{\bm{\nu}} + \varepsilon \textbf{1}_{|\alpha| \geq 1} C_4^{\bm{\nu}})  \|\mathbf{f}\|_{H_{x,v}^{s-1}}^2\\
 -\varepsilon \textbf{1}_{|\alpha| \geq 1} C_5^{\bm{\nu}} \|\mathbf{f}\|_{H_{x,v}^{s-1} (\langle v \rangle^\gamma) }^2
 \label{lower bound for nu with v derivative}
\end{aligned}
\end{equation}
\end{lemma}
\begin{proof}
  % We first try to find the lower and upper bounds for $\bm{\nu}^\varepsilon$.
  We first prove that there exist upper and lower bounds for $\langle \bm{\nu}^\varepsilon(\mathbf{f}),\mathbf{g}\rangle_{L_v^2}$. 
  The expressions \eqref{express of nu} and \eqref{explicit express of nu-ij} infer that 
\[ \langle \bm{\nu}^\varepsilon (\mathbf{f}), \mathbf{g} \rangle_{L_v^2} = \sum_{i}^N \nu_i^\varepsilon (v) f_i (v) g_i (v) \,\d v.
\]
The estimates for $M_j^\varepsilon (v_\ast)$, specifically inequalities 
\eqref{lower bound for M-i} and \eqref{upper bound for M-i}, imply that there exist two positive constants 
$\bar{\nu}_{ij}^{\delta, low}, \bar{\nu}_{ij}^{\delta,up}$, depending only on $\delta$, such that, for any 
$v \in \mathbb{R}^3$
\begin{align*}
0< \bar{\nu}_{ij}^{\delta, low} c_j \langle v \rangle^\gamma \leq \nu_{ij}^\varepsilon(v) \leq \bar{\nu}_{ij}^{\delta,up}
c_j \langle v \rangle^\gamma,
\end{align*}
with the notation $\langle v \rangle = (1 + |v|^2)^{1/2}$. Since $\nu_i^\varepsilon (v) = \sum_{j=1}^N \nu_{ij}^\varepsilon (v)$, we obtain 
two positive constants $\bar{\nu}_i^{\delta, low}$ and $\bar{\nu}_i^{\delta,up}$, such that for any $v \in \mathbb{R}^3$,
\begin{equation}
0< \bar{\nu}_i^{\delta, low} \left (\sum_{j=1}^N c_j\right ) \langle v \rangle^\gamma \leq \nu_i^\varepsilon (v) \leq  
\bar{\nu}_i^{\delta,up} \left (\sum_{j=1}^N c_j \right ) \langle v \rangle^\gamma,
\end{equation} 
which in turns deduces that there exist two positive constants $C^{\bm{\nu}, \delta}_{low}$ and $C^{\bm{\nu}, \delta}_{up}$, 
such that the inequality \eqref{up and lower bound for nu-eps} holds. Furthermore, we emphasize that the ranges of the parameter $\delta$ 
are not same in the lower and upper bounds, they are given by the conditions \eqref{delta comm-low} and \eqref{delta comm-up}. 
Details can be found in Lemma \ref{lemma for bound of M-i}.

For the second inequality in this Lemma, we follow the analysis in \cite{briant2021stability}, firstly handling the simple case that 
$|\alpha| = |\beta|= 1$. The expression of $\bm{\nu}^\varepsilon$ 
and the upper bound for $M_j^\varepsilon$ deduce the following bound, 
for any $v \in \mathbb{R}^3$ and $l = {1, 2, 3}$
\begin{align*}
  | \partial_{v_l} \nu_i^\varepsilon(v)| & = \left|\sum_{j=1}^N \int_{\mathbb{R}^3 \times \mathbb{S}^2} 
  \mathit{b}_{ij} (\cos{\theta}) \gamma |v - v_\ast|^{\gamma - 1} \frac{v_l - v_{\ast,l}}{ |v - v_\ast| } M_j^\varepsilon (v_\ast) 
  \,\d \sigma \d v_\ast \right|\\
  & \leq C(\delta) \sum_{j=1}^N c_j \int_{\mathbb{R}^3} \gamma |v - v_\ast|^{\gamma - 1} e^{-\frac{\delta m_j |v_\ast|^2}{2}} \,\d v_\ast
  \leq C(\delta) \sum_{j=1}^N c_j < +\infty,
\end{align*}
which holds for any fixed $\delta$ satisfying \eqref{delta comm-up}. The integral above is uniformly finite with respect to $v$, as $\langle v \rangle^{\gamma-1} \leq 1$ for any $v \in \mathbb{R}^3$. 

Due to the expression $\mathbf{M}^\varepsilon = \mathbf{c} \bm{\mathcal{M}}^\varepsilon$, the $x$-derivative term can be written as, for any $(t,x,v) \in \mathbb{R}_+ \times \mathbb{T}^3 \times \mathbb{R}^3$ and $k = {1, 2, 3}$,
\begin{align*}
\partial_{x_k} \nu_i^\varepsilon (t,x,v) = & \sum_{j=1}^N \int_{\mathbb{R}^3 \times \mathbb{S}^2} \mathit{b}_{ij}  ( \cos{\theta} ) 
|v - v_\ast|^\gamma \partial_{x_k} M_j^\varepsilon (v_\ast) \,\d \sigma \d v_\ast\\
 = &\varepsilon \sum_{j=1}^N  \int_{\mathbb{R}^3 \times \mathbb{S}^2} \mathit{b}_{ij} ( \cos{\theta} ) |v - v_\ast|^\gamma 
 \left[ \partial_{x_k} \tilde{c}_j + \frac{m_j c_j (v_\ast - \varepsilon u_j) \cdot \partial_{x_k} u_j}{T} \right.\\
& + \left. \left ( \frac{m_j |v_\ast - \varepsilon u_j|^2}{2 T} - \frac{3}{2} \right ) \frac{ c_j \partial_{x_k} \tilde{T}}{T} \right] 
\mathcal{M}_j^\varepsilon (v_\ast) \,\d \sigma \d v_\ast\\
 = & \varepsilon \tilde{\nu}_i^{\varepsilon, k}(t,x,v).
\end{align*}
It is not hard to notice that $\tilde{\nu}_i^{\varepsilon,k}$ and $\nu_i^\varepsilon$ share the same structure. The macroscopic quantities and 
their derivatives can be treated as factors with $L^\infty$ bounds, which only require them to 
be sufficiently regular.

Now we consider the $\partial_{v_l} \partial_{x_k}$ derivatives, writing as  
\begin{align*}
\langle \partial_{v_l} \partial_{x_k} \bm{\nu}^\varepsilon (\mathbf{f}), \partial_{v_l} \partial_{x_k} \mathbf{f} \rangle_{L_{x,v}^2}
= & \sum_{i=1}^N \int_{\mathbb{T}^3 \times \mathbb{R}^3} \partial_{v_l} \nu_i^\varepsilon \partial_{x_k} f_i 
\partial_{v_l} \partial_{x_k} f_i \,\d x\d v\\
& + \sum_{i=1}^N \int_{\mathbb{T}^3 \times \mathbb{R}^3} \nu_i^\varepsilon (\partial_{v_l} \partial_{x_k} f_i)^2  \,\d x\d v\\
& + \varepsilon \sum_{i=1}^N \int_{\mathbb{T}^3 \times \mathbb{R}^3} \partial_{v_l} \tilde{\nu}_i^{\varepsilon,k}  f_i 
\partial_{v_l} \partial_{x_k} f_i  \,\d x\d v\\
& + \varepsilon \sum_{i=1}^N \int_{\mathbb{T}^3 \times \mathbb{R}^3}  \tilde{\nu}_i^{\varepsilon,k} \partial_{v_l} f_i 
\partial_{v_l} \partial_{x_k} f_i  \,\d x\d v.
\end{align*}
Using Young's inequality and the above bound for $\partial_{v_l} \nu_i^\varepsilon$, the first term on the right hand side 
can be estimated by 
\begin{align*}
  \sum_{i=1}^N \int_{\mathbb{T}^3 \times \mathbb{R}^3} \partial_{v_l} \nu_i^\varepsilon \partial_{x_k} f_i 
  \partial_{v_l} \partial_{x_k} f_i \,\d x \d v \geq  &  \sum_{i=}^N - \frac{1}{2} \int_{\mathbb{T}^3 \times \mathbb{R}^3} 
  \frac{(\partial_{v_l} \nu_i^\varepsilon)^2}{\nu_i^\varepsilon} (\partial_{x_k} f_i)^2 \,\d x \d v\\
  & - \frac{1}{2} \int_{\mathbb{T}^3 \times \mathbb{R}^3} \nu_i^\varepsilon (\partial_{v_l} \partial_{x_k} f_i)^2 \,\d x \d v,
\end{align*}
Similarly, the third and fourth terms are bounded from below
\begin{align*}
  \varepsilon\sum_{i=1}^N \int_{\mathbb{T}^3\times\mathbb{R}^3} \partial_{v_l} \tilde{\nu}_i^{\varepsilon, k}  f_i 
  \partial_{v_l} \partial_{x_k} f_i  \,\d x \d v  \geq & \sum_{i=1}^N -\frac{\varepsilon}{2} \int_{\mathbb{T}^3 \times \mathbb{R}^3} 
  \frac{(\partial_{v_l} \tilde{\nu}_i^{\varepsilon, k})^2}{\nu_i^\varepsilon} f_i^2   \,\d x \d v\\
  & -\frac{\varepsilon}{2} \int_{\mathbb{T}^3 \times \mathbb{R}^3} \nu_i^\varepsilon( \partial_{v_l} \partial_{x_k}f_i)^2 \,\d x \d v,
\end{align*}
\begin{align*}
  \varepsilon \sum_{i=1}^N \int_{\mathbb{T}^3 \times \mathbb{R}^3}  \tilde{\nu}_i^{\varepsilon, k} \partial_{v_l} f_i 
  \partial_{v_l} \partial_{x_k} f_i \,\d x \d v  \geq & \sum_{i=1}^N -\frac{\varepsilon}{2} \int_{\mathbb{T}^3 \times \mathbb{R}^3} 
  \tilde{\nu}_i^{\varepsilon,k} (\partial_{v_l} f_i)^2  \,\d x \d v\\
  & -\frac{\varepsilon}{2} \int_{\mathbb{T}^3 \times \mathbb{R}^3} \tilde{\nu}_i^{\varepsilon, k} (\partial_{v_l} \partial_{x_k} f_i)^2 \,\d x \d v.
\end{align*}

Since $\tilde{\nu}_i^{\varepsilon,k}$ and $\nu_i^\varepsilon$ share the same structure, by summing all the above bounds, we can establish 
these positive constants as
\begin{gather*}
C_1^{\bm{\nu}} = \frac{\bar{\nu}_i^{\delta, up}}{2} \|c_{tot}\|_{L_t^\infty L_x^\infty} > 0, \quad 
C_3^{\bm{\nu}} = \max_{1\leq i\leq N} \sup_{v \in \mathbb{R}^3} \frac{(\partial_{v_l} \nu_i^\varepsilon)^2}{\nu_i^\varepsilon},\\
C_4^{\bm{\nu}} = \max_{1\leq i\leq N} \sup_{v \in \mathbb{R}^3} \frac{(\partial_{v_l} \tilde{\nu}_i^{\varepsilon,k})^2}{\nu_i^\varepsilon}.
\end{gather*}
Moreover, the constant $C_5^{\bm{\nu}}$ corresponding to the upper bound of $\tilde{\bm{\nu}}^{\varepsilon, k}$ can be established 
using similar calculations in Lemma \ref{lemma for bound of M-i}, while $C_2^{\bm{\nu}}$ corresponding to the upper bound of 
$\tilde{\bm{\nu}}^{\varepsilon, k} + \bm{\nu}^\varepsilon$ can be obtained as well. We emphasize that computations of these constants rely on both the lower and upper bounds of $M_i^\varepsilon$. Consequently, the parameter $\delta$ must be chosen to simultaneously satisfy conditions \eqref{delta comm-low} and \eqref{delta comm-up}. Gathering the four parts, we recover \eqref{lower bound for nu with v derivative} 
in the case $|\alpha| = |\beta| = 1$ (the case $|\alpha| = 0$ is same following the above calculations).

We provide a brief analysis for the general case, splitting the calculations into four parts,
\begin{align*}
&\langle \partial_v^\beta \partial_x^\alpha \bm{\nu}^\varepsilon (\mathbf{f}), \partial_v^\beta \partial_x^\alpha \mathbf{f} 
\rangle_{L_{x,v}^2} \\
& = \sum_{i=1}^N \int_{\mathbb{T}^3 \times \mathbb{R}^3} \partial_v^\beta \partial_x^\alpha (\nu_i^\varepsilon f_i) 
\partial_v^\beta \partial_x^\alpha f_i \,\d x \d v\\
& = \sum_{i=1}^N  \left ( \sum_{ \substack{|\beta^\prime| \neq 0 \\ |\alpha^\prime| = 0} } + \sum_{ \substack{ |\beta^\prime| = 0\\ 
|\alpha^\prime|=0} } + \sum_{ \substack{|\beta^\prime| \neq 0 \\ |\alpha^\prime| \neq 0} } + \sum_{\substack{ |\beta^\prime| = 0 \\ 
|\alpha^\prime|\neq 0} } \right )\int_{\mathbb{T}^3 \times \mathbb{R}^3} \partial_v^{\beta^\prime} \partial_x^{\alpha^\prime} \nu_i^\varepsilon 
\partial_v^{\beta - \beta^\prime} \partial_x^{\alpha - \alpha^\prime} f_i \partial_v^\beta \partial_x^\alpha f_i \,\d x \d v.
\end{align*} 
These four terms correspond exactly to those in the case $|\beta| = |\alpha| = 1$, and we can compute them using a similar method.
\end{proof}
The operator $\mathbf{K}^\varepsilon$ in \eqref{def of operator K} can be written as a kernel operator, by using the Carleman representation \cite{Carleman1957book}, i.e. for any $v \in \mathbb{R}^3$
\begin{equation}
\begin{aligned}
\mathit{K}_i^\varepsilon (\mathbf{f}) (v) = & \sum_{j=1}^N \int_{\mathbb{R}^3} k_{ij}^{(1)} (v, v_\ast) f_j(v_\ast) + k_{ij}^{(2)} 
(v,v_\ast) f_i (v_\ast) - k_{ij}^{(3)} (v, v_\ast) f_j (v_\ast) \,\d v_\ast\\
:= & \mathbf{K}^{(1), \varepsilon} + \mathbf{K}^{(2), \varepsilon} 
+ \mathbf{K}^{(3), \varepsilon},
\end{aligned}
\end{equation}
where the kernels are defined as
\begin{equation}
\begin{aligned}
k_{ij}^{(1)} (v, v_\ast) & = \frac{C_{ij}}{|v - v_\ast|} \int_{\tilde{E}_{v v_\ast}^{ij}} \frac{\mathit{B}_{ij} 
\left ( v - V(w, v_\ast), \frac{v_\ast - w}{ |w - v_\ast| } \right )}{ |w - v_\ast| } M_i^\varepsilon (w) 
\sqrt{\frac{\mu_j (v_\ast)}{\mu_i (v)}} \, \d\tilde{E} (w) , \\
k_{ij}^{(2)}(v, v_\ast) & = \frac{C_{ji}}{|v - v_\ast|} \int_{E_{v v_\ast}^{ij}} \frac{\mathit{B}_{ij} 
\left ( v - V(v_\ast,w), \frac{w - v_\ast}{ |w - v_\ast| } \right )}{|w - v_\ast|} M_j^\varepsilon (w) 
\sqrt{\frac{\mu_i (v_\ast)}{\mu_i (v)}} \, \d E(w) , \\
k_{ij}^{(3)} (v, v_\ast) & = \int_{\mathbb{S}^2} \mathit{B}_{ij} (|v - v_\ast|, \cos{\theta}) M_i^\varepsilon (v) 
\sqrt{\frac{\mu_j (v_\ast)}{\mu_i (v)}} \,\d \sigma.
\end{aligned}
\end{equation}
The notation $V(w, v_\ast) := v_\ast + \frac{m_i}{m_j} w - \frac{m_i}{m_j}v$, and $C_{ij}, C_{ji}$ are 
explicit positive constants that only depending on masses $m_i,m_j$ for any $1 \leq i,j \leq N$ (these constants are not symmetric with respect 
to species indices $i,j$ anymore). Moreover, the symbol $\, \d E$ records the Lebesgue measure on the hyperplane $E_{v v_\ast}^{ij}$, 
which is orthogonal to $v-v_\ast$ and passes through 
\[ V_E(v, v_\ast) = \frac{m_i + m_j}{2 m_j} v - \frac{m_i - m_j}{2 m_j} v_\ast.
\]
The symbol $\, \d \tilde{E}$ records the Lebesgue measure on the spece $\tilde{E}_{v v_\ast}^{ij}$. When $m_i = m_j$, the space 
$\tilde{E}_{v v_\ast}^{ij}$ corresponds to the hyperplane $E_{v v_\ast}^{ij}$, and when $m_i \neq m_j$, 
the space is the sphere centered at $O$ with radius $R$,
\[ O = O(v, v_\ast) := \frac{m_i}{m_i-m_j} v - \frac{m_j}{m_i-m_j} v_\ast, \quad 
R=R(v,v_\ast):=\frac{m_j}{|m_j-m_j|} |v-v_\ast|.
\]
Details on how to derive these expressions through coordinate transformations can 
be found in Section 5 in \cite{briant2016ARMAglobal}. Furthermore, we omit the explicit Carleman representation of 
$\mathbf{K}^\varepsilon$. This again uses coordinate transformations to map the variable $w$ in $k_{ij}^{(1)}$ to $\mathbb{S}^2$ when $m_i \neq m_j$, and map the variable $w$ in $k_{ij}^{(2)}$ to $\mathbb{R}^2$. These details are similar to those in \cite[Appendix A]{briant2021stability}. 

Now we present the controls for the operator $\mathbf{K}^\varepsilon$.
%%%%%%%%%%%%%%%%%%%%%%%%%%%%%%%%%%%%%%%%%%%%%%%%%%%%%%%%%%%%%%%%%%%%%
\begin{lemma}\label{lemma for K operator with v derivative}
Let $\varepsilon \in (0, 1]$ and $s \in \mathbb{N}^\ast$, $\alpha, \beta$ are two multi-indices that satisfy $|\alpha| + |\beta| = s$ and 
$|\beta| \geq 1$. Then for any $\xi \in (0, 1)$, there exist constants $C_1^\mathbf{K}(\xi)$ and $C_2^\mathbf{K}$, 
such that, for any $\mathbf{f} \in H^s(\mathbb{T}^3 \times \mathbb{R}^3)$,
\begin{equation}
\langle \partial_v^\beta \partial_x^\alpha \mathbf{K}^\varepsilon (\mathbf{f}), \partial_v^\beta \partial_x^\alpha \mathbf{f} 
\rangle_{L_{x,v}^2} \leq C_1^\mathbf{K} (\xi)  \|\mathbf{f}\|_{H_{x,v}^{s-1}}^2 + \xi C_2^\mathbf{K} 
\|\partial_v^\beta \partial_x^\alpha \mathbf{f}\|_{L_{x,v}^2}^2 \label{bound for x-alpha v-beta deri on K}
\end{equation}
\end{lemma}
\begin{proof}
We prove this result in the simplest case $|\alpha|=0, |\beta| = 1$, and we only present details for $\mathbf{K}^{(1), \varepsilon}$, 
part of the operator $\mathbf{K}^\varepsilon$, following computations in \cite{briant2021stability}. 
We choose the relative velocity 
$\eta = v - v_\ast$ as a new variable, rewriting $v_\ast = v - \eta$. Then these quantities become
\[ R = R(v, \eta) = \frac{m_j}{|m_i - m_j|} |\eta|, \quad O = O(v,\eta) = v + \frac{m_j}{m_i - m_j} \eta,
\]
and the kernel $k_{ij}^{(1)}$ (for any fixed $i,j\in \{1,\dots,N\}$) becomes
\begin{equation}
k_{ij}^{(1)} (v, \eta) = C_{ij} |\eta|^\gamma \int_{\mathbb{S}^2} \mathit{b}_{ij} (w \cdot \frac{\eta}{|\eta|})
W_{ij}^{(1)} (w \cdot \frac{\eta}{|\eta|}) M_i^\varepsilon(R w + O) \sqrt{\frac{\mu_j(v - \eta)}{\mu_i (v)}} \, \d w,
\label{explicit expression of k-ij-1}
\end{equation} 
with 
\[ M_i^\varepsilon (R w + O) = \left ( \frac{m_i}{2 \pi T} \right )^{\frac{3}{2}} \exp \left\{\frac{m_i |R w + O|^2}{2 T} 
+ \frac{\varepsilon m_i ( R w + O) \cdot u_i}{T} - \frac{\varepsilon^2 m_i |u_i|^2}{2 T} \right\},
 \]
 \begin{align*}
 W_{ij}^{(1)} (w \cdot \frac{\eta}{|\eta|}) := & \left| \frac{m_i^2 + m_j^2}{(m_i - m_j)^2} + \frac{2 m_i m_j}{(m_i - m_j) |m_i - m_j|} 
 \left ( w \cdot \frac{\eta}{|\eta|} \right )  \right|^{\gamma - 1},
 \end{align*}
 and
\[ \mathit{b}_{ij} (w \cdot \frac{\eta}{|\eta|}) := \mathit{b}_{ij} \left ( \frac{\frac{2 m_i m_j}{(m_i - m_j)^2} + 
\frac{m_i^2 + m_j^2}{ (m_i - m_j) |m_i - m_j|} \frac{\eta \cdot w}{|\eta|}}{\frac{m_i^2 + m_j^2}{(m_i - m_j)^2} 
+ \frac{2 m_i m_j}{(m_i - m_j) |m_i - m_j|} \frac{\eta \cdot w}{|\eta|}}  \right ).
\]
Since $\gamma \in [0,1]$, the quantity $W_{ij}^{(1)}$ is finite for any $w \in \mathbb{S}^2$, and for any $\eta \in \mathbb{R}^3$
\[  W_{ij}^{(1)} (w \cdot \frac{\eta}{|\eta|}) = (\frac{m_i w}{m_i - m_j} + \frac{m_j}{|m_i - m_j|} \frac{\eta}{|\eta|})^{2 (\gamma - 1)} \leq 1.
\]

Through this substitution, we notice that the $v$-derivative on $\mathbf{K}^{(1), \varepsilon}(\mathbf{f})$ actually acts on the kernels, 
$( k_{ij}^{(1)} )_{1 \leq i,j \leq N }$ and $(f_j)_{1 \leq j \leq N}$ separately
\begin{equation}
\begin{aligned}
  \langle \nabla_v \mathbf{K}^{(1), \varepsilon} (\mathbf{f}), \nabla_v \mathbf{f} \rangle & _{L_{x,v}^2}\\
  = & \sum_{i , j =1}^N \int_{\mathbb{T}^3 \times \mathbb{R}^3} \left[ \int_{\mathbb{R}^3} \nabla_v k_{ij}^{(1)} (v, \eta) f_j(v - \eta) 
  + k_{ij}^{(1)} (v, \eta) \nabla_v f_j (v - \eta) \,d\eta \right] \nabla_v f_i \,\d x \d v.\label{K-1 v-der express}
\end{aligned}
\end{equation}
To control it, we only need to find a $L^\infty$ bound for $|\nabla_v k_{ij}^{(1)}|$ in the first term. For the second term, 
we split the kernel $k_{ij}^{(1)}$ into two parts as Bondesan and Briant did in \cite{briant2021stability}. 
For a small parameter $\xi\in (0,1)$, we split it into
\[ k_{ij}^{(1)} = k_{ij}^{(1), S} + k_{ij}^{(1), R}, 
\]
where $k_{ij}^{(1), S}$ is a smooth part. It is defined by introducing a mollified indicator function 
$\textbf{1}_{\left\{ |\cdot| \geq \xi \right\}}$, i.e.
\[ k_{ij}^{(1), S} (v, \eta) = \textbf{1}_{\left\{ |\eta| \geq \xi \right\}} k_{ij}^{(1)} (v, \eta).
\]
Applying integration by parts, the second part on the right hand side of \eqref{K-1 v-der express} takes the form 
\begin{equation}
\begin{aligned}
 \int_{\mathbb{R}^3} k_{ij}^{(1)} (v, \eta) \nabla_v f_j (v - \eta) \, \d \eta = & \int_{\mathbb{R}^3} k_{ij}^{(1), R} (v,\eta) 
 \nabla_v f_j (v - \eta) \, \d \eta\\
 & + \int_{\mathbb{R}^3} \nabla_\eta k_{ij}^{(1), S} (v, \eta) f_j (v - \eta) \, \d \eta. \label{k-ij split}
\end{aligned}
\end{equation}

Since the expression $\mu_i = \bar{c}_i \mathcal{M}_i$ holds, by applying the upper bound for $M_i^\varepsilon$ in \eqref{upper bound for M-i} 
and following computations in \cite[Lemma 5.1]{briant2016ARMAglobal} and \cite[Lemmas 4.3-4.6]{Bondesan2020CPAAnonequilibrium}, 
there exists a positive constant $\bar{\delta}_{ij}$, such that, for any $v\in\mathbb{R}^3$
\[ k_{ij}^{(1),R} (v, \eta) \leq C_{ij}(\delta) \|\mathbf{c}\|_{L_t^\infty L_x^\infty} |\eta|^\gamma 
e^{-C(\delta) |\eta|^2},
\]
where the positive constants  $C_{ij} (\delta)$ and $ C (\delta)$ depend only on an arbitrary parameter 
$\delta \in ( \bar{\delta}_{ij}, \frac{1}{1 + \delta_{MS}} )$. 
The constant $\bar{\delta}_{ij}$ is defined as 
\begin{equation*}
    \bar{\delta}_{ij} = \max\{0, \frac{m_i - m_j}{2 m_i}, \frac{m_i}{8 m + m_i}\}, \quad \text{with } 
    8 m := \min\{m_i, m_j, \frac{m_i^2}{m_j} \}.
\end{equation*}
We can divide it into two cases:
\begin{align*}
\text{for } m_i \leq m_j, \quad 8 m = \frac{m_i^2}{m_j}, \quad \bar{\delta}_{ij} = \frac{m_j}{m_i + m_j},\\
\text{for } m_i \geq m_j, \quad 8 m = m_j, \quad \bar{\delta}_{ij} = \frac{m_i}{m_i + m_j}.
\end{align*}
It is not hard to see that $\bar{\delta}_{ij} \leq \frac{\max\{ m_i,m_j \}}{m_i + m_j} < 1$, which has a uniform bound, for any $1\leq i,j \leq N$
\begin{equation}
\bar{\delta}_{ij} \leq \bar{\delta} := 
\frac{ \max_{1 \leq i \leq N} m_i}{ \max_{1 \leq i \leq N} m_i + \min_{1 \leq i \leq N} m_i}.  \label{def of bar-delta}
\end{equation}
Consider the interval $(\bar{\delta}, \frac{1}{1 + \delta_{MS}} )$, it is valid if and only if 
\begin{equation}
\delta_{MS} \leq \bar{\delta}_{MS} \leq \frac{ \min_{1 \leq i \leq N} m_i}{ 2 \max_{1 \leq i \leq N} m_i}. \label{delta-MS commonds in analysis of K}
\end{equation}
Applying the definition of $k_{ij}^{(1), R}$ and the above pointwise estimate, the bound for its $L_\eta^1$ norm is
\begin{equation}
\|k_{ij}^{(1), R} (v, \eta) \|_{L_\eta^1} \leq C_{ij}^{(1),R} (\delta) \xi^{\gamma + 3},
\quad \forall v \in \mathbb{R}^3.
\end{equation}
The method to get the bound for $L_v^1$ norm of $k_{ij}^{(1), R}$ is consistent with the previous discussions, 
where we only need to choose a new variable $\eta = v_\ast - v$ and set $v = v_\ast - \eta$.
The integral for $k_{ij}^{(1),R}$ in \eqref{k-ij split} can be bounded by 
$\xi^{\gamma+3}C(\delta) \|\nabla_v \mathbf{f}\|_{L_{x,v}^2}^2$ after simple computations. 
This explains the necessity of the decomposition for $k_{ij}^{(1)}$ and the small $\xi$, 
as we require this bound to be sufficiently small 
so that it can be absorbed by the negative term $ - C_1^{\bm{\nu}}\|\nabla_v \mathbf{f}\|_{L_{x,v}^2(\langle v \rangle^\gamma) }^2$ stated in 
Lemma \ref{lemma for bound of nu and its v-derivative}. This relies on the fact that 
$\|\mathbf{f}\|_{L_{x,v}^2}^2 \leq \|\mathbf{f}\|_{L_{x,v}^2 (\langle v \rangle^\gamma) }^2$ as $\gamma \in [0, 1]$.

Next, we present the upper bound for $\nabla_v k_{ij}^{(1)}$, which firstly can be written as
\begin{align*}
\nabla_v k_{ij}^{(1)} = C_{ij} |\eta|^\gamma \int_{\mathbb{S}^2} \mathit{b}_{ij} (w \cdot \frac{\eta}{|\eta|}) W_{ij}^{(1)} 
(w \cdot \frac{\eta}{|\eta|})  \left[ \nabla_v M_i^\varepsilon (R w + O) \sqrt{\frac{\mu_j (v - \eta)}{\mu_i(v)}} \right.\\
\left. + M_i^\varepsilon(R w + O) \nabla_v  \sqrt{\frac{\mu_j (v - \eta)}{\mu_i (v)}} \right]\, \d w.
\end{align*}
The upper bound for $M_i^\varepsilon$ in \eqref{upper bound for M-i} implies that for any $\delta \in (0, \frac{1}{1 + \delta_{MS}})$,
\begin{align*}
|\nabla_v M_i^\varepsilon (R w + O)| = & |- \frac{m_i (R w + O - \varepsilon u_i) \cdot \nabla_v (R w + O)}{T} M_i^\varepsilon| = 
|- \frac{m_i (R w + O - \varepsilon u_i)}{T} M_i^\varepsilon|\\
& \leq C_{ij} (\delta) (|\eta| + |v| + |u_i|) c_i \mathcal{M}_i^\delta (R w + O).
\end{align*} 
Moreover, the following inequality holds,
\[ \nabla_v  \sqrt{\frac{\mu_j (v - \eta)}{\mu_i (v)}} = \frac{m_i v - m_j (v - \eta)}{2} \sqrt{\frac{\mu_j (v - \eta)}{\mu_i (v)}}
\leq C (|v| + |\eta|)  \sqrt{\frac{\mu_j (v - \eta)}{\mu_i (v)}}.
\]
The decomposition $v + v_\ast = V^\perp + V^\parallel$ in \cite[Section 5]{briant2016ARMAglobal} is applied, 
where $V^\parallel$ is the projection onto $\mathrm{Span} (v - v_\ast)$ in space $\mathbb{R}^3$ and $V^\perp$ is the orthogonal component. We transform them into our notations, that are $V^\perp \perp \eta$ and $V^\parallel = \langle 2 v - \eta, \frac{\eta}{|\eta|} \rangle \frac{\eta}{|\eta|}$, and obtain 
\[ |v|^2 = \frac{|V^\perp + V^\parallel + \eta|^2}{4} = \frac{|V^\perp|^2}{4} + \frac{\langle v, \eta \rangle^2}{|\eta|^2}.
\]
These two terms can be respectively absorbed by the exponential terms discarded in proof of \cite[Lemma 4.6]{Bondesan2020CPAAnonequilibrium}, controlled by a constant $C(\delta)$. The upper bound for the $v$-derivative is, for any $v \in \mathbb{R}^3$
\begin{equation}
\left| \nabla_v k_{ij}^{(1)} (v, \eta) \right| \leq C_{ij} (\delta) \max_{1 \leq i,j \leq N} \sqrt{ \frac{\bar{c}_j}{\bar{c}_i} }  
\|\mathbf{c}\|_{L_t^\infty L_x^\infty} (1 + |\eta| + \|\mathbf{u}\|_{L_t^\infty L_x^\infty}) |\eta|^\gamma e^{- C (\delta) |\eta|^2},
\end{equation}
with an arbitrary constant $\delta \in (\bar{\delta}, \frac{1}{1 + \delta_{MS}})$. The constant $\bar{\delta}$ is defined in 
\eqref{def of bar-delta} and $\delta_{MS}$ satisfies condition \eqref{delta-MS commonds in analysis of K}. 

The $\eta$-derivative of $k_{ij}^{(1), S}$ can be expanded as 
\begin{align*}
\nabla_\eta k_{ij}^{(1), S} (v, \eta) = & C_{ij} \nabla_\eta \textbf{1}_{ \left\{|\eta| \geq \xi\right\} } |\eta|^\gamma 
\int_{\mathbb{S}^2} \mathit{b}_{ij} W_{ij}^{(1)} M_i^\varepsilon (R w + O) \sqrt{\frac{\mu_j (v - \eta)}{\mu_i (v)}}\, \d w\\
& + C_{ij} \textbf{1}_{ \left\{|\eta| \geq \xi \right\} } \nabla_\eta |\eta|^\gamma \int_{\mathbb{S}^2} 
\mathit{b}_{ij} W_{ij}^{(1)} M_i^\varepsilon (R w + O) \sqrt{\frac{\mu_j (v - \eta)}{\mu_i (v)}} \, \d w\\
& + C_{ij} \textbf{1}_{ \left\{|\eta| \geq \xi \right\} } |\eta|^\gamma \int_{\mathbb{S}^2} 
\nabla_\eta \mathit{b}_{ij} W_{ij}^{(1)} M_i^\varepsilon (R w + O) \sqrt{\frac{\mu_j (v - \eta)}{\mu_i (v)}} \, \d w\\
& + C_{ij} \textbf{1}_{ \left\{|\eta| \geq \xi \right\} } |\eta|^\gamma \int_{\mathbb{S}^2} 
\mathit{b}_{ij} \nabla_\eta W_{ij}^{(1)} M_i^\varepsilon (R w + O) \sqrt{\frac{\mu_j (v - \eta)}{\mu_i (v)}} \, \d w\\
& + C_{ij}  \textbf{1}_{ \left\{|\eta| \geq \xi \right\} } |\eta|^\gamma \int_{\mathbb{S}^2} 
 \mathit{b}_{ij} W_{ij}^{(1)} \nabla_\eta M_i^\varepsilon (R w + O) \sqrt{\frac{\mu_j (v - \eta)}{\mu_i (v)}} \, \d w\\
 & + C_{ij} \textbf{1}_{ \left\{|\eta| \geq \xi \right\} } |\eta|^\gamma \int_{\mathbb{S}^2} 
 \mathit{b}_{ij} W_{ij}^{(1)} M_i^\varepsilon (R w + O) \nabla_\eta \sqrt{\frac{\mu_j (v - \eta)}{\mu_i (v)}} \, \d w\\
  = & I_1^{(1), S} + I_2^{(1), S} + I_3^{(1), S} + I_4^{(1), S} + I_5^{(1), S} + I_6^{(1), S}.
\end{align*}
The following inequalities hold
\[  \left| |\nabla_\eta |\eta \right|^\gamma | \leq \gamma |\eta|^{\gamma - 1}, \quad |\nabla_\eta \mathit{b}_{ij}| \leq C_{ij} |\eta|^{-1}, 
\quad |\nabla_\eta W_{ij}^{(1)}| \leq C_{ij} |\eta|^{-1},\]
and 
\[|\nabla_\eta \sqrt{\frac{\mu_j (v - \eta)}{\mu_i (v)}}| \leq C_{ij} (|v| + |\eta|) \sqrt{\frac{\mu_j (v - \eta)}{\mu_i (v)}}.
\]
Since $\mathit{b}_{ij}^\prime (\cos{\theta}) \leq C$, referring to the computation details in \cite[Lemma 3.4]{briant2021stability}, 
we can obtain these bounds uniformly with respect to $v\in\mathbb{R}^3$,
\begin{align*}
 |I_1^{(1), S}| & \leq C_{ij} (\delta_1) \max_{1 \leq i, j \leq N} \sqrt{\frac{\bar{c}_j}{\bar{c}_i}}
 \|\mathbf{c}\|_{L_t^\infty L_x^\infty} |\eta|^\gamma  e^{-C (\delta_1) |\eta|^2},\\
 |I_k^{(1), S}| & \leq C_{ij} (\delta_k) \max_{1 \leq i, j \leq N} \sqrt{\frac{\bar{c}_j}{\bar{c}_i}}
  \|\mathbf{c}\|_{L_t^\infty L_x^\infty} |\eta|^{\gamma-1}  e^{-C(\delta_k)|\eta|^2}, \quad  k=\{2,3,4\},\\
 |I_l^{(1), S}| & \leq C_{ij} (\delta_l) \max_{1\leq i, j\leq N} \sqrt{\frac{\bar{c}_j}{\bar{c}_i}}
 \|\mathbf{c}\|_{L_t^\infty L_x^\infty}  (1 + |\eta| + \|\mathbf{u}\|_{L_t^\infty L_x^\infty}) |\eta|^\gamma  e^{- C (\delta_l) |\eta|^2},
 \quad l = \{ 5, 6 \}
\end{align*}
for any fixed $(\delta_l)_{1 \leq l \leq 6 }$ satisfying $(\bar{\delta}, \frac{1}{1 + \delta_{MS}})$. Therefore, the $L_v^1$ norm 
and $L_\eta^1$ norm of $\nabla_\eta k_{ij}^{(1), S}$ are finite.

Back to the expressions \eqref{K-1 v-der express} and \eqref{k-ij split}, using all the above estimates, we 
finally deduce 
\[ \langle \nabla_v \mathbf{K}^{(1), \varepsilon} (\mathbf{f}), \nabla_v \mathbf{f} \rangle_{L_{x,v}^2} \leq C^{(1), \mathbf{K}}_1 (\xi) 
\|\mathbf{f}\|_{L_{x,v}^2}^2 + \xi C^{(1), \mathbf{K}}_2 \|\nabla_v \mathbf{f}\|_{L_{x,v}^2}^2,
\]
where Young's inequality with the constant $\xi$ is applied. Similar bounds can be derived for the operators $\mathbf{K}^{(2), \varepsilon}$ and $ \mathbf{K}^{(3), \varepsilon}$ with constants $C^{(2), \mathbf{K}}_k$ and $ C^{(3), \mathbf{K}}_k$, $k = 1, 2$. By suitably choosing the parameter $\delta$ and gathering these bounds, 
the bound for $\mathbf{K}^\varepsilon$ in \eqref{bound for x-alpha v-beta deri on K} is recovered, with the constants 
$C_1^\mathbf{K} (\xi) = \sum_{l=1}^3 C^{(l), \mathbf{K}}_1 (\xi)$ and $C_2^\mathbf{K} = \sum_{l=1}^3 C^{(l), \mathbf{K}}_2$. 
\end{proof}
%%%%%%%%%%%%%%%%%%%%%%%%%%%%%%%%%%%%%%%
Next we establish the local coercivity for the operator $\mathbf{L}^\varepsilon$ and provide some estimates related to it.
\begin{lemma}\label{lemma for some estimate on L-eps}
Consider collision kernels $(\mathit{B}_{ij})_{1 \leq i,j \leq N}$ that satisfy assumptions (H1)-(H2)-(H3)-(H4) 
with the parameter $\gamma \in [0,1]$, 
and let $\lambda_\mathbf{L} > 0$ be the spectral gap 
in $L^2(\mathbb{R}^3)$ of the operator $\mathbf{L}$. Let $\varepsilon\in(0,1]$, there exists an explicit constant $C_{coe}^\mathbf{L} > 0$, 
such that for any $\eta_1 > 0$ and for any $\mathbf{f} \in L^2(\mathbb{T}^3 \times \mathbb{R}^3)$,
\begin{equation}
\langle \mathbf{L}^\varepsilon (\mathbf{f}), \mathbf{f} \rangle_{L_{x,v}^2} \leq 
-(\lambda_\mathbf{L} - (\varepsilon + \eta_1) \delta_{MS} C_{coe}^\mathbf{L}) \|\mathbf{f}^\perp\|_{L_{x,v}^2(\langle v \rangle^\gamma)}^2 
+ \varepsilon^2 \frac{\delta_{MS} C_{coe}^\mathbf{L}}{\eta_1}  \|\bm{\pi}_\mathbf{L} (\mathbf{f})\|_{L_{x,v}^2(\langle v \rangle^\gamma)}^2. 
\label{hypercoercivity property for L-eps}
\end{equation}

For the upper bound of $\mathbf{L}^\varepsilon$, there exists a positive constant $C_0^\mathbf{L}$, such that for any 
$\mathbf{f}, \mathbf{g} \in L^2(\mathbb{T}^3 \times \mathbb{R}^3)$,
\begin{equation}
\langle \mathbf{L}^\varepsilon (\mathbf{f}), \mathbf{g} \rangle_{L_{x,v}^2} = 
\langle \mathbf{L}^\varepsilon (\mathbf{f}), \mathbf{g}^\perp \rangle_{L_{x,v}^2}  \leq C_0^\mathbf{L} 
\|\mathbf{f}\|_{L_{x,v}^2(\langle v \rangle^\gamma)}  \|\mathbf{g}^\perp\|_{L_{x,v}^2(\langle v \rangle^\gamma)}.
\end{equation}

Moreover, for $s \in \mathbb{N}^\ast$ and the multi-index $\alpha$ satisfies $|\alpha| = s$, then for any $\varepsilon \in (0, 1]$ and for any $\mathbf{f} \in L^2(\mathbb{T}^3 \times \mathbb{R}^3)$, there exists a positive constant $C_\alpha^\mathbf{L}$, such that
\begin{equation}
\langle \partial_x^\alpha \mathbf{L}^\varepsilon (\mathbf{f}) - \mathbf{L}^\varepsilon (\partial_x^\alpha \mathbf{f}), 
\partial_x^\alpha \mathbf{f} \rangle_{L_{x,v}^2} \leq  \varepsilon\delta_{MS} C_\alpha^\mathbf{L} 
\sum_{|\alpha^\prime| \leq s-1}\| \partial_x^{\alpha^\prime} \mathbf{f}\|_{L^2_{x,v}(\langle v \rangle^\gamma)}  
\| \partial_x^\alpha \mathbf{f}^\perp\|_{L_{x,v}^2(\langle v \rangle^\gamma)}, 
\label{sub of L of alpha-x deri}
\end{equation} 
and 
\begin{equation}
\langle \partial_x^\alpha \mathbf{L}^\varepsilon (\mathbf{f}) - \mathbf{L}^\varepsilon (\partial_x^\alpha \mathbf{f}), 
\partial_{v}^{e_k} \partial_x^{\alpha - e_k} \mathbf{f} \rangle_{L_{x,v}^2} \leq \varepsilon \delta_{MS} C_\alpha^\mathbf{L} 
\sum_{|\alpha^\prime| \leq s-1}  \|\partial_x^{\alpha^\prime} \mathbf{f}\|_{L^2_{x,v}(\langle v \rangle^\gamma)}  
\| \partial_{v}^{e_k} \partial_x^{\alpha - e_k} \mathbf{f}^\perp\|_{L_{x,v}^2(\langle v \rangle^\gamma)}.  \label{sub of L of alpha-k x-v deri}
\end{equation}  
\end{lemma}
\begin{proof}
We start to analyze the coercivity of $\mathbf{L}^\varepsilon$. The expression 
$\mathbf{M}^\varepsilon = \mathbf{c} \bm{\mathcal{M}}^\varepsilon = (\bar{\mathbf{c}} + \varepsilon\tilde{\mathbf{c}}) 
\bm{\mathcal{M}}^\varepsilon$ implies that we can split this operator defined in \eqref{def of L-eps} into the following parts, 
\begin{align*}
\mathbf{L}^\varepsilon (\mathbf{f}) = & \bm{\mu}^{-1/2} [ \mathbf{Q} (\mathbf{M}^\varepsilon, \bm{\mu}^{1/2}\mathbf{f}) + 
\mathbf{Q}(\bm{\mu}^{1/2}\mathbf{f}, \mathbf{M}^\varepsilon) ]\\
 = & \bm{\mu}^{-1/2} [ \mathbf{Q} (\bar{\mathbf{c}} \bm{\mathcal{M}}^\varepsilon , \bm{\mu}^{1/2} \mathbf{f}) 
 + \mathbf{Q} (\bm{\mu}^{1/2}\mathbf{f}, \bar{\mathbf{c}} \bm{\mathcal{M}}^\varepsilon) ]
 + \varepsilon \bm{\mu}^{-1/2} \left[ \mathbf{Q} (\tilde{\mathbf{c}} \bm{\mathcal{M}}^\varepsilon, \bm{\mu}^{1/2}\mathbf{f}) 
 + \mathbf{Q} (\bm{\mu}^{1/2} \mathbf{f}, \tilde{\mathbf{c}} \bm{\mathcal{M}}^\varepsilon) \right]\\
  := & \mathbf{L}_\infty^\varepsilon (\mathbf{f}) + \varepsilon \tilde{\mathbf{L}}^\varepsilon,
\end{align*}
We study the Dirichlet form of $\mathbf{L}^\varepsilon$ by introducing the penalization with respect to $\mathbf{L}$,
\begin{equation}
\begin{aligned}
\langle \mathbf{L}^\varepsilon (\mathbf{f}), \mathbf{f} \rangle_{L_{x,v}^2} & = \langle \mathbf{L}^\varepsilon(\mathbf{f}), 
\mathbf{f}^\perp \rangle_{L_{x,v}^2}\\
& = \langle \mathbf{L}(\mathbf{f}), \mathbf{f}^\perp \rangle_{L_{x,v}^2} + 
\langle \mathbf{L}_\infty^\varepsilon (\mathbf{f}) - \mathbf{L}(\mathbf{f}), \mathbf{f}^\perp\rangle_{L_{x,v}^2} 
+ \varepsilon \langle \tilde{\mathbf{L}}^\varepsilon (\mathbf{f}), \mathbf{f}^\perp \rangle_{L_{x,v}^2}.  \label{split for the Dirichlet form of L}
\end{aligned}
\end{equation}
The first equality holds due to the conservation laws for the collision operator $\mathbf{Q}$. 
Result of Theorem 3.3 in \cite{briant2016ARMAglobal} tells us the spectral gap for $\mathbf{L}$, that is, there exists $\lambda_\mathbf{L} > 0$, such that, 
\[ \langle \mathbf{L}(\mathbf{f}), \mathbf{f}^\perp \rangle_{L_v^2} \leq - \lambda_\mathbf{L} 
\|\mathbf{f}^\perp \|_{L_{x,v}^2(\langle v \rangle^\gamma )}^2.
\] 
The second term in \eqref{split for the Dirichlet form of L} is of order $\mathcal{O}(\varepsilon)$, 
which can be estimated following the calculations in 
\cite{Bondesan2020CPAAnonequilibrium}. 
Combining our bounds for $\mathcal{M}_i^\varepsilon - \mathcal{M}_i$ in \eqref{upper bound for sub of M-i-eps and M-i}, 
there exists a positive constant $C_{coe}^\mathbf{L}$ such that
\begin{align*}
  \langle (\mathbf{L}_\infty^\varepsilon - \mathbf{L}) (\mathbf{f}), \mathbf{f}^\perp \rangle_{L_{x,v}^2} = &  
  \langle (\mathbf{L}_\infty^\varepsilon - \mathbf{L}) (\mathbf{f}^\perp), \mathbf{f}^\perp \rangle_{L_{x,v}^2} + 
  \langle (\mathbf{L}_\infty^\varepsilon - \mathbf{L}) (\bm{\pi}_\mathbf{L} (\mathbf{f}) ), \mathbf{f}^\perp \rangle_{L_{x,v}^2}\\
  &\leq (\varepsilon + \eta_1) C_{coe}^\mathbf{L} \delta_{MS}  \|\mathbf{f}^\perp\|_{L_{x,v}^2(\langle v \rangle^\gamma)}^2 
  + \frac{\varepsilon^2 C_{coe}^\mathbf{L}}{\eta_1} \delta_{MS}  \|\bm{\pi}_\mathbf{L} (\mathbf{f})\|_{L_{x,v}^2(\langle v \rangle^\gamma)}^2.
\end{align*}
The deduction uses Young's inequality with the constant $\eta_1/\varepsilon$, and the inequality 
\[ (\|\mathbf{u}\|_{L_t^\infty L_x^\infty} + \|\tilde{T}\|_{L_t^\infty L_x^\infty}) \leq C \delta_{MS},
\]
which follows from the Sobolev embedding $H_x^2\hookrightarrow L_x^\infty$. The constant $C$ is independent of $\varepsilon$. 
For the third term, in a similar approach, we apply the upper bound for $\mathcal{M}_i^\varepsilon$ in \eqref{upper bound for M-i} and use 
Young's inequality with the constant $\frac{\eta_1}{\varepsilon}$, resulting in
\begin{align*}
  \varepsilon \langle \tilde{\mathbf{L}}^\varepsilon (\mathbf{f}), \mathbf{f}^\perp \rangle_{L_{x,v}^2} 
  = & \varepsilon \langle \tilde{\mathbf{L}}^\varepsilon (\mathbf{f}^\perp + \bm{\pi}_\mathbf{L} (\mathbf{f})), 
  \mathbf{f}^\perp \rangle_{L_{x,v}^2}\\
    & \leq (\varepsilon + \eta_1) C_{coe}^\mathbf{L} \delta_{MS} \|\mathbf{f}^\perp\|_{L_{x,v}^2(\langle v \rangle^\gamma)}^2 
  + \frac{\varepsilon^2 C_{coe}^\mathbf{L}}{\eta_1} \delta_{MS} \|\bm{\pi}_\mathbf{L} (\mathbf{f})\|_{L_{x,v}^2(\langle v \rangle^\gamma)}^2,
\end{align*}
increasing the constant $C_{coe}^\mathbf{L}$ if necessary. 
The bound $\|\tilde{\mathbf{c}}\|_{L_t^\infty L_x^\infty} \leq C \delta_{MS}$ is applied, which appears in the definition 
of $\tilde{\mathbf{L}}^\varepsilon$. Summing the above three estimates, we recover inequality 
\eqref{hypercoercivity property for L-eps}.

Furthermore, we also prove the control for operator $\mathbf{L}^\varepsilon$. Using the estimate for $M_i^\varepsilon$, 
we can construct upper bounds for $\nu_{ij}^\varepsilon$ and $ (k_{ij}^{(l)})_{1 \leq l \leq 3}$ respectively. 
Thus, there exists a positive constant $C_0^\mathbf{L}$ such that 
\begin{equation}
\langle \mathbf{L}^\varepsilon (\mathbf{f}), \mathbf{g} \rangle_{L_{x,v}^2} = 
\langle \mathbf{L}^\varepsilon (\mathbf{f}), \mathbf{g}^\perp  \rangle_{L_{x,v}^2}  \leq C_0^\mathbf{L} 
\|\mathbf{f}\|_{L_{x,v}^2(\langle v \rangle^\gamma)}  \|\mathbf{g}^\perp\|_{L_{x,v}^2(\langle v \rangle^\gamma)}.
\end{equation}
The constant has the form $C_0^\mathbf{L} = C(\delta) \max_{1 \leq i,j \leq N} \sqrt{\frac{\bar{c}_j}{\bar{c}_i}}
\|\mathbf{c}\|_{L_t^\infty L_x^\infty} > 0$, where $\delta \in (\bar{\delta}, \frac{1}{1 + \delta_{MS}})$ can be chosen arbitrarily.

Finally, there remains analysis about the terms $\langle \partial_x^\alpha \mathbf{L}^\varepsilon (\mathbf{f}), 
\partial_x^\alpha \mathbf{f} \rangle_{L_{x,v}^2}$ and $\langle \partial_x^\alpha \mathbf{L}^\varepsilon (\mathbf{f}), 
\partial_{v_k} \partial_x^{\alpha-e_k} \mathbf{f} \rangle_{L_{x,v}^2}$, which arise in the subsequent computations. 
We only control the terms involving 
$\partial_x^\alpha \mathbf{L}^\varepsilon (\mathbf{f}) - \mathbf{L}^\varepsilon (\partial_x^\alpha \mathbf{f})$, 
as the estimate for $\mathbf{L}^\varepsilon$ has been established above. It can be expressed as 
\[  \partial_x^\alpha \mathbf{L}^\varepsilon (\mathbf{f}) - \mathbf{L}^\varepsilon (\partial_x^\alpha \mathbf{f}) = 
\sum_{ \substack{\alpha^\prime \leq \alpha \\ |\alpha^\prime| \neq 0} } (\partial_x^{\alpha^\prime} \mathbf{L}^\varepsilon) 
(\partial_x^{\alpha - \alpha^\prime} \mathbf{f}),
\]
where the terms $(\partial_x^{\alpha^\prime} \mathbf{L}^\varepsilon)$ denote the actions 
that $\partial_x^{\alpha^\prime}$ derivatives ($|\alpha^\prime| \geq 1$) only act on the the kernels 
$(k_{ij}^{(l)})_{1 \leq i,j \leq N, 1 \leq l \leq 3}$ of operator $\mathbf{K}^\varepsilon$ and $(\nu_{ij}^\varepsilon)_{1 \leq i, j \leq N}$. 
It is not hard to notice that these terms are of order $\mathcal{O} (\varepsilon)$. 
Following estimates in Lemmas \ref{lemma for bound of nu and its v-derivative} 
\ref{lemma for K operator with v derivative}, there exists a constant $C_\alpha^\mathbf{L}$, such that 
\begin{align*}
\langle \partial_x^\alpha \mathbf{L}^\varepsilon (\mathbf{f}) - \mathbf{L}^\varepsilon (\partial_x^\alpha \mathbf{f}), 
\partial_x^\alpha \mathbf{f} \rangle_{L_{x,v}^2} = & 
\langle \partial_x^\alpha \mathbf{L}^\varepsilon (\mathbf{f}) - \mathbf{L}^\varepsilon (\partial_x^\alpha \mathbf{f}), 
\partial_x^\alpha \mathbf{f}^\perp \rangle_{L_{x,v}^2}\\
& \leq \varepsilon \delta_{MS} C_\alpha^\mathbf{L} \sum_{|\alpha^\prime| \leq s-1}  
\|\partial_x^{\alpha^\prime} \mathbf{f}\|_{L^2_{x,v}(\langle v \rangle^\gamma)}  
\|\partial_x^\alpha \mathbf{f}^\perp\|_{L_{x,v}^2(\langle v \rangle^\gamma)}, 
\end{align*}
with the notation $\partial_x^\alpha \mathbf{f}^\perp = \partial_x^\alpha \mathbf{f} -\bm{\pi}_{\mathbf{L}} (\partial_x^\alpha \mathbf{f})$. 
The first equality holds since $x$-derivative does not change its action on $v$ variable. Moreover, $\mathcal{O}(\delta_{MS})$ is obtained from the upper bounds for the macroscopic quantities, 
as there is at least one $x$-derivative acts on $\mathbf{M}^\varepsilon$. Similar computations can recover the inequality 
\eqref{sub of L of alpha-k x-v deri}.
\end{proof}
Controls for the nonlinear operator $\mathbf{\Gamma}$ are extensions of the same property for the Boltzmann operator in mono-species case. By adapting the calculations in Appendix A in \cite{briant2015JDEnavierstokes}, we can obtain the following result.
\begin{lemma}\label{lemma for estimates on Gamma}
The linearized operator $\mathbf{L}^\varepsilon$ and nonlinear operator $\mathbf{\Gamma}$ are orthogonal to the kernel of $\mathbf{L}$, that are 
\begin{equation}
\bm{\pi}_\mathbf{L} (\mathbf{L}^\varepsilon (\mathbf{f})) = \bm{\pi}_\mathbf{L} (\mathbf{\Gamma} (\mathbf{g}, \mathbf{h})) 
= 0,\quad \forall \mathbf{f}, \mathbf{g}, \mathbf{h} \in L^2(\mathbb{R}^3).
\end{equation}
Moreover, for any $s \in \mathbb{N}$ and multi-indices $\alpha, \beta$ satisfying $|\alpha| + |\beta| = s$, there 
exist two nonnegative functionals $\mathcal{G}_x^s$ and $\mathcal{G}_{x,v}^s$ which increase with respect to the 
index $s$ separately, such that 
\begin{equation}
  \left| \langle \partial_x^\beta \partial_x^\alpha \mathbf{\Gamma}(\mathbf{g},\mathbf{h}), \mathbf{f} \rangle_{L_{x,v}^2} \right| \leq 
    \begin{cases} 
  \mathcal{G}_x^s(\mathbf{g}, \mathbf{h})  \|\mathbf{f} \|_{L_{x,v}^2(\langle v \rangle^\gamma)} & \text{if } |\beta| = 0, \\ 
  \mathcal{G}_{x,v}^s(\mathbf{g}, \mathbf{h})  \|\mathbf{f} \|_{L_{x,v}^2(\langle v \rangle^\gamma)} & \text{if } |\beta| \geq 1.
  \end{cases}
  \end{equation}
  In particular, there exists $s_0 \in \mathbb{N}^\ast$ such that, for any integer $s \geq s_0$, there exists an explicit 
  constant $C_s^{\mathbf{\Gamma}} > 0$, such that
  \begin{equation}
  \begin{aligned}
  \mathcal{G}_x^s (\mathbf{g}, \mathbf{h}) \leq C_s^{\mathbf{\Gamma}} \left ( \|\mathbf{g}\|_{H_x^s L_v^2}  
  \|\mathbf{h}\|_{H_x^s L_v^2(\langle v \rangle^\gamma)} + \|\mathbf{h}\|_{H_x^s L_v^2}  
  \|\mathbf{g}\|_{H_x^s L_v^2(\langle v \rangle^\gamma)} \right )\\
     \mathcal{G}_{x,v}^s(\mathbf{g}, \mathbf{h}) \leq C_s^{\mathbf{\Gamma}} \left ( \|\mathbf{g}\|_{H_{x,v}^s}  
     \|\mathbf{h}\|_{H_{x,v}^s (\langle v \rangle^\gamma)} + \|\mathbf{h}\|_{H_{x,v}^s}  
     \|\mathbf{g}\|_{H_{x,v}^s (\langle v \rangle^\gamma)} \right ) .
  \end{aligned}
  \end{equation}
\end{lemma}
The first property is derived from the definitions of operators $\mathbf{L}^\varepsilon$ and $\mathbf{\Gamma}$, and the conservation laws for the collision operator $\mathbf{Q}$. 
We omit the proof details for the bounds of $\mathbf{\Gamma}$.

%%%%%%%%%%%%%%%%%%%%%%%%%%%%%%%%%%%%%%%%%%%%%%%%%%%%%%%%%%%%%%%%%%%%%%%%%%%%%%
Next, we present estimates for the source term $\mathbf{S}^\varepsilon$.
\begin{lemma}\label{lemma for estimates on source term S}
Let $\varepsilon \in (0,1]$, $s \in \mathbb{N}$ and $\mathbf{f} \in H^s(\mathbb{T}^3 \times \mathbb{R}^3)$, we introduce two 
multi-indices $\alpha, \beta$ such that $|\alpha| + |\beta|=s$. If $|\beta| \geq 1$, there exists a positive constant 
$C^\mathbf{S}_{\alpha, \beta}$, such that for any $\eta_2 > 0$,
\begin{equation}
 \langle \partial_v^\beta \partial_x^\alpha \mathbf{S}^\varepsilon, \partial_v^\beta \partial_x^\alpha \mathbf{f}\rangle_{L_{x,v}^2} 
 \leq \frac{C^\mathbf{S}_{\alpha, \beta} \delta_{MS}^2}{\eta_2} + \frac{\eta_2}{\varepsilon^2} 
 \| \partial_v^\beta \partial_x^\alpha \mathbf{f}\|_{L_{x,v}^2(\langle v \rangle^\gamma)}^2.
 \label{bound for x-alpha v-beta deri of source term S}
\end{equation}
For $\beta = e_k$ for some $k\in \{1,2,3\}$ with $\alpha_k>0$, there exists a constant $C^\mathbf{S}_{\alpha, k}$, such that for any $\eta_3 > 0$,
\begin{equation}
\langle \partial_x^\alpha \mathbf{S}^\varepsilon, \partial_v^{e_k} \partial_x^{\alpha-e_k} \mathbf{f}\rangle_{L_{x,v}^2} \leq  
\frac{C^\mathbf{S}_{\alpha, k} \delta_{MS}^2}{\varepsilon \eta_3} + \frac{\eta_3}{\varepsilon} 
\|\partial_v^{e_k} \partial_x^{\alpha - e_k} \mathbf{f}\|_{L_{x,v}^2(\langle v \rangle^\gamma)}^2.
\label{bound for x-alpha v-k deri of source term S}
\end{equation}
Moreover, for the pure $x$-derivative, we establish a stronger control on projection $\bm{\pi}_\mathbf{L} (\mathbf{S}^\varepsilon)$. There exists a 
constant $C_\alpha^\mathbf{S}$, such that for any $\eta_4, \eta_5 > 0$,
\begin{equation}
\langle \partial_x^\alpha \mathbf{S}^\varepsilon, \partial_x^\alpha \mathbf{f}\rangle_{L_{x,v}^2} \leq 
\frac{C_\alpha^\mathbf{S} \delta_{MS}^2}{\eta_4 \eta_5} + \eta_4 \|\bm{\pi}_\mathbf{L} (\partial_x^\alpha \mathbf{f})\|_{L_{x,v}^2}^2 
+ \frac{\eta_5}{\varepsilon^2} \|\partial_x^\alpha \mathbf{f}^\perp\|_{{L_{x,v}^2(\langle v \rangle^\gamma)}}^2 
\label{bound for pure x-alpha deri of source term S}
\end{equation}
\end{lemma}
\begin{proof}
The explicit expression for the source term \eqref{equ for source term} infers that we can extract a power 
$\varepsilon^{-1}$ from it, dividing it into linear and nonlinear parts:
\begin{equation}
\partial_v^\beta \partial_x^\alpha \mathbf{S}^\varepsilon = \frac{1}{\varepsilon} \left\{ - \partial_v^\beta \partial_x^\alpha 
\left (\bm{\mu}^{ - \frac{1}{2}} \partial_t\mathbf{M}^\varepsilon + \frac{1}{\varepsilon} \bm{\mu}^{-\frac{1}{2}} v \cdot \nabla_x 
\mathbf{M}^\varepsilon \right ) + \frac{1}{\varepsilon^2} \partial_v^\beta \partial_x^\alpha \left ( 
  \bm{\mu}^{-\frac{1}{2}} \mathbf{Q} (\mathbf{M}^\varepsilon, \mathbf{M}^\varepsilon) \right ) \right\}.  
\label{alpha-x beta-v express of source term}
\end{equation}
The linear term involves time and spatial derivatives. We compute, for $1 \leq i \leq N$,
\begin{equation}
\begin{aligned}
\partial_t M_i^\varepsilon + \frac{1}{\varepsilon} v \cdot \nabla_x & M_i^\varepsilon = \left\{ \partial_t c_i + 
\frac{\varepsilon m_i c_i (v - \varepsilon u_i) \cdot \partial_t u_i}{T} + \left ( \frac{m_i |v-\varepsilon u_i|^2}{2T} - \frac{3}{2} 
\right ) \frac{c_i \partial_t T}{T}\right.\\
+ & \left. \sum_{k=1}^3 v_k \left ( \partial_{x_k} \tilde{c}_i + \frac{m_i c_i (v - \varepsilon u_i) \cdot \partial_{x_k} u_i}{T} 
+ \left ( \frac{m_i |v - \varepsilon u_i|^2}{2T} - \frac{3}{2} \right ) \frac{c_i \partial_{x_k} \tilde{T}}{T} \right ) \right\} 
\mathcal{M}_i^\varepsilon,
\label{v-beta x-alpha source linear term}
\end{aligned}
\end{equation}
and the leading order of this part is $\mathcal{O} (1)$. Since $(\mathbf{c}, \mathbf{u}, T)$ is the unique solution for 
Maxwell-Stefan system stated in Theorem \ref{theorem for M-S in main result}, the following relations involving time derivatives hold,
\[ \partial_t \mathbf{c} = -  \nabla_x \mathbf{c} \cdot \mathbf{u} - \mathbf{c} \nabla_x \cdot \mathbf{u}, 
\]
\[ \partial_t T = \frac{2\alpha}{3} \frac{\Delta_x c_{tot}}{c_{tot}} + \frac{5\alpha}{3} \frac{\nabla_x c_{tot}}{c_{tot}} 
\cdot \nabla_x T.
\]
Applying time derivative on the 
flux-force relations $\nabla_x (\mathbf{c} T) = T^{\gamma/2} A (\mathbf{c}) \mathbf{u}$, one obtains
\begin{equation}
T^{ - \gamma/2} \nabla_x ( \partial_t \mathbf{c} T + \mathbf{c} \partial_t T) = \frac{\gamma}{2} \frac{ \partial_t T}{T} 
A (\mathbf{c}) \mathbf{u} + \partial_t A(\mathbf{c}) \mathbf{u} + A(\mathbf{c}) \partial_t \mathbf{u}. \label{time der on M-S rela}
\end{equation}
Taking the time derivative on the expression of the matrix $A(\mathbf{c})$, we deduce
\[ [ \partial_t A(\mathbf{c})]_{ij} = \frac{\partial_t c_i c_j}{\Delta_{ij}} - \left ( 
  \sum_{r=1}^N \frac{ \partial_t c_i c_r}{\Delta_{ir}} \right ) \delta_{ij} + \frac{ c_i \partial_t c_j}{\Delta_{ij}} 
  - \left ( \sum_{r=1}^N \frac{ c_i \partial_t c_r}{\Delta_{ir}} \right ) \delta_{ij},
\]
which shares the same matrix structure as $A(\mathbf{c})$, leading to the fact that 
$ \partial_t A (\mathbf{c}) \mathbf{u} \in \mathrm{Span} (\mathbf{1})^\perp$. Moreover, the relations 
$\langle \nabla_x (\mathbf{c} T), \mathbf{1} \rangle = 0$ and $\mathrm{Ker} A(\mathbf{c})= \mathrm{Span} ( \mathbf{1} )$ hold. 
Replacing $\partial_t \mathbf{c}, \partial_t T$ and $\partial_t A(\mathbf{c})$ in the equation \eqref{time der on M-S rela} with the above relations, we obtain the following expression
\begin{align*}
\partial_t \mathbf{u} = T^{-\gamma/2} A (\mathbf{c})^{-1} \nabla_x & \left ( - \nabla_x\mathbf{c}\cdot \mathbf{u} T 
- \mathbf{c} T (\nabla_x \cdot \mathbf{u}) + \frac{2\alpha}{3} \frac{\Delta_x c_{tot}}{c_{tot}} \mathbf{c} 
+ (\frac{5\alpha}{3} \frac{\nabla_x c_{tot}}{c_{tot}} \cdot \nabla_x T) \mathbf{c}  \right )\\
  & - \frac{\gamma}{2} \frac{\partial_t T}{T} \mathbf{u} - A (\mathbf{c})^{-1} \partial_t A (\mathbf{c}) \mathbf{u}.
\end{align*} 

Now, we replace the time derivative terms in \eqref{v-beta x-alpha source linear term}, 
where the macroscopic quantities 
$\mathbf{c}, \mathbf{u}, T, c_{tot}$ and their spatial derivatives appear as factors 
in each term on the right hand side. To obtain an upper bound of $\mathcal{O} (\delta_{MS})$, we use the Sobolev 
embedding $H_x^2 \hookrightarrow L_x^\infty$, assuming that the macroscopic quantities in Theorem \ref{theorem for M-S in main result} 
are regular enough. This leads to the following bound for any 
$1 \leq i \leq N$ and for almost any $(t,x,v)\in \mathbb{R}_+ \times \mathbb{T}^3 \times \mathbb{R}^3$,
\[ \left| \mu_i^{-1/2}( \partial_t M_i^\varepsilon + \frac{1}{\varepsilon} v \cdot \nabla_x M_i^\varepsilon) \right|^2 \leq 
\delta_{MS}^2 C (\delta) P (v) \mathcal{M}_i^{2 \delta - 1} (v),
\]
where the expression $\mu_i = \bar{c}_i \mathcal{M}_i$ is utilized. The constant $C (\delta)$ only depends on 
$\delta \in (\frac{1}{2}, \frac{1}{1 + \delta_{MS}})$, and $P (v)$ denotes a polynomial of $|v|$ whose coefficients are independent of $\varepsilon$. We emphasize that small assumptions on $\delta_{MS}$ in 
Theorem \ref{theorem for M-S in main result} guarantee $\frac{1}{1 + \delta_{MS}} > \frac{1}{2}$, ensuring that $\delta$ is alternative. 
Furthermore, we notice that the highest order of spatial derivatives of the macroscopic quantities in the linear parts of $\mathbf{S}^\varepsilon$ are terms $(\nabla_x)^2 \mathbf{c}, (\nabla_x)^2\mathbf{u}, (\nabla_x)^2 T$ and $\nabla_x \Delta_x c_{tot}$. Therefore, the macroscopic quantities 
at least need to satisfy $\mathbf{c}, T \in L_t^\infty H_x^5$ and $\mathbf{u} \in L_t^\infty H_x^4$ to ensure the upper bounds hold. 

Taking $\partial_v^\beta \partial_x^\alpha$ derivatives on the linear term, we notice that 
these derivatives only increase the number of factors depending on polynomials of $v$, and derivatives of $\mathbf{c}, \mathbf{u}, T$.
Similar estimates can deduce that there exist a polynomial $P_L(v)$ with the highest order $\mathcal{O} (|v|^{|\beta| + 1})$, and a positive 
constant $C_L(\delta)$ only depending on $\delta \in (\frac{1}{2}, \frac{1}{1 + \delta_{MS}})$, such that for any $1 \leq i \leq N$ and 
for alomst any $(t,x,v) \in \mathbb{R}_+ \times \mathbb{T}^3 \times \mathbb{R}^3$, 
\begin{equation}
\left| \partial_v^\beta \partial_x^\alpha \left ( \mu_i^{-1/2} (\partial_t M_i^\varepsilon + \frac{1}{\varepsilon} v \cdot M_i^\varepsilon)
  \right )\right|^2 \leq \delta_{MS}^2 C_L (\delta) P_L (v) \mathcal{M}_i^{2 \delta - 1}.
\end{equation}
There are norms $\|\mathbf{c}\|_{L_t^\infty H_x^{|\alpha| + 5}}, \|\mathbf{u}\|_{L_t^\infty H_x^{|\alpha| + 4}}$ and 
$\|T\|_{L_t^\infty H_x^{|\alpha| + 4}}$ hidden inside the constant $C_L (\delta)$.

Now we estimate the following inner product, by applying Young's inequality with a positive constant $\eta_2/ \varepsilon$ 
\begin{equation}
  \begin{aligned}
  \frac{1}{\varepsilon} & \left| \langle \partial_v^\beta \partial_x^\alpha \left ( \bm{\mu}^{-1/2} 
  ( \partial_t \mathbf{M}^\varepsilon + \frac{1}{\varepsilon} v \cdot \mathbf{M}^\varepsilon) \right ), 
  \partial_v^\beta \partial_x^\alpha \mathbf{f}\rangle_{L_{x,v}^2}
\right|\\
 &\leq \frac{1}{\eta_2}  \|\partial_v^\beta \partial_x^\alpha \left ( \bm{\mu}^{-1/2} 
 (\partial_t \mathbf{M}^\varepsilon + \frac{1}{\varepsilon} v \cdot \mathbf{M}^\varepsilon) \right )\| _{L_{x,v}^2}^2 
 + \frac{\eta_2}{\varepsilon^2} \|\partial_v^\beta \partial_x^\alpha \mathbf{f}\|_{L_{x,v}^2}^2\\
 & \leq \frac{C_L (\delta) \delta_{MS}^2}{\eta_2} \sum_{i=1}^N \int_{\mathbb{T}^3 \times \mathbb{R}^3} P_L (v) 
 \mathcal{M}_i^{2 \delta - 1} (v) \,\d x \d v + \frac{\eta_2}{\varepsilon^2}  \|\partial_v^\beta \partial_x^\alpha \mathbf{f}\|_{L_{x,v}^2}^2\\
 & \leq \frac{C_L (\delta) \delta_{MS}^2}{\eta_2} + \frac{\eta_2}{\varepsilon^2}  
 \|\partial_v^\beta \partial_x^\alpha \mathbf{f}\|_{L_{x,v}^2}^2,
  \end{aligned}
\end{equation}
for any $\delta \in (\frac{1}{2}, \frac{1}{1 + \delta_{MS}})$, increasing the constant $C_L (\delta)$ if necessary.

We continue our estimates on the nonlinear term $\bm{\mu}^{-1/2} \mathbf{Q} (\mathbf{M}^\varepsilon, \mathbf{M}^\varepsilon) 
= \mathbf{\Gamma} (\bm{\mu}^{-\frac{1}{2}} \mathbf{M}^\varepsilon, \bm{\mu}^{-\frac{1}{2}} \mathbf{M}^\varepsilon)$, splitting the 
Maxwellian $\mathbf{M}^\varepsilon$ into two parts 
$\mathbf{M}^\varepsilon = \mathbf{c} \bm{\mathcal{M}}_T^\varepsilon + 
(\mathbf{M}^\varepsilon - \mathbf{c} \bm{\mathcal{M}}_T^\varepsilon)$. 
The local Maxwellian is defined as $\bm{\mathcal{M}}_T^\varepsilon = (\mathcal{M}_{T,1}^\varepsilon, \dots, \mathcal{M}_{T,N}^\varepsilon)$, and for any $1 \leq i \leq N$,
\[  \mathcal{M}_{T,i}^\varepsilon (t,x,v) = \left ( \frac{m_i}{2 \pi T (t,x)} \right )^{\frac{3}{2}} 
\exp \left\{ -\frac{m_i |v|^2}{2 \pi T (t,x)} \right\}, \quad  \forall (t,x,v) \in \mathbb{R}_+ \times \mathbb{T}^3 \times \mathbb{R}^3.      
\]
The fact $\mathbf{\Gamma} (\bm{\mu}^{-\frac{1}{2}} \mathbf{c} \bm{\mathcal{M}}_T^\varepsilon, \bm{\mu}^{-\frac{1}{2}} \mathbf{c} 
\bm{\mathcal{M}}_T^\varepsilon) = \bm{\mu}^{-\frac{1}{2}} \mathbf{Q} (\mathbf{c} \bm{\mathcal{M}}_T^\varepsilon, 
\mathbf{c} \bm{\mathcal{M}}_T^\varepsilon) = 0$ is clear. Due to the perturbation $\mathbf{u} = \varepsilon \tilde{\mathbf{u}}$,
it is not hard to notice that the Maxwellians $\mathbf{M}^\varepsilon$  actually are split into a local equilibrium and a supplement, which is closed to this local equilibrium up to an order $\varepsilon^2 \|\tilde{\mathbf{u}}\|_{L_{x,v}^\infty}$. 
We can rewrite the nonlinear term as
\begin{align*}
\mathbf{\Gamma} (\bm{\mu}^{-\frac{1}{2}} \mathbf{M}^\varepsilon, \bm{\mu}^{-\frac{1}{2}} \mathbf{M}^\varepsilon) 
= & \mathbf{\Gamma} ( \bm{\mu}^{-\frac{1}{2}} \mathbf{c} \bm{\mathcal{M}}_T^\varepsilon, \bm{\mu}^{-\frac{1}{2}} 
(\mathbf{M}^\varepsilon - \mathbf{c} \bm{\mathcal{M}}_T^\varepsilon) ) + \mathbf{\Gamma} 
( \bm{\mu}^{-\frac{1}{2}} (\mathbf{M}^\varepsilon - \mathbf{c} \bm{\mathcal{M}}_T^\varepsilon), \bm{\mu}^{-\frac{1}{2}} \mathbf{c} 
\bm{\mathcal{M}}_T^\varepsilon )\\
& + \mathbf{\Gamma} ( \bm{\mu}^{-\frac{1}{2}} (\mathbf{M}^\varepsilon - \mathbf{c} \bm{\mathcal{M}}_T^\varepsilon), 
\bm{\mu}^{-\frac{1}{2}} (\mathbf{M}^\varepsilon - \mathbf{c} \bm{\mathcal{M}}_T^\varepsilon) ).
\end{align*}
The Maxwellian has the form $\mathbf{M}^\varepsilon = \mathbf{c} \bm{\mathcal{M}}^\varepsilon$, and the difference between $\bm{\mathcal{M}}^\varepsilon$ and $\bm{\mathcal{M}}_T^\varepsilon$ appears on the velocity quantities 
$\varepsilon^2 \tilde{\mathbf{u}}$. Referring to \cite{Bondesan2020CPAAnonequilibrium}, there exist polynomials $P_1 (v), P_2 (v)$ and 
positive constants $C_1 (\delta), C_2 (\delta)$ that only depend on $\delta \in (\frac{1}{2}, \frac{1}{1 + \delta_{MS}})$, such that, 
for any $1 \leq i \leq N$, and for any $(t,x,v) \in \mathbb{R}_+ \times \mathbb{T}^3 \times \mathbb{R}^3$, 
\begin{align}
\left| \partial_v^\beta \partial_x^\alpha \left ( c_i \mu_i^{-\frac{1}{2}} (\mathcal{M}_i^\varepsilon - \mathcal{M}_{T,i}^\varepsilon) 
\right ) \right| \leq \varepsilon^2 \delta_{MS} C_1 (\delta) P_1 (v) \mathcal{M}_i^{ \delta - \frac{1}{2}} (v) ,
\label{bound for v-beta x-alpha of sub of M-i and M-T-i}\\
\left| \partial_v^\beta \partial_x^\alpha (c_i \mu_i^{-\frac{1}{2}} \mathcal{M}_{T,i}^\varepsilon) \right| \leq 
C_2 (\delta) P_2 (v) \mathcal{M}_i^{\delta - \frac{1}{2}} (v).
\label{bound for v-beta x-alpha of M-T-i}
\end{align}
We further explain that the second inequality can be bound by $\mathcal{O} (\delta_{MS})$ if $|\alpha| \neq 0$.
The decomposition of $\mathbf{\Gamma} (\bm{\mu}^{-\frac{1}{2}} \mathbf{M}^\varepsilon, \bm{\mu}^{-\frac{1}{2}} \mathbf{M}^\varepsilon)$ 
consists of three parts, each comprising pairs of products formed from the two types of components mentioned above, while preserving the expected $\mathcal{O} (\delta_{MS})$.

Therefore, the nonlinear term can be bounded as follows, by applying the estimate on $\mathbf{\Gamma}$ in Lemma 
  \ref{lemma for estimates on Gamma},
  \begin{align*}
  & \frac{1}{\varepsilon^3} \left| \langle \partial_v^\beta \partial_x^\alpha 
  \mathbf{\Gamma} ( \bm{\mu}^{-\frac{1}{2}} \mathbf{M}^\varepsilon, \bm{\mu}^{-\frac{1}{2}} \mathbf{M}^\varepsilon ), 
  \partial_v^\beta \partial_x^\alpha \mathbf{f} \rangle_{L_{x,v}^2} \right|\\
   \leq & \frac{C_s^{\mathbf{\Gamma}}}{\varepsilon} \left\{ \left ( 
    2 \|\mathbf{c} \bm{\mu}^{-\frac{1}{2}} \bm{\mathcal{M}}_T^\varepsilon\|_{H_{x,v}^s} + 
    \|\mathbf{c} \bm{\mu}^{-\frac{1}{2}} (\bm{\mathcal{M}}^\varepsilon - \bm{\mathcal{M}}_T^\varepsilon)\|_{H_{x,v}^s} \right ) 
    \left\| \frac{ \mathbf{c} \bm{\mu}^{-\frac{1}{2}} (\bm{\mathcal{M}}^\varepsilon - \bm{\mathcal{M}}_T^\varepsilon) }{\varepsilon^2} 
    \right\|_{H_{x,v}^s(\langle v\rangle^\gamma)}\right.\\
    + & \left. \left ( 2 \|\mathbf{c} \bm{\mu}^{-\frac{1}{2}} \bm{\mathcal{M}}_T^\varepsilon\|_{H_{x,v}^s (\langle v \rangle^\gamma)} 
    + \|\mathbf{c} \bm{\mu}^{-\frac{1}{2}} (\bm{\mathcal{M}}^\varepsilon - \bm{\mathcal{M}}_T^\varepsilon)\|_{H_{x,v}^s(\langle v \rangle^\gamma)} 
    \right )  \left\| \frac{\mathbf{c} \bm{\mu}^{-\frac{1}{2}} (\bm{\mathcal{M}}^\varepsilon - \bm{\mathcal{M}}_T^\varepsilon)}
    {\varepsilon^2} \right\|_{H_{x,v}^s} \right\} \\
    & \times \|\partial_v^\beta \partial_x^\alpha \mathbf{f}\|_{L_{x,v}^2 (\langle v \rangle^\gamma)}.
  \end{align*}
Combining inequalities \eqref{bound for v-beta x-alpha of sub of M-i and M-T-i} and \eqref{bound for v-beta x-alpha of M-T-i}, 
and performing calculations similar to those in the linear term, there exists a positive constant $C_{NL} (\delta)$ such that 
\begin{equation}
  \frac{1}{\varepsilon^3} \left| \langle \partial_v^\beta \partial_x^\alpha \left (
  \bm{\mu}^{-1/2} \mathbf{Q} (\mathbf{M}^\varepsilon, \mathbf{M}^\varepsilon) \right ), 
  \partial_v^\beta \partial_x^\alpha \mathbf{f} \rangle_{L_{x,v}^2} \right| \leq \frac{C_{NL} (\delta) \delta_{MS}^2}{\eta_2} 
  + \frac{\eta_2}{\varepsilon^2}  \|\partial_v^\beta \partial_x^\alpha \mathbf{f}\|_{L_{x,v}^2 (\langle v \rangle^\gamma)}^2.
\end{equation}
The above integrals are finite since we choose $\delta> 1/2$. Gathering the linear and nonlinear terms, the following bound for the source term $\mathbf{S}^\varepsilon$ is derived,
\begin{equation}
\langle \partial_v^\beta \partial_x^\alpha \mathbf{S}^\varepsilon, \partial_v^\beta \partial_x^\alpha \mathbf{f} \rangle_{L_{x,v}^2} 
\leq \frac{C^\mathbf{S}_{\alpha, \beta} \delta_{MS}^2}{\eta_2} + \frac{\eta_2}{\varepsilon^2} 
\|\partial_v^\beta \partial_x^\alpha \mathbf{f}\|_{L_{x,v}^2 (\langle v \rangle^\gamma)}^2 ,
\end{equation}
where the constant is defined as $C^\mathbf{S}_{\alpha, \beta} = \max\{ C_L (\delta), C_{NL} (\delta) \}$, for any arbitrary $\delta \in (\frac{1}{2}, \frac{1}{1 + \delta_{MS}})$.

Similar analysis can be used to derive the result \eqref{bound for x-alpha v-k deri of source term S}. The distinction lies in our choice of the constant $\eta_3$ (rather than $\eta_2/\varepsilon$) when applying Young's inequality, due to the presence of an $\varepsilon$-factor multiplying this term in the target norm $\mathcal{H}_\varepsilon^s$.

Finally, we analyze terms involving spatial derivatives of the source term, by first decomposing it into $\mathrm{Ker} \mathbf{L}$ and $(\mathrm{Ker}{\mathbf{L}})^\perp$ components,
\begin{equation}
\langle \partial_x^\alpha \mathbf{S}^\varepsilon, \partial_x^\alpha \mathbf{f} \rangle_{L_{x,v}^2} = 
\langle \bm{\pi}_{\mathbf{L}} (\partial_x^\alpha \mathbf{S}^\varepsilon),\bm{\pi}_{\mathbf{L}} (\partial_x^\alpha \mathbf{f}) 
\rangle_{L_{x,v}^2} + \langle \partial_x^\alpha \mathbf{S}^{\varepsilon, \perp}, \partial_x^\alpha \mathbf{f}^\perp \rangle_{L_{x,v}^2}.
\label{x-alpha source decom}
\end{equation}
Together with the expression of $\mathbf{S}^\varepsilon$ in 
\eqref{equ for source term} and the fact that $\bm{\pi}_\mathbf{L} (\mathbf{\Gamma} (\mathbf{f}, \mathbf{g})) = 0$ for any 
$\mathbf{f},\mathbf{g} \in L^2(\mathbb{T}^3 \times \mathbb{R}^3)$, the projection $\bm{\pi}_{\mathbf{L}} (\mathbf{S}^\varepsilon)$ 
can be expressed as
\[  \bm{\pi}_{\mathbf{L}} (\mathbf{S}^\varepsilon) = - \frac{1}{\varepsilon^2} \bm{\pi}_\mathbf{L} \left ( 
   \bm{\mu}^{-1/2} ( \varepsilon \partial_t \mathbf{M}^\varepsilon + v \cdot \nabla_x \mathbf{M}^\varepsilon) \right ).
\]
Using the expression of the projection $\bm{\pi}_\mathbf{L}$ in \eqref{proje pi-L express}, 
the following equalities hold, for any $1 \leq i \leq N$,
\begin{align*}
\int_{\mathbb{R}^3} \varepsilon \partial_t M_i^\varepsilon + v \cdot \nabla_x M_i^\varepsilon \,\d v & = \varepsilon 
(\partial_t c_i + \nabla_x \cdot (c_i u_i)) , \\
\int_{\mathbb{R}^3} m_i v (\varepsilon \partial_t M_i^\varepsilon + v \cdot \nabla_x M_i^\varepsilon) \,\d v & = \varepsilon^2 m_i \left ( 
  \partial_t (c_i u_i) + \nabla_x \cdot (c_i u_i \otimes u_i ) \right ) + \nabla_x (c_i T) ,\\
\int_{\mathbb{R}^3} (m_i|v|^2 - 3 ) (\varepsilon \partial_t M_i^\varepsilon + v \cdot \nabla_x M_i^\varepsilon) \,\d v & =  3 \varepsilon \partial_t 
(c_i T) + \varepsilon^3 m_i \partial_t (c_i |u_i|^2) - 3 \varepsilon \partial_t c_i\\
& + 5 \varepsilon \nabla_x \cdot (c_i T u_i) + \varepsilon^3 m_i\nabla_x \cdot (c_i |u_i|^2 u_i) - 3 \varepsilon \nabla_x \cdot(c_i u_i),
\end{align*}
where we use computations
\begin{equation*}
\int_{\mathbb{R}^3} M_i^\varepsilon \left ( 
  \begin{matrix} 1 \\ v \\|v|^2 \end{matrix} 
  \right ) \,\d v = \left ( 
    \begin{matrix} c_i \\ \varepsilon c_i u_i \\ \frac{3}{m_i} c_i T + \varepsilon^2 c_i |u_i|^2 \end{matrix}  
    \right ),
\end{equation*}
and
\[ \int_{\mathbb{R}^3} |v|^2 v M_i^\varepsilon \,\d v =\varepsilon^3 c_i |u_i|^2 u_i + \frac{5 c_i T}{m_i} u_i.
\]
Since the quantities $(\mathbf{c}, \mathbf{u}, T)$ satisfy the Maxwell-Stefan system \eqref{M-S 1}-\eqref{M-S 2}-\eqref{M-S 3}-\eqref{M-S 4}, 
these relations hold for any $1\leq i\leq N$
\[ \partial_t c_i + \nabla_x \cdot (c_i u_i) = 0,
\]
\[ \partial_t (\sum_{i=1}^N c_i T) + \frac{5}{3} \nabla_x \cdot (\sum_{i=1}^N c_i u_i T) = 0,
\]
and
\[ \sum_{i=1}^N \nabla_x (c_i T) = 0,
\]
leading to an explicit expression as
\begin{align*}
\bm{\pi}_\mathbf{L} (\mathbf{S}^\varepsilon) = & - \frac{v}{\sum_{i=1}^N m_i \bar{c}_i} \cdot \left[ 
  \sum_{i=1}^N m_i \left ( \partial_t (c_i u_i) + \nabla_x \cdot (c_i u_i \otimes u_i)  \right ) \right] 
  (m_i \mu_i^{1/2})_{1 \leq i \leq N}\\
 & - \frac{\varepsilon}{\sqrt{6} \sum_{i=1}^N \bar{c}_i} \left[ \sum_{i=1}^N \left ( m_i \partial_t (c_i |u_i|^2) 
 + m_i \nabla_x \cdot (c_i |u_i|^2 u_i) \right ) \right] \left ( \frac{m_i |v|^2 - 3}{\sqrt{6}} \mu_i^{1/2} \right )_{1 \leq i \leq N}.
\end{align*}
Assuming the macroscopic quantities $\mathbf{c}, \mathbf{u}$ and $T$ in Theorem \ref{theorem for M-S in main result} are regular enough, we can find two constant $C_1, C_2 > 0$ such that 
\begin{align*}
\left| \sum_{i=1}^N ( \partial_t (c_i u_i) + \nabla_x \cdot (c_i u_i \otimes u_i)) \right|^2 \leq C_1 \delta_{MS}^2,\\
\left| \sum_{i=1}^N \left ( m_i \partial_t (c_i |u_i|^2) + 3 \nabla_x (c_i|u_i|^2 u_i) \right ) \right|^2 \leq C_2 \delta_{MS}^2,
\end{align*}
by repeating the previous discussions about the bound for the linear part of the source term. Based on these inequalities, there exists a constant $C_{\bm{\pi}_\mathbf{L}}>0$, such that
\[ \|\bm{\pi}_\mathbf{L} (\mathbf{S}^\varepsilon)\|_{L_{x,v}^2}^2 \leq C_{\bm{\pi}_\mathbf{L}} \delta_{MS}^2 
\|(1 + |v| + |v|^2) \bm{\mu}^{1/2}\|_{L_{x,v}^2}^2 \leq C_{\bm{\pi}_\mathbf{L}} \delta_{MS}^2.
\]
The above integral is clearly finite. 

The projection $\bm{\pi}_\mathbf{L}$ commutes with $x$-derivative since it only acts on velocity variable. 
Thus we can control the projection on the source term with pure $x$-derivatives, that is
\[ \|\bm{\pi}_\mathbf{L} ( \partial_x^\alpha \mathbf{S}^\varepsilon)\|_{L_{x,v}^2}^2 = 
\| \partial_x^\alpha \bm{\pi}_\mathbf{L} (\mathbf{S}^\varepsilon)\|_{L_{x,v}^2}^2 \leq C_{\alpha, 1}^\mathbf{S} \delta_{MS}^2,
\]
where $C_{\alpha, 1}^\mathbf{S}$ is a positive constant independent of $\varepsilon$. We conclude to get the following inequality by applying Young's inequality with a positive constant $\eta_4$,
\begin{equation}
\begin{aligned}
\langle \bm{\pi}_\mathbf{L} ( \partial_x^\alpha \mathbf{S}^\varepsilon), \bm{\pi}_\mathbf{L} (\partial_x^\alpha \mathbf{f}) 
\rangle_{L_{x,v}^2} & \leq \frac{1}{\eta_4} \| \bm{\pi}_\mathbf{L} (\partial_x^\alpha \mathbf{S}^\varepsilon)\|_{L_{x,v}^2}^2 
+ \eta_4 \|\bm{\pi}_\mathbf{L} (\partial_x^\alpha \mathbf{f})\|_{L_{x,v}^2}^2, \\
& \leq \frac{C_{\alpha,1}^\mathbf{S} \delta_{MS}^2}{\eta_4} + \eta_4 \|\bm{\pi}_\mathbf{L} 
(\partial_x^\alpha \mathbf{f})\|_{L_{x,v}^2}^2.
\label{x-derivative source term projec}
\end{aligned}
\end{equation}
For the second term of \eqref{x-alpha source decom}, we calculate it using the subtraction of two terms,
\[  \langle \partial_x^\alpha \mathbf{S}^{\varepsilon,\perp}, \partial_x^\alpha \mathbf{f}^\perp \rangle_{L_{x,v}^2} 
= \langle \partial_x^\alpha \mathbf{S}^\varepsilon, \partial_x^\alpha \mathbf{f}^\perp \rangle_{L_{x,v}^2} 
- \langle \bm{\pi}_\mathbf{L} (\partial_x^\alpha \mathbf{S}^\varepsilon), \partial_x^\alpha \mathbf{f}^\perp \rangle_{L_{x,v}^2} .
\]
Using a similar method as the estimate in velocity derivatives, the following inequality holds
\[  \langle \partial_x^\alpha \mathbf{S}^\varepsilon, \partial_x^\alpha \mathbf{f}^\perp \rangle_{L_{x,v}^2} \leq 
\frac{2 (C_L (\delta) + C_{NL} (\delta)) \delta_{MS}^2}{\eta_5} + \frac{\eta_5}{2\varepsilon^2} 
\| \partial_x^\alpha \mathbf{f}^\perp\|_{L_{x,v}^2 (\langle v \rangle^\gamma)}^2.
\]
Computations for pure spatial derivatives deduce that
\[ \langle \bm{\pi}_\mathbf{L} (\partial_x^\alpha \mathbf{S}^\varepsilon), \partial_x^\alpha \mathbf{f}^\perp \rangle_{L_{x,v}^2} 
\leq \frac{2 C_{\alpha,1}^\mathbf{S} \delta_{MS}^2}{\eta_5} + \frac{\eta_5}{2} \|\partial_x^\alpha \mathbf{f}^\perp \|_{L_{x,v}^2}^2.
\]
Since the $L_{x,v}^2 (\langle v \rangle^\gamma)$ norm controls the $L_{x,v}^2$ norm and the parameter $\varepsilon \in (0, 1]$, the inequality for 
$\partial_x^\alpha \mathbf{S}^{\varepsilon,\perp}$ finally takes the form
\begin{equation}
\langle \partial_x^\alpha \mathbf{S}^{\varepsilon,\perp}, \partial_x^\alpha \mathbf{f}^\perp \rangle_{L_{x,v}^2} \leq 
\frac{C_{\alpha, 2}^\mathbf{S} \delta_{MS}^2}{\eta_5} + \frac{\eta_5}{\varepsilon^2} 
\|\partial_x^\alpha \mathbf{f}^\perp\|_{L_{x,v}^2 (\langle v \rangle^\gamma)}^2. \label{x-derivative source term orth proj}
\end{equation} 
The positive constant has the form $C_{\alpha, 2}^\mathbf{S} = 2(C_L (\delta) + C_{NL} (\delta) + C_{\alpha, 1}^\mathbf{S})$, where 
$\delta \in (\frac{1}{2}, \frac{1}{1 + \delta_{MS}})$ is arbitrary. The result 
\eqref{bound for pure x-alpha deri of source term S} is recovered by adding the inequalities 
\eqref{x-derivative source term projec} and \eqref{x-derivative source term orth proj} together, with the constant $C_\alpha^\mathbf{S} = \eta_5 C_{\alpha,1}^\mathbf{S} + \eta_4 C_{\alpha, 2}^\mathbf{S}$.
\end{proof}
Furthermore, we present some properties involving the projection $\bm{\pi}_\mathbf{L}$.
\begin{lemma}\label{lemma for som properties of T-eps}
For any $\varepsilon \in (0,1]$, the operator $\mathbf{T}^\varepsilon = \frac{1}{\varepsilon^2} \mathbf{L} - \frac{1}{\varepsilon} 
v \cdot \nabla_x$, defined on the space $H^1 (\mathbb{T}^3 \times \mathbb{R}^3)$, satisfies the property 
$\mathrm{Ker} \mathbf{T}^\varepsilon = \mathrm{Ker} \mathbf{L} \cap \mathrm{Ker}(v \cdot \nabla_x)$. Moreover, the associated projection 
$\bm{\pi}_\mathbf{T}^\varepsilon$ has an explicit expression, for any $\mathbf{f} \in H^1 (\mathbb{T}^3 \times \mathbb{R}^3)$
\begin{equation}\label{def of pi T}
\bm{\pi}_\mathbf{T}^\varepsilon (\mathbf{f}) = \int_{\mathbb{T}^3} \bm{\pi}_\mathbf{L} (\mathbf{f}) \,\d x.
\end{equation}
\end{lemma}
\begin{lemma}\label{lemma for equivalence norms in ker L}
There exists a positive explicit constant $C_\pi$, such that for any $\mathbf{f} \in L^2 (\mathbb{R}^3)$,
\begin{equation}\label{equivalence of two norms in ker L}
\|\bm{\pi}_\mathbf{L} (\mathbf{f})\|_{L_v^2 (\langle v \rangle^\gamma )}^2 \leq C_\pi \|\bm{\pi}_\mathbf{L} (\mathbf{f})\|^2_{L_v^2},
\end{equation}
with the constant 
\[ C_\pi = (N + 4) \max_{1 \leq k, l \leq N+4} \left| \langle \bm{\phi}^{(k)}, \bm{\phi}^{(l)} \rangle_{L_v^2 (\langle v \rangle^\gamma)}
\right|.
\]
\end{lemma}
The latter Lemma states the equivalence between the norms $L_v^2$ and $L_v^2 (\langle v \rangle^\gamma)$ on 
the space $\mathrm{Ker} \mathbf{L}$. We omit the proofs for the two Lemmas, as they are the same as the results in \cite[Lemmas 3.9, 3.10]{briant2021stability}.

%%%%%%%%%%%%%%%%%%%%%%%%%%%%
We conclude our preparations with a Poincaré inequality for $\bm{\pi}_{\mathbf{L}} (\mathbf{f})$, stated as follows.
\begin{lemma}\label{lemma for Poincaré inequality of pi-L of perturbed f}
Let us denote $C_{\mathbb{T}^3}$ as the Poincaré constant on $\mathbb{T}^3$, and let $\delta_{MS} >0$. Consider a solution 
$\mathbf{f} \in H^1(\mathbb{T}^3 \times \mathbb{R}^3)$ of the perturbed Boltzmann equations \eqref{pert Bz}, with the initial data 
satisfying $\|\bm{\pi}_{\mathbf{T}}^\varepsilon (\mathbf{f}^{in})\|_{L_{x,v}^2} \leq C \delta_{MS}$ for some constant $C > 0$ 
independent of $\varepsilon$ and $\delta_{MS}$. There exists a positive constant $C^\mathbf{T}$ such that 
\begin{equation}\label{pi-L estimate for pert f}
\|\bm{\pi}_\mathbf{L} (\mathbf{f})\|_{L_{x,v}^2}^2 \leq 2 C_{\mathbb{T}^3} \|\nabla_x \mathbf{f} \|_{L_{x,v}^2}^2 
+ \delta_{MS}^2 C^\mathbf{T}.
\end{equation}
\end{lemma}
\begin{proof}
Since the projection $\bm{\pi}_{\mathbf{T}}^\varepsilon (\mathbf{f}) = \int_{\mathbb{T}^3} \bm{\pi}_{\mathbf{L}}(\mathbf{f})\,\d x$ holds, we first obtain the following inequality, by applying Poincaré inequality,
\[  \|\bm{\pi}_{\mathbf{L}} (\mathbf{f})\|_{L_{x,v}^2}^2 \leq 2 C_{\mathbb{T}^3}  
\|\nabla_x \bm{\pi}_{\mathbf{L}} (\mathbf{f})\|_{L_{x,v}^2}^2 + \frac{2}{|\mathbb{T}^3|^2} 
\|\bm{\pi}_{\mathbf{T}}^\varepsilon (\mathbf{f})\|_{L_{x,v}^2}^2.
\]
The function $\mathbf{f}$ satisfies the perturbed Boltzmann equations, thus 
we can apply projection $\bm{\pi}_{\mathbf{T}}^\varepsilon$ on \eqref{pert Bz}, leading to the equality that for almost any 
$\mathbb{R}_+ \times \mathbb{T}^3 \times \mathbb{R}^3$,
\begin{equation}
\partial_t \bm{\pi}_{\mathbf{T}}^\varepsilon (\mathbf{f}) = - \frac{1}{\varepsilon} \partial_t \bm{\pi}_{\mathbf{T}}^\varepsilon 
(\bm{\mu}^{-1/2} \mathbf{M}^\varepsilon).
\end{equation}
It uses the orthogonality of $\mathbf{L}^\varepsilon$ and $\mathbf{\Gamma}$ to $\mathrm{Ker} \mathbf{L}$. 
Integrating over $[0, t]$ with respect to time variable, the above equality transforms to
\begin{equation}
\bm{\pi}_{\mathbf{T}}^\varepsilon (\mathbf{f}) = \bm{\pi}_{\mathbf{T}}^\varepsilon (\mathbf{f}^{in}) - \frac{1}{\varepsilon} 
\bm{\pi}_{\mathbf{T}}^\varepsilon (\bm{\mu}^{-1/2}\mathbf{M}^\varepsilon - \bm{\mu}^{-1/2}\mathbf{M}^{\varepsilon, in}).
\end{equation}
Maxwellians $\mathbf{M}^\varepsilon$ and $\mathbf{M}^{\varepsilon, in}$ are respectively constructed with macroscopic quantities 
$(\mathbf{c},\varepsilon\mathbf{u},T\mathbf{1})$ and $(\mathbf{c}^{in}, \varepsilon \mathbf{u}^{in}, T^{in}\mathbf{1})$, which satisfying the Maxwell-Stefan system as stated in Theorem \ref{theorem for M-S in main result}. Therefore, we can compute 
the second term by using the 
expression of projection $\bm{\pi}_{\mathbf{L}}$ in \eqref{proje pi-L express},
\begin{align*}
   & \bm{\pi}_{\mathbf{T}}^\varepsilon  (\bm{\mu}^{-1/2} \mathbf{M}^\varepsilon - \bm{\mu}^{-1/2} \mathbf{M}^{\varepsilon, in})\\
   & = \int_{\mathbb{T}^3} \bm{\pi}_\mathbf{L} (\bm{\mu}^{-1/2} \mathbf{M}^\varepsilon - \bm{\mu}^{-1/2} 
   \mathbf{M}^{\varepsilon, in}) \,\d x\\
   &= \sum_{i=1}^N \frac{\varepsilon}{\bar{c}_i} \left[ \int_{\mathbb{T}^3} (\tilde{c}_i - \tilde{c}_i^{in}) \,\d x \right] \mu_i^{1/2} 
   \mathbf{e}^{(i)}\\
  & + \frac{\varepsilon v}{\sum_{i=1}^N m_i \bar{c}_i} \cdot \left[ \sum_{i=1}^N \int_{\mathbb{T}^3} m_i(c_i u_i - c_i^{in} u_i^{in}) \,\d x 
  \right]  (m_i \mu_i^{1/2})_{1 \leq i \leq N}\\
  & + \frac{\varepsilon}{\sqrt{6} \sum_{i=1}^N \bar{c}_i} \left[ \sum_{i=1}^N \int_{\mathbb{T}^3} 
  3(c_i \tilde{T} - c_i^{in} \tilde{T}^{in}) 
  + \varepsilon m_i (c_i |u_i|^2 - c_i^{in} |u_i^{in}|^2) \,\d x \right]  
  \left ( \frac{m_i |v|^2 - 3}{\sqrt{6}} \mu_i^{1/2} \right )_{1 \leq i \leq N}.
\end{align*}
The quantity $c_i$ satisfies the equation $\partial_t c_i + \nabla_x \cdot (c_i u_i) = 0$ for every $1 \leq i \leq N$, thus we have 
\[  \frac{\d}{\d t} \int_{\mathbb{T}^3} \tilde{c}_i \,\d x = 0, \quad \forall 1 \leq i \leq N, \quad \forall t \in \mathbb{R}_+.
\]
The explicit expression of $\bm{\pi}_\mathbf{T}^\varepsilon (\mathbf{f})$ finally takes the form
\begin{align*}
  \bm{\pi}_{\mathbf{T}}^\varepsilon (\mathbf{f}) = & \bm{\pi}_{\mathbf{T}}^\varepsilon (\mathbf{f}^{in}) - 
  \frac{v}{\sum_{i=1}^N m_i \bar{c}_i} \cdot \left[\sum_{i=1}^N \int_{\mathbb{T}^3} m_i (c_i u_i - c_i^{in} u_i^{in}) \,\d x \right] 
  (m_i \mu_i^{1/2})_{1 \leq i \leq N} - \frac{1}{\sqrt{6} \sum_{i=1}^N \bar{c}_i}\\
  & \times \left[ \sum_{i=1}^N \int_{\mathbb{T}^3} 3 (c_i \tilde{T} - c_i^{in} \tilde{T}^{in}) + \varepsilon m_i 
  (c_i |u_i|^2 - c_i^{in} |u_i^{in}|^2)  \,\d x \right]  \left ( \frac{m_i |v|^2 - 3}{\sqrt{6}} \mu_i^{1/2} \right )_{1 \leq i \leq N}.
\end{align*}

Under the assumption that macroscopic quantities are regular enough, there exists a uniform bound for the quantities 
$\mathbf{c}, \mathbf{u}$ and $ \tilde{T}$ given by Theorem \ref{theorem for M-S in main result}. By using the Sobolev embedding $H_x^2 \hookrightarrow L_x^\infty$, there exists a positive constant $C^\mathbf{T}$ such that
\begin{align*}
\left|\int_{\mathbb{T}^3} 3 (c_i \tilde{T} - c_i^{in} \tilde{T}^{in}) \,\d x \right| & \leq 3 |\mathbb{T}^3| (\max_{1 \leq i \leq N} \bar{c}_i) 
(\|\mathbf{c}\|_{H_x^2 (\bar{\mathbf{c}}^{-1})} \|\tilde{T}\|_{H_x^2} + \|\mathbf{c}^{in}\|_{H_x^2 (\bar{\mathbf{c}}^{-1})}  
\|\tilde{T}^{in}\|_{H_x^2}) \\
&\leq C^{\mathbf{T}} \delta_{MS} (1 + \varepsilon \delta_{MS}).
\end{align*}
For other terms involving the macroscopic quantities $\mathbf{c}, \mathbf{u}$, similar upper bounds can be established due to the perturbative form $u_i = \varepsilon \tilde{u}_i$. Thus we obtain the finial control for $\bm{\pi}_\mathbf{T}^\varepsilon (\mathbf{f})$:
\begin{align*}
\|\bm{\pi}_\mathbf{T}^\varepsilon (\mathbf{f}) \|_{L_{x,v}^2}^2 & \leq 2\|\bm{\pi}_\mathbf{T}^\varepsilon (\mathbf{f}^{in})\|_{L_{x,v}^2}^2
 + C^\mathbf{T} \delta_{MS}^2  \|(1 + |v| + |v|^2) \bm{\mu}^{1/2}\|_{L_v^2}^2\\
& \leq C^\mathbf{T} \delta_{MS}^2 ,
\end{align*}
appropriately increasing the constant $C^\mathbf{T}$ if necessary. Here the assumptions on the initial data $\mathbf{f}^{in}$ are applied. The norm involving $\bm{\mu}$ on the right hand side is clearly finite.

We conclude this proof with the following inequality, by using the unique orthogonal decomposition
$\nabla_x \mathbf{f} = \bm{\pi}_\mathbf{L} (\nabla_x \mathbf{f}) + \nabla_x \mathbf{f}^\perp$ in the space 
$L^2 (\mathbb{T}^3 \times \mathbb{R}^3)$ and the fact that $\bm{\pi}_\mathbf{L}$ commutes with spatial derivative,
\[ \|\nabla_x \bm{\pi}_\mathbf{L} (\mathbf{f})\|_{L_{x,v}}^2 \leq \|\nabla_x\mathbf{f}\|_{L_{x,v}}^2.
\]
\end{proof}
%%%%%%%%%%%%%%%%%%%%%%%%%%%%%%%%%%%%%%%%%%%
\subsection{Perturbative Cauchy Theory for the Multi-species Boltzmann Equations}
Let $s \in \mathbb{N}^\ast$ and $\mathbf{f} \in H^s(\mathbb{T}^3 \times \mathbb{R}^3)$ be a solution to the perturbed Boltzmann 
equations \eqref{pert Bz}, where the initial data satisfies the assumptions in Theorem \ref{theorem for perturbed f}. 
Recalling the expression of our target functional 
$\mathcal{H}_\varepsilon^s$ in \eqref{express of functional H-eps-s}, we first calculate the following estimates based on our preparations.

\textbf{Components of $\mathcal{H}_\varepsilon^1$.} 

The estimate for $L_{x,v}^2$ norm of $\mathbf{f}$:
\begin{align*}
  \frac{\d}{\d t} \|\mathbf{f}\|^2_{L_{x,v}^2} = & \frac{2}{\varepsilon^2} \langle \mathbf{L}^\varepsilon(\mathbf{f}), 
  \mathbf{f} \rangle_{L_{x,v}^2} + \frac{2}{\varepsilon} \langle \mathbf{\Gamma} (\mathbf{f},\mathbf{f}), \mathbf{f}^\perp \rangle_{L_{x,v}^2} 
  + 2 \langle \mathbf{S}^\varepsilon, \mathbf{f} \rangle_{L_{x,v}^2}\\
    \leq & - \frac{2 [\lambda_\mathbf{L} - \eta_5 - \eta_0 - (\varepsilon + \eta_1) \delta_{MS} C_{coe}^\mathbf{L}] }{\varepsilon^2}  
    \|\mathbf{f}^\perp\|_{L_{x,v}^2 (\langle v \rangle^\gamma)}^2\\
     & + 2 (\eta_4 + \frac{\delta_{MS} C_\pi C_{coe}^\mathbf{L}}{\eta_1})  \|\bm{\pi}_\mathbf{L} (\mathbf{f})\|_{L_{x,v}^2}^2 
     + \frac{2}{\eta_0} \mathcal{G}_x^0 (\mathbf{f}, \mathbf{f})^2 + \frac{2 \delta_{MS}^2 C_\alpha^\mathbf{S}}{\eta_4 \eta_5}.
\end{align*}
The fact $\langle v \cdot \nabla_x \mathbf{f}, \mathbf{f} \rangle_{L_{x,v}^2} = \sum_{i=1}^N \frac{1}{2} 
\int_{\mathbb{T}^3}\nabla_x \cdot ( \int_{\mathbb{R}^3} v f_i^2 \,\d v ) \,\d x = 0$, and the conservation laws for the Boltzmann collision operator $\mathbf{Q}$ are applied here. Moreover, we use Young's inequality with a positive constant $\frac{\eta_0}{\varepsilon}$ on the estimate for the operator $\mathbf{\Gamma}$, and employ the inequality \eqref{equivalence of two norms in ker L} in this deduction. 
By choosing the constants $\eta_1 = 1, \delta_{MS} \leq \frac{\lambda_\mathbf{L}}{8 C_{coe}^\mathbf{L}}, \eta_0 = \eta_5 = 
\frac{\lambda_\mathbf{L}}{8}$ and $\eta_4 = \delta_{MS}$, and using the Poincaré inequality \eqref{pi-L estimate for pert f}, 
we finally obtain 
\begin{equation}
\frac{\d}{\d t} \|\mathbf{f}\|^2_{L_{x,v}^2} \leq - \frac{\lambda_\mathbf{L}}{\varepsilon^2} 
\|\mathbf{f}^\perp\|_{L_{x,v}^2 (\langle v \rangle^\gamma)}^2 + \delta_{MS} C^{(1)}  \|\nabla_x \mathbf{f}\|_{L_{x,v}^2}^2 
+ \frac{16}{\lambda_\mathbf{L}} \mathcal{G}_x^0 (\mathbf{f}, \mathbf{f})^2 
+ \tilde{C} \delta_{MS}.  \label{estimate for H-1 of L2 for f}
\end{equation}

The $x$-derivative on the operator $\mathbf{\Gamma}$ does not change its action on $v$ variable, leading to
\[ \langle \nabla_x \mathbf{\Gamma} (\mathbf{f}, \mathbf{f}), \nabla_x \mathbf{f} \rangle_{L_{x,v}^2} 
= \langle \nabla_x \mathbf{\Gamma} (\mathbf{f}, \mathbf{f}), \nabla_x \mathbf{f}^\perp \rangle_{L_{x,v}^2} .
\]
Using the inequality \eqref{sub of L of alpha-x deri} and result in Lemma \ref{lemma for Poincaré inequality of pi-L of perturbed f}, the estimate for the $L_{x,v}^2$ norm of $\nabla_x \mathbf{f}$ can be written as
\begin{equation}
  \begin{aligned}
  \frac{\d}{\d t} \|\nabla_x \mathbf{f}\|^2_{L_{x,v}^2} \leq & - \frac{\lambda_\mathbf{L}}{\varepsilon^2}  
  \|\nabla_x\mathbf{f}^\perp\|_{L_{x,v}^2 (\langle v \rangle^\gamma)}^2 + \delta_{MS} C^{(2)} \|\nabla_x \mathbf{f}\|_{L_{x,v}^2}^2 
  + \delta_{MS} C^{(3)} \|\mathbf{f}^\perp\|_{L_{x,v}^2 (\langle v\rangle^\gamma)}^2   \\
   & + \frac{16}{\lambda_\mathbf{L}} \mathcal{G}_x^1 (\mathbf{f}, \mathbf{f})^2 + \tilde{C}_x \delta_{MS}. 
   \label{estimate for H-1 of L2 for x-deri f}
  \end{aligned}
\end{equation}

The estimate for $\nabla_v \mathbf{f}$ first uses the following equality
\[  - \frac{2}{\varepsilon} \langle \nabla_v (v \cdot \nabla_x \mathbf{f}), \nabla_v \mathbf{f} \rangle_{L_{x,v}^2} 
= - \frac{2}{\varepsilon} \langle \nabla_x \mathbf{f}, \nabla_v \mathbf{f} \rangle_{L_{x,v}^2}.
\]
Therefore, we can deduce
\begin{equation}
  \begin{aligned}
   \frac{\d}{\d t} \|\nabla_v \mathbf{f}\|^2_{L_{x,v}^2} \leq & - \frac{C_1^{\bm{\nu}}}{\varepsilon^2}  
   \|\nabla_v \mathbf{f}\|_{L_{x,v}^2 (\langle v \rangle^\gamma)}^2 + \frac{K_1}{\varepsilon^2}  
   \|\mathbf{f}^\perp\|_{L_{x,v}^2}^2 + \frac{C^{(4)}}{\varepsilon^2} \|\nabla_x \mathbf{f}\|_{L_{x,v}^2}^2\\
    & + \frac{16}{C_1^{\bm{\nu}}} \mathcal{G}_{x,v}^1 (\mathbf{f}, \mathbf{f})^2 + \tilde{C}_v \delta_{MS}.
    \label{estimate for H-1 of L2 for v- deri f}
  \end{aligned}
\end{equation}

The estimate for the commutator relies on this equality
\[ - \langle \nabla_x (v \cdot \nabla_x \mathbf{f}), \nabla_v \mathbf{f} \rangle{L_{x,v}^2} 
= - \langle \nabla_v (v \cdot \nabla_x \mathbf{f}), \nabla_x \mathbf{f}\rangle{L_{x,v}^2} = - \|\nabla_x \mathbf{f}\|_{L_{x,v}^2}^2.
\] 
Then the following inequality holds, for any $e > 0$,
\begin{equation}
\begin{aligned}
\frac{\d}{\d t} \langle \nabla_x \mathbf{f}, \nabla_v \mathbf{f} \rangle_{L_{x,v}^2} = & 2 
\langle \partial_t \nabla_x \mathbf{f},\nabla_v \mathbf{f} \rangle_{L_{x,v}^2}\\
& \leq \frac{2 e C_0^\mathbf{L}}{\varepsilon^3} \|\nabla_x \mathbf{f}^\perp\|_{L_{x,v}^2 (\langle v \rangle^\gamma)}^2 
+ \frac{C^{(5)}}{e\varepsilon} \|\nabla_v\mathbf{f}\|_{L_{x,v}^2 (\langle v \rangle^\gamma)}^2 + \frac{2e}{\varepsilon} 
\mathcal{G}_x^1 (\mathbf{f},\mathbf{f})^2\\
& - \frac{2 (1 - \delta_{MS} e C^{(6)})}{\varepsilon}  \|\nabla_x \mathbf{f}\|_{L_{x,v}^2}^2 + 
\frac{e \delta_{MS} K_{1, 1}}{\varepsilon}  \|\mathbf{f}^\perp\|_{L_{x,v}^2 (\langle v \rangle^\gamma)}^2 
+ \frac{\delta_{MS}^2 e \tilde{C}_{x, v}}{\varepsilon}.  \label{estimate for H-1 of commutator for x-v f}
\end{aligned}
\end{equation}
The first equality uses integration by parts, and the deduction uses a further decomposition 
\[ \langle \mathbf{L}^\varepsilon (\nabla_x \mathbf{f}), \nabla_v \mathbf{f} \rangle_{L_{x,v}^2} 
= \langle \mathbf{L}^\varepsilon (\nabla_x \mathbf{f}^\perp), \nabla_v \mathbf{f}\rangle_{L_{x,v}^2} 
+ \langle (\mathbf{L}^\varepsilon - \mathbf{L}) ( \bm{\pi}_\mathbf{L} (\nabla_x \mathbf{f})), \nabla_v \mathbf{f} \rangle_{L_{x,v}^2},
\]
the second part of which is of $\mathcal{O} (\varepsilon )$ and $\mathcal{O} ( \delta_{MS})$, seeing the proof in Lemma \ref{lemma for some estimate on L-eps}.

\textbf{Components of $\mathcal{H}_\varepsilon^s$.}

We further consider two multi-indices $\alpha,\beta$ with $|\alpha|+|\beta|=s$. The estimate for higher $x$-derivatives is 
\begin{equation}
\begin{aligned}
\frac{\d}{\d t} \| \partial_x^\alpha \mathbf{f}\|_{L_{x,v}^2}^2 \leq & - \frac{\lambda_\mathbf{L}}{\varepsilon^2}  
\|\partial_x^\alpha \mathbf{f}^\perp\|_{L_{x,v}^2 (\langle v \rangle^\gamma)}^2 + \frac{16}{\lambda_\mathbf{L}} 
\mathcal{G}_x^s (\mathbf{f}, \mathbf{f})^2 + \delta_{MS} K_\alpha \|\mathbf{f}\|^2_{H_{x,v}^{s-1} (\langle v\rangle^\gamma)} \\
   & + \delta_{MS} C^{(7)}  \|\partial_x^\alpha \mathbf{f}\|_{L_{x,v}^2}^2 + \tilde{C}_\alpha \delta_{MS}.
   \label{estimate for H-s of L2 for x-alpha f}
\end{aligned}
\end{equation}

Next, we estimate the norms of $\mathbf{f}$ that have at least one derivative in $v$, i.e. for $|\beta| \geq 1$, 
\begin{equation}
\begin{aligned}
\frac{\d}{\d t} \|\partial_v^\beta \partial_x^\alpha \mathbf{f}\|_{L_{x,v}^2}^2 \leq & - \frac{C_1^{\bm{\nu}}}{\varepsilon^2}  
\|\partial_v^\beta \partial_x^\alpha \mathbf{f}\|_{L_{x,v}^2 (\langle v \rangle^\gamma)}^2 + \frac{K_s}{\varepsilon^2}  
\|\mathbf{f}\|_{H_{x,v}^{s-1}}^2 + \frac{16}{C_1^{\bm{\nu}}} \mathcal{G}_{x,v}^s (\mathbf{f}, \mathbf{f})^2\\
  & + C^{(8)} \sum_{k, \beta_k > 0}  \|\partial_v^{\beta - e_k} \partial_x^{\alpha + e_k} \mathbf{f}\|_{L_{x,v}^2}^2 
  + \frac{2 C_5^{\bm{\nu}}}{\varepsilon}  \|\mathbf{f}\|_{H_{x,v}^{s-1} (\langle v \rangle^\gamma)}^2 
  + \tilde{C}_{\alpha, \beta} \delta_{MS},  \label{estimate for H-s of L2 for x-alpha v-beta f}
\end{aligned}
\end{equation}
where we assume $\varepsilon \leq \varepsilon_0 \leq \frac{C_1^{\bm{\nu}}}{8 C_2^{\bm{\nu}}}$. The following computations are employed
\begin{align*}
- \frac{2}{\varepsilon} & \langle \partial_v^\beta \partial_x^\alpha (v \cdot \nabla_x \mathbf{f}), 
\partial_v^\beta \partial_x^\alpha \mathbf{f} \rangle_{L_{x,v}^2}\\
& = - \frac{2}{\varepsilon} \langle \partial_v^{\beta - e_k}  \partial_x^{\alpha + e_k}\mathbf{f}, 
\partial_v^\beta \partial_x^\alpha \mathbf{f} \rangle_{L_{x,v}^2} - \frac{2}{\varepsilon} 
\langle \partial_v^{\beta - e_k} (v \cdot \nabla_x \partial_v^{e_k} \partial_x^\alpha \mathbf{f}), 
\partial_v^\beta \partial_x^\alpha \mathbf{f} \rangle_{L_{x,v}^2}\\
& = \dots = - \frac{2}{\varepsilon} \sum_{k, \beta_k > 0}  \langle \partial_v^{\beta - e_k}  \partial_x^{\alpha + e_k}\mathbf{f}, 
\partial_v^\beta \partial_x^\alpha \mathbf{f} \rangle_{L_{x,v}^2} - \frac{2}{\varepsilon} 
\langle v \cdot \nabla_x \partial_v^\beta \partial_x^\alpha \mathbf{f}, \partial_v^\beta \partial_x^\alpha \mathbf{f} \rangle_{L_{x,v}^2},
\end{align*}
in which the second term in the last equality is zero.

For the particular case with $\partial_v^{e_k} \partial_x^{\alpha - e_k}$ derivative, the above inequality in this case becomes 
\begin{equation}
\begin{aligned}
\frac{\d}{\d t} \|\partial_v^{e_k} \partial_x^{\alpha - e_k} \mathbf{f}\|_{L_{x,v}^2}^2 \leq & - \frac{C_1^{\bm{\nu}}}{\varepsilon^2}  
\|\partial_v^{e_k} \partial_x^{\alpha - e_k} \mathbf{f}\|_{L_{x,v}^2(\langle v\rangle^\gamma)}^2 + \frac{K_s}{\varepsilon^2} 
\|\mathbf{f}\|_{H_{x,v}^{s-1}}^2 + \frac{16}{C_1^{\bm{\nu}}} \mathcal{G}_{x,v}^s (\mathbf{f}, \mathbf{f})^2\\
  & + C^{(8)} \|\partial_x^{\alpha} \mathbf{f}\|_{L_{x,v}^2}^2 + \frac{2 C_5^{\bm{\nu}}}{\varepsilon} 
  \|\mathbf{f}\|_{H_{x,v}^{s-1} (\langle v \rangle^\gamma)}^2 + \tilde{C}_{\alpha, \beta} \delta_{MS}.
  \label{estimate for H-s of L2 for x-alpha v-e-k f}
\end{aligned}
\end{equation}

Finally, the upper bound for the commutator with higher derivatives reads, for any $e > 0$,
\begin{equation}
\begin{aligned}
\frac{\d}{\d t} \langle \partial_x^\alpha \mathbf{f}, \partial_v^{e_k} \partial_x^{\alpha - e_k} \mathbf{f} \rangle_{L_{x,v}^2} 
\leq & \frac{2 e C_0^\mathbf{L}}{\varepsilon^3} \| \partial_x^\alpha \mathbf{f}^\perp\|_{L_{x,v}^2(\langle v \rangle^\gamma)}^2 
+ \frac{C^{(9)}}{e\varepsilon}  \|\partial_v^{e_k} \partial_x^{\alpha - e_k} \mathbf{f}\|_{L_{x,v}^2 (\langle v \rangle^\gamma)}^2 
+ \frac{2 e}{\varepsilon} \mathcal{G}_{x}^s (\mathbf{f}, \mathbf{f})^2\\
& - \frac{2 (1 - \delta_{MS} e C^{(10)})}{\varepsilon}  \|\partial_x^\alpha \mathbf{f}\|_{L_{x,v}^2}^2 
+ \frac{e \delta_{MS} K_{\alpha, k}}{\varepsilon}  \|\mathbf{f}\|_{H_{x,v}^{s-1}}^2 + \frac{\delta_{MS} e \tilde{C}_{\alpha, k}}{\varepsilon}. 
\label{estimate for H-s of commutator for x-alpha v-e-k f}
\end{aligned}
\end{equation}

Starting from these inequalities, we can derive a priori energy estimate following steps in \cite{briant2015JDEnavierstokes}. 
We first present the results as follows.
\begin{proposition}\label{proposition preparation of priori estimate}
  For any $s \in \mathbb{N}^\ast$, there exist $\bar{\delta}_{MS} > 0$ and $\varepsilon_0 \in (0, 1]$, such that the following statements 
  hold.\\
  (i) There exist three sets of positive constants $ \{a_\alpha^{(s)}\},  \{b_{\alpha, k}^{(s)} \}$ and $\{ d_{\alpha, \beta}^{(s)}\}$,
 such that, for all $\varepsilon \in (0, \varepsilon_0]$, 
    \begin{equation}
     \begin{aligned}
     \|\cdot\|_{\mathcal{H}_\varepsilon^s} \sim \left ( \sum_{|\alpha| \leq s}  \|\partial_x^\alpha \cdot\|_{L_{x,v}^2}^2 
     + \varepsilon^2 \sum_{ \substack{|\alpha| + |\beta| \leq s \\ |\beta| \geq 1} }  
     \|\partial_v^\beta \partial_x^\alpha \cdot\|_{L_{x,v}^2}^2 \right )^{1/2}.
     \label{equivalencen of functional and standard norm}
     \end{aligned}
     \end{equation}
     (ii) There exist positive constants $K_0^{(s)}, K_1^{(s)}, K_2^{(s)}$ and $C^{(s)}$, such that for any 
     $\varepsilon \in(0, \varepsilon_0]$ and $\delta_{MS} \in [0, \bar{\delta}_{MS}]$, if 
     $\mathbf{f} \in H^s (\mathbb{T}^3 \times \mathbb{R}^3)$ is a solution for the perturbed Boltzmann equations \eqref{pert Bz} with 
     initial data satisfying $\mathbf{f}^{in} \in H^s(\mathbb{T}^3 \times \mathbb{R}^3)$ and 
     $\|\bm{\pi}_\mathbf{T}^\varepsilon (\mathbf{f}^{in})\|_{L_{x,v}^2} = \mathcal{O} (\delta_{MS})$, then for any time $t \geq 0$,
     \begin{equation}
      \begin{aligned}
      \frac{\d}{\d t} \|\mathbf{f}\|_{\mathcal{H}_\varepsilon^s}^2 \leq & - K_0^{(s)} \left ( 
        \|\mathbf{f}\|_{H_{x,v}^s (\langle v \rangle^\gamma)}^2 + \frac{1}{\varepsilon^2} \sum_{|\alpha| \leq s} 
        \|\partial_x^\alpha \mathbf{f}^\perp\|_{L_{x,v}^2 (\langle v \rangle^\gamma)}^2 \right )\\
      & + K_1^{(s)} \mathcal{G}_x^s (\mathbf{f}, \mathbf{f})^2 + \varepsilon^2 K_2^{(s)} \mathcal{G}_{x,v}^s (\mathbf{f}, \mathbf{f})^2 
      + C^{(s)}\delta_{MS},  \label{Prepared priori estimate}
      \end{aligned}
     \end{equation}
     where $\mathcal{G}_x^s$ and $\mathcal{G}_{x,v}^s$ are functionals introduced in Lemma \ref{lemma for estimates on Gamma}.
\end{proposition}
\begin{proof}
Details in Lemma \ref{lemma for estimates on source term S} require us to consider the solutions for the Maxwell-Stefan system $(\mathbf{c}, \mathbf{u}, T) \in H_x^{s + 5} \times H_x^{s + 4} \times H_x^{s + 5}$, with $\bar{\delta}_{MS}$ satisfying the smallness 
assumptions in Theorem \ref{theorem for M-S in main result}, inequality \eqref{delta-MS commonds in analysis of K} and $\bar{\delta}_{MS} \leq \frac{\lambda_\mathbf{L}}{8 C_{coe}^\mathbf{L}}$. Let $\mathbf{f}$ satisfies the statements in this Proposition, the above inequalities from \eqref{estimate for H-1 of L2 for f} up to \eqref{estimate for H-s of commutator for x-alpha v-e-k f} hold.

We first set $s = 1$, the functional $\mathcal{H}_\varepsilon^1$ reads, for any $\mathbf{f} \in H^1 (\mathbb{T}^3 \times \mathbb{R}^3)$, 
\[  \|\mathbf{f}\|_{\mathcal{H}_\varepsilon^1}^2 = A \|\mathbf{f}\|_{L_{x,v}^2}^2 + a \|\nabla_x \mathbf{f}\|_{L_{x,v}^2}^2 
+ \varepsilon b \langle \nabla_x \mathbf{f}, \nabla_v \mathbf{f} \rangle_{L_{x,v}^2} + \varepsilon^2 d \|\nabla_v\mathbf{f}\|_{L_{x,v}^2}^2.
\]
We consider the linear combination of inequalities $A \eqref{estimate for H-1 of L2 for f} 
+ a \eqref{estimate for H-1 of L2 for x-deri f} + \varepsilon b \eqref{estimate for H-1 of commutator for x-v f} 
+ \varepsilon^2 d \eqref{estimate for H-1 of L2 for v- deri f}$, and notice that $\mathcal{G}_x^0 \leq \mathcal{G}_x^1$, 
deducting that
\begin{align*}
\frac{\d}{\d t}  \|\mathbf{f}\|_{\mathcal{H}_\varepsilon^1}^2 \leq & \frac{1}{\varepsilon^2} [\varepsilon^2 \delta_{MS} (b e K_{1,1} + a C^{(3)}) 
+ \varepsilon^2 d K_1 - \frac{A \lambda_\mathbf{L}}{2} ] \|\mathbf{f}^\perp\|_{L_{x,v}^2 (\langle v \rangle^\gamma)}\\
& + \frac{1}{\varepsilon^2} [ 2 b e C_0^\mathbf{L} - \frac{a \lambda_\mathbf{L}}{2} ]  
\|\nabla_x \mathbf{f}^\perp\|_{L_{x,v}^2 (\langle v \rangle^\gamma)}^2 + [ \frac{b C^{(5)}}{e} - d C_1^{\bm{\nu}} ]  
\|\nabla_v \mathbf{f}\|_{L_{x,v}^2 (\langle v \rangle^\gamma)}^2\\
& + [ \delta_{MS} (A C^{(1)} + a C^{(2)}) + d C^{(4)} - 2 b (1 - e \delta_{MS} C^{(6)}) ]  \|\nabla_x \mathbf{f}\|_{L_{x,v}^2}^2\\
& - \frac{A \lambda_\mathbf{L}}{2 \varepsilon^2}  \|\mathbf{f}^\perp\|_{L_{x,v}^s (\langle v \rangle^\gamma)}^2 
- \frac{a \lambda_\mathbf{L}}{2 \varepsilon^2}  \|\nabla_x\mathbf{f}^\perp\|_{L_{x,v}^s (\langle v \rangle^\gamma)}^2\\
& + K_1^{(1)} \mathcal{G}_x^1 (\mathbf{f}, \mathbf{f})^2 + \varepsilon^2 K_2^{(1)} \mathcal{G}_{x,v}^1 (\mathbf{f}, \mathbf{f})^2 
+ C^{(1)} \delta_{MS},
\end{align*}
where $K_1^{(1)}, K_2^{(1)}$ and $C^{(1)}$ are positive constants that depend only on the choices of $A, a, b$ and $d$. 
We further choose 
\[  \bar{\delta}_{MS} \leq \min\{ \frac{d K_1}{b e K_{1, 1} + a C^{(3)}}, \frac{d C^{(4)}}{A C^{(1)} +a C^{(2)}}, \frac{1}{2 e C^{(6)}} \} ,
\]
while the constants $A, a, b, d, e$ satisfy the following requirements:
\begin{gather*}
\text{choose $d$ such that } -d C_1^{\bm{\nu}} < - 1, \quad A \text{ is big enough such that } 2 d K_1 
- \frac{A \lambda_\mathbf{L}}{2} \leq - 1,\\
b \text{ is big enough such that } 2 d C^{(4)} - b \leq - 1, \quad e \text{ is big enough such that } 
\frac{b C^{(5)}}{e} - d C_1^{\bm{\nu}} \leq - 1,\\
a \text{ is big enough such that } 2 b e C_0^\mathbf{L} - \frac{a \lambda_\mathbf{L}}{2} \leq - 1.
\end{gather*}
%  $-dC_1^{\bm{\nu}}<1$, then A is big enough such that $2dK_1-\frac{A\lambda_\mathbf{L}}{2}\leq 1$, 
% and b is big enough such that $ 2dC^{(4)}-b\leq1$, e is big enough such that $\frac{bC^{(5)}}{e}-dC_1^{\bm{\nu}} \leq 1$, a is big enough 
% such that $2beC_0^\mathbf{L}-\frac{a\lambda_\mathbf{L}}{2}\leq1$. 
The equivalence between the functional $\mathcal{H}_\varepsilon^1$ and the standard Sobolev norm $H_x^1L_v^2$ can be obtained by repeating discussions in \cite{briant2015JDEnavierstokes}.

Furthermore, due to the Poincaré inequality \eqref{pi-L estimate for pert f} and the inequality \eqref{equivalence of two norms in ker L}, there exist two positive constants $C_1^\prime$ and $C_2^\prime$, such that 
\[  \|\mathbf{f}\|_{L_{x,v}^2 (\langle v \rangle^\gamma)}^2 \leq C_1^\prime \left ( 
  \|\mathbf{f}^\perp\|_{L_{x,v}^2 (\langle v \rangle^\gamma)}^2 + \frac{1}{2} \|\nabla_x\mathbf{f}\|_{L_{x,v}^2}^2 
  + \delta_{MS}^2 \right )  
  \] 
\[  \|\nabla_x \mathbf{f}\|_{L_{x,v}^2 (\langle v \rangle^\gamma)}^2 \leq C_2^\prime \left ( 
  \|\nabla_x \mathbf{f}^\perp\|_{L_{x,v}^2 (\langle v \rangle^\gamma)}^2 + \frac{1}{2}  \|\nabla_x\mathbf{f}\|_{L_{x,v}^2}^2 \right ). 
  \] 
Based on the above discussions, we deduce the existence of a positive constant $K_0^{(1)}$, such that 
\begin{align*}
\frac{\d}{\d t} \|\mathbf{f}\|_{\mathcal{H}_\varepsilon^2}^2 \leq - K_0^{(1)} \|\mathbf{f}\|_{H_{x,v}^1 (\langle v \rangle^\gamma)}^2 
- \frac{\lambda_\mathbf{L} \min\{A,a\}}{2 \varepsilon^2} \left ( \|\mathbf{f}^\perp\|_{L_{{x,v}^2 (\langle v \rangle^\gamma)}}^2 
+ \|\nabla_x\mathbf{f}^\perp\|_{L_{x,v}^2 (\langle v \rangle^\gamma)}^2  \right )\\
+ K_1^{(1)} \mathcal{G}_x^1 (\mathbf{f}, \mathbf{f})^2 + \varepsilon^2 K_2^{(1)}\mathcal{G}_{x,v}^1 (\mathbf{f}, \mathbf{f})^2 + C^{(1)} \delta_{MS}.
\end{align*}
Lowering the value of $K_0^{(1)}$ to ensure $0 < K_0^{(1)} \leq \frac{\lambda_\mathbf{L} \min\{A,a\}}{2}$, we recover the inequality \eqref{Prepared priori estimate} in the case $s = 1$.

Next we assume that the results hold up to integer $s - 1$, and prove that they are true for $s$. 
The first statement of the equivalence between two norms is straightforward. We only present details of the second property, firstly define, for any $t \geq 0$,
\begin{align*}
F_s (t) = \varepsilon^2 B \sum_{ \substack{|\alpha| + |\beta| = s \\ |\beta| \geq 2} } 
\|\partial_v^\beta \partial_x^\alpha \mathbf{f}\|_{L_{x,v}^2}^2 + & \sum_{ \substack{|\alpha| = s \\ k, \alpha_k > 0} } 
\left[ a \|\partial_x^\alpha \mathbf{f}\|_{L_{x,v}^2}^2 \right.\\
& \left. + \varepsilon b \langle \partial_v^{e_k} \partial_x^{\alpha - e_k} \mathbf{f}, \partial_x^\alpha \mathbf{f} \rangle_{L_{x,v}^2} 
+ \varepsilon^2 d \|\partial_v^{e_k} \partial_x^{\alpha - e_k} \mathbf{f}\|_{L_{x,v}^2}^2 \right].
\end{align*}
The following equality holds
\[  \sum_{ \substack{|\alpha| + |\beta| = s \\ |\beta| \geq2} } \sum_{k, \beta_k > 0} 
\|\partial_v^{\beta - e_k} \partial_x^{\alpha + e_k} \mathbf{f}\|_{L_{x,v}^2}^2 = \sum_{ \substack{|\alpha| + |\beta| = s \\ |\beta| \geq2} }  
\| \partial_x^\beta \partial_x^\alpha \mathbf{f}\|_{L_{x,v}^2}^2 + \sum_{ \substack{|\alpha| = s \\ k, \alpha_k > 0} }  
\| \partial_v^{e_k} \partial_x^{\alpha - e_k} \mathbf{f}\|_{L_{x,v}^2}^2.
\]
Using the linear combination of inequalities \eqref{estimate for H-s of L2 for x-alpha f} up to 
\eqref{estimate for H-s of commutator for x-alpha v-e-k f}, we obtain
\begin{align*}
\frac{\d}{\d t} F_s (t) \leq & \sum_{ \substack{|\alpha| = s\\ k, \alpha_k > 0}} \left\{ \frac{1}{\varepsilon^2} \left ( 
  2 b e C_0^\mathbf{L} - \frac{a \lambda_\mathbf{L}}{2} \right )  
  \|\partial_x^\alpha \mathbf{f}^\perp\|_{L_{x,v}^2 (\langle v \rangle^\gamma)}^2 \right.\\
& + \left[ a C^{(7)} \delta_{MS} + \varepsilon^2 d C^{(8)} - 2 b (1- e C^{(10)} \delta_{MS}) \right] 
\|\partial_x^\alpha \mathbf{f}\|_{L_{x,v}^2}^2\\
& \left. + \left[ \frac{b C^{(9)}}{e} + \varepsilon^2 B C^{(8)} -d C_1^{\bm{\nu}} \right]  
\|\partial_v^{e_k} \partial_x^{\alpha - e_k} \mathbf{f}\|_{L_{x,v}^2 (\langle v \rangle^\gamma)}^2\right\}\\
& + \sum_{ \substack{|\alpha| + |\beta| = s \\ |\beta| \geq2} } (\varepsilon^2 B C^{(8)} - B C_1^{\bm{\nu}})  
\|\partial_v^\beta \partial_x^\alpha \mathbf{f}\|_{L_{x,v}^2 (\langle v \rangle^\gamma)}^2 - 
\sum_{ \substack{|\alpha| = s \\ k, \alpha_k > 0} } \frac{a \lambda_\mathbf{L}}{2 \varepsilon^2}  
\|\partial_x^\alpha \mathbf{f}^\perp\|_{L_{x,v}^2 (\langle v \rangle^\gamma)}\\
  & + \tilde{K}_0^s  \|\mathbf{f}\|_{H_{x,v}^{s-1} (\langle v \rangle^\gamma)}^2 + \tilde{K}_1^s \mathcal{G}_x^s (\mathbf{f}, \mathbf{f})^2 
  + \varepsilon^2 \tilde{K}_2 \mathcal{G}_{x,v}^s (\mathbf{f}, \mathbf{f})^2 + \tilde{C}^s \delta_{MS},
\end{align*}
where $\tilde{K}_l^s, l = 0, 1, 2$ and $\tilde{C}^s$ are positive constants independent of $\varepsilon$. We further require the bounds satisfy
\[ \bar{\delta}_{MS} \leq \min\{ \frac{d C^{(8)}}{a C^{(7)}} , \frac{1}{2 e C^{(10)}} \}, \quad
\varepsilon_0 \leq \min\{ \frac{b C^{(9)}}{e B C^{(8)}} , \frac{C_1^{\bm{\nu}}}{2C^{(8)}}
   \},
\]
and suitably choose the constants $a, b, d, e$ and $B$. Then, there exists a positive constant $K_s^\prime$, such that
\begin{align*}
\frac{\d}{\d t} F_s (t) \leq - K_s^\prime \sum_{ \substack{|\alpha| = s \\ k, \alpha_k > 0} } \left ( 
  \|\partial_x^\alpha \mathbf{f}^\perp\|_{L_{x,v}^2 (\langle v \rangle^\gamma)}^2 + \|\partial_x^\alpha \mathbf{f}\|_{L_{x,v}^2}^2 
  + \|\partial_v^{e_k} \partial_x^{\alpha - e_k} \mathbf{f}\|_{L_{x,v}^2 (\langle v \rangle^\gamma)}^2 \right )\\
 - C_1^{\bm{\nu}} \sum_{ \substack{|\alpha| + |\beta| = s \\ |\beta| \geq 2} } 
 \|\partial_v^\beta \partial_x^\alpha \mathbf{f}\|_{L_{x,v}^2 (\langle v \rangle^\gamma)}^2 - \frac{a \lambda_\mathbf{L}}{2 \varepsilon^2} 
 \sum_{ \substack{|\alpha| = s \\ k, \alpha_k > 0} } \|\partial_x^\alpha \mathbf{f}^\perp\|_{L_{x,v}^2 (\langle v \rangle^\gamma)}^2\\
 + \tilde{K}_0^s \|\mathbf{f}\|_{H_{x,v}^{s-1} (\langle v \rangle^\gamma)}^2 + \tilde{K}_1^s \mathcal{G}_x^s (\mathbf{f}, \mathbf{f})^2 
 + \varepsilon^2 \tilde{K}_2 \mathcal{G}_{x,v}^s (\mathbf{f}, \mathbf{f})^2 + C^{(s)} \delta_{MS}.
\end{align*}
Following discussions similar to those in the case $s = 1$, the inequality 
\[  \|\partial_x^\alpha \mathbf{f}\|_{L^2_{x,v} (\langle v\rangle^\gamma)}^2 \leq 
C_2^\prime ( \|\partial_x^\alpha \mathbf{f}^\perp\|_{L^2_{x,v} (\langle v \rangle^\gamma)}^2 + \frac{1}{2} 
\|\partial_x^\alpha \mathbf{f}\|_{L_{x,v}^2}^2 )
\]
holds for any $|\alpha| = s$ with $s \in \mathbb{N}^\ast$. Decreasing the value 
of $K_s^\prime$ if necessary, we can obtain
\begin{align*}
\frac{\d}{\d t} F_s (t) \leq - K_s^\prime \left ( \sum_{|\alpha| + |\beta| = s} 
\|\partial_v^\beta \partial_x^\alpha \mathbf{f}\|_{L_{x,v}^2 (\langle v\rangle^\gamma)}^2 + \frac{1}{\varepsilon^2} 
\sum_{|\alpha| = s} \|\partial_x^\alpha \mathbf{f}^\perp\|_{L_{x,v}^2 (\langle v \rangle^\gamma)}^2 \right )\\
 + \tilde{K}_0^s  \|\mathbf{f}\|_{H_{x,v}^{s-1} (\langle v \rangle^\gamma)}^2 + \tilde{K}_1^s \mathcal{G}_x^s (\mathbf{f}, \mathbf{f})^2 
 + \varepsilon^2 \tilde{K}_2 \mathcal{G}_{x,v}^s (\mathbf{f}, \mathbf{f})^2 + C^{(s)} \delta_{MS}.
\end{align*}
Since the second statement holds for $s - 1$, we choose a positive constant $D$ big enough to ensure 
$\tilde{K}_0^s < \frac{D K_0^{(s-1)}}{2}$, and set $\|\mathbf{f}\|_{\mathcal{H}_\varepsilon^s}^2 
= D \|\mathbf{f}\|_{\mathcal{H}_\varepsilon^{s-1}}^2 + F_s$. The functionals $\mathcal{G}_x^s$ and $\mathcal{G}_{x,v}^s$ increase 
with respect to the index $s$. Thus, it is not hard to find a constant $K_0^{(s)}$ (decreasing it if necessary), along with 
constants $K_1^{(s)}, K_2^{(s)}$ and $C^{(s)}$ (increasing them if necessary), 
to recover inequality \eqref{Prepared priori estimate} in the general case.
\end{proof}
Based on the above Proposition, we can establish a uniform priori estimate for the perturbed solution $\mathbf{f}$.
\begin{proposition}\label{proposition of priori estimate}
There exists $s_0\in\mathbb{N}^\ast$, for any integer $s \geq s_0$, there exist $\bar{\delta}_{MS} > 0$, $\varepsilon_0 \in (0, 1]$ and $\delta_B > 0$, 
such that for any $\varepsilon \in (0, \varepsilon_0]$ and $\delta_{MS} \in [0, \bar{\delta}_{MS}]$, if 
$\mathbf{f} \in H^s (\mathbb{T}^3 \times \mathbb{R}^3)$ solves the perturbed Boltzmann equations \eqref{pert Bz} 
with the initial data satisfying
\[  \|\mathbf{f}^{in}\|_{\mathcal{H}_\varepsilon^s} \leq \frac{\delta_B}{2}, \quad 
\|\bm{\pi}_\mathbf{T}^\varepsilon (\mathbf{f}^{in})\|_{L_{x,v}^2} = \mathcal{O} (\delta_{MS}),
\]
then $\|\mathbf{f}\|_{\mathcal{H}_\varepsilon^s} (t) \leq \delta_B$ for any $t\geq0$.
\end{proposition}
\begin{proof}
Since we use the same assumptions as those in Proposition \ref{proposition preparation of priori estimate}, the following inequality can be directly applied 
\[\frac{\d}{\d t} \|\mathbf{f}\|_{\mathcal{H}_\varepsilon^s}^2 \leq - K_0^{(s)} \|\mathbf{f}\|_{L_{x,v}^2 (\langle v \rangle^\gamma)}^2 
+ K_1^{(s)} \mathcal{G}_x^s (\mathbf{f}, \mathbf{f})^2 + \varepsilon^2 K_2^{(s)} 
\mathcal{G}_{x,v}^s (\mathbf{f}, \mathbf{f})^2 + C^{(s)} \delta_{MS}.
\]
The equivalence between the norm $\mathcal{H}_\varepsilon^s$ and the norm given in \eqref{equivalencen of functional and standard norm} 
implies that there exist two constants $C_{eq}$ and $C_{EQ}$, such that 
\begin{equation}
C_{eq} \left (  \sum_{|\alpha| \leq s}  \|\partial_x^\alpha \cdot\|_{L_{x,v}^2}^2 + \varepsilon^2  
     \sum_{ \substack{|\alpha| + |\beta| \leq s \\ |\beta| \geq 1} } \|\partial_v^\beta \partial_x^\alpha \cdot\|_{L_{x,v}^2}^2  
     \right )\leq \|\cdot\|_{\mathcal{H}_\varepsilon^s}^2 \leq C_{EQ} \|\cdot\|_{H_{x,v}^s}^2.
\end{equation}
The results in Lemma \ref{lemma for estimates on Gamma} infer that there exists $s_0 \in \mathbb{N}^\ast$, such that for any integer $s \geq s_0$, 
functionals $\mathcal{G}_x^s$ and $\mathcal{G}_{x,v}^s$ can be estimated as
\[  \mathcal{G}_x^s (\mathbf{f}, \mathbf{f})^2 \leq 4 (C_s^\mathbf{\Gamma})^2 \|\mathbf{f}\|_{H_x^s L_v^2}^2  
\|\mathbf{f}\|_{H_x^sL_v^s (\langle v \rangle^\gamma)}^2 \leq \frac{4 (C_s^\mathbf{\Gamma})^2}{C_{eq}}  
\|\mathbf{f}\|_{\mathcal{H}_\varepsilon^s}^2 \|\mathbf{f}\|_{H_x^sL_v^s (\langle v \rangle^\gamma)}^2,
\] 
\[  \mathcal{G}_{x,v}^s (\mathbf{f}, \mathbf{f})^2 \leq \frac{4 (C_s^\mathbf{\Gamma})^2}{\varepsilon^2 C_{eq}}  
\|\mathbf{f}\|_{\mathcal{H}_\varepsilon^s}^2  \|\mathbf{f}\|_{H_x^s L_v^s (\langle v \rangle^\gamma)}^2.
\]
Therefore, we obtain
\[  \frac{\d}{\d t} \|\mathbf{f}\|_{\mathcal{H}_\varepsilon^s}^2 \leq \left ( \frac{4(C_s^\mathbf{\Gamma})^2}{C_{eq}} 
(K_1^{(s)} + K_2^{(s)}) \|\mathbf{f}\|_{\mathcal{H}_\varepsilon^s}^2 - K_0^{(s)} \right ) 
\|\mathbf{f}\|_{H_{x,v}^s (\langle v \rangle^\gamma)}^2 + C^{(s)} \delta_{MS} .
\]
If we choose $\delta_B > 0$ satisfying
\begin{equation}
\frac{4(C_s^\mathbf{\Gamma})^2}{ C_{eq}} (K_1^{(s)} + K_2^{(s)}) \delta_B^2 \leq \frac{K_0^{(s)}}{2},
\end{equation}
then the following inequality holds as long as $\|\mathbf{f}\|_{\mathcal{H}_\varepsilon^s}\leq \delta_B$, 
\[  \frac{\d}{\d t} \|\mathbf{f}\|_{\mathcal{H}_\varepsilon^s}^2 \leq - \frac{K_0^{(s)}}{2} 
\|\mathbf{f}\|_{H_{x,v}^s (\langle v \rangle^\gamma)}^2 + C^{(s)} \delta_{MS} \leq - \frac{K_0^{(s)}}{2 C_{EQ}}  
\|\mathbf{f}\|_{\mathcal{H}_\varepsilon^s}^2 + C^{(s)} \delta_{MS}.
\]
% We define 
% \[T=\sup\{t>0:\|\mathbf{f}(s)\|_{\mathcal{H}_\varepsilon^s}\leq \delta_B \quad\forall s\in[0,t]\},\]
Denoting $\lambda_B = \frac{K_0^{(s)}}{2 C_{EQ}}$, and using Grönwall's inequality, we can deduce this upper bound
\[  \|\mathbf{f}\|_{\mathcal{H}_\varepsilon^s}^2 \leq \|\mathbf{f}^{in}\|_{\mathcal{H}_\varepsilon^s}^2 e^{- \lambda_B t} 
+ \frac{C^{(s)} \delta_{MS}}{\lambda_B} (1 - e^{- \lambda_B t}),
\]
holds as long as $\|\mathbf{f}\|_{\mathcal{H}_\varepsilon^s}^2 \leq \delta_B$. We introduce
\[  T^\ast = \sup \{t>0 : \|\mathbf{f}\|_{\mathcal{H}_\varepsilon^s} (s) \leq \delta_B, \quad \forall s \in [0, t]\}.
\]
If $T^\ast$ is finite, the above inequality infers that at $t=T^\ast$, 
\[  \|\mathbf{f}\|_{\mathcal{H}_\varepsilon^s} (T^\ast) \leq \max \left\{ \frac{\delta_B}{2}, \left ( 
  \frac{C^{(s)} \delta_{MS}}{\lambda_B} \right )^{1/2} \right\}.
\]
By imposing 
\begin{equation}
\bar{\delta}_{MS} \leq \frac{\delta_B^2 \lambda_B}{4 C^{(s)}},
\end{equation}
we deduce that $\|\mathbf{f}\|_{\mathcal{H}_\varepsilon^s} (T^\ast) \leq \frac{\delta_B}{2} < \delta_B$, contradicting the definition of $T^\ast$. 
Therefore, $T = + \infty$, and we finally prove that $\|\mathbf{f}\|_{\mathcal{H}_\varepsilon^2}(t) \leq \delta_B$ for any $t \geq 0$.
\end{proof}
\textbf{Existence and uniqueness of the perturbation $\mathbf{f}$}
\begin{proof}[Proof of Theorem \ref{theorem for perturbed f}]
  Based on the assumptions in Theorem \ref{theorem for perturbed f}, we can directly apply results in Propositions
  \ref{proposition preparation of priori estimate} and \ref{proposition of priori estimate}. 
  For the existence, we use a standard iterative method initially on the time interval $[0,T_0]$, and then extend the time interval to $\mathbb{R}_+$. 
  
Let $s \geq s_0$ and consider $(\mathbf{c}, \mathbf{u}, T)$ is the perturbation solution for the Maxwell-Stefan system in Theorem 
\ref{theorem for M-S in main result}. If we consider $(\mathbf{c}, \mathbf{u}, T) \in L^\infty (\mathbb{R}_+, H_x^{s + 6}) 
\times L^\infty(\mathbb{R}_+, H_x^{s + 5}) \times L^\infty(\mathbb{R}_+, H_x^{s + 6})$, then the source term $\mathbf{S}^\varepsilon$ constructed with Maxwellian $\mathbf{M}^\varepsilon$, belongs to $C^0 (\mathbb{R}_+, H_{x,v}^s)$.

We initially set 
\begin{equation}
\mathbf{f}^{(0)} (t,x,v) 
= \begin{cases}
\mathbf{f}^{in}, \quad t = 0,\\ 0, \quad t > 0;
\end{cases}
\quad T_0 = \frac{\delta_B \min \left\{1, \frac{K_0^{(s)}}{2} \right\}}{4 C^{(s)} \delta_{MS}}.
\end{equation} 
Following steps in \cite[Section 6]{briant2015JDEnavierstokes}, we introduce a functional defined on $H_{x,v}^s$ on $[0, T_0]$,
\[  E_{[0, T_0]} (\mathbf{f}) = \sup_{t \in [0, T_0]} \left ( \|\mathbf{f}\|_{\mathcal{H}_\varepsilon^s}^2 (t) 
+ \int_0^t \|\mathbf{f}\|_{H_{x,v}^s (\langle v \rangle^\gamma)}^2 (\tau) \,d\tau \right ),
\]
omitting the subscript $[0, T_0]$ when considering $t \in[0, T_0]$ for simplicity. Suppose that on the time interval $[0, T_0]$ 
we have constructed 
functions $f^{(n)}$ ($n \in \mathbb{N}$), satisfying for any $t \in[0, T_0]$,
\[  \mathbf{f}^{(n)} \in H_{x,v}^s, \quad \|\bm{\pi}_\mathbf{T}^\varepsilon(\mathbf{f}^{(n)})\|_{L_{x,v}^s} 
= \mathcal{O} (\delta_{MS}), \quad E(\mathbf{f}^{(n)}) \leq \delta_B.
\]
By induction, we defined the function $\mathbf{f}^{(n+1)}$ such that 
\begin{equation} \left\{
  \begin{aligned} 
  & \partial_t \mathbf{f}^{(n+1)} + \frac{1}{\varepsilon} v \cdot \nabla_x \mathbf{f}^{(n+1)} = \frac{1}{\varepsilon^2} 
  \mathbf{L}^\varepsilon (\mathbf{f}^{(n+1)}) + \frac{1}{\varepsilon} \mathbf{\Gamma}(\mathbf{f}^{(n)}, \mathbf{f}^{(n+1)}) 
  + \mathbf{S}^\varepsilon,\\
& \mathbf{f}^{(n+1)}|_{t = 0} = \mathbf{f}^{in} . \label{equation for n+1 perturbed f}
  \end{aligned}   \right.
\end{equation}
The existence of the solution $\mathbf{f}^{(n+1)} \in H^s(\mathbb{T}^3 \times \mathbb{R}^3)$ satisfying 
$\|\bm{\pi}_\mathbf{T}^\varepsilon(\mathbf{f}^{(n+1)})\|_{L_{x,v}^2} \leq C^\mathbf{T}\delta_{MS}$ for any $t \in [0, T_0]$ is a 
classical result, where we rewrite the equation as 
\[ \partial_t \mathbf{f}^{(n+1)} = \mathbf{T}^\varepsilon (\mathbf{f}^{(n+1)}) + \frac{1}{\varepsilon^2} 
(\mathbf{L}^\varepsilon - \mathbf{L}) (\mathbf{f}^{(n+1)}) + \frac{1}{\varepsilon} \mathbf{\Gamma} (\mathbf{f}^{(n)}, \mathbf{f}^{(n+1)}) 
+ \mathbf{S}^\varepsilon.
\]
The operator $\mathbf{T}^\varepsilon$ generates a strongly continuous semigroup in $H^s (\mathbb{T}^3 \times \mathbb{R}^3)$ 
(general theory can be found in \cite{Kato1995bookperturbation} and details for multi-species Boltzmann equation is in Section 4.3 of \cite{briant2016ARMAglobal}). Operators $\mathbf{L}^\varepsilon - \mathbf{L}$ and 
$\mathbf{\Gamma} (\mathbf{f}^{(n)}, \cdot)$ are bounded linear operators from $(H_{x,v}^s; E(\cdot))$ to $(H_{x,v}^s; \|\cdot\|_{H_{x,v}^s})$, and therefore the sequence $(\mathbf{f}^{(n)})_{n \in \mathbb{N}}$ remains well-defined. 
Then we prove that $(E(\mathbf{f}^{(n)}))_{n\in\mathbb{N}}$ has a uniform bound. 
Applying the inequality in Proposition \ref{proposition preparation of priori estimate}, and using the estimates for $\mathcal{G}_x^s$ and 
$\mathcal{G}_{x,v}^s$ in Lemma \ref{lemma for estimates on Gamma}, we can deduce 
\begin{align*}
\frac{\d}{\d t} \|\mathbf{f}^{(n+1)}\|_{\mathcal{H}_\varepsilon^s}^2 \leq & ( \lambda_K \|\mathbf{f}^{(n)}\|_{\mathcal{H}_\varepsilon^s}^2 
- K_0^{(s)} ) \|\mathbf{f}^{(n+1)}\|_{H_{x,v}^2 (\langle v \rangle^\gamma)}^2 + \lambda_K 
\|\mathbf{f}^{(n+1)}\|_{\mathcal{H}_\varepsilon^s}^2  \|\mathbf{f}^{(n)}\|_{H_{x,v}^2 (\langle v \rangle^\gamma)}^2 
+ C^{(s)} \delta_{MS}\\
  \leq & (\lambda_K E(\mathbf{f}^{(n)}) - K_0^{(s)})  \|\mathbf{f}^{(n+1)}\|_{H_{x,v}^2 (\langle v \rangle^\gamma)}^2 
  + \lambda_K E(\mathbf{f}^{(n+1)}) \|\mathbf{f}^{(n)}\|_{H_{x,v}^2 (\langle v \rangle^\gamma)}^2 + C^{(s)} \delta_{MS},
\end{align*}
with the constant defined as $\lambda_K = \frac{2(C_s^{\mathbf{\Gamma}})^2}{C_{eq}} (K_1^{(s)} + K_2^{(s)})$. If we suppose that, 
\begin{equation}
\delta_B \leq \min \left\{1, \frac{K_0^{(s)}}{2} \right\}/ 2 \lambda_K, \quad \delta_B \leq \min \left\{1, \frac{K_0^{(s)}}{2} \right\}/4,
\end{equation}
then $\lambda_K E(\mathbf{f}^{(n)}) \leq \frac{K_0^{(s)}}{2}$, leading to
\[  \frac{\d}{\d t} \|\mathbf{f}^{(n+1)}\|_{\mathcal{H}_\varepsilon^s}^2 + \frac{K_0^{(s)}}{2} 
\|\mathbf{f}^{(n+1)}\|_{H_{x,v}^2 (\langle v \rangle^\gamma)}^2 \leq \lambda_K E(\mathbf{f}^{(n+1)})  
\|\mathbf{f}^{(n)}\|_{H_{x,v}^2 (\langle v \rangle^\gamma)}^2 + C^{(s)} \delta_{MS}.
\]
Integrating this inequality on $[0, t]$, with $t \leq T_0$, we can obtain
\[  E(\mathbf{f}^{(n+1)}) \leq \frac{2}{\min \left\{1, \frac{K_0^{(s)}}{2} \right\}}  \|\mathbf{f}^{in}\|_{\mathcal{H}_\varepsilon^s}^2 
+ \frac{2 C^{(s)} \delta_{MS}}{\min\left\{ 1, \frac{K_0^{(s)}}{2} \right\}} T_0 \leq \delta_B.
\]
Therefore, $(E(\mathbf{f}^{(n)}))_{n \in \mathbb{N}}$ is uniformly bounded by $\delta_B$, and $(\mathbf{f}^{(n)})_{n \in \mathbb{N}}$ 
is uniformly bounded in the norm  $L^\infty (0, T_0; H_{x,v}^s) \cap L^2 (0, T_0; H_{x,v}^s (\langle v \rangle^\gamma))$. By applying 
compact embeddings into less regular Sobolev spaces (Rellich Theorem), we can take the limit as $n\rightarrow +\infty$ in 
\eqref{equation for n+1 perturbed f}, since $\mathbf{T}^\varepsilon, \mathbf{L}^\varepsilon - \mathbf{L}$ and $\mathbf{\Gamma}$ are 
continuous. In particular, we can extract a subsequence that converges to a function 
$\mathbf{f}^{(\infty)} \in C^0 ([0, T_0]; H^s (\mathbb{T}^3 \times \mathbb{R}^3))$, which solves the problem
  \begin{equation*}
 \left\{
  \begin{aligned} 
  & \partial_t \mathbf{f}^{(\infty)} + \frac{1}{\varepsilon} v \cdot \nabla_x \mathbf{f}^{(\infty)} = \frac{1}{\varepsilon^2} 
  \mathbf{L}^\varepsilon(\mathbf{f}^{(\infty)}) + \frac{1}{\varepsilon} \mathbf{\Gamma} (\mathbf{f}^{(\infty)}, \mathbf{f}^{(\infty)}) 
  + \mathbf{S}^\varepsilon,\\
& \mathbf{f}^{(\infty)}|_{t = 0} = \mathbf{f}^{in} .
  \end{aligned}   \right.
  \end{equation*}
Lemma \ref{lemma for Poincaré inequality of pi-L of perturbed f} infers that 
$\|\bm{\pi}_{\mathbf{T}}^\varepsilon (\mathbf{f}^{(\infty)})\|_{L_{x,v}^2} (t) \leq C^\mathbf{T} \delta_{MS}$ for any $t \in [0, T_0]$, 
and the proof in Proposition \ref{proposition of priori estimate} deduces the inequality 
$\|\mathbf{f}^{(\infty)}\|_{\mathcal{H}_\varepsilon^s} (T_0) \leq \frac{\delta_B}{2} < \delta_B$. The extension of the time interval to 
$\mathbb{R}_+$ is analogous to the discussions in the Maxwell-Stefan system. Therefore, we recover the existence of 
$\mathbf{f} \in C^0 (\mathbb{R}_+; H_{x,v}^s)$ with a uniform control 
$ \|\mathbf{f}\|_{\mathcal{H}_\varepsilon^s} (t) \leq \delta_B $ 
for any $t \geq 0$.

% Thus by simply restarting above process on 
% internal $[T_0,2T_0]$ with $\mathbf{f}^{(\infty)}(T_0)$ as an initial data and $E_{[T_0,2T_0]}$ as the functional, we can repeat the 
% computations and finally recover the existence of $\mathbf{f}\in C^0(\mathbb{R}_+;H_{x,v}^s)$ with a uniform control that 
% $\|\mathbf{f}\|_{\mathcal{H}_\varepsilon^s}\leq \delta_B$ for any $t\geq0$.\\
The uniqueness starts from the assumption that $\mathbf{f}, \mathbf{g}$ are two solutions to the equations \eqref{pert Bz} with the same initial data 
$\mathbf{f}^{in}$. We first discuss the case on $t \in [0, T_0]$, setting $\mathbf{h} = \mathbf{f} - \mathbf{g}$. 
Subtracting the equations for $\mathbf{f}$ and $\mathbf{g}$, we obtain that $\mathbf{h}$ is a solution on time interval $[0,T_0]$ of
\begin{equation*}\left\{
 \begin{aligned} 
  & \partial_t \mathbf{h} + \frac{1}{\varepsilon} v \cdot \nabla_x \mathbf{h} = \frac{1}{\varepsilon^2} \mathbf{L}^\varepsilon (\mathbf{h}) 
  + \frac{1}{\varepsilon} \left ( \mathbf{\Gamma} (\mathbf{h}, \mathbf{f}) + \mathbf{\Gamma} (\mathbf{h}, \mathbf{f}) \right ) , \\
&\mathbf{h}|_{t = 0} = 0 .
  \end{aligned}   \right.
  \end{equation*}
Applying the computations in the previous analysis, we can deduce the following inequality
\begin{align*}
\frac{\d}{\d t} \|\mathbf{h}\|_{\mathcal{H}_\varepsilon^s}^2 \leq \left ( \frac{\lambda_K}{2} ( E(\mathbf{f}) + E(\mathbf{g}) ) - K_0^{(s)} 
\right ) \|\mathbf{h}\|_{H_{x,v}^s (\langle v \rangle^\gamma)}^2 + \frac{\lambda_K}{2} 
\left (\|\mathbf{f}\|_{H_{x,v}^s (\langle v \rangle^\gamma)}^2 + \|\mathbf{g}\|_{H_{x,v}^s (\langle v \rangle^\gamma)}^2 \right ) E (\mathbf{h}).
\end{align*}
The previous results infer that $E(\mathbf{f}), E(\mathbf{g}) \leq \delta_B$ with $\delta_B \leq \frac{K_0^{(s)}}{2 \lambda_K}$. From this, 
we can further derive an inequality, and integrate it over $[0, t]$ with $t \leq T_0$, leading to
\[  E(\mathbf{h}) \leq \frac{\lambda_K \delta_B}{\min\left\{ 1, \frac{K_0^{(s)}}{2} \right\}} E (\mathbf{h}).
\]
Lowering the value of the positive constant $\delta_B$ if necessary, we can conclude $E(\mathbf{h}) = 0$ by contraction, 
which means $\mathbf{f} = \mathbf{g}$ for $t\in[0,T_0]$. 
Resetting $\mathbf{f} (T_0)$ as the initial data and repeating the above analysis along with the functional $E_{[T_0, 2 T_0]}$, 
we can recover $\mathbf{f} = \mathbf{g}$ on $[T_0, 2 T_0]$. 
By iteration, we finally deduce $\mathbf{f} = \mathbf{g}$ for any $t \geq 0$.
\end{proof}

\section*{Acknowledgments}
This work was supported by National Key R\&D Program of China under the grant 2023YFA1010300. The author N. Jiang is supported by the grants from the National Natural Science Foundation of China under contract No. 11971360 and No.11731008, and also supported by the Strategic Priority Research Program of Chinese Academy of Sciences under grant No. XDA25010404. The author Y.-L. Luo is supported by grants from the National Natural Science Foundation of China under contract No. 12201220, the Guang Dong Basic and Applied Basic Research Foundation under contract No. 2024A1515012358, and the Fundamental Research Funds for the Central Universities of Hunan University under contract No. 531118011008.

\vspace*{1cm}

{\footnotesize
%%%%%%%%%%%%%%%%%%%%%%%%%%%%%%%%%%%
\bibliographystyle{aomvar}

\bibliography{ref} }
%%%%%%%%%%%%%%%%%%%%%%%%%%%
\end{document}